\documentclass[12pt]{amsart}
\usepackage{fullpage}
\usepackage[active]{srcltx}
\usepackage{amssymb,amsthm,amsmath,amscd}

\usepackage{mathtools}

\RequirePackage[dvipsnames,usenames]{xcolor}

\usepackage[colorlinks, linkcolor=Mahogany, citecolor=ForestGreen]{hyperref}

\input{kmacros3.sty}

\usepackage{tikz}
\usepackage{tikz-cd}
\usetikzlibrary{cd}
\usepackage{graphicx}
\usepackage[all,cmtip]{xy}

\usepackage{mabliautoref}
\numberwithin{equation}{theorem}

\usepackage{pifont}
\usepackage{upgreek}
\newcommand{\mytau}{{\uptau}}
\usepackage{eucal}
\usepackage{ulem}
\usepackage{stmaryrd}

\numberwithin{equation}{theorem}

\usepackage{enumerate}
\usepackage{calc}

\usepackage{verbatim}
\usepackage{alltt}

\def\todo#1{\textcolor{Mahogany}%
{\footnotesize\newline{\color{Mahogany}\fbox{\parbox{\textwidth-15pt}{\textbf{todo: } #1}}}\newline}}

\def\commentbox#1{\textcolor{Mahogany}%
{\footnotesize\newline{\color{Mahogany}\fbox{\parbox{\textwidth-15pt}{\textbf{comment: } #1}}}\newline}}

\renewcommand{\O}{\mathcal O}

\renewcommand{\m}{\mathfrak{m}}

\renewcommand{\fram}{\mathfrak{m}}

\newcommand{\perfd}{\textnormal{perfd}}

\newcommand{\utau}{\tau}
\renewcommand{\bZ}{{\bf Z}}
\renewcommand{\bN}{{\bf N}}
\renewcommand{\bC}{{\bf C}}

\begin{document}

\title[Test ideals in mixed characteristic]{Test ideals in mixed characteristic: \\a unified theory up to perturbation}

\author[Bhatt, Ma, Patakfalvi, Schwede, Tucker, Waldron, Witaszek; appendix by Datta]{Bhargav Bhatt, Linquan Ma, Zsolt Patakfalvi, Karl Schwede, Kevin Tucker, Joe Waldron, Jakub Witaszek, and an appendix by Rankeya Datta}

\dedicatory{\color{black}Dedicated to Karen E. Smith on the occasion of her 60th birthday.}

\address{Department of Mathematics, IAS/Princeton and University of Michigan, Princeton and Ann Arbor, NJ and MI, USA}
\email{bhargav.bhatt@gmail.com}
\address{Department of Mathematics, Purdue University, West Lafayette, IN, USA}
\email{ma326@purdue.edu}
\address{\'Ecole Polytechnique F\'ed\'erale de Lausanne (EPFL), Lausanne, Switzerland}
\email{zsolt.patakfalvi@epfl.ch}
\address{Department of Mathematics, University of Utah, Salt Lake City, UT, USA}
\email{schwede@math.utah.edu}
\address{Department of Mathematics, University of Illinois at Chicago, Chicago, IL, USA}
\email{kftucker@uic.edu}
\address{Department of Mathematics, Michigan State University, East Lansing, MI, USA}
\email{waldro51@msu.edu}
\address{Department of Mathematics, Fine Hall, Washington Road, Princeton, NJ, USA}
\email{jwitaszek@princeton.edu}
\address{Department of Mathematics, University of Missouri, Columbia, MO, USA}
\email{rankeya.datta@missouri.edu}
\begin{abstract}
Let $X$ be an integral scheme of finite type over a complete DVR of mixed characteristic.  We provide a definition of a test ideal which agrees with the multiplier ideal after inverting $p$, is computed from a sufficiently large alteration, agrees with previous mixed characteristic BCM test ideals after completing at any point of residue characteristic $p$ (up to small perturbation), and which satisfies the full suite of expected properties of a multiplier or test ideal. This object is obtained via the $p$-adic Riemann-Hilbert functor. 
\end{abstract}

\maketitle

\setcounter{tocdepth}{1}
\tableofcontents

\section{Introduction}


\subsection{Background}
In characteristic zero algebraic geometry, it is enormously useful to study the singularities of algebraic varieties by considering resolutions of singularities. In characteristic $p > 0$ however, even when a resolution is known to exist, looking solely at birational maps is insufficient to deduce properties of singularities, due to a lack of vanishing theorems. Instead, it turns out to be more fruitful to study properties of the Frobenius morphism, or more generally, of all finite covers or even all alterations.
To wit, we call a commutative Noetherian ring $R$ a \emph{splinter} if every finite surjective map $\Spec S \to \Spec R$, the induced map $R \to S$ splits as a map of $R$-modules; this is an important class of singularities in characteristic $p$ algebraic geometry\footnote{Indeed, in characteristic $p > 0$, $R$ being a splinter is conjecturally the same as $R$ being \emph{strongly $F$-regular}, and this equivalence is known when $R$ is $\bQ$-Gorenstein by \cite{SinghQGorensteinSplinters} \cf \cite{BlickleSchwedeTuckerTestAlterations,ChiecchioEnescuMillerSchwede}.  Strongly $F$-regular singularities are arguably the most important class of singularities coming out of tight closure theory \cite{HochsterHunekeTC1}. 
 They are also known to correspond to KLT singularities in characteristic zero (\cite{HaraWatanabeFRegFPure,HaraRatImpliesFRat,MehtaSrinivasRatImpliesFRat,SmithMultiplierTestIdeals,HaraYoshidaGeneralizationOfTightClosure,TakagiInterpretationOfMultiplierIdeals}) and are thus important in birational geometry.}. 

While the definition of a splinter makes perfect sense in mixed characteristic, relatively little was known about this notion until quite recently. The breakthrough was Andr\'e's proof of Hochster's direct summand conjecture (\cite{AndreDirectsummandconjecture}), showing that regular rings are splinters. Based on this and the related works \cite{HochsterHunekeInfiniteIntegralExtensionsAndBigCM,AndreWeaklyFunctorialBigCM,BhattDirectsummandandDerivedvariant,GabberMSRINotes,HeitmannMaBigCohenMacaulayAlgebraVanishingofTor}, the second and the fourth author helped develop the theory of BCM singularities (see \cite{HochsterSolidClosure,SchoutensCanonicalBCM,MaSchwedeSingularitiesMixedCharBCM,SatoTakagiArithmeticAndGeometricDeformationsOfFPure,PerezRGTestIdeals21,JiangClosureOpsMixedChar,DattaTuckerOpenness,CaiLeeMaSchwedeTuckerPerfectoidSignature1,MurayamaUniformBoundsOnSymbolicPowers,YamaguchiBCMTestIdealsInEqualChar,NakazatoShimomotoVariantPerfectoidAbhyankar}, \cf \cite{BrennerRescueSolidClosure}) which was inspired by the theory of tight closure and multiplier ideals in characteristic $p > 0$ and 0 respectively, see \cite{HochsterHunekeTC1,HaraYoshidaGeneralizationOfTightClosure,TakagiInterpretationOfMultiplierIdeals,EsnaultViehwegSurUneMinoration,NadelMultiplierIdealSheaves,LipmanAdjointsOfIdealsInRegularLocal,LazarsfeldPositivity2}.  This, in conjunction with the result of the first author on the Cohen-Macaulayness of the absolute integral closure $R^+$ (\cite{BhattAbsoluteIntegralClosure}), allowed for a thorough study of splinters.  

In this paper, we are interested in measuring the failure of a ring $R$ to be a splinter. To quantify this, note that a finite extension  $R \subseteq S$ splits if and only if the evaluation-at-1 map $\Hom_R(S,R) \to R$ surjects. Tautologically then, the  ideal 
\begin{equation}
\label{eq:intro-splinter}
    \bigcap_{\textrm{ finite } R \subseteq S} \Image \Big( \Hom_R(S, R) \xrightarrow{\text{evaluation at}\,  1} R \Big) = \bigcap_{\textrm{ finite } R \subseteq S} \Tr_{S/R}(\omega_{S/R}) \subset R
\end{equation}
equals $R$ exactly when $R$ is a splinter; in general, we regard this ideal as a measure of the failure of $R$ to be a splinter.
A priori, the infinite intersection appearing above is quite difficult to control. Nonetheless, one of the goals of this paper is to prove some finiteness properties of this infinite intersection, such as those discussed next, by accessing it topologically via the $p$-adic Riemann--Hilbert functor.

\subsection{The motivating question}

To formulate statements and obtain objects that work uniformly in all characteristics, it is convenient to replace finite maps with alterations in the index set of  the intersection in \eqref{eq:intro-splinter} above\footnote{It is known that replacing finite covers with alterations in \eqref{eq:intro-splinter} does not change the intersection in geometric situations in positive or mixed characteristic, see \cite[Corollary 4.13]{BMPSTWW1}, but it is essential to include birational maps in characteristic zero.}. With this replacement, we shall study the following fundamental question about mixed characteristic singularities:

\begin{question*}[Localization of alteration-test ideal] \label{question:key-intro}
Let $X = \Spec R$ with $R$ Noetherian normal and Gorenstein.  Define
\[
    \tau_{\alt}(R) = \bigcap_{f : Y \to X}  \Image \Big( \Tr_{Y/X} : H^0(X, f_*\omega_{Y/X}) \to R \Big),
\]
where the intersection runs over all alterations of $X$\footnote{More precisely, in this and all such intersections that appear in this paper, we always mean the following: choose a geometric generic point of $X$ (i.e., an algebraic closure of $K(R)$), and intersect over all alterations of $X$ equipped with a lift of this point.}. Is it true that 
\[ \tau_{\alt}(R) = \Image \big( \Tr_{Y/X} : H^0(X, f_*\omega_{Y/X}) \to R \big)\] 
for a single sufficiently large alteration $f \colon Y \to X$? 

Note that a positive answer would, in particular, show that $\tau_{\rm alt}(R[1/g]) = \tau_{\rm alt}(R)[1/g]$ for all $g \in R$., i.e., the formation of $\tau_{\rm alt}(R)$ commutes with localization. 
\end{question*}

Let us briefly provide some context for this question. In mixed characteristic, the ideal $\tau_{\alt}(R)$ appeared in \cite{BMPSTWW1,TakamatsuYoshikawaMMP, HaconLamarcheSchwede} among other places; the fact that its formation is not known to commute with localization (and then completion at a maximal ideal) has been a major obstacle.\footnote{A variant of $\tau_{\alt}(R)$ has been defined in \cite{HaconLamarcheSchwede} for quasi-projective varieties by passage to the affine cone; this commutes with localization essentially by fiat but its behavior after completion at a stalk was unclear.} Note that localization does not commute with infinite intersections in general. 
In equal characteristic zero, this intersection stabilizes since $\tau_{\alt}(R)$ is the multiplier ideal, and is in fact computed by any single resolution of singularities.  In characteristic $p > 0$, a resolution is not enough even if it exists, but there is still a single sufficiently large alteration which computes it by \cite{SmithTightClosureParameter,BlickleSchwedeTuckerTestAlterations}, \cf \cite{HochsterHunekeInfiniteIntegralExtensionsAndBigCM,HunekeLyubeznikAbsoluteIntegralClosure,ChiecchioEnescuMillerSchwede,DattaTuckerOpenness}. 

A primary goal of our article is to provide a positive answer to Question~\ref{question:key-intro} ``up to small $p$-perturbation'', as made precise next.

\subsection{The main results}
In order to state the main theorem, we first fix a system of compatible $p$-power roots $\{p^{1/p^e}\}_{e \in \bZ_{\geq 0}}$ of $p$ within a fixed algebraic closure of the fraction field of $R$ -- in the end, the resulting object will be independent of such choices.
We need to define a $p$-perturbed alteration variant of $\tau_{\alt}(R)$ for $0 < \epsilon = 1/p^e \ll 1$ (the choice of $0 < \epsilon  \ll 1$ is irrelevant by Noetherianity):
\[
\tau^a_{\alt}(R) := \bigcap_{f : Y \to X}  \Image \Big(  H^0(X, f_*\omega_{Y/X}) \xrightarrow{p^\epsilon} H^0(X, f_*\omega_{Y/X}) \xrightarrow{\Tr_{Y/X}} R \Big),
\]
where $X := \Spec R$ and the intersection is taken over alterations such that
$p^{\epsilon} \in \cO_Y$. We emphasize that although the above $p$-perturbation may look technical, in most applications it is completely harmless. One should note that building in small perturbation is the key part of tight closure theory -- \emph{test elements} are small perturbations -- and of course statements up to small $p$-perturbation are common applications of almost ring theory \cite{FaltingsAlmostEtale,GabberRameroAlmostringtheory,GabberRameroFoundationsAlmostRingTheory}.   This perturbation lets us prove that the intersection stabilizes for $f$ sufficiently large. 
\begin{theoremA*}[{\autoref{thm.tauAltNonComplete}, \autoref{cor.TauRHInvertPEqualsJNoetherianNoPair}, \autoref{cor.TauAgreesWithCompletionComputation}}]
Let $R$ be a normal Gorenstein\footnote{Versions of this result also hold for (log-)$\bQ$-Gorenstein $X$ if $\omega_{Y/X}$ is replaced by $\cO_Y(\lceil K_Y - f^* K_X\rceil) $.  See the discussion below, \autoref{rem.AddingRoundings}, as well as \autoref{rem:a-test-ideal-definition} and \autoref{cor.TauRHVsTau+ForDivisorPairs} for more details.} domain of finite type over a DVR $V$ of mixed characteristic.
Then for $1 \gg \epsilon > 0$, there exists an alteration $f : Y \to X = \Spec R$ with $p^\epsilon \in \cO_Y$ such that
\[
\tau^a_{\alt}(R) = \Image \Big(  H^0(X, f_*\omega_{Y/X}) \xrightarrow{p^\epsilon} H^0(X, f_*\omega_{Y/X}) \xrightarrow{\Tr_{Y/X}} R \Big).
\]
In particular, $\tau^a_{\alt}(R[1/g]) = \tau^a_{\alt}(R)[1/g]$ for every $g \in R$, and furthermore $\tau^a_{\rm alt}(R)[1/p]$ agrees with the multiplier ideal $\mathcal{J}(R[1/p])$.  Its formation also commutes with completion at a maximal ideal: $\tau^a_{\alt}(R) \cdot \widehat{R} = \tau^a_{\alt}(\widehat{R})$.
\end{theoremA*}

To prove this theorem, our main idea is to interpret $\tau^a_\alt(R)$ intrinsically in terms of the intersection cohomology complex (with $\mathbf{Z}_p$-coefficients) of $\mathrm{Spec}(R[1/p])$ via the $p$-adic Riemann-Hilbert functor \cite{BhattLuriepadicRHmodp}; this is a mixed characteristic counterpart of the fact (reviewed in  \autoref{ss:intro-char0}) that the multiplier ideal sheaf on a complex variety is naturally encoded in the Hodge filtration on the intersection cohomology $\mathcal{D}$-module provided by Saito's theory of Hodge modules; see \autoref{ss:IdeaProof} for a more detailed summary of the proof.


By incorporating small perturbations along arbitrary divisors and not just ${\rm div}(p)$, we provide a comprehensive theory of a test ideal sheaf $\tau(\cO_X)$. 
This notion can be generalized to pairs and even triples $(X, \Delta, \fra^t)$, see \autoref{sec.TestIdealsDivisorPairs} and \autoref{sec.TestIdealsNonPrincipal}, where $\Delta$ is a $\bQ$-divisor such that $K_X+\Delta$ is $\bQ$-Cartier, $\fra$ is an ideal sheaf, and $t \in \bQ_{\geq 0}$.  
In that setting, for any sufficiently large Cartier divisor $G > 0$ on $X$ (a test element analog), and $1 \gg \epsilon > 0$, we can define 
\[
    \tau(\cO_X, \Delta, \fra^t) = \Tr_{Y/X}\big( f_* \cO_Y(K_Y - f^*(K_X + \Delta + \epsilon G) - tM)\big)
\]
where $f : Y \to X$ is a sufficiently large alteration such that $\fra \cO_Y = \cO_Y(-M)$, where $M$, $tM$ and $f^*(K_X + \Delta + \epsilon G)$ are Cartier.  In particular, the intersection over such $Y$ stabilizes (\autoref{thm:onealterationtorulethemall},
\autoref{thm.FinalAlterationStabilizationForTestIdeal}, \cf \autoref{prop:tau-nonprincipal-is-global-intersection}).

In this context, we are able to show that $\tau(\cO_X, \Delta, \fra^t)$ satisfies the full suite of properties one expects from multiplier ideals in characteristic zero or test ideals in characteristic $p > 0$.


\begin{theoremB*}[Properties of test ideals, {\autoref{cor.EffectiveGlobalGenerationTauDivisor}, \autoref{cor.FinalCompletionGlobalizationLocalizationForTestIdeal}, \autoref{thm.FinalPropertiesOfTestIdealsTriples}}]
Suppose $X$ is normal, integral, flat, and of finite type over a DVR $V$ of mixed characteristic.  Suppose additionally that $\Delta$ is a $\bQ$-divisor such that $K_X + \Delta$ is $\bQ$-Cartier, $\fra, \frb$ are ideal sheaves and $s,t \geq 0$ are rational numbers.  

Then the ideal sheaf $\tau(\cO_X, \Delta, \fra^t)$ is coherent and satisfies the following properties.
\begin{enumerate}
\setlength\itemsep{0.2em}
\item \textnormal{(Multiplier ideals):}  Inverting $p$, it becomes the multiplier ideal: $\tau(\cO_X,\Delta, \fra^t)[1/p]=\mathcal{J}(\cO_{X[1/p]}, \Delta|_{X[1/p]}, \fra^t[1/p])$.  
\item \textnormal{(Smooth pullback):}  If $f : Y \to X$ is smooth, then $\utau(\cO_Y, f^*\Delta, (\fra \cO_Y)^t) = f^* \utau(\cO_X, \Delta, \fra^t)$. 
\item \textnormal{(Finite maps):}  If $f : Y \to X$ is a finite surjective map, then  for $K_Y+\Delta_Y = f^*(K_X+\Delta)$ we have that $\Tr_{Y/X} \big( \utau(\cO_Y, \Delta_Y, (\fra \cO_Y)^t) \big) =\utau(\cO_X, \Delta,  \fra^t)$.  
\item \textnormal{(Restriction)}  If $H$ is a normal Cartier divisor on $X$, then $\tau(\cO_X, \Delta, \fra^t) \cdot \cO_H \supseteq \tau(\cO_H, \Delta|_H, (\fra \cO_H)^t)$.  
\item \textnormal{(Skoda):}  If $t \geq \dim(X)$, then $\tau(\cO_X, \Delta, \fra^t) = \fra \cdot \tau(\cO_X, \Delta, \fra^{t-1})$.  
\item \textnormal{(Summation):}  $\utau(\cO_X,\Delta, (\fra + \frb)^t) = \sum_{t_1 + t_2 = t} \utau(\cO_X,\Delta, \fra^{t_1} \frb^{t_2})$, where $t_1, t_2 \geq 0$.
\item \textnormal{(Subadditivity):}  If $X$ is nonsingular, then $\tau(\cO_X, \fra^t \frb^s) \subseteq \tau(\cO_X, \fra^t) \cdot \tau(\cO_X, \frb^s)$.
\item \textnormal{(Perturbation):} $\tau(\cO_X,\Delta + \epsilon D, \fra^t) = \tau(\cO_X, \Delta, \fra^t)$ for any Cartier divisor $D \geq 0$ and $0 < \epsilon \ll 1$. 
\item \textnormal{(Effective global generation):}  Suppose $X$ is projective over $V$, $A$ is a globally generated ample Cartier divisor and $L$ is such that $L - K_X - \Delta$ is big and nef, then there exists Cartier $H > 0$ so that for $1 \gg \epsilon > 0$ and for all $n \geq \dim X_{p=0}$,
\[
    \tau(\cO_X, \Delta) \otimes \cO_X(L + nA)
\]
is globally generated by $\myB^0(X, \Delta + \epsilon H, \cO_X(L + nA)) \subseteq H^0(X, \cO_X(L + nA))$. 
\item \textnormal{(Unification):}  $\tau(\cO_X, \Delta, \fra^t)$ agrees with and unifies previous mixed characteristic test ideals up to small perturbation \cite{MaSchwedeSingularitiesMixedCharBCM,BMPSTWW1,DattaTuckerOpenness,SatoTakagiArithmeticAndGeometricDeformationsOfFPure,RobinsonBCMTestIdealsMixedCharToric,MurayamaUniformBoundsOnSymbolicPowers} (cf.\ \cite{BhattAbsoluteIntegralClosure}).  Hence, for $x \in X_{p=0}$ and if $R = \widehat{\cO_{X,x}}$, then using notation of the references:
\[
    \tau(\cO_X, \Delta, \frc^t) \cdot R = \tau_B(R, \Delta|_R, (\frc R)^t)
\]
for all sufficiently large perfectoid big Cohen-Macaulay $R^+$-algebras $B$.
Furthermore,
\[ 
H^0(X, \utau(\cO_X, \Delta, \fra) \otimes \sL^n) = \myB^0(X, \Delta + \epsilon G; \sL^n)
\] 
for $X$ projective over a complete DVR $V$, $\sL$ ample, $n \gg 0$, some (test element-like-)divisor $G \geq 0$, and every $0 < \epsilon \ll 1$, where $\myB^0$ is defined as in \cite{BMPSTWW1,TakamatsuYoshikawaMMP}.  
\end{enumerate}
\end{theoremB*}
For variants and generalizations of these results in other contexts, see \autoref{thm.SmoothPullbackForRHOnNoetherian}, \autoref{thm.ComparisonTauRHvsHLSNoPair}, \autoref{prop.PerturbedTestIdealIsReal}, \autoref{prop.PropertiesOfUltTauOmegaForPairs}, \autoref{thm:BlowUpSubadditivity}, \autoref{thm.SkodaForLocalCase}, \autoref{cor.SummationInLocalCase}, \autoref{thm.MainResultOnNonprincipalV2}.  
In  positive characteristic for $F$-finite schemes, $\tau(\cO_X, \Delta, \fra^t)$ agrees with the usual test ideal detecting strongly $F$-regular singularities \cite{BlickleSchwedeTuckerTestAlterations}.  In mixed characteristic, after localization and completion, it detects perfectoid BCM-regular singularities in the sense of \cite{MaSchwedeSingularitiesMixedCharBCM,SatoTakagiArithmeticAndGeometricDeformationsOfFPure,RobinsonBCMTestIdealsMixedCharToric,MurayamaUniformBoundsOnSymbolicPowers} (also see \cite{BrennerRescueSolidClosure,PerezRGTestIdeals21}) as implied by (Unification) above.

The unification property above is quite useful.  Indeed, for some variants of test ideals we had subadditivity, and for others we had Skoda and effective global generation, and so by showing that our ideal agreed with previously defined ones, those properties formally work on our ideal as well.  We believe that some of the above properties (for instance the summation formula) were not known for any mixed characteristic test ideal variant.

{\color{black}
\begin{remark}[More precise test elements]
    In the above definition of $\tau(\cO_X, \Delta, \fra^t)$ and the unification statement, we needed to fix a divisor $G$.  
    We can make the choice of $G$ precise.  One simply needs any $G$ such that 
    \[
        \Supp G \supseteq \Div(p) \cup \Supp(\Delta) \cup V(\fra) \cup X_{\text{sing}}.
    \]
    The $\Div(p)$ support condition can also be weakened frequently, and $\Supp(\Delta)$ can replaced with the non-Cartier locus of $K_X + \Delta$.  See instance \autoref{cor.CompleteLocalOptimalTestElementChoice} for other more precise statements in the local case.
    
    Furthermore, set $\frg$ to be an ideal sheaf such that $V(\frg) \supseteq V(p) \cup \Supp \Delta \cup V(\frc)$ and such that $X \setminus V(\frg)$ is smooth over $V$.  
    Then for $1 \gg \epsilon > 0$ there exists an alteration $f : Y \to X$ with $\frc \cO_Y = \cO_Y(-M)$, 
    $\frg \cO_Y = \cO_Y(-N)$,
    $p^\epsilon \in \cO_Y$,  and $\epsilon N$ and $t M$ Cartier, such that 
    \[
        \tau(X, \Delta, \frc^t) = \Image\Big( H^0(X, f_* \cO_Y(K_Y - f^*(K_X + \Delta) -\epsilon N - t M)  \Big).
    \]
    The same equality also holds for any larger alteration.  See \autoref{def.UltimateTestIDEALDefinition} and \autoref{thm.FinalAlterationStabilizationForTestIdeal} for the full statement.
\end{remark}}

\begin{remark}[Adding roundings]
\label{rem.AddingRoundings}
    In our work, based on the theory of multiplier ideals \cite{LazarsfeldPositivity2,NadelMultiplierIdeals,NadelMultiplierIdealSheaves} \cf \cite{EsnaultViehwegLecturesOnVanishing,LipmanAdjointsOfIdealsInRegularLocal}, it is natural to consider 
    \[
        \bigcap_{\pi : Y \to X}  \Tr_{W/X}\Big( f_* \cO_W(\lceil K_W - f^*(K_X + \Delta + \epsilon G) - tM) \rceil \Big)
    \]
    where the intersection runs over normal alterations $\pi : W \to X$ where $\frc \cO_W = \cO_Y(-M)$ is a line bundle.  However, any such alteration $W \to X$ is dominated by one where $K_X + \Delta + \epsilon G$ and $tM$ pull back to Cartier divisors.  Since we are rounding up ($\lceil - \rceil$), it follows that one may restrict the intersection to such $Y$.  We thus avoid roundings in this paper as the intersections are the same.
\end{remark}

Throughout most of the article we work with $\tau(\omega_X)$ and variants for pairs and triples, as this formulation is more convenient for our purpose, see the discussion of strategy below, noting that if $\Delta = -K_X$ then $\tau(\cO_X, \Delta) = \tau(\omega_X)$.  Note $\tau(\omega_X)$ is called the \emph{(parameter) test (sub)module} in positive characteristic and the analogous notion in characteristic zero is called the  \emph{Grauert-Riemenschneider sheaf} or \emph{multiplier (sub)module}. 

\subsection{Other applications}

We discuss some additional applications. 

\subsubsection*{Openness of splinter loci} 
Let $R$ be a normal Noetherian domain. Suppose that $R_{\mathfrak{q}}$ is a splinter for $\mathfrak{q} \in X_{p=0}$, where $X = \Spec R$. Is $R[1/f]$ a splinter for some $f \not \in \mathfrak{q}$? This is known to be true in characteristic $p>0$ by \cite{DattaTuckerOpenness}. In what follows we answer this question for $\bQ$-Gorenstein $p$-almost splinters in mixed characteristic (see \autoref{ss:open-splinter} for details). 
\begin{cor*}[{\autoref{thm:p-almost-splinter-locus-open}}]
The $p$-almost splinter locus is open for normal $\bQ$-Gorenstein domains of finite type over a DVR of mixed characteristic.
\end{cor*}

We expect a similar statement for the splinter locus itself, but our methods, which need the $p$-perturbation, cannot handle this case yet.

\subsubsection*{Multiplier ideals via finite covers}

Suppose $(R, \fram)$ is a Gorenstein local domain of mixed characteristic $(0,p>0)$ essentially of finite type over a DVR $V$ of mixed characteristic.  
Define $\tau_{\mathrm{fin}}(R) = \bigcap_{\textrm{ finite } R \subseteq S} \Tr_{S/R}(\omega_{S/R})$ as in \autoref{eq:intro-splinter},
where $S$ runs over finite extensions of $R$ contained in a fixed choice of $R^+$. As a consequence of our main result, we obtain the following.

\begin{cor*}[{\autoref{cor.MultiplierIdealViaFiniteCovers}}]
    With notation as above, $\tau_{\mathrm{fin}}(R)[1/p] = \mJ(R[1/p])$.  In particular, the multiplier ideal in characteristic zero can be computed from finite covers in mixed characteristic.
\end{cor*}
\noindent Note that this Corollary implies that the infinite intersection defining $\tau_{\mathrm{fin}}(R)$ {\it cannot} stabilize since the multiplier ideal is not equal to $R[1/p]$ for any normal but non-KLT $R[1/p]$.  
{
\color{black}
\subsubsection*{Uniform approximation of Abhyankar valuation ideals}
An appendix to this paper, written by Rankeya Datta, establishes analogs of \cite{EinLazSmithValuations,DattaUniformApproximation} in mixed characteristic.  The following result is obtained.  

\begin{theorem*}[Datta, {
\autoref{thm:uniform-approximation-Abhyankar-valuation-ideals}}]
Let $(V,\fram_V,\kappa_V)$ be a DVR of mixed characteristic $(0,p)$ such that $\kappa_V$ is perfect. Let $(B,\fram_B,\kappa_B)$ be a regular local ring that is an essentially of finite type extension domain of $V$ that dominates $V$. Let $F$ be the fraction field of $B$ and $\nu$ be an $\bR$-valued valuation of $F$ such that $(B,\fram_B,\kappa_B)$ is an Abhyankar center of $\nu$. Suppose $\Gamma_\nu/\nu(\Frac(V)^\times)$ is torsion-free. Then there exists $e \in \mathbb{R}_{\geq 0}$ such that for all $m \in \mathbb{R}$ and $\ell \in \mathbb{Z}_{> 0}$, 
\[\ba_{\nu,m}(B)^\ell \subseteq \ba_{\nu,\ell m}(B) \subseteq \ba_{\nu,m-e}(B)^\ell.\]
\end{theorem*}
See \autoref{Datta-valuation-ideals} for additional consequences, including an Izumi-type theorem.
}

\subsubsection*{Non-archimedean Monge-Amp\`ere (MA) equations} The paper \cite{fang2022nonarchimedean} obtained results on solutions of non-archimedean MA equations in mixed characteristic assuming resolutions of singularities and subadditivity of $+$-test ideals (see \cite[Theorem 1.2]{fang2022nonarchimedean}). As far as we understand, our theory of $+$-test ideals is now sufficient for their applications (see e.g.\ Theorem B (subadditivity) and \autoref{rem.TestIdealsOfLinearSeries}), and so their result is now contingent on resolutions of singularities only. 

\subsection{Idea of the proof}
\label{ss:IdeaProof}
In what follows we explain the idea of the proof of Theorem A. Let $X$ be a normal integral scheme of finite type over a complete DVR $V$ of mixed characteristic. For simplicity, we shall work with the sheaf \[
\tau^a_{\alt}(\omega_X) = \bigcap_{Y \to X}\Image \Big( f_*\omega_{Y} \xrightarrow{p^\epsilon}  f_*\omega_{Y} \xrightarrow{\Tr_{Y/X}}  \omega_X \Big)
\]
with intersection taken over alterations. This sheaf, called a \emph{test module} (or in characteristic zero the \emph{Grauert-Riemenschneider sheaf} or \emph{multiplier module}), agrees with $\tau^a_{\alt}(\cO_X) \otimes \omega_X$ when $X$ is Gorenstein (\cf \cite{GRVanishing,SmithTightClosureParameter,SmithTestIdeals,BlickleMultiplierIdealsAndModulesOnToric}). We will show that $\tau^a_{\alt}(\omega_X)$ is coherent (i.e.\ satisfies localization), as this is easier to explain than the existence of a single alteration which calculates it. Note, with some work however, a strong enough version of the former implies the latter for formal reasons (\cf \autoref{lem.StabilizingIntersection}).

The main difficulty in showing that $\tau^a_{\alt}(\omega_X)$ is coherent is that infinite intersections of modules do not commute with localization in general; or, in other words, an infinite intersection of coherent sheaves may not be quasi-coherent. To circumvent this problem, we find an intrinsic definition of $\tau^a_{\alt}(\omega_X)$ that is visibly coherent using methods inspired by topology. Specifically, we shall read-off this Grauert-Riemenschneider sheaf from the intersection complex\footnote{The intersection complex originates from the theory of intersection cohomology introduced by Goresky and MacPherson in the topological setting and Deligne algebraically. Specifically, they generalized the usual singular cohomology $H^i(X,\bC)$ constructing intersection cohomology $I^pH^i(X,\bC) := \bH^{i-\dim}(X, {\rm IC}_X)$. This cohomology theory agrees with $H^i(X,\bC)$, when $X$ is smooth, but is much better behaved when $X$ is singular (for example, it always satisfies Poincare duality if $X$ is compact).} ${\rm IC}_{X[1/p]} \in D^b_{\rm cons}(X[1/p], \mathbf{Z}_p)$ which is a constructible complex of sheaves on the characteristic $0$ variety $X[1/p]$. What allows for accessing the algebro-geometric object $\tau^a_{\alt}(\omega_X)$ in mixed characteristic using the topological object $\mathrm{IC}_{X[1/p]}$ living in characteristic $0$ is the $p$-adic Riemann--Hilbert functor from \cite{BhattLuriepadicRHmodp}, as we explain next.\footnote{It is perhaps initially surprising that one can 
detect mixed characteristic phenomena using characteristic $0$ objects. However, note that it has been understood for a long time in $p$-adic Hodge theory that cohomology theories with $\mathbf{F}_p$ or $\bZ_p$ coefficients on a characteristic zero variety carry information on the geometry of reductions mod $p$ or integral models.}

\begin{remark}
The $p$-adic Riemann--Hilbert functor from \cite{BhattLuriepadicRHmodp} can be regarded as an analog of a functor provided by Saito's theory of mixed Hodge modules over $\mathbf{C}$. To make this analogy clear, in \autoref{ss:intro-char0} below, we explain how to read off the multiplier ideal of a complex algebraic variety from the intersection complex with $\bC$-coefficients using the framework of Hodge modules (\autoref{thm:mainthm-in-char-zero}); this procedure, which is well-known to experts, is the complex geometric motivation leading to \eqref{eq:def-omegaRH-intro} below.
\end{remark}

For simplicity of exposition, we work over $V=\bZ_p$; let $X/V$ be a flat projective scheme. The $p$-adic Riemann-Hilbert functor is easiest to work with over a perfectoid base, so we shall first solve our problem after base change to a perfectoid extension $V_\infty$ of $V$, and then trace the solution down to $V$. For definiteness, let us take $V_\infty := (\bZ_p[p^{1/p^\infty}])^{\wedge p}$, and set $X_{\infty} := X \otimes_V V_{\infty}$. In this setup, \cite{BhattLuriepadicRHmodp} provides a functor: 
\[
\RH : D^b_{\cons}(X_\infty[1/p], \bZ_p) \to D^b_{\acoh}(X_\infty),
\]
assigning to a complex of $\bZ_p$-constructible sheaves on the characteristic zero fiber $X_\infty[1/p]$, a complex of ($p$-)almost coherent sheaves on $X_\infty$; the latter are quasi-coherent sheaves on $X_\infty$ that enjoy a strong finiteness property: they are coherent up to multiplication by $p^{\epsilon}$ for any $0< \epsilon \ll 1$. Note that almost coherence only makes sense because we adjoined all the roots of $p$ in $V_{\infty}$. The functor $\RH$ enjoys many nice properties; for our purposes, let us note that $\RH(\bZ_p)$ equals to the perfectoidification $\cO_{X_\infty, \rm perfd}$ as in \cite{BhattScholzepPrismaticCohomology}, and $\RH$ commutes with proper pushforward as well as $p$-completed direct limits.

With that being said, one can define the analog of $\tau^a_{\alt}(\omega_X)$ in this context as:
\begin{equation} \label{eq:def-omegaRH-intro}
\omega^{\RH}_{X_\infty/V_\infty} := {\rm Image}\big(\psi \colon H^0{\rm RH} ({\rm IC}_{X_\infty[1/p]}) \to \omega_{X_\infty/V_\infty}\big),
\end{equation}
where the intersection complex is defined with $\bZ_p$ coefficients
(see \autoref{thm:mainthm-in-char-zero} in characteristic zero). Note that there are no infinite intersections involved here, and in fact this object is visibly almost coherent (as almost coherent sheaves form an Abelian category). To explain the relation to alterations, it is easier to work with the dual variant; in the dual form, the key result is the following. 
\begin{theoremC*}[\autoref{GRviaAIC}] Assume $X_\infty$ is integral, and let $\pi:X_\infty^+ \to X_\infty$ be an absolute integral closure.
Pick a closed point $x \in X_{p=0}$ and let 
$\fram$ be the corresponding maximal ideal of $\cO_{X,x}$. Then
\[
\Image \big(H^d_{\fram}(\cO_{X_\infty}) \to H^d_{\fram}(\pi_*\cO_{X_\infty}^+)\big) = \Image \big(H^d_{\fram}(\cO_{X_\infty}) \to H^d_{\fram}(\RH(j_*\bZ_p))\big),  
\]
up to $p$-almost mathematics, where $d=\dim(X)$ 
and $j \colon U \to X_\infty[1/p]$ is an inclusion of any regular dense open affine subscheme.
\end{theoremC*}
Before sketching the proof, let us explain the relevance of this theorem to describing $\omega^{\RH}_{X_\infty/V_\infty}$ in terms of alterations. First, the map in the theorem is induced by the composition 
$$\cO_{X_{\infty}}\to \cO_{X_\infty, \rm perfd}= {\rm RH}(\bZ_p) \to {\rm RH}(j_*\bZ_p),$$ where the second map is the standard map.  One can then check that if we replace $j_*\bZ_p$ by ${\rm IC}_{X_\infty[1/p]}[-\dim(X_\infty[1/p])]$ in the statement of the theorem, the right hand side does not change\footnote{This follows from the perverse left $t$-exactness of Bhatt-Lurie's Riemann-Hilbert functor and the fact that ${\rm IC}_{X_\infty[1/p]}[-\dim(X_\infty[1/p])] = j_{!*}\bZ_p \to j_*\bZ_p$ is injective in the category of perverse sheaves (up to a shift) by definition.}. Using duality, one then learns from the theorem that $\omega^{\RH}_{X_\infty/V_\infty}$ almost equals 
\[ \bigcap_{f : Y \to X_\infty}  \Image \Big( f_* \omega_{Y/V_\infty} \to \omega_{X_\infty/V_\infty} \Big),\]
where the intersection runs over all finitely presented finite covers of $X_\infty$ dominated by $X^+$; this provides the desired relationship of $\omega^{\RH}_{X_\infty/V_\infty}$ to finite covers, and the passage to alterations can be accomplished using a cohomology killing result from \cite{BhattAbsoluteIntegralClosure}.

\begin{proof}[Sketch of the proof]
It is enough to construct a map 
\begin{equation}  \label{eq:intro-proof-Bhatt}
 \RH(j_*\bZ_p) \to \pi_*\cO^+_{X_\infty}   
\end{equation}
and show that it is almost injective on the $d$-th local cohomology.

Since  $\cO^+_{X_\infty}$ is perfectoid (up to $p$-completion), $\pi_*\cO^+_{X_\infty}= \RH(\pi_*\bZ_p)$. Given that the fraction field of $\cO^+_{X_\infty}$ is algebraically closed, one can check that $\bZ_p = j'_*\bZ_p$ for an open inclusion $j' : U' \to X_\infty^+[1/p]$ based changed from $j : U \to X_\infty[1/p]$ under $X_
\infty^+ \to X_\infty$ (\cite[Theorem 3.11]{BhattAbsoluteIntegralClosure}, cf.\ \autoref{thm:AIC}). Thus $\pi_*\cO^+_{X_\infty} = \RH(\pi_*j'_*\bZ_p)$, and so we can construct \autoref{eq:intro-proof-Bhatt} by applying $\RH$ to the map
\[
j_*\bZ_p \to \pi_*j'_*\bZ_p.
\]
This map is injective (up to a shift) in the category of perverse sheaves.\!\footnote{We prove this using \cite[Theorem 3.11]{BhattAbsoluteIntegralClosure}, see  \autoref{thm:AIC}. However, this result should not come as a surprise. Intuitively, if we replace  $\pi : X_\infty^+ \to X_\infty$ by a finite cover $\pi : Y_\infty \to X_\infty$ and make $U$ smaller, we have that $U' \to U$ is \'etale, in which case $j_*\bZ_p \to \pi_*j'_*\bZ_p$ just splits.} Now, a key property of the $p$-adic Riemann-Hilbert functor, called perverse left $t$-exactness, turns injections of perverse sheaves into almost injections of local cohomology modules under applying $\myR\Gamma_{\fram}\RH$:
\[
H^d_{\fram}(\RH(j_*\bZ_p[d])) \hookrightarrow H^d_{\fram}(\RH(\pi_*j'_*\bZ_p[d])) = H^d_{\fram}(\pi_*\cO^+_{X_\infty}),
\]
which concludes this sketch of the proof.
\end{proof}


The above discussion gives a satisfactory finitistic approach to the test module \autoref{eq:def-omegaRH-intro} for projective schemes over $V_\infty$.  It remains to descend this solution down to $V$.  The main difficulty here is that $\bZ_p \to \bZ_p[p^{1/p^\infty}]$ is only integral and not finite; as such, there is no naturally defined trace map. Nonetheless, our key idea is to construct a map $\kev : V_{\infty} \to V$  to play the role of a trace map in the descent from $V_\infty$ down to $V = \bZ_p$. Noting that $\bZ_p[p^{1/p^\infty}]$ is a free $\bZ_p$-module with basis $p^{\alpha}$ for $\alpha \in \bZ[1/p]$ and $0 \leq \alpha < 1$, one may simply take $\kev$ to be the (completion of) the unique map determined by setting $$\kev(p^\alpha) = \begin{cases}
    1 \quad \text{ if } \alpha = 1-1/p^e \mbox{ for some } e\in \bZ_{\geq 0} \\ 0 \quad \mbox{otherwise}
\end{cases}$$
and observe that the restriction $\kev|_{V_e}$ is a trace map for each of the finite extensions $\bZ_p \to V_e := \bZ_p[p^{1/p^e}]$, \textit{i.e.} a generator for $\Hom_{V}(V_e,V)$ as a $V_e$-module. Similar constructions are possible for arbitrary complete DVRs in mixed characteristic, see \autoref{sec:SingOverDVR}.

Thus, we can prove the following -- which also demonstrates independence of the choice of the auxiliary map $\kev$ above.
\begin{theoremD*}[see the Proof of \autoref{thm.SingleAlterationWithPToEpsilon}]
With notation as above, 
$$
\tau^a_{\rm alt}(\omega_X) = \kev((p^{1/p^\infty}) \cdot \omega^{\RH}_{X_\infty/V_
\infty}).$$
\end{theoremD*}
\noindent
Note  that, while the construction of $\omega^{\RH}_{X_\infty/V_
\infty}$ on $X_\infty$ relies heavily on the use of almost mathematics, multiplying by $(p^{1/p^\infty})$ yields an honest sheaf.
In particular, $\tau^a_{\rm alt}(\omega_X)$ is coherent as it is a quasi-coherent subsheaf of the coherent sheaf $\omega_{X}$. Thus, the shadow of almost mathematics in the non-Noetherian world of almost coherent sheaves on $X_{\infty}$ persists as a small $p$-perturbation in the definition of $\tau^a_{\rm alt}(\omega_X)$ on the Noetherian scheme $X$.

\subsection*{Acknowledgements}  Various subsets of the authors worked on this while visiting the American Institute of Mathematics at an AIM square in 2023, while visiting Luminy in 2023, and while visiting the various authors home institutions, the IAS, Michigan, Princeton, Utah, EPFL in 2022, 2023, and 2025.  The authors thank  Jacob Lurie and Shunsuke Takagi for valuable conversations.  They also thank Rankeya Datta and Walter Gubler for comments on previous drafts.  

Bhatt was supported by NSF Grant DMS \#1801689 and \#1840234, NSF FRG Grant \#1952399, a Packard Fellowship, and the Simons Foundation.  
Ma was supported by NSF Grant DMS \#2302430, NSF FRG Grant DMS \#1952366, a fellowship from the Sloan Foundation, and by a grant from the Institute for Advanced Study School of Mathematics. 
Patakalvi was partially supported by the following grants: grant \#200020B/192035 from the Swiss National Science Foundation,  ERC Starting grant \#804334. 
Schwede was supported by NSF Grant \#2101800, NSF FRG Grant \#1952522 and a Fellowship from the Simons Foundation. 
Tucker was supported by NSF Grant DMS \#2200716. 
Waldron was supported by NSF Grant DMS \#2401279, the Simons Foundation Gift ID \#850684 and also gratefully acknowledges support from the Institute for Advanced Study during his Spring 2024 Membership funded by the Infosys Membership Fund. 
Witaszek was supported by NSF Grant No.\ DMS-\#2101897. 
Datta was supported by a grant from the Simons Foundation MP-TMS-00002400.

\section{Preliminaries}
Let $R$ be a domain. We denote by $R^+$  \emph{the absolute integral closure} of $R$ (that is, the integral closure of $R$ in an algebraic closure of $K(R)$). This object is unique up to isomorphism. Specifically, 
 $R^+ = \bigcup_{R \subseteq S}S$, where the union is taken over all finite extensions $R \subseteq S$ contained in the fixed algebraic closure $\overline{K(R)}$ of $K(R)$. 




\subsection{Grothendieck and local duality}
In this subsection, we briefly review Grothendieck and local duality. Recall that any complete Noetherian local ring $(R, \fram)$ has a dualizing complex $\omega_R^{\mydot}$, since such an $R$ is a quotient of a regular ring (\cite[\href{https://stacks.math.columbia.edu/tag/032A}{Tag 032A}]{stacks-project},
\cite[\href{https://stacks.math.columbia.edu/tag/0A7I}{Tag 0A7I}]{stacks-project},
\cite[\href{https://stacks.math.columbia.edu/tag/0A7J}{Tag 0A7J}]{stacks-project}). 
We always choose $\omega_R^{\mydot}$ to be normalized in the sense of \cite{HartshorneResidues}, that is $\myH^{-i} \omega_R^{\mydot} = 0$ for $i > \dim R$ and $\myH^{-\dim R} \omega_R^{\mydot} \neq 0$.  If $\pi : X \to \Spec(R)$ is a finite type morphism, then we define 
\begin{enumerate}
    \item The normalized dualizing complex $\omega_X^\mydot$ to be $\pi^! \omega_R^\mydot$ and the dualizing sheaf $\omega_X$ to be $\myH^{-\dim X}(\omega_X^\mydot)$.
    \item The relative dualizing complex $\omega_{X/R}^\mydot$ to be $\pi^!R$ and the relative dualizing sheaf $\omega_{X/R}$ to be $\myH^{-\dim X+ \dim R}(\omega_{X/R}^\mydot)$.
\end{enumerate}
Note that, if $R$ is Gorenstein, then the relative dualizing complex is a dualizing complex (but it is not normalized unless $R$ is a field) and we have $\omega_X\cong \omega_{X/R}$. In \autoref{sec:SingOverDVR}, we will be frequently working in the scenario that $R=V$ is a complete DVR and $X$ is finite type over $R$, see \autoref{notation.ForSchemesOverADVR}. In \autoref{sec:preliminaries_perverse_sheaves} and \autoref{sec:RH_subsheaves}, we will also encounter the dualizing complexes for certain $p$-adic formal schemes, and we refer to \autoref{notation:CoherentDualizingSheavesFormalScheme} for more details in that setting. 


Back in the case of a Noetherian complete local ring $(R,\m)$, fix $E = E_R(R/\fram)$ to be an injective hull of the residue field.  This provides an exact Matlis duality functor $(-)^{\vee} := \Hom_R(-, E)$ which induces an anti-equivalence of categories of Noetherian $R$-modules with Artinian $R$-modules \cite[\href{https://stacks.math.columbia.edu/tag/08Z9}{Tag 08Z9}]{stacks-project}; by exactness, Matlis duality extends to the derived category as well, and we continue to denote it by $(-)^\vee$.  In particular, since $E$ is injective, we may harmlessly identify $\Hom_R(-, E)$ and $\myR \Hom_R(-, E)$. Since we work with normalized dualizing complexes, we have an isomorphism $\myR\Gamma_{\mathfrak{m}}(\omega^{\mydot}_R) \simeq E$  \cite[\href{https://stacks.math.columbia.edu/tag/0A81}{Tag 0A81}]{stacks-project}. This isomorphism together with the complete-torsion equivalence \cite[\href{https://stacks.math.columbia.edu/tag/0A6X}{Tag 0A6X}]{stacks-project} shows the following compatibility of Grothendieck and Matlis duality:  for any $K \in D_{\coherent}^b(R)$, the following natural maps give isomorphisms
\[ \myR\Hom_R(K,\omega_R^{\mydot}) \simeq \myR\Hom_R\big(\myR\Gamma_{\mathfrak{m}}(K), \myR\Gamma_{\mathfrak{m}}(\omega^{\mydot}_R)\big) \simeq \Hom_R\big(\myR\Gamma_{\mathfrak{m}}(K),E\big) = \myR\Gamma_{\mathfrak{m}}(K)^\vee. \]
As $R$ is complete and $\Hom_R(-,E)$ induces an anti-equivalence of Noetherian and Artinian $R$-modules, this also yields
\[
    \big(\myR \Hom_R(K, \omega_R^{\mydot})\big)^{\vee} \simeq \myR \Gamma_{\fram}(K)
\]
for $K \in D^b_{\coherent}(R)$. For more details see for instance \cite{HartshorneLocalCohomology, HartshorneResidues, BrunsHerzog} and \cite[\href{https://stacks.math.columbia.edu/tag/0A81}{Tag 0A81}]{stacks-project}. 

If $X$ is a scheme and $x\in X$ is a (not necessarily closed) point, for $K\in D^b_{\qcoh}(X)$ we will use $\myR\Gamma_x(K_x)$ to denote the local cohomology functor; if $R=\cO_{X,x}$, $\m=\m_x$, and $M=K_x$, then this is the same as $\myR\Gamma_\m(M)$.

\begin{remark} \label{remark:MD-non-noetherian}
We record the following form of duality over a non-Noetherian base in a relative setting. This follows from the Noetherian case and is not used directly in our paper (but it underlies some arguments). Let $V_0$ be a DVR with uniformizer $\varpi$ and let $R_0$ be a finitely generated $\varpi$-torsion free $V_0$-algebra. Let $\{V_j\}$ be    a directed system of DVRs flat over $V_0$ and let $V:=\widehat{\varinjlim}_j V_j$ where the completion is $\varpi$-adic (note $V$ is typically non-Noetherian as we will be taking roots of $\varpi$). Let $R_j=R_0\otimes_{V_0}V_j$, $R=R_0\otimes_{V_0}V$, and let $\m$ be a maximal ideal of $R_0$. We consider the relative dualizing complex $\omega^\mydot_{R/V} := \omega^{\mydot}_{R_0/V_0}\otimes_{V_0}V$, for which we have that 
\[
    E := E_{R/V} : = \myR\Gamma_\m(\omega^\mydot_{R/V}) = \myR\Gamma_\m(\omega^\mydot_{R_0/V_0}) \otimes_{V_0} V = H^1(\myR\Gamma_\m(\omega^\mydot_{R/V}))[-1] = (E_{R_0} \otimes_{V_0} V)[-1].
\]
Note that $E[1]$ is in general not an injective module over $R$. On the other hand, for $K \in {D^b_{\qcoh}(R)}$, we have 
\begin{align*}
\RHom_R(\myR\Gamma_\m (K), E) & = \RHom_R(K\otimes_{R}^{{{\bf L}}} \myR\Gamma_\m(R), E) \\
& = \RHom_{R}(K, \RHom_R(\myR\Gamma_\m(R), E)) \\
& = \RHom_{R}(K, (\myR\Gamma_\m\omega^\mydot_{R/V})^{\wedge\m}) \\
& = \RHom_R(K, (\omega^\mydot_{R/V})^{\wedge\m})\\
& = \RHom_R(K, \omega^\mydot_{R/V})^{\wedge\m}
\end{align*}
where $(-)^{\wedge\m}$ denotes derived $\m$-completion. Moreover, if $K=K_j\otimes_{R_j}R$ where $K_j\in D^b_{\coherent}(R_j)$, then we also have
\begin{align*}
\RHom_R(\RHom_R(K, \omega^\mydot_{R/V}), E) & = \RHom_R(\RHom_R(K, \omega^\mydot_{R/V}), \myR\Gamma_\m\omega^\mydot_{R/V}) \\
& = \myR\Gamma_\m\RHom_R(\RHom_R(K, \omega^\mydot_{R/V}), \omega^\mydot_{R/V}) \\
& = \myR\Gamma_\m(K)
\end{align*}
where the second equality follows since $\RHom_R(K, \omega^\mydot_{R/V})\in D^b_{\coherent}(R)$, and the third equality follows from local duality over $R_j$. 
\end{remark}

\subsection{Perfectoid rings and almost mathematics} 

We will freely and frequently use the language of perfectoid rings as in \cite[Section 3]{BhattMorrowScholzeIHES}. Fix a prime number $p$. A ring $S$ is {\it perfectoid} if it is $\varpi$-adically complete for some element $\varpi\in S$ such that $\varpi^p$ divides $p$, the Frobenius on $S/pS$ is surjective, and the kernel of Fontaine's map $\theta$: $W(S^\flat)\to S$ is principal. One can always choose $\varpi$ such that it admits a compatible system of $p$-power roots $\{\varpi^{1/p^e}\}_{e=1}^{\infty}$ in $S$ (\cite[Lemma 3.9]{BhattMorrowScholzeIHES}) and we will always assume this. In our context, it is often the case that either $p=0$ in $S$ or $p$ is a nonzerodivisor in $S$. If $S$ has characteristic $p$, then a perfectoid ring is the same as a perfect ring, see \cite[Example 3.15]{BhattMorrowScholzeIHES}. On the other hand, if $S$ is $p$-torsion free, then $S$ is perfectoid if and only if it is $p$-adically complete and the Frobenius map $S/\varpi \to S/\varpi^p$ is bijective, see \cite[Lemma 3.10]{BhattMorrowScholzeIHES}. If $\varpi=p^{1/p}$ (e.g., $S$ is an algebra over $R^+$ for a Noetherian complete local domain $R$), then this definition is compatible with the definition given in \cite[2.2]{AndreWeaklyFunctorialBigCM} or \cite[Definition 2.2]{MaSchwedeSingularitiesMixedCharBCM}. Perfectoid rings are always reduced.

We record two main examples of perfectoid rings that will be relevant in this paper, see \cite[Example 3.8]{BhattIyengarMaRegularRingsPerfectoid} for more details.
\begin{example}
\label{example:PerfectoidRings}
\begin{enumerate}[(1)]
    \item $(\mathbf{Z}_p[p^{1/p^\infty}])^{\wedge p}$ is a perfectoid valuation ring (see \autoref{prop:kev_is_generator} below for a more general construction).
    \item Let $(R,\m)$ be a Noetherian complete local domain of residue characteristic $p>0$. Both the $p$-adic completion $(R^+)^{\wedge p}$ and the $\m$-adic completion $(R^+)^{\wedge \m}$ of the absolute integral closure of $R$ are perfectoid.
\end{enumerate}
\end{example}

We will also freely use some basic language of almost mathematics following \cite{GabberRameroAlmostringtheory}. 
Recall that an almost mathematics setup is a pair $(V,\m_V)$, where $\m_V$ is an ideal such that $\m_V^2=\m_V$, and for our purpose we always have $\m_V$ is flat. For $M$  an almost $V$-module, we have the $!$-realization $M_!:=M\otimes_V\m_V$ and he $*$-realization $M_*:= \Hom(\m_V, M)$. We will frequently fix a perfectoid ring $S$ (e.g., a perfectoid valuation ring) and work with the almost mathematics setup $(S,(\varpi^{1/p^\infty}))$ in later sections.

\subsection{$+$-stable sections ($\myB^0$)}  \label{ss:+-stable-sections}
One of the key problems in positive characteristic birational geometry is the fact that Kodaira vanishing fails in general. To address this problem, one introduces the space of Frobenius stable sections $S^0(X,M) \subseteq H^0(X, \cO_X(M))$, where $M$ is a Weil divisor, which behaves as if Kodaira vanishing was valid for them. Motivated by the work in \cite{BlickleSchwedeTuckerTestAlterations}, we introduced in \cite{BMPSTWW1} (\cf \cite{TakamatsuYoshikawaMMP}) the space of $+$-stable sections in mixed characteristic.

\begin{definition}[{\cite[Definition 4.2]{BMPSTWW1}, \cite{TakamatsuYoshikawaMMP}}] 

Let
\begin{itemize}
\item $X \xrightarrow{\pi} \Spec R$ be a normal and integral scheme proper over a Noetherian complete local domain $(R, \fram)$ such that $R/\m$ has characteristic $p>0$,
\item $\Delta \geq 0$ be a $\mathbf{Q}$-divisor on $X$,
\item $M$ be a $\mathbf{Z}$-divisor on $X$ with $\sM = \sO_X(M)$. 
\end{itemize}
Fix an algebraic closure $\overline{K(X)}$ of $K(X)$. 
We define 
\begin{equation*}
\myB^0(X, \Delta; \sM) := \bigcap_{\substack{f \colon Y \to X\\ \textnormal{finite}}}\Image \left( H^0(Y, \sO_Y( K_Y + \lceil{f^* (M - K_X - \Delta)}\rceil)) \xrightarrow{\rm Tr} H^0(X, \sM) \right) 
\end{equation*}
where the intersection is taken over all finite maps $f \colon Y \to X$ with $Y$ integral and $K(Y)$ contained in the fixed algebraic closure $\overline{K(X)}$ of $K(X)$. If $\Delta=0$, then we use the notation: $\myB^0(X; \sM):=\myB^0(X, \Delta; \sM)$.
\end{definition}

If $(-)^{\vee}$ denotes Matlis duality on $R$, then we also have that  
\begin{equation}
    \label{eq.B0DefinitionViaLocalCohom}
    \myB^0(X, \Delta; \sM) = \Image\Big( H^d \myR\Gamma_{\pi^{-1}\fram}(X, \sM) \to H^d \myR\Gamma_{\pi^{-1}\fram}(X^+, \sM \otimes \cO_{X^+}(\rho^*\Delta)) \Big)^{\vee}
\end{equation}
where $\cO_{X^+}(\Delta)$ is the colimit of the $\cO_{Y}(f^*\Delta)$ running over $f : Y \to X$ with $f^* \Delta$ integral.

When $M-(K_X+\Delta)$ is $\bQ$-Cartier, then we proved in \cite[Corollary 4.13]{BMPSTWW1} that $\myB^0(X, \Delta; \sM)$ is equal to 
\begin{equation*}
\myB^0_{\rm alt}(X, \Delta; \sM) := \bigcap_{\substack{f \colon Y \to X\\\textnormal{alteration}}}\Image \left( H^0(Y, \sO_Y( K_Y + \lceil f^* (M - K_X - \Delta)\rceil)) \xrightarrow{\rm Tr} H^0(X, \sM) \right)
\end{equation*}
where the intersection runs over all alterations $f:Y \to X$ from a normal integral schemes such that $K(Y)$ is contained in the fixed algebraic closure $\overline{K(X)}$ of $K(X)$.

Finally, in the case when $X$ is a disjoint union of normal schemes of the same dimension, we define $\myB^0(X, \Delta; \sM)$ to be the direct sum of $\myB^0$ for each connected component of $X$.

\subsection{Different notions of test ideals}
\label{ss:different-notions-test-ideals}
A central goal of this paper is to unify various existing definitions of test ideals in mixed characteristic. In what follows we recall some of these definitions. Unless otherwise stated,  we work in the following setting.

\begin{setting} \label{setting:local-complete-setting}
Assume that $(R,\fram)$ is a complete Noetherian local domain such that $R/\fram$ is of characteristic $p>0$. We fix an algebraic closure $\overline{K(R)}$ of the fraction field $K(R)$. If $R$ is additionally normal, we let $\Delta$ and $D$ be $\bQ$-divisors on $\Spec(R)$ so that $K_R+\Delta$ and $D$ are $\bQ$-Cartier. 
\end{setting}

\noindent \textbf{BCM-test ideals.} We first recall that an $R$-algebra $B$ is said to be a \emph{(balanced) big Cohen-Macaulay $R^+$-algebra} if $B$ is an $R^+$-algebra, $B\neq \m B$, and every system of parameters $x_1,\ldots, x_d \in R$ is a regular sequence on $B$, \textit{i.e.} $x_i$ is a non-zero-divisor in $B/(x_1,\ldots, x_{i-1})$ for every $i$ (\textit{cf.} \cite[Section 2.2]{BMPSTWW1}). We now recall the definition of BCM-test ideals of pairs as introduced by the second and the fourth author in \cite{MaSchwedeSingularitiesMixedCharBCM}.  Also compare with \cite{HochsterSolidClosure,HaraWatanabeFRegFPure,HaraYoshidaGeneralizationOfTightClosure,BrennerRescueSolidClosure,TakagiInterpretationOfMultiplierIdeals,PerezRGTestIdeals21}.

\begin{definition}\label{def:tauR+} Working in \autoref{setting:local-complete-setting}, 
suppose first that $D$ is effective. Fix a big Cohen-Macaulay $R^+$-algebra $B$ and take $m \in \bN$ and $f \in R$ such that $D = \frac{1}{m}{\rm div}(f)$. Following \cite[Definition 6.2]{MaSchwedeSingularitiesMixedCharBCM}, define
\[
\tau_B(\omega_R, D) := \big(\Image(H^d_{\fram}(R) \xrightarrow{f^{1/m}} H^d_{\fram}(B))\big)^{\vee} \subseteq \omega_R,
\]
where the map is induced by $R \hookrightarrow B \xrightarrow{f^{1/m}} B$, and $(-)^{\vee}$ denotes Matlis duality. When $D$ is not-effective, we pick a Cartier divisor $E \geq 0$ such that $D + E \geq 0$, and define $\tau_B(\omega_R, D) := \tau_B(\omega_R, D+E) \otimes_R R(-E)$.
Taking $D = K_R + \Delta$, we can then 
 define:
\[
\tau_B(R, \Delta) := \tau_B(\omega_R, K_R+\Delta).
\]
When $\Delta \geq 0$, we have that $H^d_{\fram}(R) \xrightarrow{f^{1/m}} H^d_{\fram}(B)$ naturally factors through $H^d_{\fram}(\omega_R)$, and thus $\tau_B(R, \Delta) \subseteq R$ (see \cite[Lemma 6.8 an Definition 6.9]{MaSchwedeSingularitiesMixedCharBCM}).

Finally, with notation as above, we define:
\begin{align*}
\tau_\mathcal{B}(\omega_R,D) &:= \bigcap_{B} \tau_B(\omega_R,D), \text{ and } \\
\tau_\mathcal{B}(R,\Delta) &:= \bigcap_{B} \tau_B(R,\Delta), 
\end{align*}
where the intersection is taken over all perfectoid big Cohen-Macaulay $R^+$-algebras $B$, whose existence is guaranteed by Andr\'e's work \cite{AndreWeaklyFunctorialBigCM}. It is proved in \cite[Proposition 6.4 and Proposition 6.10]{MaSchwedeSingularitiesMixedCharBCM} that $\tau_\mathcal{B}(\omega_R,D)$ (resp. $\tau_\mathcal{B}(R,\Delta)$) is equal to $\tau_B(\omega_R,D)$ (resp. $\tau_{B}(R,\Delta)$) for all sufficiently large perfectoid big Cohen-Macaulay $R^+$-algebras $B$. 
\end{definition}

\noindent \textbf{$+$-test ideals.} Motivated by \cite{BlickleSchwedeTuckerTestAlterations}, one can define test ideals using finite maps.
    
\begin{definition} \label{def:+test} Again working in \autoref{setting:local-complete-setting}, suppose first that $D$ is effective and write $mD = {\rm div}(f)$ for some $m \in \bN$ and $f \in R$. Define
\[
\tau_+(\omega_R,D) := \bigcap_{\substack{R \, \subseteq\, S\\ \textnormal{finite}}}\Image \left( \omega_{S}(\lceil{-f^*D}\rceil) \xrightarrow{\rm Tr} \omega_R\right),
\]
where the intersection is taken over all finite extensions $R \subseteq S$ contained in the fixed algebraic closure $\overline{K(R)}$ of $K(R)$. This agrees with $\myB^0(\Spec R, \Delta; \omega_R)$, and the definition is once more extended to non-effective $D$ by setting $\tau_+(\omega_R,D) = \tau_+(\omega_R, D + E) \otimes_R R(-E)$ for a Cartier divisor $E \geq 0$ with $D + E \geq 0$. Taking $D = K_R+\Delta$, we can then define
\[
\tau_+(R, \Delta) := \tau_+(\omega_R, K_R+\Delta);
\]
\noindent when $\Delta \geq 0$, we have that $\tau_+(R,\Delta) \subseteq R$.
\end{definition}

\noindent If $D \geq 0$, then $\tau_+(\omega_R,D)$ can also be calculated using local cohomology
\[
\tau_+(\omega_R,D) = \big(\Image(H^d_{\fram}(R) \xrightarrow{f^{1/m}} H^d_{\fram}(R^+))\big)^{\vee};
\]
see \cite[Lemma 4.8(a)]{BMPSTWW1} for details. Moreover, the $p$-adic completion $B=(R^+)^{\wedge p}$ of $R^+$ is a perfectoid and (balanced) big Cohen-Macaulay $R^+$-algebra by \cite{BhattAbsoluteIntegralClosure}, and for arbitrary $\bQ$-divisors $D$ and $\Delta$ we have
\[
\tau_+(\omega_R,D) = \tau_{B}(\omega_R,D) \quad \text{ and } \quad \tau_+(R,\Delta) = \tau_{B}(R,\Delta)
\]
with $B=(R^+)^{\wedge p}$ by \cite[Theorem 2.5]{BMPSTWW1}.

\begin{remark}
Finally, $\tau_+(\omega_R,D)$ is also equal to:
\[
\tau_{\rm alt}(\omega_R, D) := \begin{cases} \bigcap_{\substack{f : Y \to X\\ \textnormal{alteration}}}\Image \left( H^0(Y, \cO_Y(K_Y + \lceil{f^*(-D)}\rceil)) \xrightarrow{\rm Tr} \omega_R\right) \text{ if } D \geq 0;\\[1em]
\tau_{\rm alt}(\omega_R, D+E)(-E) \text{ for a sufficiently effective } E \geq 0,
\end{cases}
\]
where the intersection runs over all alterations $f:Y \to X := \Spec{(R)}$ from  normal integral schemes such that $K(Y)$ is contained in the fixed algebraic closure $\overline{{K}(R)}$ of ${K}(R)$. Equivalently, $\tau_+(R,\Delta) = \tau_{\rm alt}(R, \Delta) := \tau_{\rm alt}(\omega_R, K_R+\Delta)$. The latter statement follows from \cite[Corollary 4.13]{BMPSTWW1}.  


Note that when working with non-local rings, one usually cannot detect singularities using only finite maps (since this is not possible for characteristic zero rings), and so our ultimate definition of a test ideal will be based on $\tau_{\rm alt}(R, \Delta)$.
\end{remark}

\noindent \textbf{Hacon-Lamarche-Schwede test ideals.} The above definitions 
 were only introduced in the setting of local rings. The first attempt to give a meaningful definition of a global version of mixed characteristic test ideals in the general setting has been undertaken in \cite{HaconLamarcheSchwede} by Hacon, Lamarche, and the fourth author. We recall their construction in this subsection, starting with the case when $X$ is projective, after which we consider the quasi-projective case.

\begin{setting} \label{setting:proper-HLS}
Let $(R,\fram)$ be a complete Noetherian local domain such that $R/\fram$ is of characteristic $p>0$. Let $X$ be a projective normal integral scheme over $\Spec(R)$. Let $\sL$ be a very ample line bundle on $X$,  let $\Delta$ be a $\mathbf{Q}$-divisor on $X$ such that $K_X+\Delta$ is $\mathbf{Q}$-Cartier, and let $D$ be a $\mathbf{Q}$-Cartier $\mathbf{Q}$-divisor on $X$. 
\end{setting}

\noindent Briefly speaking, the idea of \cite{HaconLamarcheSchwede} is that one expects any reasonably defined test ideal $\tau(X,\Delta)$ to satisfy: $H^0({X}, \tau({X},\Delta) \otimes \cO_{X}(nL)) = \myB^0(X,\Delta; \cO_X(nL))$ for an ample line bundle $\sL = \cO_X(L)$ and $n\gg 0$. We leverage this equality to define $\tau_{\myB^0}({X},\Delta)$.     

\begin{definition}[{\cite[Definition 4.3, 4.14, 4.15, Section 6.1]{HaconLamarcheSchwede}}]
\label{def.HaconLamarcheSchwede} \label{def:HLS_test_ideal} 
We work in the situation of \autoref{setting:proper-HLS}. First we assume that $\Gamma \geq 0$ is a $\bQ$-Cartier $\bQ$-divisor. In this case, let $\mathcal{N}_i\subseteq \omega_{X}\otimes\sL^i$ be the subsheaf of $\omega_{X}\otimes\sL^i$ generated by $\myB^0( X, \Gamma; \omega_{X}\otimes\sL^i)\subseteq H^0(X, \omega_{X}\otimes\sL^i)$. Then we have that $\mathcal{N}_i\otimes\sL^{-i}\subseteq \omega_{X}$ is a subsheaf and is stable when $i\gg0$. One then defines
$$\tau_{\myB^0}(\omega_{X}, \Gamma) := \mathcal{N}_i\otimes\sL^{-i} \text{ for $i\gg0$. } $$
It was shown in \cite{HaconLamarcheSchwede} that $\tau_{\myB^0}(\omega_{X}, \Gamma)$ does not depend on the choice of $\sL$. When $\Gamma$ is a $\mathbf{Q}$-Cartier divisor $\lambda\Div(f)$, we write $\tau_{\myB^0}(\omega_{X}, f^{\lambda})$ for $\tau_{\myB^0}(\omega_{X}, \lambda\Div(f))$. When $\Gamma$ is not necessarily effective, we pick a Cartier divisor $E \geq 0$ such that $\Gamma + E \geq 0$ and define $\tau_{\myB^0}(\omega_{X},\Gamma)$ as $\tau_{\myB^0}(\omega_{X},\Gamma+E)(-E)$. Finally, we define $\tau_{\myB^0}(\sO_X,\Delta) :=\tau_{\myB^0}(\omega_X, K_X+\Delta)$.

Now suppose additionally that $\frc \subseteq \cO_X$ is an ideal sheaf and $t \geq 0$ is a rational number.  Let $f : Z \to X$ be a projective birational map from a normal scheme dominating the blowup of $\frc$ (for example, the normalized blowup) and write $\frc \cO_Z = \cO_Z(-M)$.  Suppose $\sL$ on $X$ is ample.  For each $i \geq 0$, define $\cN_i$ to be the subsheaf of $\omega_X \otimes \sL^i$ globally generated by $\myB^0(Z, tM + f^* \Gamma; \omega_Z \otimes f^* \sL^i) \subseteq H^0(X, \omega_X \otimes \sL^i)$.  

As before, we define 
\[
    \tau_{\myB^0}(\omega_X, \Gamma, \frc^t) := \cN_i \otimes \sL^{-i}
\]
for $i \gg 0$.  As above, it is independent of the choice of $\sL$. 
\end{definition}

In the case when $X$ is only a quasi-projective normal integral scheme over $\Spec(R)$, one can define $\tau_{\myB^0}(\omega_X,\Gamma)$ and $\tau_{\myB^0}(X,\Delta)$ by considering a compactification of $X$ and restricting the above definition to $X$. It was shown in \cite[Section 5]{HaconLamarcheSchwede} that this is independent of the compactification and we refer to \cite[Definition 5.10]{HaconLamarcheSchwede} for more details.




{\color{black}
\subsection{A lemma on alterations}
In this subsection we record a result domination of alterations under base change that will be used in \autoref{sec:SingOverDVR} to remove the completeness assumptions of the base. We begin with a general lemma that is well-known to experts.
\begin{lemma}
\label{lemma: small Galois lemma}
Let $R$ be a domain with fraction field $F$ and let $R\to S$ be an essentially \'etale extension with $K$ the fraction field of $S$. Let $L$ be a finite Galois extension of $F$ that contains $K$ and $R'$ be the integral closure of $R$ in $L$. Then 
$$R'\otimes_RS \cong \prod_{G/H} S_i$$
where $i$ runs over all cosets of $H\leq G$ where $G$ and $H$ are the Galois groups of $L/F$ and $L/K$ respectively. Furthermore, we have $S_i\cong S'$ for all $i$, where $S'$ is the integral closure of $S$ in $L$.
\end{lemma}
\begin{proof}
We first notice that, since $R'$ is normal and $R\to S$ is essentially \'etale, $R'\otimes_RS$ is normal. Thus $R'\otimes_RS$ is the integral closure of $S$ in $L\otimes_FK$. Since $L$ is Galois over $F$, we have $L\otimes_FL\cong \prod_{G}L$. Since $K=L^H$, we have 
$$L\otimes_FK = (\prod_G L)^H\cong \prod_{G/H} L.$$
The image of $K$ in $L\otimes_FK$ can be identified with $\prod_{G/H} K_i$ where each $K_i$ corresponds to an embedding of $K\to L$ induced by $g_i\in G$ where each $g_i$ lives in different coset of $H$ in $G$. Thus the integral closure of $S$ in $\prod_{G/H} L$ is isomorphic to $\prod_{G/H}S_i$, where $S_i$ is the integral closure of $S$ in $L$ under the map
$$S \to K \to L \xrightarrow{g_i} L.$$
Thus $S_i\cong S'$ for each $i$.
\end{proof}

\begin{lemma}
\label{damn_lemma}
    Suppose $W \to V$ is an essentially \'etale map of Noetherian integral schemes and $X \to V$ is projective with $X$ normal.  Suppose that there exists some normal alteration $Y \to X_W := X \times_V W$.  Then there exists a normal alteration $Z \to X$ so that the base change 
    \[
        Z_W := Z \times_V W \to X_W
    \]
    can be factored through $Y \to X_W$.
\end{lemma}
It will be important to us though that we think about the map $Y \to W$ in various different ways (via some Galois action).  An alternate proof, via a study of Riemann-Zariski spaces over $X^+$, is sketched below in \autoref{rem.RZ-damn-lemma}.
\begin{proof}
    Consider the finite field extension $F = K(V) \subseteq K(W) = K$.  Set $L$ to be the Galois closure of this extension.  Let $V' \subseteq L$ be the normalization of $V$ in $L$ and set $W' \subseteq L$ to be the normalization of $W$ in $L$ (we see that $W'$ is a localization at some multiplicative set of $V'$).  Using \autoref{lemma: small Galois lemma}, we see that $V' \times_V W = \coprod W_i$ where each $W_i$ is isomorphic to $W'$ over $V$.  Each $W_i$ then can be viewed as a different open subset of $V'$.   Without loss of generality, enlarging the alteration if necessary, we can assume that $Y \to X_W$ factors through $X_{W'} \to X_W$.  This yields the following diagram where subscript denotes base change.
    \[
    \xymatrix{
    && Y \ar[dd] \ar[rd]& \\
    & X_{V'} \ar[ld] \ar[dd] && X_{W'} \ar[dd]\ar[ld]\ar[ll] \\
    X \ar[dd] && X_W \ar[dd] \ar[ll] & \\
    & V'\ar[ld] && W'\ar[ld] \ar[ll] \\
    V && W\ar[ll] & 
    }.
    \]
We observe that $X_{V'} = X \times_V V'$ need be neither normal or integral, even if $X$ was. Note that 
$$X\times_V V'\times_V W \cong \coprod X\times_V W_i,$$ 
It follows that each $X \times_V W_i$ embeds as a different open set of $X \times_V V'$ (although they are all isomorphic as above).  

    Via the isomorphism $X \times_V W' \cong X \times_V W_i$, we obtain alterations 
    \[
        Y_i := Y \to X \times_V W' \xrightarrow{\sim} X \times_V W_i,
    \]
    We form the following diagram:
    \[
        \xymatrix{
            Y_i \ar@{^{(}->}[r]^-{\text{open}} \ar[d] & \overline{Y_i} \ar[d] \\
            X \times_V W_i \ar@{^{(}->}[r]_-{\text{open}} & X\times_V {V'}. 
        }
    \]
    where $\overline{Y_i} \to X\times_V{V'}$ is a normal projective morphism compactifying $Y_i \to X \times_V W_i$.  Choose $Z \to X_{V'}$ a normal alteration with factorizations $Z \to \overline{Y_i} \to X\times_V{V'}$ for each $i$.  Notice that $\overline{Y_i} \times_{V'} W_i \cong Y_i$. 

    Consider the map $Z \to X_{V'} \to X$.  We form the base change with $W$, 
    \[
        Z \times_V W \to X \times_V V' \times_V W \to X \times_V W.
    \]
    By \autoref{lemma: small Galois lemma} we have that 
    $$X \times_V V' \times_V W \cong \coprod_j X \times_V W_j.$$  We can thus write $Z_W := Z \times_V W = \coprod_j Z_j$ where $Z_j$ is the components of $Z_W$ mapping into $X \times_V W_j$ via the map above.  Indeed, we can also obtain this decomposition via 
    $$Z_W= Z \times_{V'} V' \times_V W = \coprod_j Z \times_{V'} W_j$$ 
    so we see that $Z_j = Z \times_{V'} W_j$.  Now, we have the following map
    \[
        Z_W = \coprod_j Z \times_{V'} W_j \to \coprod_j \overline{Y_j} \times_{V'} W_j = \coprod_j Y_j \to \coprod_j X \times_V V' = X \times_V V' \times_V W \to X_W.
    \]
    In particular, we have a factorization $Z_j \to Y_j \to X_W$ as desired.
\end{proof}

\begin{remark}
\label{rem.RZ-damn-lemma}
We also give an explanation of \autoref{damn_lemma} via Riemann-Zariski spaces. All schemes and morphisms in this remark are qcqs. With the same notations as in \autoref{damn_lemma}, fix a geometric generic point $\overline{\eta} \to X$ and let $X^+ \to X$ be the corresponding absolute integral closure of $X$, i.e., the inverse limit over all finite covers of $X$ dominated by $\overline{\eta}$.  Let $\mathrm{RZ}(X^+)\to X^+$ be the Riemann-Zariski space of the integral scheme $X^+$, i.e., $\mathrm{RZ}(X^+) = \lim_j Z_i$ the inverse limit (in locally ringed spaces) of all proper maps $Z_j \to X^+$ which are isomorphisms over the generic point $\overline{\eta} \in X^+$.

Consider the base change $X^+ \times_V W$; this is \'etale over $X^+$. As any \'etale $X^+$-scheme is a finite disjoint union of open subschemes of $X^+$, we can write $X^+ \times_V W = \sqcup F_i$, with each $F_i$ mapping to $X^+$ via an open immersion. In particular, each $F_i$ is an absolute integrally closed integral schemes. As the projection to $X_W$ is pro-finite, it follows that $F_i \to X_W$ is an absolute integral closure of $X_W$ for each $i$.

Next, consider the base change $\mathrm{RZ}(X^+) \times_V W = \lim_j Z_j \times_V W$ (in locally ringed spaces). The base change map $\mathrm{RZ}(X^+) \times_V W \to X^+ \times_V W$ becomes $R_i := \mathrm{RZ}(F_i) \to F_i$ over the component $F_i \subset X^+ \times_V W$: indeed, $F_i \to X^+ \times_V W \to X^+$ is an open immersion and the formation of Riemann-Zariski spaces commutes with restrictions to open subsets of the base.  In particular, for any $i$, the composition $R_i \to F_i \to X_W$ must lift through any given alteration $Y \to X_W$ (everything in $X_W$-schemes). Choosing such lifts for each $i$ separately, we learn that the map $\mathrm{RZ}(X^+) \times_V W = \lim_j Z_j \times_V W \to X_W$ has the same property. As maps from a cofiltered limit to a finitely presented scheme factor through some finite stage, it follows that for $j$ large enough, we have a factorization $Z_j \times_V W \to Y \to X_W$ of $X_W$-schemes. Writing $Z_j$ as the inverse limit of alterations of $X$, another approximation argument then provides the desired alteration $Z \to X$.
\end{remark}
}
\section{A review of perverse sheaves}
\label{sec:preliminaries_perverse_sheaves}
We start by giving a brief introduction for non-experts to perverse $t$-structures for varieties over a field, both in the setting of (quasi) coherent sheaves, as well as constructible sheaves with coefficients $K=\mathbf{Z}/n$. In the rest of this section, we will discuss perverse $t$-structures: 
\begin{itemize}
    \item  for quasi-coherent sheaves on formal schemes $\mathcal X$ over a perfectoid ring ${\cO_C}$.
    \item  for constructible sheaves on the generic fibre $\mathfrak X_{C}$ with $\mathbf{Z}_p$-coefficients.
\end{itemize}


\subsection{Review: perverse coherent sheaves and perverse constructible sheaves} We give a quick exposition on perverse sheaves; for simplicity, we assume $X$ is a finite type separated scheme over a field in this subsection, and we work in the bounded derived category $D^b$, see \cite{Gabber.tStruc} for more general results (for general Noetherian schemes and regarding unbounded derived category). We will work with the middle perverse $t$-structure.
\begin{definition}
\label{def.PervQCohScheme}
For $K \in D^b_{\qcoh}(X)$, we say that
\begin{enumerate}
\item[(1)] $K \in {}^pD^{\leq 0}(X)$ if and only if $K_x \in D^{\leq - {\dim} \overline{\{x\}}}$ for all $x \in X$,
\item[(2)] $K \in {}^pD^{\geq 0}(X)$ if and only if $\myR\Gamma_x(K_x) \in D^{\geq - {\dim} \overline{\{x\}}}$ for all $x \in X$.
\end{enumerate}
This defines a $t$-structure on $D^b_{\qcoh}(X)$ so that 
\[
{\rm Perv}_{\qcoh}(X) := D^b_{\qcoh}(X)^{\heartsuit} = {}^pD^{\leq 0}(X) \cap {}^pD^{\geq 0}(X)
\]
is an Abelian category. Furthermore, the $t$-structure on $D^b_{\qcoh}(X)$ above restricts to a $t$-structure on $D^b_{\coherent}(X)$, which we call the perverse $t$-structure on $D^b_{\coherent}(X)$ and we have
\[
{\rm Perv}_{\coherent}(X) := D^b_{\coherent}(X)^{\heartsuit} = D^b_{\qcoh}(X)^{\heartsuit} \cap D^b_{\coherent}(X).
\]
Finally, for $K \in D^b_{\qcoh}(X)$, we define its perverse cohomology 
$${}^pH^i(K) := ({}^p\tau^{\leq 0})({}^p\tau^{\geq 0})(K[i]) \in {\rm Perv}_{\qcoh}(X). $$
\end{definition}

First, we point out that condition (1) above can be checked by local cohomology (see also \autoref{rem:local-cohomology-is-perverse-t-exact} and \autoref{remark:equivalent-characterisation-of-perverse-structure-coherent-case}):

\begin{lemma}
\label{lem.LocalCohomPerv}
With notation as above, we have $K\in {}^pD^{\leq 0}(X)$ if and only if for all $x\in X$, we have $\myR\Gamma_x(K_x) \in D^{\leq -\dim(\overline{\{x\}})}$. 
\end{lemma}
\begin{proof}
First of all, we note that $K\in {}^pD^{\leq 0}(X)$ if and only if $\dim \Supp H^i(K) \leq -i$.\footnote{Since we are not assuming $K\in D^b_{\coherent}(X)$, $\Supp H^i(K)$ might not be closed, and we define the dimension of this support as Krull dimension, i.e., $\dim \Supp H^i(K)= \sup\{\dim(\overline{\{y\}})\}$ where $y\in \Supp H^i(K)$.} It is clear that this condition implies that $\myR\Gamma_x(K_x)\in D^{\leq -\dim(\overline{\{x\}})}$ for all $x\in X$: this comes down to the standard fact that over a Noetherian local ring $(R,\m)$, any $R$-module $M$ with $\dim\Supp(M)\leq j$ has $H^{>j}\myR\Gamma_\m(M)=0$. 

Conversely, suppose $\myR\Gamma_x(K_x)\in D^{\leq -\dim(\overline{\{x\}})}$ for all $x\in X$. If $\dim\Supp H^i(K) > -i$ for some $i$, then we choose a minimum such $i$ (which is at worst $-d$ since the condition is vacuous for $i< -d$) and choose a point $y\in \Supp H^i(K)$ such that $$\dim(\overline{\{y\}})=\dim \Supp H^i(K) > -i.$$

Since $\dim \Supp H^j(K)_y \leq -j$, and by our choice of $i$ and $y$ we have that 
$$\dim \Supp H^j(K)_y \leq -j - \dim\overline{\{y\}} < i-j \text{ for all $j<i$, and } \dim \Supp H^i(K)_y =0.$$
Using the standard fact again, it follows that $H^i\myR\Gamma_y(K_y) \cong H^i(K)_y\neq 0$ which contradicts the assumption $\myR\Gamma_y(K_y) \in D^{\leq -\dim(\overline{\{y\}})}$.
\end{proof}


\begin{remark}
By \autoref{lem.LocalCohomPerv}, a complex $K \in {\rm Perv}_{\coherent}(X)$ if and only if 
\[
\myR\Gamma_x(K_x)[-{\dim} \overline{\{x\}}] \in {\rm Coh(X)}
\]
for all $x\in X$. Therefore, up to shift, \emph{perverse complexes} are closely related to \emph{maximal Cohen-Macaulay complexes} at the stalks\footnote{Note that, there is a non-triviality condition on maximal Cohen-Macaulay complexes which may not be satisfied for all perverse complexes. For example, if $X$ is equidimensional, then $\omega_X^\mydot$ is always perverse under our definition, but $\omega_{X,x}^\mydot[-\dim\overline{\{x\}}]$ may not be maximal Cohen-Macaulay in the sense of \cite{IyengarMaSchwedeWalkerMCM} by the main example in \cite{MaSinghWaltherKoszulvsLocal}.}, see \cite{IyengarMaSchwedeWalkerMCM}. In fact, for a coherent sheaf $K$, $K[\dim(X)]$ is perverse if and only if $K_x$ is a maximal Cohen-Macaulay $\cO_{X,x}$-module for all $x\in X$. 
\end{remark}

By Matlis and local duality, the perverse $t$-structure on $D^b_{\coherent}(X)$ is the Grothendieck-dual of the usual $t$-structure on $D^b_{\coherent}(X)$. More precisely, define $\mathbf{D} \colon D^b_{\coherent}(X) \to D^b_{\coherent}(X)$ via
\[
\mathbf{D}(M) := \myR\sHom(M,\omega^{\mydot}_X).
\]
We have the following characterization.

\begin{lemma} \label{lem:coherent-perverse-dual-baby}
Suppose $X$ is equidimensional and $K \in D^b_{\coherent}(X)$. Then
\begin{enumerate}
\item[(a)] $K \in {}^pD^{\leq 0}(X)$ if and only if $\mathbf{D}(K) \in D^{\geq 0}$,
\item[(b)] $K \in {}^pD^{\geq 0}(X)$ if and only if $\mathbf{D}(K) \in D^{\leq 0}$.
\end{enumerate}
In particular, $K \in {\rm Perv}_{\coherent}(X)$ if and only if $\mathbf{D}(K)$ is concentrated in degree $0$. 
\end{lemma}
\begin{proof}
This follows from \autoref{def.PervQCohScheme}, \autoref{lem.LocalCohomPerv}, and local duality (note that, equidimensionality of $X$ is used to guarantee that $\omega_{X,x}^\mydot[-\dim\overline{\{x\}}]$ is a normalized dualizing complex for $\cO_{X,x}$).
\end{proof}

\begin{remark} \hphantom{a}

\begin{enumerate}
\item Since connectiveness of coconnectiveness of $\mathbf{D}(K)$ can be checked after localization at closed points, to show $K \in {}^pD^{\leq 0}(X)$ (resp. $K \in {}^pD^{\geq 0}(X)$), we only need to show $\myR\Gamma_x(K_x)\in D^{\leq 0}$ (resp. $\myR\Gamma_x(K_x)\in D^{\geq 0}$)  for all closed points $x \in X$.
\item \autoref{lem:coherent-perverse-dual-baby} does not hold for quasi-coherent sheaves: in fact, ${\bf D}(M)$ may not be a quasi-coherent sheaf. This is a crucial source of difficulties in the next subsections.
\item It follows from \autoref{lem:coherent-perverse-dual-baby} that $H^{-i}(\mathbf{D}(M)) = \mathbf{D}({}^pH^{i}(M))$ when $X$ is equidimensional.
\end{enumerate}
\end{remark}

\begin{example} Let $i \colon Z \hookrightarrow X$ be a closed immersion and suppose $X, Z$ are equidimensional. Then we have
\begin{enumerate}
    \item $i_*\cO_Z[\dim(Z)]$ is perverse if $Z$ is Cohen-Macaulay;
    \item $i_*\omega_Z^\mydot$ is perverse, in particular, $\omega^\mydot_X$ is perverse.
\end{enumerate}
\end{example}

\vspace{0.5em}

We next recall the definition of the perverse $t$-structure on the derived category of constructible $\mathbf{Z}/n$-sheaves for the \'etale topology, where $n \in \bN$ is invertible in $X$. We note that everything also holds true for other coefficients such as $\bQ$ or $\bC$, if $X$ is defined over the complex numbers and we work in the analytic topology. We refer to \cite{BBDG82} for more details and proofs (therein the coefficients are set to be equal to $\overline{\bQ_l}$, but as stated on page 101, the result holds for coefficients $\mathbf{Z}/n$ as well).


\begin{definition}
For $K \in D^b_{\cons}(X,\mathbf{Z}/n)$, we say that
\begin{enumerate}
\item[(1)] $K \in {}^pD^{\leq 0}(X, \mathbf{Z}/n)$ if and only if $K_x \in D^{\leq - {\dim} \overline{\{x\}}}$ for all geometric points $i \colon \overline{x} \to X$,
\item[(2)] $K \in {}^pD^{\geq 0}(X, \mathbf{Z}/n)$ if and only if $\myR\Gamma_x(K_x) \in D^{\geq - {\dim} \overline{\{x\}}}$ for all geometric points $i \colon \overline{x} \to X$.
\end{enumerate}
Here $K_x=i^*K$ and $\myR\Gamma_x(K_x)=i^!K$. 
This defines a $t$-structure on $D^b_{\cons}(X,\mathbf{Z}/n)$ so that 
\[
{\rm Perv}_{\cons}(X,\mathbf{Z}/n) := D^b_{\cons}(X,\mathbf{Z}/n)^{\heartsuit} = {}^pD^{\leq 0}(X,\mathbf{Z}/n) \cap {}^pD^{\geq 0}(X,\mathbf{Z}/n)
\]
is an Abelian category. We define perverse cohomology in the constructible setting similarly, that is, for $K \in D^b_{\cons}(X)$,  
$${}^pH^i(K) := {}^p\tau_{\leq 0}{}^p\tau_{\geq 0}(K[i]) \in {\rm Perv}_{\cons}(X).$$
\end{definition}

\begin{remark}
Note that, $i^!K \cong \myR\Gamma_{x}(\mathrm{Spec}(\cO_{X,\bar{x}}^{sh}),K_x)$ and we will use the latter notation in the sequel (see \autoref{PervIntBig}). Alternatively, let $Z\subseteq X$ be the irreducible closed subscheme of $X$ whose generic point is $x$, then $i^!K$ can be computed as $i^*\underline{\myR\Gamma}_Z(K)$, where $\underline{\myR\Gamma}_Z(-)$ is the sheaf version of $\myR\Gamma_Z(-)$.\footnote{That is, \'etale locally, it is the the fiber of $\myR\Gamma(X, -)\to \myR\Gamma(X\backslash Z, -)$.} To see this, note that $i: \overline{x}\to X$ can be factored in two ways: $\overline{x}\xrightarrow{a} Z\xrightarrow{b} X$ and $\overline{x}\xrightarrow{c} \Spec(\cO_{X,\bar{x}}^{sh}) \xrightarrow{d} X$. In general, for a closed immersion $f$: $Y\to X$, $f^!(K)=f^*\underline{\myR\Gamma}_Y(K)$ and for an (ind-)\'etale morphism $g$, $g^!=g^*$. Using these two facts, $i^!=a^!b^!$ is computed by $i^*\underline{\myR\Gamma}_Z(K)$, while $i^!=c^!d^!$ is also computed by $\myR\Gamma_{x}(\mathrm{Spec}(\cO_{X,\bar{x}}^{sh}),K_x)$.
\end{remark}

\begin{example}\hphantom{a}

\begin{enumerate}
    \item If $i \colon Z \hookrightarrow X$ is a closed immersion, then $i_*\mathbf{Z}/n[\dim(Z)] \in {}^pD^{\leq 0}$ (in particular, $\mathbf{Z}/n[\dim(X)]\in {}^pD^{\leq 0}$). If $X, Z$ are both smooth, then $i_*\mathbf{Z}/n[\dim(Z)]$ is perverse (in particular, $\mathbf{Z}/n[\dim(X)]$ is perverse when $X$ is smooth). 
    \item In general, the \emph{intersection complex} ${\rm IC}_X$ is always perverse when $X$ is integral, where  
\[
{\rm IC}_X := \Image(j_!\mathbf{Z}/n[\dim(X)] \to j_* \mathbf{Z}/n[\dim(X)]) 
\]    
for an open embedding $j \colon U \hookrightarrow X$ of a smooth affine variety $U$. Here we are using Artin's vanishing which states that $j_!\mathbf{Z}/n[\dim(X)]$ and $j_* \mathbf{Z}/n[\dim(X)]$ are perverse as $U$ is smooth affine (see \autoref{remark:all-properties-of-perverse-Fp-coefficients}(h) below). 
\end{enumerate}  
\end{example}

\begin{remark}
We caution the readers that the analog of \autoref{lem.LocalCohomPerv} for perverse $t$-structure of constructible $\mathbf{Z}/n$-complexes is false. For example, if $X$ is a smooth curve over $\mathbf{C}$ and let $i: x \to X$ be a closed point. Then we know that $H^2(i^!\mathbf{Z}/n)\neq 0$ (see \cite[Proposition 14.3 and Remark 14.4]{milneLEC}). Now we have $\mathbf{Z}/n[1]\in {}^pD^{\leq 0}$ but $\myR\Gamma_x((\mathbf{Z}/n[1])_x) \notin D^{\leq 0}$. 
\end{remark}

We have Verdier duality $\mathbf{D} \colon D^b_{\cons}(X, \mathbf{Z}/n) \to D^b_{\cons}(X, \mathbf{Z}/n)$, which is defined via
\[
\mathbf{D}(M)= \myR\sHom(M,\omega_{X,\et}),
\] 
where $\omega_{X,\et} = f^! k$ for the projection $f \colon X \to \Spec(k)$. In contrast to the setting of coherent sheaves, we have $\mathbf{D}({}^pD^{\leq 0}(X, \mathbf{Z}/n)) = {}^pD^{\geq 0}(X, \mathbf{Z}/n)$ and $\mathbf{D}({}^pD^{\geq 0}(X, \mathbf{Z}/n)) = {}^pD^{\leq 0}(X, \mathbf{Z}/n)$. In particular, Verdier duality induces a self duality: 
\[
\mathbf{D} \colon {\rm Perv}_{\cons}(X,\mathbf{Z}/n) \to {\rm Perv}_{\cons}(X,\mathbf{Z}/n).
\]

We next record some fundamental properties of constructible sheaves and the six-functor formalism in the following remark.

\begin{remark} \label{remark:all-properties-of-perverse-Fp-coefficients} Let $f \colon X \to Y$ be a morphism of schemes of finite type over a field of characteristic zero. Let $M \in D^b_{\cons}(X, \mathbf{Z}/n)$ and $N \in D^b_{\cons}(Y, \mathbf{Z}/n)$. We consider derived functors $f_*, f_!, f^*, f^!$ between $D^b_{\cons}(X, \mathbf{Z}/n)$ and $D^b_{\cons}(Y, \mathbf{Z}/n)$. Then
\begin{enumerate}
    \item $\mathbf{D}(f_*M) = f_!\mathbf{D}(M)$ and $\mathbf{D}(f^*N) = f^!\mathbf{D}(N)$ (this follows from Verdier duality, see for example \cite[Corollary 17.16]{BhattPerverseNotes}). 
    \item If dimension of the fibers of $f$ are at most $d$, then (\cite[4.2.4]{BBDG82}):
    \begin{itemize}
        \item $f_!$ and $f^*$ are of cohomological amplitude at most $d$, 
        \item $f^!$ and $f_*$ are of cohomological amplitude at least $-d$.
    \end{itemize}
In particular, if $d=0$, then $f_!$, $f^*$ are $t$-right exact, while $f^!$,  $f_*$ are $t$-left exact.
    \item If $f$ is affine, then $f_!$ is left $t$-exact \cite[Theorem 4.1.2]{BBDG82}. 
    \item If $f$ is proper, then $f_!=f_*$ (\cite[p.\ 108]{BBDG82}). 
    \item If $f$ is finite, then $f_!=f_*$ is $t$-exact (by (b), (c), (d)). 
    \item If $f$ is smooth of relative dimension $d$, then $f^! = f^*[2d](d)$ (\cite[p.\ 108]{BBDG82}).
    \item If $f$ is an open immersion, then $f^!=f^*$ and these are $t$-exact (by (b), (f)). 
    \item If $f$ is an affine open immersion, then $f_*$ and $f_!$ are $t$-exact (by (a), (b), (c)).
\end{enumerate}
\end{remark}

\begin{remark} 
\label{remark:only-check-coconnective} 
In what follows, we will implicitly use the fact that given a perverse structure on a derived category, it is uniquely characterized by the coconnective (or equivalently, the connective) part. This follows from the exact triangle ${}^p\tau^{\leq 0}(K) \to K \to {}^p\tau^{\geq 1}(K) \xrightarrow{+1}$.
\end{remark}

\subsection{Review: integral perverse coherent sheaves}
\label{subsec:PervQC}

In this section, we give a quick exposition of the theory of perverse quasi-coherent sheaves on finitely presented flat $p$-adic formal schemes $\mathcal{X}$ over a perfectoid valuation ring. Our main goal is to gain access to the notion of perverse almost coherent sheaves, which one could simply define as Grothendieck duals of usual almost coherent sheaves. It will however be quite convenient to define the perverse $t$-structure on the entirety of $D_{\qcoh}(\mathcal{X})$ before passing to the subcategory of almost coherent objects. Consequently, we first construct (\autoref{Pervqcbig}) the perverse $t$-structure without Grothendieck duality by relying on results of Gabber \cite{Gabber.tStruc} and then give the promised description via duality on almost coherent objects (\autoref{PervAlmostCoh}). The presentation borrows from \cite{BhattLuriepadicRHmodp}, and omitted details can be found there.

\begin{notation}[$p$-adic formal schemes and their quasi-coherent derived categories]
\label{not:padicformal}
Let $\cO_C$ be the ring of integers of a perfectoid extension $C/\mathbf{Q}_p$. Let $\mathcal{X}/\cO_C$ be a finitely presented flat $p$-adic formal scheme of relative dimension $d$, i.e., $\mathcal{X}$ is qcqs and for each affine open $\mathrm{Spf}(R) \subset \mathcal{X}$, the ring $R$ is $p$-complete and $p$-torsionfree with $R/p$ finitely presented over $\cO_C/p$ and having Krull dimension $d$.

We shall write $D_{\qcoh}(\mathcal{X})$ for the quasi-coherent derived category of $\mathcal{X}$; recall that this is defined as the full subcategory of $D(\mathcal{X}, \cO_{\mathcal{X}})$ spanned by derived $p$-complete complexes $K$ with the property that $K/p \in D_{\qcoh}(\mathcal{X}_{p=0}) \subset D(\mathcal{X}, \cO_{\mathcal{X}}/p)$. When $\mathcal{X}=\mathrm{Spf}(R)$ is affine, the global sections functor yields an equivalence $D_{\qcoh}(\mathcal{X}) \simeq D_{\pcomp}(R)$; in this case, the global sections functor $\myR\Gamma(\mathcal{X},-)$ is simply the forgetful functor $D_{\pcomp}(R) \to D(\mathrm{Ab})$ (which does not commute with colimits).

Since $\cO_c$ is a perfectoid valuation ring, $(\cO_C, \m_{\cO_C})$ is an almost mathematics setup. We write $D_{\qcoh}(\mathcal{X})^a$ for almost variant of $D_{\qcoh}(\mathcal{X})$, i.e., the quotient of the latter by its full subcategory of almost zero complexes. Let $D^b_{\acoh}(\mathcal{X}) \subset D_{\qcoh}(\mathcal{X})^a$ be the full subcategory spanned by bounded almost coherent complexes. Note that objects in $D^b_{\acoh}(\mathcal{X})$ only live in the almost world; we shall occasionally use the $!$-realization to view them as honest complexes. We refer to \cite{ZavyalovAlmostCoh} for an exposition of the theory of almost coherent sheaves.
\end{notation}

Let us explain the shape of the perverse $t$-structure on $D_{\qcoh}(\mathcal{X})$. The analog for $D_{\qcoh}(\mathcal{X}_{p^n=0})$ is classical and can be immediately deduced from general results in \cite{Gabber.tStruc}. To pass to the limit, the main idea is to use the complete-torsion equivalence to transport the perverse $t$-structure on $p^\infty$-torsion complexes (resulting, essentially, from the previous sentence) to one on $p$-complete complexes.

\begin{theorem}[Perverse quasi-coherent complexes]
\label{Pervqcbig}
There is a perverse $t$-structure on $D_{\qcoh}(\mathcal{X})$ characterized by either of the following conditions: 
\begin{enumerate}
\item $K \in {}^p D^{\leq 0}_{\qcoh}(\mathcal{X})$ if and only if for any $x \in \mathcal{X}$, we have $\myR\Gamma_p(K)_x \in D^{\leq -\dim(\overline{\{x\}})+1}$.
\item $K \in {}^p D^{\geq 0}_{\qcoh}(\mathcal{X})$ if any of the following equivalent conditions are satisfied:
\begin{enumerate}
\item For any $x \in \mathcal{X}$, we have $\myR\Gamma_{x}(\myR\Gamma_p(K)_x) \in D^{\geq -\dim( \overline{\{x\}})+1}$.
\item For any $x \in \mathcal{X}$, we have $\myR\Gamma_{x}(K_x/p) \in D^{\geq -\dim( \overline{\{x\}})}$.
\end{enumerate}
\end{enumerate}
\end{theorem}
Here $\myR\Gamma_p(K)$ is the local cohomology complex of $K$ with respect to $(p)$ computed in the big derived category $D(\mathcal{X}, \cO_{\mathcal{X}})$ of all $\cO_{\mathcal{X}}$-modules on $\mathcal{X}$, i.e., $\myR\Gamma_p(K) = \left(K[1/p]/K\right)[-1]$, where $K[1/p] = K \otimes_{\mathbf{Z}}^L \mathbf{Z}[1/p]$. Note that the inclusion $D_{\qcoh}(\mathcal{X}) \subset D(\mathcal{X}, \cO_{\mathcal{X}})$ does not preserve colimits, and in fact $\myR\Gamma_p(K)$ typically does not lie in $D_{\qcoh}(\mathcal{X})$.



\begin{remark}[Why use $\myR\Gamma_p(K)?$]
The main reason we prefer to write $\myR\Gamma_p(K)_x$ in the statement of the theorem is that the functor $K \mapsto \myR\Gamma_p(K)_x$ commutes with filtered colimits, unlike the functor $K \mapsto K_x$. In particular, the compatibility of the $t$-structure with filtered colimits (\autoref{PervCohColimit}) is obvious from the above definitions in terms of $\myR\Gamma_p(K)$, but less obvious otherwise.

Another advantage of using $\myR\Gamma_p(K)$ is the interpretation of $\myR\Gamma_{x}(\myR\Gamma_p(K)_x)$: to define $\myR\Gamma_{x}(M_x)$, one can work locally to assume $x$ is a closed point of $\mathcal{X}={\rm Spf}(R)$ and $M_x$ corresponds to a complex of $R$-modules, so $\myR\Gamma_{x}(M_x)$ is the local cohomology complex. Interpreted this way, in the statement of \autoref{Pervqcbig}, we could replace each occurrence of $\myR\Gamma_p(K)_x$ by $K_x$ since applying $\myR\Gamma_x$ annihilates the cone of $\myR\Gamma_p(K)_x \to K_x$ as $p$ acts invertibly on this cone (on a $p$-adic formal scheme $\mathcal{X}$, any $x \in \mathcal{X}$ is already a point of characteristic $p$, i.e., $p$ lies in the ideal sheaf of $x$). However, at least when $x\in \mathcal{X}$ is a closed point, one can also define $\myR\Gamma_x(-)$ as the fiber of $\myR\Gamma(\mathcal{X}, -) \to \myR\Gamma(\mathcal X \backslash\{x\}, -)$. Under this definition, $\myR\Gamma_{x}(\myR\Gamma_p(M))$ and $\myR\Gamma_x(M)$ are usually {\it not} the same for $M\in D^b_{\qcoh}(\mathcal X)$: the latter object, being derived $p$-complete as it is the fiber of two derived $p$-complete complexes, is exactly the derived $p$-completion of the former. In the sequel, what we need is $\myR\Gamma_{x}(\myR\Gamma_p(M))$ and this is another reason we use $\myR\Gamma_p(K)$ -- since then there are no ambiguities of these two interpretations of $\myR\Gamma_x(-)$. 
\end{remark}

\begin{proof}[Proof of \autoref{Pervqcbig}]
We first explain the equivalence between the two conditions appearing in (2). The complex $M = \myR\Gamma_x(\myR\Gamma_p(K)_x)$ is a $p^\infty$-torsion complex with $M/p = \myR\Gamma_x(K_x/p)$, 
so the claim follows from a general fact: for a $p^\infty$-torsion complex $M$, we have $M \in D^{\geq 0}$ if and only if $M/p \in D^{\geq -1}$.

To construct the $t$-structure, assume first that $\mathcal{X}=\mathrm{Spf}(R)$ is affine. We shall construct a perverse $t$-structure on $D(R)$ using \cite{Gabber.tStruc}, and apply  it to induce one on $D_{\pcomp}(R)$ via the complete-torsion equivalence. To implement this strategy, we first recall some basic properties of $\mathrm{Spec}(R)$. 
\begin{itemize}
\item The scheme $\mathrm{Spec}(R/p)$ is topologically Noetherian, i.e., the topological space $\mathrm{Spec}(R/p)$ is Noetherian: in fact, $(R/p)_{\reduced}$ is a finitely presented algebra over the residue field of $\cO_C$.
\item The scheme $\mathrm{Spec}(R[1/p])$ is Noetherian of Krull dimension equal to $\dim(\mathrm{Spec}(R/p))$: this is a standard fact in the basic theory of affinoid algebras \cite{BGR}.
\item The space $\mathrm{Spec}(R)$ is topologically Noetherian: this follows from the previous two properties. 
\item For any $p$-torsionfree quotient $S = R/J$, the ring $S$ is a topologically finitely presented $\cO_C$-algebra ({that is, it is finitely presented modulo $p$}, see \cite[Proposition 1.1 (c)]{BLR1}), and hence $\mathrm{Spec}(S)$ enjoys the same properties as the ones listed above. 
\end{itemize}


Using these properties, we construct a perversity function.  Specifically,  define a function $p:\mathrm{Spec}(R) \to \mathbf{Z}$ by setting $-p(x)+1$ to be the dimension of the closure of $\{x\}$ in the fibre of $\mathrm{Spec}(R) \to \mathrm{Spec}(\cO_C)$ containing $x$. We claim that this is a strong perversity function in the sense of \cite[\S 1]{Gabber.tStruc}, i.e., for each integer $n$, the sets 
\[ 
    \mathrm{Spec}(R)_{\geq n} = \{x \in \mathrm{Spec}(R) \mid p(x) \geq n\}
\]
are ind-constructible and closed under specialization. In fact, it suffices to show just the latter:  any set closed under specialization can be written as a union of closed sets (by taking closures of the points within the set), and any closed set is constructible by topological Noetherianity. The closedness under specialization is clear for specializations entirely contained in a single fibre of $\mathrm{Spec}(R) \to \mathrm{Spec}(\cO_C)$, so fix a specialization $x \rightsquigarrow y$ with $x \in \mathrm{Spec}(R[1/p]) \subset \mathrm{Spec}(R)$ and $y \in \mathrm{Spec}(R/p) \subset \mathrm{Spec}(R)$. The closure $\overline{\{x\}} \subset \mathrm{Spec}(R)$ has the form $\mathrm{Spec}(S)$ for a topologically finitely presented $\cO_C$-flat quotient $S$ of $R$ by the properties recalled above; in particular, we have $\dim(\mathrm{Spec}(S[1/p])) = \dim(\mathrm{Spec}(S/p))$. As the closure of $x$ in $\mathrm{Spec}(R[1/p])$ equals $\mathrm{Spec}(S[1/p])$ while the closure of $y$ in $\mathrm{Spec}(R/p)$ is contained in $\mathrm{Spec}(S/p)$, we conclude that $-p(x)+1 \geq -p(y)+1$, as wanted.

Thanks to the previous paragraph and the results of \cite[Remark 6.1 (6)]{Gabber.tStruc},
we obtain a $t$-structure on $D(R)$ characterized as follows:
\begin{itemize}
\item $K \in {}^p D^{\leq 0}(R)$ if and only if for any $x \in \mathrm{Spec}(R)$, we have $K_x \in D^{\leq p(x)}$.
\item $K \in {}^p D^{\geq 0}(R)$ if and only if for any $x \in \mathrm{Spec}(R)$, we have $\myR\Gamma_{x}(K_x) \in D^{\geq p(x)}$.
\end{itemize}
We remark that the definition in \cite[\S 1]{Gabber.tStruc} also requires that $K \in D^+$ for $K \in {}^p D^{\geq 0}$; in our case, as the function $p$ is bounded, this follows from the condition appearing above and \cite[Lemma 3.1]{Gabber.tStruc}. 

The exact same conditions as above (with $\mathrm{Spec}(R)$ replaced by $\mathrm{Spec}(R[1/p])$) also induce a perverse $t$-structure on $D(R[1/p])$ by similar reasoning, and the localization functor $D(R) \to D(R[1/p])$ is $t$-exact. It follows that the kernel $D_{p-nilp}(R) \subset D(R)$ of this localization also inherits a perverse $t$-structure. Note that for checking perverse (co)connectivity of $K \in D_{\ptors}(R)$, the conditions appearing above only need to be checked for $x \in \mathrm{Spec}(R)_{p=0} \subset \mathrm{Spec}(R)$: the complex $K_x$ vanishes if $x \in \mathrm{Spec}(R) - \mathrm{Spec}(R)_{p=0} = \mathrm{Spec}(R[1/p])$. Transporting this $t$-structure along the complete-torsion equivalence $D_{\ptors}(R) \simeq D_{\pcomp}(R)$, then yields a perverse $t$-structure on $D_{\pcomp}(R)$. Unwinding definitions, this $t$-structure is characterized as follows: 
\begin{itemize}
\item $K \in {}^p D^{\leq 0}_{\pcomp}(R)$ if and only if for any $x \in \mathrm{Spec}(R)_{p=0}$, we have $(\myR\Gamma_p(K))_x \in D^{\leq p(x)}$.
\item $K \in {}^p D^{\geq 0}_{\pcomp}(R)$ if and only if for any $x \in \mathrm{Spec}(R)_{p=0}$, we have $\myR\Gamma_x(\myR\Gamma_p(K)_x) \in D^{\geq p(x)}$.
\end{itemize}
As these are exactly the conditions appearing in \autoref{Pervqcbig}, we win in the affine case.

The case of a general $\mathcal{X}$ can be proven by glueing together the statement for affine $\mathcal{X}$ shown above; we omit the details. 
\end{proof}

\begin{remark}[Compatibility with filtered colimits]
\label{PervCohColimit}
The categories ${}^p D^{\leq 0}$ and ${}^p D^{\geq 0}$ appearing in \autoref{Pervqcbig} contain and are stable under filtered colimits in $D_{\qcoh}(\mathcal{X})$. Consequently, the functor of taking perverse cohomology sheaves commutes with filtered colimits.
\end{remark}

\begin{remark}[Passage to the almost category]
The $t$-structure on $D_{\qcoh}(\mathcal{X})$ constructed in \autoref{Pervqcbig} induces one on the full subcategory of almost zero objects with the inclusion of the latter being $t$-exact. Consequently, one also obtains a perverse $t$-structure on the almost category $D_{\qcoh}(\mathcal{X})^a$. 
\end{remark}

\begin{remark}[Local cohomology is perverse $t$-exact] 
\label{rem:local-cohomology-is-perverse-t-exact}
For any $x \in \mathcal{X}$, the functor $D_{\qcoh}(\mathcal{X}) \to D(\mathrm{Ab})$ given by $K \mapsto \myR\Gamma_{x}(\myR\Gamma_p(K)_x)$ is $t$-exact for the perverse $t$-structure on the source and the $(-\dim(\overline{\{x\}})+1)$-shift of the standard $t$-structure on the target. To check this, by the condition in (b)(i), it suffices to show $t$-right exactness, i.e., that $K \mapsto \myR\Gamma_{x}(\myR\Gamma_p(K)_x)$ carries ${}^p D^{\leq 0}$ into $D^{\leq -\dim(\overline{\{x\}})+1}$. 

Now the condition $K \in {}^ p D^{\leq 0}$ implies that 
\begin{equation} 
\label{eq:lc-perverse-t-exact-support}
\mathrm{Supp}(H^i(\myR\Gamma_p(K))) = \{x \in \mathcal X \mid H^i(\myR\Gamma_p(K)_x) \neq 0\} \subset \{x \in \mathcal X \mid \dim(\overline{\{x\}}) \leq 1-i\}.
\end{equation}
Set $M := H^i(\myR\Gamma_p(K))$ and consider $\Supp M \subseteq \mathcal{X}_{p=0}$, we claim that 
\[ 
\myR\Gamma_x (M_x) \in D^{\leq \dim (\Supp M) - \dim \overline{\{x\}}}.
\]
By filtering $M$ by the approximations to it, this can be reduced to checking that: if $Z \subset \mathcal{X}_{p=0}$ is a closed subscheme with $\dim(Z) \leq j$, then $\myR\Gamma_x(\cO_Z) \in D^{\leq j - \dim \overline{\{x\}}}$ for all points $x \in \mathcal{X}_{p=0}$. This is a well known fact, and so the claim follows.

In particular, by \autoref{eq:lc-perverse-t-exact-support} we see that 
\[
\myR\Gamma_x(H^i(\myR\Gamma_p(K))_x) \in D^{\leq 1-i - \dim(\overline{\{x\}})},
\]
and so $\myR\Gamma_x(H^i(\myR\Gamma_p(K)_x)[-i]) \in D^{\leq 1 - \dim(\overline{\{x\}})}$. Therefore, $\myR\Gamma_x(\myR\Gamma_p(K)_x) \in D^{\leq 1 - \dim(\overline{\{x\}})}$, as required.
\end{remark}


{\begin{remark} 
\label{remark:equivalent-characterisation-of-perverse-structure-coherent-case}
In fact, the converse of the \autoref{rem:local-cohomology-is-perverse-t-exact} is also true. More precisely, we have that
\begin{enumerate}
\item $K \in {}^p D^{\leq 0}_{\qcoh}(\mathcal{X})$ if and only if for every $x \in \mathcal{X}$, we have $\myR\Gamma_x(\myR\Gamma_p(K)_x) \in D^{\leq -\dim(\overline{\{x\}})+1}$.
\item $K \in {}^p D^{\geq 0}_{\qcoh}(\mathcal{X})$ if and only if for every $x \in \mathcal{X}$, we have $\myR\Gamma_x(\myR\Gamma_p(K)_x) \in D^{\geq -\dim(\overline{\{x\}})+1}$.
\end{enumerate}
In particular, the perverse quasi-coherent sheaves are simply those $K \in D_{\qcoh}(\mathcal{X})$ such that $\myR\Gamma_x(\myR\Gamma_p(K)_x)$ is supported in degree $-\dim \overline{\{x\}} + 1$ for each $x \in \mathcal X$. The proof of the ``if" direction of (a) (which is the only thing that remains to prove) is very similar to the proof of \autoref{lem.LocalCohomPerv} (with $\myR\Gamma_p(K)$ in place of $K$ and with a shift of $-\dim\overline{\{x\}}$ by $1$) and we leave the details to the interested readers.
\end{remark}
}

\begin{example}[The case of a point]
Specializing to the case $\mathcal{X} = \mathrm{Spf}(\cO_C)$, we obtain a perverse $t$-structure on $D_{\pcomp}(\cO_C)$. As $\mathrm{Spf}(\cO_C)$ has a single point, this $t$-structure is simply obtained by transporting a shift of the standard $t$-structure on $D_{\ptors}(\cO_C)$ along the complete-torsion equivalence. One checks that it has the following intrinsic characterization: given $K \in D_{\pcomp}(\cO_C)$, we have
\begin{itemize}
\item $K \in {}^p D^{\geq 0}$ exactly when $K \in D^{\geq 0}$ and $H^0(K)$ is $p$-torsionfree.
\item $K \in {}^p D^{\leq 0}$ exactly when $K \in D^{\leq 1}$ and $H^1(K)$ is $p^\infty$-torsion.
\end{itemize}
In fact, since $p^\infty$-torsion $p$-complete objects have bounded $p^\infty$-torsion by the Banach open mapping theorem, the condition on $H^1(K)$ in the second item above can be replaced with ``$p^c \cdot H^1(K) = 0$ for some $c \geq 1$''.
\end{example}


We next recall the definition and some basic stability properties of perverse coherent sheaves on $\mathcal{X}$.


\begin{notation}[Coherent dualizing sheaves]
\label{notation:CoherentDualizingSheavesFormalScheme}
Write $\omega^{\mydot}_{\mathcal{X}/\cO_C}$ for the normalized dualizing complex of $\mathcal{X}/\cO_C$. This means that $\omega^{\mydot}_{\mathcal{X}/\cO_C} \in D^{\geq -d}$ and the $-d$-th cohomology is nonzero, where $d$ is the dimension of $\mathcal{X}_{p=0}$. Suppose $\mathcal{X}_{p=0}$ is equidimensional, then since $\mathcal{X}_{p=0}$ is finitely presented over $\cO_C/p$, local rings of $\mathcal{X}_{p=0}$ at closed points are also of dimension $d$ by the dimension formula. In particular, for any closed point $x \in \mathcal{X}_{p=0}$, the object $\myR\Gamma_x(\omega^{\mydot}_{\mathcal{X}/\cO_C}/p)$ is concentrated in degree $0$. {More generally, for $x \in \mathcal{X}_{p=0}$ an arbitrary point, $\myR\Gamma_x((\omega^{\mydot}_{\mathcal{X}/\cO_C}/p)_x)$ is concentrated in degree $-\dim \overline{\{x\}}$ and $\myR\Gamma_x(\myR\Gamma_p(\omega^{\mydot}_{\mathcal{X}/\cO_C})_x)$ is concentrated in degree $-\dim \overline{\{x\}}+1$. In particular, $\omega^{\mydot}_{\mathcal{X}/\cO_C}$ is a $\mathfrak{p}^+$-perverse sheaf.}

The formal scheme $\mathcal{X}$ is coherent, so there is a standard $t$-structure on $D^b_{\coherent}(\mathcal{X})$. Write $\mathbf{D}_\mathcal{X}(-) = \underline{\RHom}_{\cO_\mathcal{X}}(-, \omega^{\mydot}_{\mathcal{X}/\cO_C})$ for the resulting Grothendieck duality on $D^b_{\coherent}(\mathcal{X})$ and similarly write $\mathbf{D}_{\mathcal{X}_{p^n=0}}(-) = \underline{\RHom}_{\cO_\mathcal{X}/p^n}(-, \omega^{\mydot}_{\mathcal{X}/\cO_C}/p^n)$ for the resulting Grothendieck duality on $D^b_{\coherent}(\mathcal{X}_{p^n=0})$ (cf.\ \cite[Tag 0A7T]{stacks-project}). Thus, we have
 \begin{equation}
 \label{eq:coherent-duality-mod-p}
  \mathbf{D}_\mathcal{X}(-)/p^n \simeq \mathbf{D}_{\mathcal{X}_{p^n=0}}(-/p^n),
  \end{equation}
 i.e.,  the formation of Grothendieck duals commutes with reduction modulo $p^n$ (see \cite[0A6A(4)]{stacks-project}). On the other hand, given $M \in D^b_{\coherent}(\mathcal{X}_{p^n=0})$, the objects $\mathbf{D}_\mathcal{X}(M)$ and $\mathbf{D}_{\mathcal{X}_{p^n=0}}(M)$ differ by a shift. The functor $\mathbf{D}_{\mathcal{X}}(-)$ induces an autoequivalence of $D^b_{\acoh}(\mathcal{X})$ that we shall also call Grothendieck duality. 
 \end{notation}

\begin{proposition}[Perverse almost coherent sheaves]
\label{PervAlmostCoh}
The perverse $t$-structure on $D_{\qcoh}(\mathcal{X})^a$ restricts to one on $D^b_{\acoh}(\mathcal{X})$. Moreover, the induced $t$-structure is the Grothendieck dual of the standard $t$-structure, i.e., we have 
\begin{enumerate}
\item $K \in {}^p D^{\geq 0}$ if and only if $\mathbf{D}_{\mathcal{X}}(K) \in D^{\leq 0}$ (or equivalently $\mathbf{D}_{\mathcal{X}}(K)/p \in D^{\leq 0}$).
\item $K \in {}^p D^{\leq 0}$ if and only if $\mathbf{D}_{\mathcal{X}}(K) \in D^{\geq 0}$.
\end{enumerate}
\end{proposition}

Here we remind the readers that, since we are working in the almost category, $D^{\leq 0}$ (resp., $D^{\geq 0}$) should be interpreted in the almost category, i.e., those complexes that are almost concentrated in non-positive (resp., non-negative) cohomological degree.


\begin{proof}
The equivalence in the parenthetical in (a) is immediate from derived $p$-completeness. For the rest, note that as $\mathbf{D}_{\mathcal{X}}$ is an auto-equivalence of $D^b_{\acoh}(\mathcal{X})$, the conditions described in the proposition clearly describe a $t$-structure on $D^b_{\acoh}(\mathcal{X})$. To prove the proposition, it then suffices to identify the coconnective parts. Since once this is done, the connective parts are also identified by the exact triangle ${}^p\tau^{\leq 0}(K) \to K \to {}^p\tau^{\geq 1}(K) \xrightarrow{+1}$ applied to both $t$-structures (see \autoref{remark:only-check-coconnective}). Specifically, we need to show that 
\[
\mathbf{D}_{\mathcal{X}}(K)/p \overset{\autoref{eq:coherent-duality-mod-p}}{=} \mathbf{D}_{\mathcal{X}_{p=0}}(K/p) \in D^{\leq 0}
\]
if and only if $\myR\Gamma_x(K_x/p) \in D^{\geq -\dim(\overline{\{x\}})}$ for every $x \in \mathcal X$ (see \autoref{Pervqcbig}(b)(ii))). But this is an almost coherent analog of the following standard result in commutative algebra {applied to $K/p$}: if $R$ is a Noetherian ring admitting a dualizing complex $\omega_R^\mydot$ normalized to ensure $\myR\Gamma_{\mathfrak{m}}(\omega_R^\mydot)$ is in degree $0$ for all maximal ideals $\mathfrak{m}$, then $\RHom_R(-, \omega_R^\mydot)$ identifies $D^{\leq 0, b}_{\coherent}(R)$ with the full subcategory of $D^b_{\coherent}(R)$ spanned by complexes $K$ with $\myR\Gamma_x(K_x) \in D^{\geq -d(x)}$ for all $x \in \mathrm{Spec}(R)$, where $d(x)$ is the dimension function determined by $\omega_R^\mydot$ (i.e., $\myR\Gamma_x((\omega_R^\mydot)_x)$ is concentrated in degree $-d(x)$). We omit the details here, and refer to \cite{BhattLuriepadicRHmodp} for a proof.
\end{proof}

\begin{example}[Perversity of the structure sheaf]
\label{ex:PervStr}
We have $\cO_{\mathcal{X}}[d] \in {}^p D^{\leq 0}$: indeed, this assertion is dual to saying that $\omega_{\mathcal{X}/\cO_C}^{\mydot} \in D^{\geq -d}$, which is a standard property of the normalized dualizing complexes. Alternately, this follows from the definition in \autoref{Pervqcbig} (1), the description $\myR\Gamma_p(\cO_{\mathcal{X}}[d]) = \colim_n \cO_{\mathcal{X}}/p^n[d-1]$, and the fact that $d \geq \dim(\overline{\{x\}})$ for all $x \in \mathcal{X}$. 

The first argument also shows that $\cO_{\mathcal{X}}[d]$ is perverse coherent exactly when $\omega^{\mydot}_{\mathcal{X}/\cO_C}$ is concentrated in degree $-d$, i.e., $\mathcal{X}$ is Cohen-Macaulay over $\cO_C$. 

For future reference, we observe that since containment in ${}^p D^{\leq 0}$ can be checked locally on $\mathcal{X}$, the containment $\cO_{\mathcal{X}}[d] \in {}^p D^{\leq 0}$ formally yields that $D^{\leq -d}_{\qcoh}(\mathcal{X}) \in {}^p D^{\leq 0}_{\qcoh}(\mathcal{X})$. 
\end{example}

\begin{remark}[Finite pushforwards]
\label{PervAlmostCohFinPush}
If $f:\mathcal{X} \to \mathcal{Y}$ is a finite morphism of finitely presented flat $p$-adic formal $\cO_C$-schemes, then $f_*:D^b_{\acoh}(\mathcal{X}) \to D^b_{\acoh}(\mathcal{Y})$ commutes with Grothendieck duality on $D^b_{\acoh}(-)$ and is $t$-exact for the usual $t$-structure. Consequently, $f_*$ is also $t$-exact for the perverse $t$-structure on $D^b_{\acoh}(-)$. 

More generally, the same reasoning shows that for any proper map $f:\mathcal{X} \to \mathcal{Y}$ between finitely presented flat $p$-adic formal $\cO_C$-schemes, the pushforward $f_*:D^b_{\acoh}(\mathcal{X}) \to D^b_{\acoh}(\mathcal{Y})$ is $t$-right exact for the perverse $t$-structure: indeed, the dual statement is that $f_*$ is $t$-left exact for the standard $t$-structure, which is clear.
\end{remark}

\begin{remark}[Smooth pullback]
\label{PervAlmostCohSmoothPull}
If $f:\mathcal{X} \to \mathcal{Y}$ is a smooth morphism of relative dimension $r$ between finitely presented flat $p$-adic formal $\cO_C$-schemes, then $f^*(-)[r]$ is $t$-exact
for the perverse $t$-structure on $D^b_{\acoh}(-)$: indeed, conjugating by Grothendieck duality, this amounts checking that $f^!(-)[-r]$ is $t$-exact 
for the standard $t$-structure, which follows from the smoothness of $f$ thanks to the formula $f^!(-) = f^*(-) \otimes \omega_{\mathcal{X}/\mathcal{Y}}[r]$.
\end{remark}

\subsection{Review: perverse $\mathbf{Z}_p$-sheaves}
\label{subsec:perverseZpSheaves}

In this section, let $C$ be any field where $p$ is invertible.  Let $Y/C$ be a variety of dimension $d$.  We recall some basic definitions and properties of perverse $\mathbf{Z}_p$-sheaves on $Y$. Most references on perverse sheaves focus on the case of $\mathbf{Z}/p^n$ or $\mathbf{Q}_p$-coefficients (primarily as one obtains a self-dual theory in these cases). However, working with $\mathbf{Z}_p$-coefficients is important for our purposes: we cannot invert $p$ as that would loose valuable torsion information, and we cannot work modulo a fixed power of $p$ as we ultimately want to define invariants of singularities for $p$-torsionfree $\mathbf{Z}_p$-schemes. Moreover, it is quite convenient for our applications (especially those that rely on the absolute integral closure) to work with all complexes, instead of merely the constructible ones. Consequently, we explain (\autoref{PervIntBig}) how to construct the perverse $t$-structure on all $\mathbf{Z}_p$-complexes using results of Gabber \cite{Gabber.tStruc}. Using this $t$-structure, we also explain the shape for $\mathbf{Z}_p$-coefficients of certain standard results and constructions for $\mathbf{Z}/p^n$-coefficients, e.g., one has non-self dual $\mathrm{IC}$-sheaves in this setting (\autoref{Def:ICInt}).
{We emphasize that we need to work with $p$-complete objects as opposed to merely constructible ones, since the sheaves $\pi_*\mathbf{Z}_p/p^n$ for $\pi \colon X^+ \to X$ (which are key for our applications) are {\it not} constructible.}

\begin{notation}[\'Etale derived categories]
Write $D_{\pcomp}(Y, \mathbf{Z}_p) \subset D(Y_{\et}, \mathbf{Z})$ for the full subcategory spanned by (derived) $p$-complete objects, i.e., those $K \in D(Y_{\et}, \mathbf{Z})$ such that the natural map $K \to \myR\lim_n K/p^n$ is an isomorphism. The natural maps give an equivalence $D_{\pcomp}(Y, \mathbf{Z}_{p}) = \lim D(Y_{\et}, \mathbf{Z}/p^n)$
provided we work with stable $\infty$-categories (see \cite[Proposition 4.3.9]{GL_Weil}). The ``local cohomology at $p$ functor'' ({$K \mapsto \colim K/p^n[-1]$}) and the ``derived $p$-completion functor'' ({$K \mapsto \myR\lim K/p^n$}) induce mutually inverse equivalences $D_{\pcomp}(Y, \mathbf{Z}_{p}) \simeq D_{\ptors}(Y, \mathbf{Z})$, where the latter denotes the full subcategory of $K \in D(Y_{\et}, \mathbf{Z})$ with $K[1/p] = 0$. As explained above,  this perspective is often helpful as objects in $D_{\ptors}(Y, \mathbf{Z})$ can be written as filtered colimits of constructible objects in $D(Y_{{\et}}, \mathbf{Z}/p^n)$ for varying $n$. {For example, $\myR\Gamma_p(\mathbf Z_p) = \mathbf Q_p / \mathbf Z_p[-1]$.}

\begin{remark}
We remind the readers that the ``constant sheaf" $\mathbf{Z}_p\in D_{\pcomp}(Y, \mathbf{Z}_p)$ denotes $\myR\lim_n \mathbf{Z}/p^n$ calculated on the \'etale site.
In fact, to compute $\myR\lim_n \mathbf{Z}/p^n$ in $D(Y_{\et}, \mathbf{Z})$, one needs to take an injective resolution $I_n^{\mydot}$ of $\mathbf{Z}/p^n$ (in the category of \'etale sheaves) and then compute the homotopy fiber of $\prod_n I_n^{\mydot} \to \prod_n I_n^{\mydot}$ where the map is the identity minus the natural map induced by $I^{\mydot}_n\to I^{\mydot}_{n-1}$. 
Alternatively, one can view $\mathbf{Z}/p^n$ as local systems on the pro-\'etale site of $Y$, then $\myR\lim_n \mathbf{Z}/p^n = \lim_n\mathbf{Z}/p^n$ is a sheaf on $Y_{{\proet}}$ which we also denoted by ${\mathbf{Z}_p}$ by abusing notation, see \cite[Proposition 3.1.10]{BhattScholzeProetale}. More precisely, for any $U\in Y_{{\proet}}$, the value of ${\mathbf{Z}_p}$ on $U$ is the continuous functions from $U$ to $\mathbf{Z}_p$ where $\mathbf{Z}_p$ is given the $p$-adic topology. 
We then have $\myR\lim_n\mathbf{Z}/p^n \in D(Y_{{\et}}, \mathbf{Z})$ is equal to $\myR\nu_*{\mathbf{Z}_p}$ where $\nu: Y_{{\proet}}\to Y_{{\et}}$, 
see \cite[Example 5.2.2]{BhattScholzeProetale}.
\end{remark}

Write $D^b_{\cons}(Y, \mathbf{Z}_p) \subset D_{\pcomp}(Y, \mathbf{Z}_p)$ for the usual constructible derived category, i.e., $K \in D_{\pcomp}(Y, \mathbf{Z}_p)$ such that $K/p \in D^b_{\cons}(Y, \mathbf{Z}/p)$ is constructible in the usual sense.  Write $\omega^{\et}_{Y} \in D^b_{\cons}(Y, \mathbf{Z}_p)$ for the Verdier dualizing complex of $Y$, normalized to ensure that for any closed point $y \in Y$, the object $\myR\Gamma_y(\omega^{\et}_{Y}/p)$ is concentrated in degree $0$\footnote{Unlike the coherent story, as $\omega_{Y}^{\et}$ is $p$-complete, this also means that $\myR\Gamma_y(\omega^{\et}_{Y})$ is concentrated in degree $0$.}; 
write $\mathbf{D}_{Y}^{\et}(-) = \underline{\RHom}_{Y}(-, \omega^{\et}_{Y})$ for the resulting Verdier duality on $D^b_{\cons}(Y, \mathbf{Z}_p)$. 
 
 We shall use the standard  middle perverse $t$-structure $({}^p D^{\leq 0}(Y, \mathbf{Z}/p^n), {}^p D^{\geq 0}(Y, \mathbf{Z}/p^n))$ on  $D(Y, \mathbf{Z}/p^n)$ with heart  $\mathrm{Perv}(Y, \mathbf{Z}/p^n)$. Note that this $t$-structure restricts to a self-dual one on $D^b_{\cons}(Y, \mathbf{Z}/p^n)$ with heart given by the category $\mathrm{Perv}_{\cons}(Y, \mathbf{Z}/p^n)$ of perverse constructible sheaves.
\end{notation}

The reduction mod $p^n$ functor $D(Y_{\et}, \mathbf{Z}/p^{n+1}) \to D(Y_{\et}, \mathbf{Z}/p^n)$ is not $t$-exact for the middle perverse $t$-structure,
so it does not (at least in an obvious way) induce a $t$-structure on the limit $D_{\pcomp}(Y, \mathbf{Z}_p)$. However, its right adjoint (aka the restriction of scalars functor) $D(Y_{\et}, \mathbf{Z}/p^{n}) \to D(Y_{\et}, \mathbf{Z}/p^{n+1})$ is $t$-exact for the middle perverse $t$-structures. Consequently, by passage to the limit, one can define perverse truncation functors on $D_{\ptors}(Y,\mathbf{Z})$,  and then also on $D_{\pcomp}(Y, \mathbf{Z}_p)$ by transport of structure along the complete-torsion equivalence. This analysis is formalized next:

\begin{theorem}[Perverse $\mathbf{Z}_p$-complexes]
\label{PervIntBig}
There is a $\mathfrak{p}^+$-perverse $t$-structure on $D_{\pcomp}(Y, \mathbf{Z}_p)$ characterized by either of the following conditions:
\begin{enumerate}
\item $K \in {}^{\mathfrak{p}^+} D^{\leq 0}$ if and only if $\myR\Gamma_p(K)_{{y}} := \myR\Gamma(\mathrm{Spec}(\cO_{Y,\bar{y}}^{sh}),R\Gamma_p(K)) \in D^{\leq -\dim(\overline{\{y\}})+1}$ for every geometric point $\bar{y} \to Y$ supported at $y \in Y$.
\item $K \in {}^{\mathfrak{p}^+} D^{\geq 0}$ if and only if any of the following equivalent conditions are satisfied:
\begin{enumerate}
\item $\myR\Gamma_y(\myR\Gamma_p(K)_y) := \myR\Gamma_y(\mathrm{Spec}(\cO_{Y,\bar{y}}^{sh}),R\Gamma_p(K)_y) \in D^{\geq -\dim(\overline{\{y\}})+1}$ for all geometric points $\bar{y} \to Y$ supported at $y \in Y$.
\item $\myR\Gamma_y(K_y/p) := \myR\Gamma_y(\mathrm{Spec}(\cO_{Y,\bar{y}}^{sh}),K_y/p) \in D^{\geq -\dim(\overline{\{y\}})}$ for all geometric points $\bar{y} \to Y$ supported at $y \in Y$.
\item $K/p \in {}^p D^{\geq 0}(Y, \mathbf{F}_p)$.
\end{enumerate}
\end{enumerate}
\end{theorem}
\begin{proof}
First, we explain the equivalence between the conditions appearing in (b). The equivalence $(ii)\Leftrightarrow (iii)$ is definitional. For ${(i)} \Leftrightarrow {(ii)}$, observe that the complex $\myR\Gamma_y(\myR\Gamma_p(K)_y)$ is $p^\infty$-torsion with $\myR\Gamma_y(\myR\Gamma_p(K)_y)/p = \myR\Gamma_y(K_y/p)$; {the equivalence then follows by the general statement that for a $p^\infty$-torsion complex $M$, we have $M \in D^{\geq 0}$ if and only if $M/p \in D^{\geq -1}$. }

We now construct the $t$-structure. Consider the function $q:Y \to \mathbf{Z}$ given by $q(y) = -\dim(\overline{\{y\}})+1$. The function $q$ is a finite bounded strong perversity function in the sense of \cite[\S 1]{Gabber.tStruc}. Consequently, by \cite[\S 6]{Gabber.tStruc}, there is a $t$-structure $\mathfrak{q}$ on either $D(Y, \mathbf{Z})$ or $D(Y, \mathbf{Z}[1/p])$ characterized as follows:
\begin{itemize}
\item $M \in {}^{\mathfrak{q}} D^{\leq 0}$ if and only if $M_y \in D^{\leq -\dim(\overline{\{y\}})+1}$ for every geometric point $\bar{y} \to Y$ supported at $y \in Y$.
\item $M \in {}^{\mathfrak{q}} D^{\geq 0}$ if and only if $\myR\Gamma_y(M_y) \in D^{\geq -\dim(\overline{\{y\}})+1}$ for every geometric point $\bar{y} \to Y$ supported at $y \in Y$.
 \end{itemize}
As the formation of both $*$ and $!$ stalks for objects in $D(Y, \mathbf{Z})$ commutes with filtered colimits, the localization functor $D(Y, \mathbf{Z}) \to D(Y, \mathbf{Z}[1/p])$ is clearly $t$-exact, so its kernel $D_{\ptors}(Y, \mathbf{Z})$ also inherits a $t$-structure $\mathfrak{q}$ characterized as above. Transporting this $t$-structure along the complete-torsion equivalence $D_{\pcomp}(Y, \mathbf{Z}_p) \simeq D_{\ptors}(Y, \mathbf{Z})$ then yields the desired $t$-structure.
\end{proof}

\begin{example}[The case of a point]
Assume $Y=\mathrm{Spec}(C)$, so $D_{\pcomp}(Y,\mathbf{Z}_p)$ identifies with the $p$-complete objects $D_{\pcomp}(\mathbf{Z}_p)$ in the derived category of $\mathbf{Z}_p$-modules.
As $Y$ has a single point, the $\mathfrak{p}^+$-perverse $t$-structure on $D_{\pcomp}(\mathbf{Z}_p)$ is obtained by transporting the standard $t$-structure along the complete-torsion equivalence. In particular, it
 can be described as follows:
\begin{itemize}
\item $K \in {}^{\mathfrak{p}^+} D^{\leq 0}$ exactly when $K \in D^{\leq 1}$ and $H^1(K)$ is $p^\infty$-torsion.
\item $K \in {}^{\mathfrak{p}^+} D^{\geq 0}$ exactly when $K \in D^{\geq 0}$ with $H^0(K)$ being $p$-torsionfree.
\end{itemize}
\end{example}

\begin{remark}[Compatibility with filtered colimits]
The categories ${}^{\mathfrak{p}^+} D^{\leq 0}$ and ${}^{\mathfrak{p}^+} D^{\geq 0}$ appearing in \autoref{PervIntBig} are stable under filtered colimits in $D_{\pcomp}(Y, \mathbf{Z}_p)$. Consequently, the functor of taking $\mathfrak{p}^+$-cohomology sheaves commutes with filtered colimits. 
\end{remark}

\begin{proposition}[Perverse constructible $\mathbf{Z}_p$-sheaves]
\label{PropEtaleIntPerv}
The $\mathfrak{p}^+$-perverse $t$-structure on {\allowbreak} $D_{\pcomp}(Y, \mathbf{Z}_p)$ restricts to a $\mathfrak{p}^+$-perverse $t$-structure on $D^b_{\cons}(Y, \mathbf{Z}_p)$ characterized by either of the following conditions:
\begin{enumerate}
\item  $K \in {}^{\mathfrak{p}^+} D^{\geq 0}$ if and only if $K/p \in {}^p D^{\geq 0}(Y, \mathbf{F}_p)$ (or equivalently $\mathbf{D}_Y^{\et}(K)/p \in {}^p D^{\leq 0}(Y, \mathbf{F}_p)$).
\item $K \in {}^{\mathfrak{p}^+} D^{\leq 0}$ if and only if $K/p^n \in {}^p D^{\leq 1}(Y, \mathbf{Z}/p^n)$ for all (or equivalently any) $n \geq 1$ and such that ${}^p H^1(K/p^n)$ is annihilated by $p^c$ for some fixed $c$ independent of $n$. 
\end{enumerate}
In particular, given $K \in D^b_{\cons}(Y, \mathbf{Z}_p)$, if $K/p \in \mathrm{Perv}(Y, \mathbf{F}_p)$, then $K \in \mathrm{Perv}^+(Y, \mathbf{Z}_p)$.
\end{proposition}

\begin{remark} \label{remark:perverse-with-Zpn-coeffs} 
Suppose that $K \in D^b_{\cons}(Y, \mathbf{Z}_p)$ is such that $K/p \in {}^p D^{\geq 0}(Y, \mathbf{F}_p)$. Then $K/p^n \in {}^p D^{\geq 0}(Y, \mathbf{Z}/p^n)$ for every $n\geq 1$. This is immediate from the exact triangles:
\[
K/p \to K/p^n \to K/{p^{n-1}} \xrightarrow{+1}
\]
and the long exact sequence of perverse cohomology as the restriction of scalars functor $D^b_{\cons}(Y, \mathbf{Z}/p^{n-1}) \to D^b_{\cons}(Y, \mathbf{Z}/p^n)$ is $t$-exact for the middle perverse t-structure. The same argument shows that if $K/p \in {}^p D^{\leq 0}(Y, \mathbf{F}_p)$, then $K/p^n \in {}^p D^{\leq 0}(Y, \mathbf{Z}/p^n)$ for every $n \geq 1$. In particular, if $K/p \in \mathrm{Perv}(Y, \mathbf{F}_p)$, then $K/p^n \in \mathrm{Perv}(Y, \mathbf{Z}/{p^n})$.
\end{remark}

{In the proofs in this section we shall use that applying natural functors such as $f^!, f^*, f_!, f_*$ to objects in $D_{\pcomp}(\mathbf{Z}_p)$ commute with derived restriction mod $p$, because this is simply tensoring by $\mathbf{Z}_p \xrightarrow{p} \mathbf{Z}_p$.}

\begin{proof}[Proof of \autoref{PropEtaleIntPerv}]
The existence of a $t$-structure characterized by {(a)} or {(b)} is already \cite{BBDG82}; see also \cite[\S 4.2]{BhattHansen} for a relatively recent exposition.  To identify this $t$-structure with the one induced from \autoref{PervIntBig}, it suffices to note that the coconnective parts are the same (see \autoref{remark:only-check-coconnective}), that is (a) holds, which is clear by (iii)(b) of \autoref{PervIntBig}. The last sentence is immediate from (a), (b), and \autoref{remark:perverse-with-Zpn-coeffs}.
\end{proof}


\begin{warning}[Non-self duality]
\label{NoDual}
Unlike its mod $p^n$ counterpart, the perverse $t$-structure from \autoref{PropEtaleIntPerv} is not self-dual. In fact, applying $\mathbf{D}_Y^{\et}$ to the $\mathfrak{p}^+$-perverse $t$-structure gives the so-called $\mathfrak{p}$-perverse $t$-structure on $D^b_{\cons}(Y, \mathbf{Z}_p)$, characterized by the following: $K \in {}^{\mathfrak{p}} D^{\leq 0}$ exactly when $K/p \in {}^p D^{\leq 0}(Y, \mathbf{F}_p)$.
\end{warning}


Despite the above warning, many of the basic results on perverse sheaves with $\mathbf{Z}/p^n$-coefficients extend to $\mathbf{Z}_p$-coefficients. 

\begin{remark}[Acyclicity of finite pushforwards]
\label{FinPushPervEt}
If $f:Y' \to Y$ is a finite morphism, then $f_*:D^b_{\cons}(Y', \mathbf{Z}_p) \to D^b_{\cons}(Y, \mathbf{Z}_p)$ is $t$-exact for the $\mathfrak{p}^+$-perverse $t$-structure. 

To prove this, we will show that if $K \in {}^{\mathfrak{p}^+}D^{\leq 0}(Y', \mathbf{Z}_p)$, then $f_*K \in {}^{\mathfrak{p}^+}D^{\leq 0}(Y, \mathbf{Z}_p)$. The case of ${}^{\mathfrak{p}^+}D^{\geq 0}(Y', \mathbf{Z}_p)$ is analogous but simpler. By the above assumption and \autoref{PropEtaleIntPerv}  we have that $K/p^n \in {}^pD^{\leq 1}(Y', \mathbf Z/p^n)$ and ${}^pH^1(K/p^n)$ is annihilated by $p^c$ for some fixed $c$ independent of $n$. Since $f_*$ is $t$-exact for the middle perverse $t$-structure (see \autoref{remark:all-properties-of-perverse-Fp-coefficients}), we also have that $f_*K/p^n = f_*(K/p^n) \in {}^pD^{\leq 1}(Y, \mathbf Z/p^n)$ and ${}^pH^1(f_*(K/p^n)) = f_*{}^pH^1(K/p^n)$ is annihilated by the same $p^c$. Thus,  $f_*K \in {}^{\mathfrak{p}^+}D^{\leq 0}(Y, \mathbf{Z}_p)$ by \autoref{PropEtaleIntPerv} again. 
\end{remark}

\begin{remark}[Artin vanishing with $\mathbf{Z}_p$-coefficients]
\label{ArtinIntCoeff}
If $j:U \to W$ is an affine \'etale morphism, then the functors $j_!,j_*:D^b_{\cons}(U, \mathbf{Z}_p) \to D^b_{\cons}(W, \mathbf{Z}_p)$ are $t$-exact for the $\mathfrak{p}^+$-perverse $t$-structure. 



{To prove this, we first note that by the Artin vanishing for $\mathbf{Z}/p^n$ coefficient, we know that $j_*$ and $j_!$: $D^b_{\cons}(U, \mathbf{Z}/p^n) \to D^b_{\cons}(W, \mathbf{Z}/p^n)$ are both $t$-exact for the middle perverse $t$-structure (\autoref{remark:all-properties-of-perverse-Fp-coefficients}). Suppose $K\in {}^{\mathfrak{p}^+}D^{\geq 0}$, then $K/p\in {}^{p}D^{\geq 0}(Y, \mathbf{F}_p)$.  Thus $(j_*K)/p=j_*(K/p) \in {}^{p}D^{\geq 0}(Y, \mathbf{F}_p)$, and so $j_*K\in {}^{\mathfrak{p}^+}D^{\geq 0}$ (by \autoref{PropEtaleIntPerv}). If $K\in {}^{\mathfrak{p}^+}D^{\leq 0}$, then $K/p^n\in {}^{p}D^{\leq 1}(Y, \bZ/p^n)$ and ${}^pH^1(K/p^n)$ is annihilated by $p^c$ for some $c$ independent of $n$. It follows that $(j_*K)/p^n=j_*(K/p^n) \in{}^{p}D^{\leq 1}(Y, \mathbf{Z}/p^n)$ and that ${}^pH^1((j_*K)/p^n)= {}^pH^1(j_*(K/p^n))=j_*{}^pH^1(K/p^n)$ is annihilated by $p^c$. Therefore $j_*K\in{}^{\mathfrak{p}^+}D^{\leq 0}$ (again, by \autoref{PropEtaleIntPerv}). This proves the $t$-exactness of $j_*$ for the $\mathfrak{p}^+$-perverse $t$-structure (the same argument shows that $j_*$ is also $t$-exact for the $\mathfrak{p}$-perverse $t$-structure). Now for $j_!$, simply notice that $j_!= \mathbf{D}_W^{\et} j_* \mathbf{D}_U^{\et}$ and so the $t$-exactness of $j_!$ for the $\mathfrak{p}^+$-perverse (resp.\ $\mathfrak{p}$-perverse $t$-structure) follows from the $t$-exactness of $j_*$ for the  $\mathfrak{p}$-perverse (resp. $\mathfrak{p}^+$-perverse $t$-structure).}
\end{remark}

\begin{example}[The constant sheaf is perverse]
\label{ex:SmoothArtin}
If $Y$ is smooth, then $\mathbf{Z}_p[d]$ is $\mathfrak{p}^+$-perverse by the last sentence of \autoref{PropEtaleIntPerv}. Note that $\mathbf{Z}_p[d]$ is also self-dual (ignoring Tate twists), so $\mathbf{Z}_p[d]$ is also $\mathfrak{p}$-perverse in the sense of \autoref{NoDual}. Without smoothness assumptions on $Y$, we always have $\mathbf{Z}_p[d] \in {}^{\mathfrak{p}^+} D^{\leq 0}$: this follows as $\mathbf{Z}/p^n[d] \in {}^p D^{\leq 0}(Y, \mathbf{Z}/p^n)$. 
\end{example}

To define $\mathrm{IC}$-sheaves, we need the following observation.

\begin{lemma}
\label{lem:NoSubQuotConstSmooth}
Suppose $Y$ is smooth of dimension $d$. For any closed subset $Z \subset Y$ of dimension $\leq d-1$, the $\mathfrak{p}^+$-perverse sheaf $\mathbf{Z}_p[d]$ (\autoref{ex:SmoothArtin}) has no non-trivial subobject or quotient object supported on $Z$. 
\end{lemma}
\begin{proof}
Write $i:Z \to Y$ for the defining closed immersion. To rule out subobjects supported on $Z$, it suffices to show that $i^! \mathbf{Z}_p[d] \in {}^{\mathfrak{p}^+} D^{\geq 1}$ {since $i^!$ is left $t$-exact (\autoref{remark:all-properties-of-perverse-Fp-coefficients})}. Using \autoref{PropEtaleIntPerv} (1), this amounts to checking that $\mathbf{D}_Z^{\et}(i^! \mathbf{Z}_p[d])/p \in {}^p D^{\leq -1}$.
But using the self-duality of $\mathbf{Z}_p[d]$ resulting from the smoothness of $Y$, we  have 
\[ \mathbf{D}_Z^{\et}(i^! \mathbf{Z}_p[d])/p = i^* \mathbf{D}_Y^{\et}(\mathbf{Z}_p[d])/p = i^* \mathbf{Z}_p[d]/p = \mathbf{F}_p[d],\]
so the claim follows as $\dim(Z) < d$. 

Similarly, to rule out quotient objects supported on $Z$, we must show that $i^* \mathbf{Z}_p[d] \in {}^{\mathfrak{p}^+} D^{\leq -1}$ {since $i^*$ is right $t$-exact}; this is immediate from \autoref{PropEtaleIntPerv} {(b)} and does not use the smoothness of $Y$.
\end{proof}

\begin{definition}[The $\mathrm{IC}$-sheaf with integral coefficients]
\label{Def:ICInt}
If $j:U \subset Y$ is a dense affine open immersion with $U_{\reduced}$ smooth, then we define the $\mathfrak{p}^+$-intersection cohomology complex as
\[ \mathrm{IC}^+_{Y,\mathbf{Z}_p} = \mathrm{Image}\left(j_! \mathbf{Z}_p[d] \to j_* \mathbf{Z}_p[d]\right) \in  D^b_{\cons}(Y, \mathbf{Z}_p)^{\heartsuit},\]
where we implicitly use \autoref{ArtinIntCoeff} to make sense of the above expression. One checks that this object is canonically independent of $U$ by \autoref{lem:NoSubQuotConstSmooth}.
\end{definition}

As the $\mathfrak{p}$ and the $\mathfrak{p}^+$-perverse $t$-structures are different, the $\mathrm{IC}$-sheaf defined above need not be self-dual in general. {Specifically, applying duality to $j_! \mathbf{Z}_p[d] \to j_* \mathbf{Z}_p[d]$ preserves the map, but the images are taken in a different perverse structure, and so they need not be equal.}

\begin{example}
If $Y$ is smooth, then $\mathrm{IC}^+_{Y,\mathbf{Z}_p} = \mathbf{Z}_p[d]$. {Indeed, by \autoref{lem:NoSubQuotConstSmooth}, we have that $\mathbf Z_p[d]$ has no non-trivial subobject supported on $Y - U$, and so the natural map $\mathbf Z_p[d] \to j_*\mathbf Z_p[d]$ is injective in the category of $\mathfrak{p}^+$-perverse sheaves. In particular, the factorization 
\[
j_! \mathbf{Z}_p[d] \to \mathbf{Z}_p[d] \to j_* \mathbf{Z}_p[d]
\]
yields an injective map $\mathrm{IC}^+_{Y,\mathbf{Z}_p} \to \mathbf Z_p[d]$ which is an isomorphism on $U$. Since $\mathbf Z_p[d]$ has no non-trivial quotient objects supported on $Y - U$, this map $\mathrm{IC}^+_{Y,\mathbf{Z}_p} \to \mathbf Z_p[d]$ is in fact an isomorphism.} 
\end{example}

We shall repeatedly use the canonical map relating the shifted constant sheaf with the $\mathrm{IC}$-sheaf.

\begin{proposition}[The fundamental class in intersection cohomology]
\label{CanMapIC}
There is a unique morphism 
\[ \mathbf{Z}_p[d] \to \mathrm{IC}^+_{Y,\mathbf{Z}_p}\]
on $Y$ characterized by the requirement that its formation commutes with passing to open subsets and that it equals the identity over $(Y_{\reduced})_{sm}$. Moreover, in the factorization
\[ \mathbf{Z}_p[d] \to {}^{\mathfrak{p}^+} \tau^{\geq 0} \mathbf{Z}_p[d] \simeq {}^{\mathfrak{p}^+} H^0(\mathbf{Z}_p[d]) \to \mathrm{IC}^+_{Y,\mathbf{Z}_p}\]
induced by the last sentence of \autoref{ex:SmoothArtin}, the second map is a surjection of $\mathfrak{p}^+$-perverse sheaves.
\end{proposition}

\begin{proof}
To construct this map, we may assume $Y$ is reduced. Choose $j:U \hookrightarrow Y$ a sufficiently small dense affine open with $U$ smooth. Then we have a natural map $\mathbf{Z}_p[d] \to j_* \mathbf{Z}_p[d]$. We first show that this map naturally factors as $\mathbf{Z}_p[d]\to \mathrm{IC}^+_{Y,\mathbf{Z}_p} \to j_* \mathbf{Z}_p[d]$; this also immediately shows the uniqueness. To see this note that the quotient $Q = j_* \mathbf{Z}_p[d] / \mathrm{IC}^+_{Y,\mathbf{Z}_p}$ is a $\mathfrak{p}^+$-perverse sheaf on $Y$ supported on $Z=Y-U$, so the map $\mathbf{Z}_p[d] \to Q$ factors canonically as $$\mathbf{Z}_p[d] \xrightarrow{a} i_* i^* \mathbf{Z}_p[d] \xrightarrow{b} Q.$$ But as we have seen in the second half of \autoref{lem:NoSubQuotConstSmooth} that $i_* i^* \mathbf{Z}_p[d] \in  {}^{\mathfrak{p}^+} D^{\leq -1}$ (here we used the fact that $i_*$ is $t$-exact, see \autoref{remark:all-properties-of-perverse-Fp-coefficients}). Thus the map $b$ is $0$ and so we have the factorization as wanted. 

To see surjectivity, following the above construction, it suffices to show that the composition $j_! \mathbf{Z}_p[d] \to \mathbf{Z}_p[d] \to \mathrm{IC}^+_{Y,\mathbf{Z}_p}$ is the natural (surjective) map appearing in the definition of $\mathrm{IC}^+_{Y,\mathbf{Z}_p}$ as an image. By adjunction, it suffices to check the claim after applying $j^*$, where it is clear.
\end{proof}
\section{A summary of the Riemann-Hilbert functor and related results}
\label{sec:ReviewRH}

In this section, we recall some results that will appear in \cite{BhattLuriepadicRHmodp} and will be used below\footnote{While the natural geometric context for \cite{BhattLuriepadicRHmodp} is that of $p$-adic formal schemes and their rigid-analytic generic fibres, in the interest of keeping the exposition more accessible, we have tried to formulate statements entirely in the language of schemes; this leads to certain properness assumptions in Theorem~\ref{thm:RH}.}.

\begin{notation}
\label{not:RH1}
Fix a perfectoid extension $C/\mathbf{Q}_p$. We adopt the following convention: fraktur fonts refer to schemes over $\cO_C$, while the corresponding calligraphic and roman fonts refer to the $p$-completion over $\cO_C$. Thus, if $\mathfrak{X}$ is a $\cO_C$-scheme, then $\mathcal{X}$ is its $p$-adic formal completion (and hence a $p$-adic formal $\cO_C$-scheme). We will write $\mathfrak{X}_C$ to denote its generic fibre. 
\end{notation}

\begin{theorem}[The Riemann-Hilbert functor]
\label{thm:RH}
Let $\mathfrak{X}$ be a finitely presented flat $\cO_C$-scheme. There is a naturally defined exact coproduct preserving lax symmetric monoidal functor
\[ \RH:D_{\pcomp}(\genX, \mathbf{Z}_p) \to  D_{\qcoh}(\mathcal{X})^a\]
with the following features:
\begin{enumerate}
\item Almost coherence:  For constructible $F \in D^b_{\cons}(X, \mathbf{Z}_p)$, we have 
\[ \RH(F) \in D^b_{\acoh}(\mathcal{X}) \subset D_{\qcoh}(\mathcal{X})^a.\]
If $\mathfrak{X}/\cO_C$ is proper, then formal GAGA gives $D^b_{\acoh}(\mathfrak{X}) \simeq D^b_{\acoh}(\mathcal{X})$, so we can then regard the $p$-localization of $\RH(-)$ as an exact functor
\[ \RH:D^b_{\cons}(\genX,\mathbf{Q}_p) \to D^b_{\coherent}(\genX).\]

\item The value on the constant sheaf: There is a natural isomorphism.
\[ \RH(\mathbf{Z}_p) \simeq \cO_{\mathcal{X},\pfd} \in D_{\qcoh}(\mathcal{X})^a.\] 
Moreover, if $\mathfrak{X}$ is proper with $\genX$ smooth and $j:U \to \genX$ is an open immersion whose complement $D$ is an SNC divisor (possibly empty), then we have a (non-canonical) decomposition
\[ \RH(j_* \mathbf{Q}_p) \simeq \bigoplus_i \Omega^i_{(\genX,D)/C}(-i)[-i] \in D^b_{\coherent}(\genX),\]
where $\Omega^i_{(\genX,D)/C}$ is the sheaf of $i$-forms with logarithmic poles along $D$. Such decompositions can be chosen compatibly for any finite diagram of such pairs $(\mathfrak{X},j:U \to \mathfrak{X}_C)$.  Moreover, if $\mathfrak{X}$ and $j$ are defined over a discretely valued subfield $K \subset C$, there is a unique such decomposition compatible with Galois actions; it is functorial in the descent to $K$.


\item Pullback compatibility: Via the first isomorphism in (b), we obtain an induced coproduct preserving lax symmetric monoidal exact functor
 \[ \RH:D_{\pcomp}(\genX, \mathbf{Z}_p) \to  D_{\rm qc}(\mathcal{X}, \cO_{\mathcal{X},\pfd})^a.\]
This functor is symmetric monoidal and commutes with arbitrary pullbacks.

\item Proper pushforward compatibility: If $f:\mathfrak{Y} \to \mathfrak{X}$ is a proper map of finitely presented flat $\cO_C$-schemes with generic fibre $f_{C}$, then there is a natural isomorphism $\RH \circ \myR f_{C,*} \simeq \myR f_* \circ \RH$ of functors $D_{\pcomp}(Y, \mathbf{Z}_p) \to  D_{\rm qc}(\mathcal{X})^a$.

\item Duality: The functor $\RH:D^b_{\cons}(\genX, \mathbf{Z}_p) \to D^b_{\acoh}(\mathcal{X})$ intertwines Verdier duality on the source with Grothendieck duality on the target. 

\item Perverse $t$-structures: the functor $\RH:D_{\pcomp}(\genX,\mathbf{Z}_p) \to D_{\qcoh}(\mathcal{X})^a$ is:
\begin{itemize}
\item $t$-right exact for the $\mathfrak{p}^+$-perverse $t$-structure on the source and the standard $t$-structure on the target.
\item $t$-left exact for the $\mathfrak{p}^+$-perverse $t$-structures on the source and the perverse $t$-structure on the target.
\end{itemize}
\end{enumerate}
\end{theorem}

In applications of \autoref{thm:RH}, it is convenient to have a large supply of perverse sheaves. In practice, these are provided by classical techniques in perverse sheaf theory (e.g., Artin vanishing or the decomposition theorem) as well as the last assertion of the following result. 

\begin{theorem}[{Absolute integral closures, \cite[Theorem 3.11]{BhattAbsoluteIntegralClosure}}]
\label{thm:AIC}
Let $X$ be an integral scheme. If  $\pi:X^+ \to X$ is an absolute integral closure, then the constant sheaf $\mathbf{F}_p$ and $\mathbf{Z}_p$ on $X^+$ are $*$-extended from the generic point (and hence from any non-empty open subset). If $X$ is a $d$-dimensional variety over a field of characteristic $\neq p$, then $\pi_* \mathbf{F}_{p}[d]$ is ind-perverse and $\pi_* \mathbf{Z}_{p}[d]$ is $\mathfrak{p}^+$-perverse. 
\end{theorem}
\begin{proof}
The first conclusion for $\mathbf{F}_p$ follows from \cite[Proposition 3.10]{BhattAbsoluteIntegralClosure}, which implies the conclusion for $\mathbf{Z}_p$: to prove $\eta_*\mathbf{Z}_p\cong \mathbf{Z}_p$ in $D_{\pcomp}(X^+, \mathbf{Z}_p)$ where $\eta: U\to X^+$, it is enough to check this modulo $p$ by derived Nakayama. The second conclusion for $\pi_* \mathbf{F}_{p}[d]$ is contained in \cite[Theorem 3.11]{BhattAbsoluteIntegralClosure} (we recalled the argument in \autoref{rmk:noncompleteAIC} below), and for $\pi_* \mathbf{Z}_{p}[d]$, note that our definition of $\mathfrak{p}^+$-perverse $\mathbf{Z}_p$-complexes is stable under colimit, the result then follows from the $\mathbf{F}_p$-coefficient case via \autoref{PropEtaleIntPerv}. Alternatively, one can also deduce it from the first conclusion and Artin vanishing with $\mathbf{Z}_p$-coefficient (see \autoref{ArtinIntCoeff}).
\end{proof}

\begin{remark}
\label{rmk:noncompleteAIC}
For our applications in later sections, we often deal with the situation that $\cO_C$ is perfectoid and is the $p$-adic completion of a Henselian valuation ring $\cO_C^{nc}$, and we have an integral scheme $\mathfrak{X}$ over $\cO_C$ that is based changed from an integral scheme over $\cO_{C}^{nc}$, i.e., $\mathfrak{X}= \mathfrak{X}^{nc}\otimes_{\cO_C^{nc}}\cO_C$. We can apply \autoref{thm:AIC} to $X= \mathfrak{X}[1/p] (=\mathfrak{X}_C)$ and we have $\RH(\pi_*\mathbf{Z}_p)=\pi_*\cO_{\mathcal{X}^+}$ (in the almost category) by property (b) and (d) of \autoref{thm:RH}, where $\mathcal{X}^+$ denotes the $p$-adic completion of $\mathfrak{X}^+$ (which is perfectoid). 

However, for our applications, we also need a version for the $p$-adic completion of $(\mathfrak{X}^{nc})^+$. More precisely, we consider the map 
$$\pi: (\mathfrak{X}^{nc})^+\otimes_{\cO_C^{nc}}\cO_C \to \mathfrak{X}^{nc}\otimes_{\cO_C^{nc}}\cO_C = \mathfrak{X}$$
and by abusing notations we will also use $\pi$ to denote the corresponding maps after inverting $p$ and after $p$-adic completion. Then we have:
\begin{enumerate}
    \item $\pi_* \mathbf{F}_{p}[d]$ is ind-perverse and $\pi_* \mathbf{Z}_{p}[d]$ is $\mathfrak{p}^+$-perverse. 
    \item $\RH(\pi_*\mathbf{Z}_p)=\pi_*\cO_{(\mathcal{X}^{nc})^+}$, where $(\mathcal{X}^{nc})^+$ denotes the $p$-adic completion of ${(\mathfrak{X}^{nc})^+}$.
\end{enumerate}
By the same argument as in \autoref{thm:AIC}, part (a) above for $\pi_* \mathbf{Z}_{p}[d]$ follows from the statement for $\pi_* \mathbf{F}_{p}[d]$, and the latter follows from the same proof as in \cite[Theorem 3.11]{BhattAbsoluteIntegralClosure} and we give some details. First, $\pi_*\mathbf{F}_p[d]$ is the colimit of the ind-object $\{f'_*\mathbf{F}_p[d]\}$ where $f': \mathfrak{X}'\otimes_{\cO_C^{nc}}\cO_C \to\mathfrak{X}$ runs over all finite cover $\mathfrak{X}'\to \mathfrak{X}^{nc}$. We choose a dense affine open $j$: $U\to \mathfrak{X}^{nc}$ such that $U[1/p]$ is smooth over $\cO_{C}^{nc}[1/p]$ and $\mathfrak{X}'\times_{\mathfrak{X}^{nc}} U[1/p]\to \mathfrak{X}^{nc}[1/p]$ is finite \'etale (and thus remains finite \'etale after base change to $\cO_C[1/p]$). It follows that $j_*(f'_*\mathbf{F}_p[d])|_{U\otimes_{\cO_C^{nc}}\cO_C[1/p]}$\footnote{Again, we are abusing notations a bit and we still use $j$ (resp., $f'$) to denote the corresponding map after the base change $-\otimes_{\cO_C^{nc}}\cO_C[1/p]$ (resp., after inverting $p$).} is perverse by Artin vanishing \cite[Corollary 4.1.3]{BBDG82}. It is then enough to observe that there exists a map $\mathfrak{X}''\to \mathfrak{X}'$ such that the pullback: $f'_*\mathbf{F}_p[d]\to f''_*\mathbf{F}_p[d]$ factors through $j_*(f'_*\mathbf{F}_p[d])|_{U\otimes_{\cO_C^{nc}}\cO_C[1/p]}$. This can be checked after taking a colimit by compactness of constructible complexes, and thus follows from the fact that the constant sheaf $\mathbf{F}_p$ on $(\mathfrak{X}^{nc})^+[1/p]$ is $*$-extended from its generic point by \cite[Proposition 3.10]{BhattAbsoluteIntegralClosure}, and this remains true after a base change from $\cO_C^{nc}[1/p]\to \cO_C[1/p]$ (since this is a field extension of characteristic $0$). Part (b) above follows from the general lemma below and by noting that the $p$-adic completion of ${(\mathfrak{X}^{nc})^+}$ is the same as the $p$-adic completion of $(\mathfrak{X}^{nc})^+\otimes_{\cO_C^{nc}}\cO_C$ (and it is perfectoid).
\end{remark}

\begin{lemma}
Suppose $\pi$: $\mathfrak{Y}\to \mathfrak{X}$ is an integral map such that the $p$-adic completion of $\mathfrak{Y}$ is perfectoid.  Then we have 
$\RH(\pi_ * \mathbf{Z}_p)$ almost agrees with the structure sheaf of the $p$-adic completion of the structure sheaf of $\mathfrak{Y}$.
\end{lemma}
\begin{proof}
As $\mathfrak{Y}$ is integral over $\mathfrak{X}$, we can write $\mathfrak{Y}$ as a cofiltered limit $\varprojlim_i \mathfrak{Y}_i$ of finite covers of $\mathfrak{X}$. Set $Y_i = \mathfrak{Y}_i[1/p]$ to be the generic fibre, so $Y = \varprojlim_i Y_i$. The RH functor commutes with $p$-completed direct limits, so it takes $\pi_*\mathbf{Z}_p$ so the direct limit of RH applied to the constant sheaf $\mathbf{Z}_p$ on each $Y_i$. The latter gives (almost) the perfectoidization of the structure sheaf of $\mathfrak{Y}_i$. As perfectoidization commutes with $p$-completed direct limits, we get that $\RH(\pi_*\mathbf{Z}_p)$ is the perfectoidization of the structure sheaf of $\mathfrak{Y}$. But the latter coincides with the $p$-adic completion of $\mathfrak{Y}$ by hypothesis.
\end{proof}

\section{Singularities over the perfectoid $\cO_C$ - the RH-subsheaves}
\label{sec:RH_subsheaves}

\begin{notation}
\label{not:RH}
Continue with the notation and conventions from \autoref{not:RH1}. 
Fix a finitely presented flat $\cO_C$-scheme $\mathfrak{X}$, with $\mathcal{X}$ denotes its $p$-adic formal completion\footnote{The existence and use of the $\cO_C$-scheme $\mathfrak{X}$ algebraizing the $p$-adic formal scheme $\mathcal{X}$ is a crutch: it allows us to use \'etale sheaf theory on the scheme-theoretic generic fibre $\mathfrak{X}_{C}$ instead of contemplating Zariski-constructible sheaves on the adic generic fibre $\mathcal{X}_\eta$ of $\mathcal{X}$ in the sense of \cite{BhattHansen}. This is sufficient for our purposes and allows us to avoid some technicalities. Nonetheless, we expect that all constructions/results discussed in this section go through without assuming $\mathcal{X}$ admits an algebraization $\mathfrak{X}$. In particular, the object $\cO_{\mathcal{X}}^{\RH}$ ought to only depend on $\mathcal{X}$ and not the algebraization $\mathfrak{X}$.} and $\mathfrak{X}_{C}$ denotes its scheme-theoretic generic fibre. Write $d=\dim(\genX) = \dim(\mathcal{X}_{p=0})$. As before, we write $\omega^{\mydot}_{\mathcal{X}/\cO_C}$ for the normalized dualizing complex; we also write $\omega_{\mathcal{X}/\cO_C} := H^{-d}(\omega_{\mathcal{X}/\cO_C}^{\mydot})$ for the dualizing sheaf.
\end{notation}

In this section, we introduce the main objects: the RH structure sheaf $\cO_{\mathcal{X}}^{\RH}$ and its dual $\omega_{\mathcal{X}}^{\RH}$. These objects are variants of the structure sheaf and the dualizing sheaf that measure the singularities and are defined in terms of the Riemann-Hilbert functor (and will consequently only be \emph{almost} coherent, rather than coherent). They are defined in \autoref{ss:GRDef}. We then study properties of this construction, including behavior under finite maps (\autoref{ss:GRfin}), alterations (\autoref{ss:GRAlt}) and smooth maps (\autoref{ss:GRSm}). Using these basic properties, we explain two alternative constructions of these sheaves without any reference to \'etale sheaf theory. First, \autoref{GRviaAIC} explains a direct construction of $\cO_{\mathcal{X}}^{\RH}$ in terms of the absolute integral closure (thus giving it an arithmetic flavor). On the other hand, \autoref{AlmostStabTraceAlt} explains how to approximately compute $\omega_{\mathcal{X}}^{\RH}$ and $\cO_{\mathcal{X}}^{\RH}$ in terms of sufficiently large alterations (giving it a geometric flavour). Using this last construction, we deduce that 
base changing $\omega_{\mathcal{X}}^{\RH}$ with $p$-inverted agrees with what is known as the Grauert-Riemenschneider sheaf or the multiplier submodule in classical algebraic geometry (\autoref{GRchar0Mult}).


\begin{remark}
We emphasize that throughout this whole section, we often work in the almost category, for instance, when we say $\omega_{\mathcal{X}}^{\RH}\to \omega_{\mathcal{X}/\cO_C}$ is an injection of almost coherent sheaves (see \autoref{def:GRSheaf} below), the injection should be interpreted as an almost injection. In later sections we will use $!$-realization and view $(\omega_{\mathcal{X}}^{\RH})_!$ as an honest subsheaf of $\omega_{\mathcal{X}/\cO_C}$.
\end{remark}

\subsection{The definition}
\label{ss:GRDef}

To explain the definition, recall that we have canonical maps
\[ \cO_{\mathcal{X}} \to \cO_{\mathcal{X},\pfd} = \RH(\mathbf{Z}_p) \to \RH(\mathrm{IC}^+_{\genX,\mathbf{Z}_p}[-d]),\]
where the second map is the fundamental class from \autoref{CanMapIC}. Now $\cO_{\mathcal{X}} \in {}^p D^{\leq d}$ by \autoref{ex:PervStr} while $\RH(\mathrm{IC}^+_{\genX,\mathbf{Z}_p}[-d]) \in {}^p D^{\geq d}$ by \autoref{thm:RH} (f), so the above map factors uniquely as
\[ 
  \begin{array}{rcl}
    \cO_{\mathcal{X}} & \to & {}^p \tau^{\geq d} \cO_{\mathcal{X}} \\
    & \simeq & {}^p H^d(\cO_{\mathcal{X}})[-d]\\
    & \xrightarrow{c_\mathcal{X}} & {}^p H^0(\RH(\mathrm{IC}^+_{\genX,\mathbf{Z}_p}))[-d] \\
    & \simeq & {}^p \tau^{\leq d} \RH(\mathrm{IC}^+_{\genX,\mathbf{Z}_p}[-d]) \\
    & \to & \RH(\mathrm{IC}^+_{\genX,\mathbf{Z}_p}[-d]),
  \end{array} 
\]
where the unlabelled maps are the canonical truncation maps. Using the map $c_{\mathcal{X}}$, we obtain:

\begin{definition}[The Riemann-Hilbert structure sheaf]
\label{def:GRstr}
The object
\[ \cO_\mathcal{X}^{\RH} := \mathrm{Image}\big( \underbrace{{}^p H^d(\cO_\mathcal{X})}_{{}^p H^0(\cO_\mathcal{X}[d])} \xrightarrow{c_{\mathcal{X}}} {}^p H^0(\RH(\mathrm{IC}_{\genX,\mathbf{Z}_p}^+)) \big)[-d] \in \mathrm{Perv}_{\acoh}(\mathcal{X})[-d]\]
is the RH structure sheaf on $\mathcal{X}$.
\end{definition}

\begin{remark}[Alternate description]
\label{GRstrAlt}
Choose a dense affine open subset $j:U \hookrightarrow \mathfrak{X}_{C}$ with $U_{\reduced}$ smooth over $C$ (necessarily of dimension $d$). Then the canonical map $\mathrm{IC}_{\genX,\mathbf{Z}_p}^+ \to j_* \mathbf{Z}_p[d]$ is an injection of $\mathfrak{p}^+$-perverse sheaves by \autoref{Def:ICInt}. Using \autoref{thm:RH} (f), one then finds the following alternate description of $\cO_{\mathcal{X}}^{\RH}$: 
\[ \cO_\mathcal{X}^{\RH} \simeq \mathrm{Image}\left( {}^p H^d(\cO_\mathcal{X}) \xrightarrow{\can} {}^p H^0(\RH(j_* \mathbf{Z}_p[d]))\right)[-d] \in \mathrm{Perv}_{\acoh}(\mathcal{X})[-d]. \]
The downside to this description is that it is not canonical: there are many possible choices of $U$. On the other hand, it is perhaps easier to use than \autoref{def:GRstr} as it avoids contemplating the $\mathrm{IC}$-construction.
\end{remark}

\begin{remark}[Insensitivity to nilpotents]
If $i:\mathcal{X}_{\reduced} \subset \mathcal{X}$ is the inclusion of the reduced locus, then we have a natural isomorphism $\cO_{\mathcal{X}}^{\RH} \simeq i_* \cO_{\mathcal{X}_{\reduced}}^{\RH}$. Indeed, by \autoref{PervAlmostCohFinPush} and \autoref{FinPushPervEt}, the functor $i_*$ is $t$-exact for the $\mathfrak{p}^+$-perverse $t$-structure on constructible sheaves and the perverse $t$-structure on almost coherent sheaves; the claim now follows from the topological invariance of the \'etale site as well as the surjectivity of ${}^p H^d(\cO_{\mathcal{X}}) \to {}^p H^d(\cO_{\mathcal{X}_{\reduced}})$ (which can be checked using the last sentence in \autoref{ex:PervStr}). 
\end{remark}

By construction, the map ${}^p H^d(\cO_\mathcal{X}) \to \cO_\mathcal{X}^{\RH}[d]$ is a surjection of perverse almost coherent sheaves. Dualizing it thus gives a sheaf:

\begin{definition}[The Riemann-Hilbert canonical sheaf]
\label{def:GRSheaf}
The RH dualizing sheaf (or just the RH sheaf) on $\mathcal{X}$ is defined as
\[ \omega_\mathcal{X}^{\RH} := \mathbf{D}_\mathcal{X}(\cO_\mathcal{X}^{\RH}[d]) \in D^b_{\acoh}(\mathcal{X}).\]
Note this is indeed a sheaf: in fact, dualizing the surjection ${}^p H^d(\cO_\mathcal{X}) \to \cO_\mathcal{X}^{\RH}[d]$ in $\mathrm{Perv}_{\acoh}(\mathcal{X})$ gives an injection $\omega_\mathcal{X}^{\RH} \to \omega_{\mathcal{X}/\cO_C}$ of almost coherent sheaves on $\mathcal{X}$ (\autoref{PervAlmostCoh}), so we can regard $\omega_\mathcal{X}^{\RH}$ as an almost coherent subsheaf of the dualizing sheaf  $\omega_{\mathcal{X}/\cO_C}$.
\end{definition}

\subsection{Smooth pullback}
\label{ss:GRSm}


\begin{proposition}[Smooth pullback]
\label{prop:SmoothGR}
If ${f}:\mathfrak{X}' \to \mathfrak{X}$ is a smooth map with $p$-completion $f:\mathcal{X}' \to \mathcal{X}$ of relative dimension $r$, then the natural map gives an isomorphism 
\[ f^* \cO_\mathcal{X}^{\RH} \simeq \cO_{\mathcal{X}'}^{\RH}\]
of quotients of ${}^p H^{d+r}(\cO_{\mathcal{X}'})[-d-r]$ in $\mathrm{Perv}_{\acoh}(\mathcal{X}')[-d-r] \subset D^b_{\acoh}(\mathcal{X}')$. 
\end{proposition}

In the proof below, we shall write
\[ f^*:D_{\qcoh}(\mathcal{X}) \to D_{\qcoh}(\mathcal{X}') \quad \text{and} \quad f_\pfd^*:D_{\qcoh}(\mathcal{X},\cO_{\mathcal{X},\pfd}) \to D_{\qcoh}(\mathcal{X}', \cO_{\mathcal{X}',\pfd})\]
for the pullback functors (left adjoint to the evident pushforward functors); thus, we have $f_\pfd^*(-) = f^*(-) \widehat{\otimes}_{f^*\cO_{\mathcal{X},\pfd}} \cO_{\mathcal{X}',\pfd}$. We shall use crucially that the refined $\RH$ functor 
\[ \RH:D_{\pcomp}(\genX, \mathbf{Z}_p) \to D_{\qcoh}(\mathcal{X}, \cO_{\mathcal{X},\pfd})\]
commutes with arbitrary pullback (see \autoref{thm:RH} {(c)}).

\begin{proof}
We may assume $f$ is smooth of relative dimension $r$, so $\mathcal{X}'$ has relative dimension $d+r$ over $\cO_C$. Recall from Remark~\autoref{PervAlmostCohSmoothPull} that the shifted pullback functor $f^*[r]:D^b_{\acoh}(\mathcal{X}) \to D^b_{\acoh}(\mathcal{X}')$ is $t$-exact with respect to the perverse $t$-structure. Choose a dense affine open immersion $j:U \hookrightarrow \mathfrak{X}_{C}$ with $U_{\reduced}$ smooth. As $f$ is smooth, the pullback $j':U' \hookrightarrow \mathfrak{X}_{C}'$ is also a dense affine open immersion with $U'_{\reduced}$ smooth. Consequently, we can compute the RH structure sheaves using $U$ and $U'$. Moving shifts around, we learn that
\[ \cO_\mathcal{X}^{\RH} := \mathrm{Image}\left( {}^p H^0(\cO_\mathcal{X}[d]) \xrightarrow{\can} {}^p H^0(\RH(j_* \mathbf{Z}_p)[d])\right){[-d]} \]
and
\[ \cO_{\mathcal{X}'}^{\RH} := \mathrm{Image}\left( {}^p H^{0}(\cO_{\mathcal{X}'}[d+r]) \xrightarrow{\can} {}^p H^{0}(\RH(j'_* \mathbf{Z}_p)[d+r])\right){[-d-r]}. \]
By the perverse $t$-exactness of $f^*[r]$, the first formula above pulls back to show
\[ f^* \cO_{\mathcal{X}}^{\RH} = \mathrm{Image}\left( {}^p H^0(\cO_{\mathcal{X}'}[d+r]) \xrightarrow{\can} {}^p H^0(f^* \RH(j_* \mathbf{Z}_p)[d+r])\right){[-d-r]}.\]
Consequently, it suffices to show that the natural map
\[ f^* \RH(j_* \mathbf{Z}_p)[d+r] \to \RH(j'_* \mathbf{Z}_p[d+r]) \simeq f_\pfd^* \RH(j_* \mathbf{Z}_p[d+r])\]
is injective on ${}^p H^0(-)$, where we used the pullback compatibility of $$\RH:D^b_{\cons}(\genX, \mathbf{Z}_p) \to D_{\qcoh}(\mathcal{X}, \cO_{\mathcal{X},\pfd})$$ in the last isomorphism above. This can be checked locally on $\mathfrak{X}'$, so we may assume both $\mathfrak{X}$ and $\mathfrak{X}'$ are affine. In fact, as all our constructions are compatible with \'etale localization, we may then reduce by induction to the case $\mathfrak{X}' = \mathfrak{X} \times_{\mathrm{Spec}(\cO_C)} \mathbf{G}_m$.  In this case,  the claim follows from the following stronger assertion applied to $M=\RH(j_* \mathbf{Z}_p)$:

\begin{itemize}
\item[$(\ast)$] For any $M \in D_{\qcoh}(\mathcal{X}, \cO_{\mathcal{X},\pfd})$, the map
\[ \eta_M:f^* M \to f_{\pfd}^* M := f^*M {\otimes}_{f^*\cO_{\mathcal{X},\pfd}} \cO_{\mathcal{X'},\pfd}\]
admits a left-inverse (i.e., is split) in $D_{\qcoh}(\mathcal{X}', f^* \cO_{\mathcal{X},\pfd})$.
\end{itemize}

The assertion $(\ast)$ immediately reduces to the case $M=\cO_{\mathcal{X},\pfd}$ itself: we have $\eta_M  = \eta_{\cO_{\mathcal{X},\pfd}} \otimes \mathrm{id}_{f^* M}$. Thus, we must show that $f^* \cO_{\mathcal{X},\pfd} \to \cO_{\mathcal{X}',\pfd}$ admits a left-inverse in $D_{\qcoh}(\mathcal{X}', f^* \cO_{\mathcal{X},\pfd})$. Since $\mathfrak{X}' = \mathfrak{X} \times_{\mathrm{Spec}(\cO_C)} \mathbf{G}_m$, this reduces to the following assertion: for any $p$-complete $\cO_C$-algebra $R$, the map
\[ R_\pfd[x^{\pm 1}]^{\wedge}  := R_\pfd \widehat{\otimes}_{\cO_C} \cO_C[x^{\pm 1}]^\wedge \to R[x^{\pm 1}]^\wedge_\pfd \simeq R_\pfd \widehat{\otimes}_{\cO_C} \cO_C[x^{\pm 1}]^\wedge_\pfd  \]
admits an $R_\pfd[x^{\pm 1}]^\wedge$-linear splitting. It suffices to solve this problem when $R=\cO_C$: the general case then follows by tensoring with $R_\pfd$. Thus, we must check that 
\[  \cO_C[x^{\pm 1}]^\wedge \to  \cO_C[x^{\pm 1}]^\wedge_\pfd\]
admits a $ \cO_C[x^{\pm 1}]^\wedge$-linear splitting. This can be checked using the standard presentation of $ \cO_C[x^{\pm 1}]^\wedge_\pfd$ coming from the almost purity theorem: {
\[ 
\cO_C[x^{\pm 1}]^\wedge_\pfd\simeq \widehat{\bigoplus_{i \in \mathbf{Z}[1/p]}} \left(\cO_C x^i \xrightarrow{\underline{\epsilon}^i - 1} p^{1/p^\infty} \cO_C x^i\right),\]
where $\underline{\epsilon} \in \cO_C^\flat$ is a compatible system of $p$-power roots of unity (normalized so that $\underline{\epsilon}^{1/p^n}$ maps to a primitive $p^n$-th root of $1$ in $\cO_C$); in particular, the differential appearing above is $0$ for $i \in \mathbf{Z}$.} 
\end{proof}

\begin{remark}[Smooth pullback: dual form]
\label{SmoothGRdual}
In the situation of \autoref{prop:SmoothGR}, if $f$ is smooth, then we have a natural identification $f^* \omega_{\mathcal{X}}^{\RH} \otimes \omega_{\mathcal{X}'/\mathcal{X}} \simeq \omega_{\mathcal{X}'}^{\RH}$ of almost coherent subsheaves of the dualizing sheaf of $\mathcal{X}'$. To see this, write $d$ and $d+r$ for the relative dimensions of $\mathcal{X}/\cO_C$ and $\mathcal{X}'/\cO_C$ respectively. Keeping track of shifts, applying Grothendieck duality $\mathbf{D}_{\mathcal{X}'}(-)$ to the isomorphism in \autoref{prop:SmoothGR} gives a natural identification $f^! (\omega_{\mathcal{X}}^{\RH}[d]) \simeq \omega_{\mathcal{X}'}^{\RH}[d+r]$ compatible with the isomorphism $f^! \omega_{\mathcal{X}/\cO_C}^{\mydot} \simeq \omega_{\mathcal{X}'/\cO_C}^{\mydot}$. The rest follows as $f^!(-) \simeq f^*(-) \otimes \omega_{\mathcal{X}'/\mathcal{X}}[r]$, since $f$ is smooth of relative dimension $r$.
\end{remark}

\subsection{Finite pushforward}
\label{ss:GRfin}

\begin{proposition}[Finite pushforward functorialty]
\label{Prop:FinPushGR}
Let $f:\mathfrak{X}' \to \mathfrak{X}$ be a finite surjective morphism with $p$-completion also denoted $f:\mathcal{X}' \to \mathcal{X}$, so $f_* \cO_{\mathcal{X}'}^{\RH}[d]$ is a perverse almost coherent sheaf (\autoref{PervAlmostCohFinPush}). Then there is a unique  map $\alpha_f:\cO_{\mathcal{X}}^{\RH} \to f_* \cO_{{\mathcal{X}'}}^{\RH}$ in $\mathrm{Perv}_{\acoh}(\mathcal{X})[-d]$  rendering the following diagram commutative in $D^b_{\acoh}(\mathcal{X})$:
\[ \xymatrix{ \cO_{\mathcal{X}} \ar[r] \ar[d] & \cO_{\mathcal{X}}^{\RH} \ar[d]^{\alpha_f} \\
f_* \cO_{\mathcal{X}'} \ar[r] & f_* \cO_{\mathcal{X}'}^{\RH}. }\]
Moreover, the map $\alpha_f$ is an injection of perverse almost coherent sheaves. 
\begin{proof}
Write $d=\dim(\mathcal{X}/\cO_C)$.  Let us first explain the uniqueness. \autoref{ex:PervStr} gives a unique factorization of the top horizontal map as
\[ \cO_{\mathcal{X}} \xrightarrow{\can} {}^p \tau^{\geq d} \cO_{\mathcal{X}} \simeq {}^p H^d(\cO_{\mathcal{X}})[-d] \to \cO_{\mathcal{X}}^{\RH}, \]
with the last map being a surjection of perverse sheaves placed in cohomological degree $d$. Applying the same analysis to $\mathcal{X}'$ and using the $t$-exactness of $f_*$, the uniqueness assertion in the proposition is immediate. 

For existence, we can make choices, so choose $j:U \hookrightarrow \mathfrak{X}_{C}$ a dense affine open immersion such that the preimage $j':U' \hookrightarrow \mathfrak{X}_{C}'$ is also a dense affine open immersion and such that both $U_{\reduced}$ and $U'_{\reduced}$ are smooth and the map $U'_{\reduced} \to U_{\reduced}$ is a finite \'etale cover. As $f_*$ is perverse $t$-exact, functoriality then gives a commutative square
\[ \xymatrix{ {}^p H^d(\cO_{\mathcal{X}}) \ar[r] \ar[d] & {}^p H^d(\RH(j_* \mathbf{Z}_p)) \ar[d] \\
f_* {}^p H^d(\cO_{\mathcal{X}'}) \ar[r] & f_* {}^p H^d(\RH(j'_* \mathbf{Z}_p)) }\]
of perverse almost coherent sheaves on $\mathcal{X}$.

Taking the induced map on images (and using $t$-exactness of $f_*$ once more) then gives the desired map $\alpha_f$ (cf.\ \autoref{GRstrAlt}).

It remains to explain why $\alpha_f$ is injective. Following the construction above, it suffices to show that the right vertical map in the square above is injective. The perverse $t$-exactness of $f_*$ and the finite pushforward compatibility of $\RH$ allow us to rewrite the target as 
\[  f_* {}^p H^d(\RH(j'_* \mathbf{Z}_p)) \simeq {}^p H^d(f_* \RH(j'_* \mathbf{Z}_p)) \simeq {}^p H^d \RH(f_* j'_* \mathbf{Z}_p).\]
We are therefore reduced to showing that the map  $j_* \mathbf{Z}_p[d] \to f_* j'_* \mathbf{Z}_p[d]$ of $\mathfrak{p}^+$-perverse sheaves induces an injection on applying ${}^p H^0 \RH(-)$. Now $\RH$ is perverse $t$-left exact (\autoref{thm:RH} (f)), so it suffices to show that $j_* \mathbf{Z}_p[d] \to f_* j'_* \mathbf{Z}_p[d]$ is an injection of $\mathfrak{p}^+$-perverse sheaves on $\mathfrak{X}_{C}$. As $j_*$ is $t$-exact (\autoref{ArtinIntCoeff}), it suffices show that $\mathbf{Z}_p[d] \to g_* \mathbf{Z}_p[d]$ is injective on $U$, where $g:U' \to U$ is the map induced by $f$. This assertion can be checked \'etale locally on $U$; but $g$ is a finite \'etale cover, so we may then reduce to the case where $g$ is split (i.e., $U'_{\reduced} = \sqcup_{i=1}^n U_{\reduced}$ with $g$ being the obvious map), in which case the claim is clear as the map $\mathbf{Z}_p[d] \to g_* \mathbf{Z}_p[d]$ even has a left-inverse.
\end{proof}
\end{proposition}

\begin{remark}[Dual form of finite pushforward compatibility]
\label{FinPushGRdual}
In the situation of \autoref{Prop:FinPushGR}, one has a dual statement that is often useful: there is a unique map $\beta_f:f_* \omega_{\mathcal{X}'}^{\RH} \to \omega_{\mathcal{X}}^{\RH}$ of  almost coherent sheaves rendering the following diagram commutative in $D^b_{\acoh}(\mathcal{X})$:
\[ \xymatrix{ f_* \omega_{\mathcal{X}'}^{\RH} \ar[r] \ar[d]^-{\beta_f} & f_* \omega^{\mydot}_{\mathcal{X}'/\cO_C}[-d] \ar[d] \\
\omega_{\mathcal{X}}^{\RH} \ar[r]  & \omega^{\mydot}_{\mathcal{X}/\cO_C}[-d], }\]
where the map on the right is the Grothendieck trace map. Moreover, the map $\beta_f$ is a surjection of almost coherent sheaves.
\end{remark}

\begin{proposition}[Describing $\cO_{\mathcal{X}}^{\RH}$ via the absolute integral closure]
\label{GRviaAIC}
Let $\mathfrak{X}$ be a finitely presented integral $\cO_C$-scheme with absolute integral closure $\pi:\mathfrak{X}^+ \to \mathfrak{X}$. Then we have a natural identification
\[ \cO_{\mathcal{X}}^{\RH}[d] = \mathrm{Image}\left({}^p H^0(\cO_{\mathcal{X}}[d]) \to {}^p H^0(\cO_{\mathcal{X}^+}[d])\right)\]
of perverse quotients of ${}^p H^0(\cO_{\mathcal{X}}[d])$.
\end{proposition}
\begin{proof}
Choose a smooth dense affine open $j:U \subset \mathfrak{X}_{C}$ with preimage $j^+:U^+ \to \mathfrak{X}_{C}^+$ in the absolute integral closure. Consider the maps
\[ j_* \mathbf{Z}_p[d] \to \pi_* j^+_* \mathbf{Z}_p[d] \xleftarrow{\simeq} \pi_* \mathbf{Z}_p[d]\]
in $D(\genX, \mathbf{Z}_p)$, where the last isomorphism comes from the first part of \autoref{thm:AIC}. Following the strategy of proof of \autoref{Prop:FinPushGR}, it suffices to show the first map has cone in ${}^{\mathfrak{p}^+} D^{\geq 0}$ (i.e., the first map is an injection of $\mathfrak{p}^+$-perverse sheaves). As $j_*$ is $t$-exact, it suffices to show that $\mathrm{cone}\left(\mathbf{Z}_p[d] \to (\pi|_U)_* \mathbf{Z}_p[d]\right) \in {}^{\mathfrak{p}^+} D^{\geq 0}(U, \mathbf{Z}_p)$. By definition of the $\mathfrak{p}^+$-perverse $t$-structure {(\autoref{PervIntBig}(b)(iii))}, this is equivalent to showing $\mathrm{cone}\left(\mathbf{F}_p[d] \to (\pi|_U)_* \mathbf{F}_p[d]\right) \in {}^p D^{\geq 0}(U, \mathbf{F}_p)$. But $\mathbf{F}_p[d]$ is a simple perverse sheaf as $U$ is smooth and the map is obviously nonzero; thus, the map is injective on all perverse cohomology sheaves, and we are reduced to checking that $(\pi|_U)_* \mathbf{F}_p[d] \in {}^p D^{\geq 0}$, which  follows from the second part of \autoref{thm:AIC}. 
\end{proof}

{
\begin{remark}
\label{GRviaAIC-noncompleted}
Suppose we are in the scenario of \autoref{rmk:noncompleteAIC}, i.e., $\cO_C$ is the $p$-adic completion of a Henselian valuation ring $\cO_C^{nc}$ and $\mathfrak{X}=\mathfrak{X}^{nc}\otimes_{\cO_C^{nc}}\cO_C$ is integral. Then we also have a natural identification
\[ \cO_{\mathcal{X}}^{\RH}[d] = \mathrm{Image}\left({}^p H^0(\cO_{\mathcal{X}}[d]) \to {}^p H^0(\cO_{(\mathcal{X}^{nc})^+}[d])\right)\]
of perverse quotients of ${}^p H^0(\cO_{\mathcal{X}}[d])$, where $(\mathcal{X}^{nc})^+$ denotes the $p$-adic completion of $(\mathfrak{X}^{nc})^+$. This follows from the same argument as in the proof of \autoref{GRviaAIC} by choosing $U\subseteq \mathfrak{X}_C$ such that $U$ is base changed from a smooth dense open affine of $\mathfrak{X}^{nc}[1/p]$ and replacing the use of \autoref{thm:AIC} by \autoref{rmk:noncompleteAIC}.
\end{remark}
}

\subsection{Behavior under alterations}
\label{ss:GRAlt}

\begin{proposition}
\label{prop:AltContain}
Let $f:\mathfrak{X}' \to \mathfrak{X}$ be an alteration between integral $\cO_C$-schemes with $p$-completion $f:\mathcal{X}' \to \mathcal{X}$. Then there is a unique surjection 
\[ \mathrm{Image}\left({}^p H^0(\cO_{\mathcal{X}}[d]) \xrightarrow{f^*} {}^p H^0(f_* \cO_{\mathcal{X}'}[d])\right) \to \cO_{\mathcal{X}}^{\RH}[d]\] 
in $\mathrm{Perv}_{\acoh}(\mathcal{X})$ such that the following diagram commutes 
\[ \xymatrix{  & {}^p H^0(\cO_{\mathcal{X}}[d]) \ar@{->>}[ld]^{\mathrm{\can}} \ar@{->>}[rd] \\ 
\mathrm{Image}\left({}^p H^0(\cO_{\mathcal{X}}[d]) \xrightarrow{\can} {}^p H^0(f_* \cO_{\mathcal{X}'}[d])\right)  \ar[rr] & & \cO_{\mathcal{X}}^{\RH}[d],}\]
\end{proposition}

\begin{proof}
Choose $j:U \subset \mathfrak{X}_{C}$ a sufficiently small dense affine open such that $U$ is smooth and the induced map $U' = f^{-1}(U) \to U$ is finite \'etale; write $j':U' \to \mathfrak{X}_{C}'$ for the pullback of $j$ to $\mathfrak{X}_{C}'$. This choice gives rise to the diagram
\[ \xymatrix{ \cO_{\mathcal{X}}[d] \ar[r] \ar[d] & \RH(j_* \mathbf{Z}_p[d]) \ar[d] \\
f_* \cO_{\mathcal{X}'}[d] \ar[r] & f_* \RH(j'_* \mathbf{Z}_p[d]) }\]
By the recipe in \autoref{GRstrAlt}, the image of ${}^p H^0(-)$ of the top horizontal map is exactly $\cO_{\mathcal{X}}^{\RH}[d]$. By the commutative diagram above, it then suffices to show that ${}^p H^0(-)$ of the right vertical map is injective. As $\RH$ commutes with $f_*$ and is $t$-left exact for the perverse $t$-structures (\autoref{thm:RH} (6)), it suffices to show that $j_* \mathbf{Z}_p[d] \to f_* j'_* \mathbf{Z}_p[d]$ has cone in ${}^{\mathfrak{p}^+} D^{\geq 0}$; in fact, we shall prove this map is an injection of $\mathfrak{p}^+$-perverse sheaves. We may rewrite $f_* j'_* \mathbf{Z}_p$ as $j_* g_* \mathbf{Z}_p$, where $g:U' \to U$ is the map induced by $f$. As $j$ is an affine \'etale map, the functor $j_*$ is $t$-exact for the $\mathfrak{p}^+$-perverse $t$-structure (\autoref{ArtinIntCoeff}), so it suffices to show that the canonical map $\mathbf{Z}_p[d] \to g_* \mathbf{Z}_p[d]$ is an injection of $\mathfrak{p}^+$-perverse sheaves. But this follows from the same argument as in \autoref{Prop:FinPushGR}: since $g$ is a finite \'etale cover, locally on the base, $g$ splits into a finite disjoint union of copies of the identity map. 
\end{proof}

\begin{remark}[Dual form of alteration compatibility]
\label{AltContainDual}
Continue with notation as in \autoref{prop:AltContain}. The dual version of the proposition asserts that we have a containment
\[ \omega_{\mathcal{X}}^{\RH} \subset \mathrm{Image}\left(H^{-d}(f_* \omega^{\mydot}_{\mathcal{X}'/\cO_C})  \xrightarrow{\Tr} H^{-d}(\omega^{\mydot}_{\mathcal{X}/\cO_C}) = \omega_{\mathcal{X}/\cO_C}\right)\]
as subsheaves of the dualizing sheaf $\omega_{\mathcal{X}/\cO_C}$. Noting that $f_*$ is $t$-left exact for the standard $t$-structure as well as $\omega^{\mydot}_{\mathcal{X}'/\cO_C} \in D^{\geq -d}$, we learn that
\[ \omega_{\mathcal{X}}^{\RH} \subset \mathrm{Image}\left(R^0 f_* \omega_{\mathcal{X}'/\cO_C}  \xrightarrow{\Tr} \omega_{\mathcal{X}/\cO_C}\right)\]
as subsheaves of the dualizing sheaf $\omega_{\mathcal{X}/\cO_C}$. Intuitively, this says that a volume form lying in $\omega_{\mathcal{X}}^{\RH}$ lifts to every alteration of $\mathcal{X}$; we shall see later (\autoref{AlmostStabTraceAlt:explicit}) that this property characterizes $\omega_{\mathcal{X}}^{\RH}$.
\end{remark}

\subsection{Describing $\omega_{\mathcal{X}}^{\RH}$ via alterations}
\label{ss:GRAlt2}

For any alteration $f:\mathfrak{Y} \to \mathfrak{X}$, we saw in \autoref{AltContainDual} that there was a natural inclusion 
\[ \omega_{\mathcal{X}}^{\RH} \subset \mathrm{Image}(R^0 f_* \omega_{\mathcal{Y}/\cO_C} \to \omega_{\mathcal{X}/\cO_C})\]
of subsheaves of $\omega_{\mathcal{X}/\cO_C}$. We shall now show that this  inclusion approximates an almost isomorphism as we vary over all possible such $\mathfrak{Y}$'s. 

\begin{notation}[The category of pointed alterations]
\label{not:GRDesAlt}
With notation as in \autoref{not:RH}, assume additionally that $\mathfrak{X}$ is integral (whence the same holds for $\mathfrak{X}_{C}$). Fix an algebraic closure $\overline{K(\genX)}$ of the function field $K(\genX)$, giving a geometric generic point $\eta_{\mathfrak{X}}$ of $\mathfrak{X}$. Let $\mathcal{P}$ denote the category of geometrically pointed alterations of $\mathfrak{X}$, i.e., all pairs $(f_{\mathfrak{Y}}:\mathfrak{Y} \to \mathfrak{X}, \eta_{\mathfrak{Y}})$, where $f_\mathfrak{Y} $ is an alteration of integral $\cO_C$-schemes and $\eta_{\mathfrak{Y}}$ is a lift of the geometric generic point of $\mathfrak{X}$ along $f_{\mathfrak{Y}}$; note that $\mathcal{P}$ is a poset. As in \autoref{not:RH}, we shall write $\mathcal{Y} = \widehat{\mathfrak{Y}}$ and $\genY$ for the $p$-adic formal completion and scheme-theoretic generic fibre of $\mathfrak{Y} \in \mathcal{P}$ respectively. 

Furthermore, if we are in the scenario of \autoref{rmk:noncompleteAIC}, i.e., $\cO_C$ is the $p$-adic completion of a Henselian valuation ring $\cO_C^{nc}$ and $\mathfrak{X}=\mathfrak{X}^{nc}\otimes_{\cO_C^{nc}}\cO_C$ is integral, then we let $\mathcal{P}^{nc}\subseteq \mathcal{P}$ denote the subcategory of those $\mathfrak{Y}$ that come from $\mathfrak{X}^{nc}$, i.e., all pairs $(f_{\mathfrak{Y}}:\mathfrak{Y} \to \mathfrak{X}, \eta_{\mathfrak{Y}})$ as above such that $\mathfrak{Y}=\mathfrak{Y}^{nc}\otimes_{\cO_C^{nc}}\cO_C$ for some alteration $\mathfrak{Y}^{nc}$ of $\mathfrak{X}^{nc}$.

\end{notation}

We shall use the following notion to tame certain infinite intersections that will appear:

\begin{definition}[Almost pro- and ind- constancy]
\label{def:AlmostIndConst}
A pro-object $\{M_i\}$ in stable $\cO_C$-linear $\infty$-category $\mathcal{C}$ (such as almost derived $\infty$-category $D_{\qcoh}(\mathcal{X})^a$) is called {\em almost pro-zero} if for every $\epsilon \in \mathfrak{m}_{\cO_C}$, the pro-object $\{M_i\}$ is pro-isomorphic to an $\epsilon$-torsion\footnote{This refers to an $\cO_C/\epsilon$-linear object of $\mathcal{C}$.} pro-object. Similarly, an ind-object $\{M_i\}$ in $\mathcal{C}$ is called {\em almost ind-zero} if for every $\epsilon \in \mathfrak{m}_{\cO_C}$, the ind-object $\{M_i\}$ is ind-isomorphic to an $\epsilon$-torsion ind-object.

For objects in the heart of a $t$-structure, being almost pro-zero is equivalent to the following: for each $\epsilon \in \mathfrak{m}$ and each index $i$, there exists a map $M_j \to M_i$ in the pro-system such that $\mathrm{Image}(M_j \to M_i)$ is annihilated by $\epsilon$. Similarly, being almost ind-zero is equivalent to the condition: for each $\epsilon \in \mathfrak{m}$ and each index $i$, there exists a map $M_i \to M_j$ in the ind-system such that $\mathrm{Image}(M_i \to M_j)$ is annihilated by $\epsilon$. More generally, for objects that are bounded with respect to a $t$-structure, being almost pro-zero or almost ind-zero is equivalent to the same condition on its cohomology groups.

A pro-object $\{M_i\}$ in $\mathcal{C}$ is called {\em almost pro-constant} if one can find a map $\{M\} \to \{M_i\}$ of pro-objects from a constant pro-object with almost pro-zero cone. An ind-object $\{M_i\}$ in $\mathcal{C}$ is called {\em almost ind-constant} if one can find a map $\{M_i\}\to \{M\}$ to a constant ind-object with almost ind-zero cone.  
\end{definition}

\begin{remark}
We caution the readers that, in the definition of almost ind-constancy (resp.\ almost pro-constancy), it is important that the map goes to (resp.\ from) the constant system $\{M\}$. If we swap $\{M\}$ and $\{M_i\}$, then we get a strictly stronger condition. For instance, if we have $\{M\}\to \{M_i\}$ of ind-objects whose cone is almost ind-zero, then $M$ is almost isomorphic to the colimit of $\{M_i\}$, and we obtain that $\{M_i\}\to \{\colim M_i\}$ has an almost ind-zero cone and thus $\{M_i\}$ is almost ind-constant under our definition. On the other hand, suppose $\varpi\in \m$ admits a compatible system of $p$-power roots, then the system $\{M_i\}=\{(\varpi^{1/p^i})\}$ is almost ind-constant (take $M=\cO_C$), but one cannot find $M$ such that the cone of $\{M\}\to \{M_i\}$ is almost ind-zero. 
\end{remark}

Given a Noetherian scheme, an ind-object in coherent sheaves has coherent colimit if and only if it is an essentially constant ind-object. We need a variant of this observation for the almost coherence.

\begin{lemma}
\label{AlmostIndConstCrit}
Let $\{M_i\}$ be an ind-object in $D^b_{\acoh}(\mathcal{X})$. Assume the following:
\begin{enumerate}[(1)]
\item There exists $N \gg 0$ such that $M_i \in D^{[-N,N]}$ for all $i$.
\item There exists some integer $m > 0$ and an index $i_0$ such that for all $j \geq i_0$ the cone of any map $M_{i_0} \to M_j$ in the diagram is annihilated by $p^m$. 
\item The colimit $\colim_i M_i$ is almost coherent.
\end{enumerate}
Then $\{M_i\}$ is almost ind-constant.
\end{lemma}

Note that condition (2) is necessary: as we are working on a $p$-adic formal scheme, colimits involve an implicit $p$-completion. For example, the diagram 
\[ \cO_{\mathcal{X}} \xrightarrow{p} \cO_{\mathcal{X}} \xrightarrow{p} \cO_{\mathcal{X}} \xrightarrow{p} ... \]
has $0$ colimit in $D_{\qcoh}(\mathcal{X})$, and hence verifies (1) and (3) but is not almost ind-constant.

\begin{proof}
By replacing $\{M_j\}$ with $\{M_j/M_{i_0}\}$ and possibly enlarging $m$, we may assume that $\{M_i\}$ is a diagram in ${D}^b_{\acoh}(\mathcal{X}_{p^m=0})$; in particular, as $\mathcal{X}_{p^m=0}$ is a scheme (instead of merely a $p$-adic formal scheme), colimits are computed at the level of underlying sheaves. Passing to fibres of the map to the direct limit, we may assume $\colim_i M_i = 0$. Filtering by cohomology sheaves, we may then also assume that each $M_i$ is in degree $0$. In this case, it suffices to prove the following statement:

\begin{itemize}
\item[$(\ast)$] Suppose $R$ is a finitely presented (and hence coherent) $\cO_C$-algebra. Let $\{M_i\}$ be a filtered diagram of almost coherent $R$-modules with $0$ direct limit. Then $\{M_i\}$ is almost ind-zero.
\end{itemize}

Fix some $\epsilon \in \mathfrak{m}_{\cO_C}$. We want to show that the inclusion $\{M_i[\epsilon]\} \to \{M_i\}$ is an ind-isomorphism, i.e., for each $M_i$, there exists some map $M_i \to M_j$ in the system with image contained in $M_j[\epsilon]$. As $M_i$ is almost coherent, we can find a finitely generated $R$-submodule $M \subset M_i$ with $\epsilon \cdot M_i/M = 0$. As $\colim_i M_i = 0$ and $M$ is finitely generated, we can find a map $M_i \to M_j$ in the system carrying $M$ to $0$. But then $M_i \to M_j$ factors as $M_i \to M_i/M \to M_j$, so the image must be $\epsilon$-torsion as $M_i/M$ is so. 
\end{proof}

\begin{theorem}[Almost stabilization of the image of the Grothendieck trace]
\label{AlmostStabTraceAlt}
The pro-object
\[ \{\mathrm{Image}\left(R^0 f_{\mathcal{Y},*} \omega_{\mathcal{Y}/\cO_C} \xrightarrow{\Tr} \omega_{\mathcal{X}/\cO_C}\right)\}_{\mathcal{P}} \]
is almost pro-constant with limit $\omega_{\mathcal{X}}^{\RH}$. Similarly, the ind-object
\[ \{\mathrm{Image}\left({}^p H^0(\cO_{\mathcal{X}}[d]) \to {}^p H^0(f_{\mathcal{Y,*}} \cO_{\mathcal{Y}}[d])\right)\}_{\mathcal{P}} \]
is almost ind-constant with colimit $\cO_{\mathcal{X}}^{\RH}{[d]}$. 
\end{theorem}
\begin{proof}
We put aside the question of describing the limiting values until the end of the proof; it will fall out of the argument. Next we note that the two almost constancy statements in the theorem are Grothendieck dual to each other, so it suffices to prove either one. In what follows, we will check the ind-constancy of the second system. 


Thus, we must verify the hypotheses (1), (2) and (3) from \autoref{AlmostIndConstCrit}. In fact, (1) is clear as the cohomological amplitude of the Grothendieck dual of $f_{\mathcal{Y,*}} \cO_{\mathcal{Y}}[d]$ is bounded independently of $\mathcal Y$. 

We now verify hypothesis (2) in \autoref{AlmostIndConstCrit}. Let $f_{\mathfrak{Y}}$: $\mathfrak{Y}\to \mathfrak{X}$ be an alteration such that $\mathfrak{Y}_C$ is regular and such that there exists a dense open affine $j: U\hookrightarrow \mathfrak{Y}_C$ with $U\to \mathfrak{X}_C$  \'etale and $E:=\mathfrak{Y}_C - U$ is an SNC divisor on $\mathfrak{Y}_C$ (such a $\mathfrak{Y}$ exists by the existence of log resolution of singularities of the generic fiber $\mathfrak{X}_C$). We need to show that for any further alteration $\mathfrak{Z}\to \mathfrak{Y}$, 
\begin{equation}
\label{equation 1 in proof of AlmostStabTrace}
\Image\left({}^p H^0(\cO_{\mathcal{X}}[d]) \to {}^p H^0(f_{\mathcal{Y,*}} \cO_{\mathcal{Y}}[d])\right) \twoheadrightarrow  
\Image\left({}^p H^0(\cO_{\mathcal{X}}[d]) \to {}^p H^0(f_{\mathcal{Z,*}} \cO_{\mathcal{Z}}[d])\right)
\end{equation}
has kernel annihilated by a fixed power of $p$ (that is independent of $\mathfrak{Z}$). 

Let $k: V\hookrightarrow \mathfrak{Z}_C$ be a dense open regular affine subset. By \autoref{Prop:FinPushGR}, we know that
\begin{equation}
\label{equation 2 in proof of AlmostStabTrace}
\Image\left({}^p H^0(\cO_{\mathcal{X}}[d]) \to {}^p H^0(f_{\mathcal{Y,*}}\RH( j_* \mathbf{Z}_p[d]))\right) \cong 
\Image\left({}^p H^0(\cO_{\mathcal{X}}[d]) \to {}^p H^0(f_{\mathcal{Z,*}}\RH( k_* \mathbf{Z}_p[d]))\right)
\end{equation}
and by the argument in \autoref{prop:AltContain}, the displayed object in \autoref{equation 2 in proof of AlmostStabTrace} is a common quotient of both sides of \autoref{equation 1 in proof of AlmostStabTrace}. Thus it is enough to show that the canonical map
\begin{equation}
\label{equation 3 in proof of AlmostStabTrace}
\Image\left({}^p H^0(\cO_{\mathcal{X}}[d]) \to {}^p H^0(f_{\mathcal{Y,*}} \cO_{\mathcal{Y}}[d])\right) \twoheadrightarrow
\Image\left({}^p H^0(\cO_{\mathcal{X}}[d]) \to {}^p H^0(\RH(f_{\mathcal{Y,*}} j_* \mathbf{Z}_p[d]))\right)
\end{equation}
has kernel annihilated by a fixed power of $p$. By the almost coherency of the displayed objects on both sides of \autoref{equation 3 in proof of AlmostStabTrace}, it is enough to show that \autoref{equation 3 in proof of AlmostStabTrace} becomes an isomorphism after inverting $p$. For this, it is enough to show the stronger statement that the natural map 
\begin{equation}
\label{equation 4 in proof of AlmostStabTrace}
 f_{\mathfrak{Y},*}\cO_{\mathfrak{Y}} \to 
 \RH (f_{\mathcal{Y,*}} j_* \mathbf{Z}_p) \simeq f_{\mathfrak{Y},*}\RH(j_*\mathbf{Z}_p)
\end{equation}
is split injective after inverting $p$, where we used \autoref{thm:RH} (d). But after inverting $p$, we have
\begin{equation}
\label{RHlog}
f_{\mathfrak{Y}_C,*}\RH(j_*\mathbf{Q}_p) \cong f_{\mathfrak{Y}_C,*}\left(\bigoplus_{i=0}^d \Omega^i_{(\mathfrak{Y}_C, E)}[-i]\right)
\end{equation}
by \autoref{thm:RH} (b) and (d). Since $(\mathfrak{Y}_C, E)$ is log smooth, we know that $\Omega^0_{(\mathfrak{Y}_C, E)}\cong \cO_{\mathfrak{Y}_C}$. Thus after inverting $p$, the map in \autoref{equation 4 in proof of AlmostStabTrace} is the inclusion of a direct summand, as wanted.

We now verify hypothesis (3) in \autoref{AlmostIndConstCrit}. Write $\mathcal{P}^{\fin} \subset \mathcal{P}$ for the full subcategory spanned by those alterations $\mathfrak{Y} \to \mathfrak{X}$ that are actually finite. We know from \cite[Theorem 4.19 (1)]{BhattAbsoluteIntegralClosure} that
 \[ \colim_{\mathcal{P}^{\fin}} f_{\mathcal{Y},*} \cO_{\genY} \simeq  \colim_{\mathcal{P}} f_{\mathcal{Y},*} \cO_{\genY} \in D_{\qcoh}(\mathcal{X}).\]

As the functor of applying  ${}^p H^0(-)$ or taking images both commute with filtered colimits for the perverse $t$-structure on $D_{\qcoh}(\mathcal{X})$ (\autoref{PervCohColimit}), we then learn that 
   \begin{align*}
  & \colim_{\mathcal{P}} \mathrm{Image}\left({}^p H^0(\cO_{\mathcal{X}}[d]) \to {}^p H^0(f_{\mathcal{Y,*}} \cO_{\mathcal{Y}}[d])\right) \\
   \simeq\; & 
  \colim_{\mathcal{P}^{\fin}} \mathrm{Image}\left({}^p H^0(\cO_{\mathcal{X}}[d]) \to {}^p H^0(f_{\mathcal{Y,*}} \cO_{\mathcal{Y}}[d])\right) \\
   \simeq\; & \mathrm{Image}\left( {}^p H^0 (\cO_{\mathcal{X}}[d]) \to {}^p H^0(\cO_{\mathcal{X}^+}[d])\right) \\
   = \;& \cO_{\mathcal{X}}^{\RH}[d],
  \end{align*}
where the last identification comes from \autoref{GRviaAIC}. As all hypotheses in \autoref{AlmostIndConstCrit} are verified, we obtain the promised almost ind-constancy. Finally, we note that this argument also gives the desired identification of the limiting values, proving the entire theorem.
\end{proof}

\begin{corollary}[Almost stabilization of the image of the Grothendieck trace: explicit form]\
\label{AlmostStabTraceAlt:explicit}
In what follows we work in the almost category.
\begin{enumerate}
\item We have
\[ \bigcap_{\mathfrak{Y} \in \mathcal{P}} \mathrm{Image}\left(R^0 f_{\mathcal{Y},*} \omega_{\mathcal{Y}/\cO_C} \xrightarrow{\Tr} \omega_{\mathcal{X}/\cO_C}\right) = \omega_{\mathcal{X}}^{\RH} \]
as subsheaves of $\omega_{\mathcal{X}/\cO_C}$.
\item The intersection in (1) is almost finite, i.e., for each $\epsilon \in \mathfrak{m}_{\cO_C}$, there exists some $\mathfrak{Y} \in \mathcal{P}$ such that the cokernel of the inclusion
\[ \omega_{\mathcal{X}}^{\RH} \subset \mathrm{Image}\left(R^0 f_{\mathcal{Y},*} \omega_{\mathcal{Y}/\cO_C} \xrightarrow{\Tr} \omega_{\mathcal{X}/\cO_C}\right) \]
from \autoref{prop:AltContain} is annihilated by $\epsilon$.
\end{enumerate}
\end{corollary}
\begin{proof}
\autoref{AlmostStabTraceAlt} gives (a) on taking limits, while (b) simply spells out almost pro-constancy in the case of interest. 
\end{proof}

{
\begin{remark}
\label{AlmostStabTraceAlt-noncomplete}
Suppose we are in the scenario of \autoref{rmk:noncompleteAIC}, i.e., $\cO_C$ is the $p$-adic completion of a Henselian valuation ring $\cO_C^{nc}$ and $\mathfrak{X}=\mathfrak{X}^{nc}\otimes_{\cO_C^{nc}}\cO_C$ is integral. Then both \autoref{AlmostStabTraceAlt} and \autoref{AlmostStabTraceAlt:explicit} hold if we replace $\mathcal{P}$ by $\mathcal{P}^{nc}$ (see \autoref{not:GRDesAlt}). More precisely, we have 
\[ \bigcap_{\mathfrak{Y} \in \mathcal{P}^{nc}} \mathrm{Image}\left(R^0 f_{\mathcal{Y},*} \omega_{\mathcal{Y}/\cO_C} \xrightarrow{\Tr} \omega_{\mathcal{X}/\cO_C}\right) = \omega_{\mathcal{X}}^{\RH} \]
in the almost category, and for each $\epsilon \in \mathfrak{m}_{\cO_C}$, there exists some $\mathfrak{Y} \in \mathcal{P}^{nc}$ (i.e., $\mathfrak{Y}=\mathfrak{Y}^{nc}\otimes_{\cO_C^{nc}}\cO_C$) such that the cokernel of the inclusion
\[ \omega_{\mathcal{X}}^{\RH} \subset \mathrm{Image}\left(R^0 f_{\mathcal{Y},*} \omega_{\mathcal{Y}/\cO_C} \xrightarrow{\Tr} \omega_{\mathcal{X}/\cO_C}\right) \]
from \autoref{prop:AltContain} is annihilated by $\epsilon$. To prove these statements, it is enough to show the analog of \autoref{AlmostStabTraceAlt} for $\mathcal{P}^{nc}$ (the analog of \autoref{AlmostStabTraceAlt:explicit} follows immediately). Now all we need to observe is that in the proof of \autoref{AlmostStabTraceAlt}, we can replace the use of \autoref{GRviaAIC} by \autoref{GRviaAIC-noncompleted}, and that \cite[Theorem 4.19 (1)]{BhattAbsoluteIntegralClosure} holds for $\mathcal{P}^{nc, fin}\subseteq \mathcal{P}^{nc}$ as well. 
\end{remark}
}

\begin{corollary}[The characteristic $0$ localization]
\label{GRchar0Mult}
Assume $\mathfrak{X}/\cO_C$ is proper, so we can identify $D^b_{\acoh}(\mathfrak{X}) \simeq D^b_{\acoh}(\mathcal{X})$ via formal GAGA.  Then the generic fibre functor
\[ D^b_{\acoh}(\mathcal{X}) \simeq D^b_{\acoh}(\mathfrak{X}) \xrightarrow{\text{invert }p} D^b_{\coherent}(\genX) \]
carries $\omega_{\mathcal{X}}^{\RH}$ to the classical Grauert-Riemenschneider canonical sheaf $\omega_{\genX}^{\rm GR} = \mathcal{J}(\genX, \omega_{\genX})$.
\end{corollary}
\begin{proof}
\autoref{AlmostStabTraceAlt} and resolution of singularities implies that we can find an alteration $f_{\mathfrak{Y}}:\mathfrak{Y} \to \mathfrak{X}$ of integral $\cO_C$-schemes such that $\genY$ is smooth and such that the inclusion
\[ \omega_{\mathcal{X}}^{\RH} \subset \mathrm{Image}\left(R^0 f_{\mathcal{Y},*} \omega_{\mathcal{Y}/\cO_C} \xrightarrow{\Tr} \omega_{\mathcal{X}/\cO_C}\right)\]
of subsheaves of $H^{-d}(\omega_{\mathcal{X}/\cO_C}^{\mydot})$ has cokernel annihilated by $p$. In particular, inverting $p$ and using GAGA, we learn that
\[ \omega_{\mathcal{X}}^{\RH}[1/p]  = \mathrm{Image}\left(R^0 f_{\genY,*} \omega_{\genY/C} \xrightarrow{\Tr} \omega_{\genX/C}\right).\] 
But $\genY$ is smooth and $\genY \to \genX$ is an alteration; it follows from the classical algebraic geometry that the right side above agrees with $\omega_{\genX}^{\RH}$, as wanted.
\end{proof}

\begin{remark}
\label{rmk:tauRHInfinityNonProper}
By using \autoref{GRchar0Mult}, we can algebraize the construction of the RH sheaves: for any integral finitely presented $\mathfrak{X}/\cO_C$, there is a unique object $\omega_{\mathfrak{X}}^{\RH} \in D^b_{\coherent}(\mathfrak{X})$ enjoying the analog of the properties discussed in \autoref{ss:GRfin}, \autoref{ss:GRAlt},  \autoref{ss:GRSm} and \autoref{ss:GRAlt2} at the level of schemes, i.e., without passage to $p$-adic formal completions. {Specifically, we can compactify $\mathfrak{X}$ to $\overline{\mathfrak{X}}$ and define  $\omega_{\mathfrak{X}}^{\RH} := \omega_{\overline{\mathfrak{X}}}^{\RH}|_{\mathfrak{X}}$. This definition is independent of the choice of the compactification by Beauville-Laszlo gluing \cite{BeauvilleLaszlo}, see \cite[Tag 0BNI]{stacks-project} and \cite[Section 5]{BhattTannakaDuality}, because on the generic fiber it agrees with the multiplier ideal sheaf (cf.\ the beginning of the proof of \autoref{thm.SmoothPullbackForRHOnNoetherian}).} 
\end{remark}

\section{Singularities over a DVR}
\label{sec:SingOverDVR}

The goal of this section is to descend the results of the previous section to the Noetherian case. Throughout this section we let $(V, \varpi, k)$ be a DVR of mixed characteristic $(0, p > 0)$.

\begin{proposition}
\label{prop:kev_is_generator} 
Suppose $(V,\varpi, k)$ is a {\color{black}Henselian} DVR of mixed characteristic. Then there exists inclusions of  valuation domains {\color{black}with $V_{\infty}$ $\varpi$-adically complete}
\[(V,\varpi, k) \to (\widehat{V},\varpi, k) \to (V_{\infty},\varpi^{1/p^\infty}, k^{1/p^\infty})\]
such that 
\begin{enumerate}

    \item There are {\color{black}{Henselian}} DVRs $\{V_{i}\}_{i\in I}$ which are finite free over $V$ {\color{black}(and complete if $V$ is)}, and forming a direct system, so that we can define 
    \[V_{\infty}^{\mathrm{nc}}:=\mathrm{colim}_{i\in I} V_{i}.\]
     \item $V_\infty:=\widehat{V_\infty^{\mathrm{nc}}}$ is perfectoid.
    \item 
    {\color{black}If $V$ is complete,} there exists a $V$-module homomorphism $V_{\infty}\xrightarrow{\kev} V$ such that $\kev|_{V_{i}}$ generates $\Hom_V(V_{i}, V)$ as a $V_i$-module for all $i\in I$.
    \item 
    {\color{black}If $V$ is complete,} for each $i$, there exists a {$V_i$-linear} splitting $\kev_i:V_{\infty}\to V_i$ of the natural inclusion such that $\kev=\kev|_{V_i}\circ \kev_i$.

\end{enumerate}
\end{proposition}
\begin{proof}
To begin the construction, fix a $p$-basis $\{\lambda_j\}_{j\in J}$ of $k$ by \cite[\href{https://stacks.math.columbia.edu/tag/07P2}{Lemma 07P2}]{stacks-project}, and fix lifts $\widetilde{\lambda}_{j\in J}$ of each $\lambda_j$ to $V$.  Note that if $k$ is already perfect, this set is empty.  Let $I$ be the set of pairs $(\Sigma,e)$ where $\Sigma\subset J$ is a finite subset and $e\in\bN$, which forms a partially ordered set with $(\Sigma_1,e_1)\leq (\Sigma_2, e_2)$ if and only if $\Sigma_1\subseteq \Sigma_2$ and $e_1\leq e_2$.

For each $(\Sigma,e)\in I$, define 
\[
V_{(\Sigma,e)}:=\frac{V[y_e,x_{j,e}\mid j\in \Sigma]}{(x_{j,e}^{p^e}-\widetilde{\lambda}_j,y_e^{p^e}-\varpi)}
\]
Set $\varpi^{1/p^e}$ and $\tilde\lambda_j^{1/p^e}$ to be the images of $y_e$ and $x_{j,e}$ in $V_{(\Sigma,e)}$, respectively. 
We now claim that $V_{(\Sigma,e)}$ is a DVR, which we check inductively for a fixed $e$ by adjoining the $p^e$-th roots one at a time.  Assume that $\Sigma=\{1,...,n\}$.
Suppose we have already shown that 
\[
U_{m}:= \frac{V[x_{j,e}\mid j\leq m]}{(x_{j,e}^{p^e}-\widetilde{\lambda_j})}
\]
is a DVR with residue field $k(\lambda^{1/p^e}_1,\dots,\lambda_{m}^{1/p^e})$.
Then $U_{m+1}:=U_{m}[x_{{m+1},e}]/(x^{p^e}_{m+1,e}-\widetilde{\lambda}_{m+1})$.  Then the minimal polynomial of $\lambda_{m+1}$ over $k(\lambda^{1/p^e}_1,\dots,\lambda_{m}^{1/p^e})$ is $x^{p^e}-\lambda_{m+1}$ since the $\lambda_i$ form a $p$-basis.  Therefore $x_{m+1,e}^{p^e}-\widetilde{\lambda}_{m+1}$ is irreducible over $U_m$ since its image is irreducible in $k(\lambda^{1/p^e}_1,\dots,\lambda_{m}^{1/p^e})[x]$ 
and so $U_{m+1}$ is a domain.  Since it is a finite domain extension of the {\color{black}\emph{Henselian}} DVR $V$, {we know that $U_{m+1}$ is local}. To show that $U_{m+1}$ it is a DVR, it is enough to show that the maximal ideal is principal, which is the case since $U_{m+1}/(\varpi)$ is the field $k(\lambda^{1/p^e}_1,\dots,\lambda_{m}^{1/p^e}, \lambda_{m+1}^{1/p^e})$, hence  $(\varpi)$ is the maximal ideal as required. Now by induction we see that
\[
U_{n}:= \frac{V[x_{j,e}\mid j\leq n]}{(x_{j,e}^{p^e}-\widetilde{\lambda_j})}
\]
is a DVR with residue field $k(\lambda_1^{1/p^e},\dots,\lambda_n^{1/p^e})$.  Finally $V_{(\Sigma,e)}\cong U_n[y]/(y^{p^e}-\varpi)$ is a domain since $y^{p^e}-\varpi$ is irreducible by Eisenstein's criterion.  It is a DVR since the maximal ideal is $(\varpi^{1/p^e})$ because 
\[V_{(\Sigma,e)}/(\varpi^{1/p^e})\cong U_n/\varpi\cong k(\lambda_1^{1/p^e},\dots,\lambda_n^{1/p^e}).\]
Note that $V_{(\Sigma,e)}$ is a free $V$-module on the basis 
\[
\mathcal{B}_{(\Sigma,e)}:=
\{
\varpi^\frac{b}{p^e}\Pi_i \widetilde{\lambda_i}^\frac{a_i}{p^e}\mid 0\leq a_i,b\leq p^e-1
\}.
\]
Given $(\Sigma_1,e_1)\leq (\Sigma_2, e_2)$, we have a mapping $V_{(\Sigma_1,e_1)} \to V_{(\Sigma_2,e_2)}$ determined by sending $\varpi^{1/p^{e_1}} \mapsto (\varpi^{1/p^{e_2}})^{p^{e_2-e_1}}$ and $\widetilde\lambda_j^{1/p^{e_1}} \mapsto (\widetilde\lambda^{1/p^{e_2}})^{p^{e_2-e_1}}$ for $j \in \Sigma_1$, and the $V_{(\Sigma,i)}$ form a directed system with respect to the partial order given above. We may view $V_\infty^{\mathrm{nc}}:=\colim_{(\Sigma, e)} V_{{\Sigma, e}}$ as a union as all of the maps in the system are injective, and moreover since $\mathcal{B}_{(\Sigma_1,e_1)}\subseteq \mathcal{B}_{(\Sigma_2,e_2)}$ if $(\Sigma_1,e_1)\leq (\Sigma_2,e_2)$ it also follows that $V_\infty^{\mathrm{nc}}$ is also a free $V$ module with basis given by $\bigcup_{(\Sigma,e)} \mathcal{B}_{(\Sigma,e)}$.
Since each $V_{{\Sigma, e}}$ is a DVR, it follows that $V_\infty^{\mathrm{nc}}$ is a (non-discrete) valuation domain with maximal ideal $(\varpi^{1/p^\infty})$ which is Henselian by \cite[\href{https://stacks.math.columbia.edu/tag/0A04}{Tag 0A04}]{stacks-project}.

{\color{black}From this point forward, we assume $V$ is complete.}
Now define  
$\kev^{\mathrm{nc}} \in \Hom_V( V_\infty^{\mathrm{nc}}, V)$ to be the unique mapping specified on the basis above by the following formula.
\[
 \kev^{\mathrm{nc}}(\varpi^{c}\Pi_i\widetilde{\lambda_i}^{d_i}) = \begin{cases}
            1 & \text{ if } c = \frac{p^e-1}{p^e} \text{\ for\ some\ } e\in \bZ_{\geq 0} \text{\ and\ } d_i=0\ \forall i \\
            0 & \text{otherwise}
        \end{cases}
\]

On the other hand, we claim that
\[
        \Hom_V(V_{(\Sigma,e)}, V) = \Phi_{(\Sigma,e)} \cdot V_{(\Sigma,e)}
    \]
    where 
    \[
        \Phi_{(\Sigma,e)}(\varpi^{\frac{a}{p^e}}\Pi_i\widetilde{\lambda_i}^\frac{b_i}{p^e}) =  
        \begin{cases}
            1 & \text{ if } a = p^e-1\text{\ and\ } b_i=0\ \forall i \\
            0 & \text{otherwise}
        \end{cases}
    \]
    To see this it is sufficient to show that the projection onto every basis vector can be written in this way, and the projection onto the 
   $\varpi^\frac{a}{p^e}\Pi_i \widetilde{\lambda_i}^\frac{b_i}{p^e} $ component
    can be obtained as 
    \[\Phi_{(\Sigma,e)}\left( \left(\varpi^\frac{p^e-a}{p^e}\Pi_i \widetilde{\lambda_i}^{-\frac{b_i}{p^e}}\right)\cdot - \right)\] for some $a$ and $b_i$.   Here we use that $\widetilde{\lambda_i}$ are units since $V_{(\Sigma,e)}$ is local and their images in the residue field are non-zero.
    
    Now we show that $\kev^{\mathrm{nc}}|_{V_{(\Sigma, e)}}$ is a unit-multiple of $\Phi_{(\Sigma,e)}$.  Observe that 
    \[
        \kev^{\mathrm{nc}}|_{V_{(\Sigma,e)}}( - ) = \Phi_{V_{(\Sigma,e)}}\Big( (1 + \varpi^{1/p^{e-1} - 1/p^e} + \varpi^{1/p^{e-2} - 1/p^e}+ \dots + \varpi^{1/p - 1/p^e})\cdot - \Big)
    \]
    and notice that $1 + \varpi^{1/p^{e-1} - 1/p^e} + \varpi^{1/p^{e-2} - 1/p^e}+ \dots + \varpi^{1/p - 1/p^e}$ is a unit.

Finally, apply the $p$-completion functor to $\kev^{\mathrm{nc}}$ to obtain $\kev: V_{\infty}\to V$, noting that $V_\infty^{\mathrm{nc}}$ and hence $V_{(\Sigma,e)}$ inject into $V_{\infty}$ since $V_{\infty}^\mathrm{nc}$ is $\varpi$-adically separated.


For fixed $(\Sigma',e')$ we define $\kev^{\rm nc}_{(\Sigma',e')}: V_{\infty}^{\rm nc}\to V_{(\Sigma', e')}$.
First note that $V_{(\Sigma',e')}$ is a free $V$ module with basis 
\[\Bigg\{
\varpi^{\frac{b}{p^{e'}}}\Pi_i \widetilde{\lambda}^{\frac{a_i}{p^{e'}}}
\;\Big|\; 0\leq b,a_i\leq p^{e'}-1
\Bigg\}.
\]

Also $V_{\infty}^{\mathrm{nc}}$ is a free $V_{(\Sigma', e')}$ module whose basis elements are given by 
\[
\Bigg\{
\left(\varpi^{\frac{1}{p^{e'}}}\right)^{\frac{b}{p^{e}}}\Pi_{i\in\Sigma'} \left(\widetilde{\lambda_i}^{\frac{1}{p^{e'}}}\right)^\frac{a_i}{p^{e}}\Pi_{j\notin\Sigma'} \widetilde{\lambda_j}^{\frac{c_j}{p^{e}}} \;\Big|\; e\in\bN,\ \ 0
\leq b,a_i,c_j\leq p^{e}-1 
\Bigg\}.
\]

Note that the given basis of $V_{\infty}^{\mathrm{nc}}$ over $V$ is the product of the elements of the two bases given above. 

Now define the map by 
\[\kev^{\mathrm{nc}}_{(\Sigma',e')}\left(
\left(\varpi^{\frac{1}{p^{e'}}}\right)^{\frac{b}{p^{e}}}\Pi_{i\in\Sigma'} \left(\widetilde{\lambda_i}^{\frac{1}{p^{e'}}}\right)^\frac{a_i}{p^{e}}\Pi_{j\notin\Sigma'} \widetilde{\lambda_j}^{\frac{c_j}{p^{e}}}\right)
= \begin{cases}
            1 & \text{\ when\ } b = {p^{e}-1} \text{\ for\ some\ } e\in \bZ_{\geq 0},\\
            &  \qquad \quad a_i=0\ \forall i \text{\ and\ }  c_j=0\ \forall j\\
            0 & \text{otherwise}
        \end{cases}
%
\]
Here $\kev^{\mathrm{nc}}_{(\Sigma',e')}$ is $V_{(\Sigma',e')}$-linear by construction since we extended it from a $V_{(\Sigma',e')}$-basis, and in particular it is a splitting since it sends the basis vector $1$ to $1$. 
Furthermore, we can see that it satisfies $\kev^{\mathrm{nc}}=\kev^{\mathrm{nc}}|_{V_{(\Sigma',e')}}\circ \kev^{\mathrm{nc}}_{(\Sigma',e')}$ by noting that the images of the $V$-basis of $V_{\infty}^{\mathrm{nc}}$ are the same  under both maps, which are both $V$-linear.
Now obtain the splitting of the statement by taking the $p$-completion, and noting that each $V_{(\Sigma,e)}$ is already complete.
\end{proof}

Below we summarize the notation used in this section. We emphasize that this notation differs from the one used in the previous section.
\begin{notation}
    \label{notation.ForSchemesOverADVR}
Throughout the rest of this section, unless otherwise indicated, we set
\begin{enumerate}
    \item $X$ to be an integral scheme flat over $\Spec(V)$ of relative dimension $d$ (thus $\dim X=d+1$);
    \item $X_e:= X\otimes_VV_e$, $X_\infty^{\mathrm{nc}}=X\otimes_VV_\infty^{\mathrm{nc}}$, and $X_\infty:= X \otimes_VV_\infty$;
    \item $\pi \colon X_\infty \to X$ to be the natural projection;
    \item $\widehat{X}_\infty$ to be the $p$-adic formal scheme associated to $X_\infty$;
    \item $\kev$: $V_\infty\to V$ and $\kev_e$: $V_\infty\to V_e$ to be the maps in \autoref{prop:kev_is_generator} {\color{black} in the case that $V$ is complete.}
    \end{enumerate}
In particular, $X_\infty$ plays the role of $\mathfrak{X}$ and $\widehat{X}_\infty$ plays the role of $\mathcal{X}$ from the previous sections. 

We further set $\omega^\mydot_{X/V}$ be the relative dualizing complex over $V$, i.e., such that the lowest nonvanishing cohomology degree of $\omega^{\mydot}_{X/V}$ is $-d$, and we set $\omega^\mydot_{X_e/V_e}:=\omega^\mydot_{X/V}\otimes_VV_e$ and $\omega^\mydot_{X_\infty/V_\infty}:= \omega^\mydot_{X/V}\otimes_VV_\infty$, which are  relative dualizing complexes of $X_e$ and $X_\infty$ over $V_e$ and $V_\infty$, respectively. We will abuse notations a bit and also use $\kev$ to denote the induced map $$\kev: \pi_*\omega_{X_\infty/V_\infty} = \pi_* \left(\omega_{X/V} \otimes_{V} V_\infty \right) \to \omega_{X/V}.$$ 
\end{notation}


We summarize the following results from \autoref{sec:RH_subsheaves}, applied in the setup of \autoref{notation.ForSchemesOverADVR}.

\begin{theorem}
\label{thm.InputsFromSectionRH}
With notation as in \autoref{notation.ForSchemesOverADVR}. Suppose $V$ is {\color{black}Henselian}, $X\to \Spec(V)$ is proper and $X_\infty$ is integral. Then there exists a quasi-coherent, almost coherent subsheaf $\omega^\RH_{X_\infty/V_\infty}$ of $\omega_{X_\infty/V_\infty}$ such that
\begin{enumerate}
    \item In the almost category, $\omega^\RH_{X_\infty/V_\infty}$ can be identified with 
    $$\bigcap_{f: Y_\infty\to X_\infty} \Tr(f_*\omega_{Y_\infty/V_\infty})$$
    where $Y_\infty$ runs over all alterations of $X_\infty$ in a fixed algebraic closure $\overline{K(X_\infty)}$ of the fraction field $K(X_\infty)$ of $X_\infty$ such that $Y_\infty\cong Y_e\otimes_{V_e}V_\infty$ for some $e$ and some alteration $Y_e$ of $X_e$.  
    \item For every $m>0$, there exists an alteration $f \colon Y_\infty \to X_\infty$ as above (i.e., $Y_\infty\cong Y_e\otimes_{V_e}V_\infty$ for some $e$ and some alteration $Y_e$ of $X_e$) such that the cokernel of
    \[
    \omega^{\RH}_{X_\infty/V_\infty} \to \Tr (f_*\omega_{Y_\infty/V_\infty})
    \]
    is annihilated by $p^{1/m}$.
    \item $\omega^{\RH}_{X_\infty/V_\infty}[1/p]$ agrees with the Grauert-Riemenschneider canonical sheaf / multiplier submodule  $\omega^{\emph{GR}}_{X_\infty[1/p]}=\mJ(X_\infty[1/p], \omega_{X_\infty[1/p]})$ (note that $X_\infty[1/p]$ is a proper scheme over the characteristic zero field $V_\infty[1/p]$).
\end{enumerate}
\end{theorem}
\begin{proof}
By \autoref{def:GRSheaf} applied to $\cO_C=V_\infty$, $\mathfrak{X}=X_\infty$ and $\mathcal{X}=\widehat{X}_\infty$, we have an almost coherent subsheaf $\omega_{\widehat{X}_\infty/V_\infty}^{\RH}\subseteq \omega_{\widehat{X}_\infty/V_\infty}$ (on the formal scheme $\widehat{X}_\infty$). Since $X_\infty\to\Spec(V_\infty)$ is proper, by formal GAGA, we know that $\omega_{\widehat{X}_\infty/V_\infty}^{\RH}$ corresponds to an almost coherent subsheaf of $\omega_{X_\infty/V_\infty}$ (in the almost category). Let $\omega_{{X}_\infty/V_\infty}^{\RH}$ be the $!$-realization of that sheaf. We have 
$$\omega^\RH_{X_\infty/V_\infty} = (\omega^\RH_{X_\infty/V_\infty})_!^a = (\varpi^{1/p^\infty}) \otimes \omega^\RH_{X_\infty/V_\infty}$$
is a quasi-coherent, almost coherent, (honest) subsheaf of $\omega_{{X}_\infty/V_\infty}$. Since $V_\infty$ is the $p$-adic completion of $V_\infty^{\mathrm{nc}}$ and $X_\infty=X_\infty^{\mathrm{nc}}\otimes_{V_\infty^{\mathrm{nc}}}V_\infty$, (a) and (b) both follow from \autoref{AlmostStabTraceAlt-noncomplete}: every alteration $Y^{\mathrm{nc}}_\infty$ of $X^{\mathrm{nc}}_\infty$ satisfies $Y^{\mathrm{nc}}_\infty=Y_e\otimes_{V_e}V^{\mathrm{nc}}_\infty$ for some $e\gg 0$ and some alteration $Y_e$ of $X_e$ since $V^{\mathrm{nc}}_\infty=\varinjlim_eV_e$. Finally, (c) follows from \autoref{GRchar0Mult}.
\end{proof}

\begin{definition}
\label{def.tauRH}
With notation as in \autoref{notation.ForSchemesOverADVR}, suppose $V$ is complete, $X\to \Spec(V)$ is proper and $X_\infty$ is integral. Then we define
$$\tau_{\RH}(\omega_{X/V}):= \kev(\pi_*\omega^\RH_{X_\infty/V_\infty}).$$
\end{definition}

{We emphasize that the image of a quasi-coherent sheaf under a qcqs map is quasi-coherent (see \cite[Tag 01XJ]{stacks-project}). In particular, $\tau_{\RH}(\omega_{X/V})$ is a coherent sheaf on $X$ since it is a quasi-coherent subsheaf of a coherent sheaf $\omega_{X/V}$.}

\begin{theorem}
    \label{thm.SingleAlterationWithPToEpsilon}
    With notation as in \autoref{def.tauRH}, there exists $e_0 > 0$ such that for every $e \geq e_0$, and every alteration $Y \to X$ with $\varpi^{1/p^e} \in \cO_Y$, we have that 
    \[
        \tau_{\RH}(\omega_{X/V}) \subseteq \Tr_{Y/X}(f_* \varpi^{1/p^e}\omega_{Y/V}).
    \]
    Furthermore, for every $e \geq e_0$, there exists an alteration $Y_e \to X$ with $\varpi^{1/p^e} \in \cO_{Y_e}$ such that for every further alteration $Y \to Y_e$ with composition $f : Y \to Y_e \to X$, we have that
    \[
        \tau_{\RH}(\omega_{X/V}) = \Tr_{Y/X}(f_* \varpi^{1/p^e}\omega_{Y/V}).
    \]
    In particular, for every $e \geq e_0$
    \[
        \tau_{\RH}(\omega_{X/V}) =  \bigcap_{Y \to X} \Tr_{Y/X}(f_* \varpi^{1/p^e}\omega_{Y/V}).
    \]
    {\color{black}Finally, if $V$ is the completion of a Henselian DVR $V'$ and $X \to \Spec V$ is the base change of $X' \to \Spec V'$, then we may take $Y_e \to X$ to be the base change of an alteration $Y_e' \to X'$.}
\end{theorem}
\begin{proof}
 We do this in two steps. First, we show that there exists $e_0$ such that for all $e\geq e_0$ and all alterations $f_e : Y \to X_e$ over $X_e$, we have that $\tau_{\RH}(\omega_{X/V})$ is contained inside $\Tr_{Y/X}(f_* \varpi^{1/p^e}\omega_{Y/V})$. To see this,  we form the base change $f_{\infty} : Y_{\infty} := Y \times_{V_e} V_{\infty} \to X_{\infty} := X_e \times_{V_e} V_{\infty}$. Consider the following diagram
\[\xymatrix{
(f_* \omega_{Y/V_e}) \otimes_{V_e} V_{\infty} \ar@{=}[r] & (f_{\infty})_* \omega_{Y_\infty/V_\infty} \ar[d] \ar[r]^-{\Tr} & \omega_{X_\infty/V_\infty} \ar[d] \ar@/^4pc/[ddd]^{\kev}\\
& f_* \omega_{Y/V_e} \ar[d]_{\cong} \ar[r]^{\Tr} & \omega_{X_e/V_e}\ar[d]_{\cong} \ar@/^2pc/[dd]^<<<<{\kev|_{V_e}} \\
& f_* \omega_{Y/V} \ar@/_1pc/[rd]_{\Tr_{Y/X}} \ar[r]^{\Tr} & \omega_{X_e/V} \ar[d]_{\Tr}\\
&  & \omega_{X/V}.
}
\]%
Here the two upper unlabeled vertical arrows are simply $\mathbb{T}_e : V_{\infty} \to V_e$ base changed to $Y$ and $X$ respectively, and the isomorphism is chosen to make the composition $$\omega_{X_e/V_e}\cong \omega_{X_e/V}\xrightarrow{\Tr_{X_e/X}}\omega_{X/V}$$ equal to the map induced by $\kev_{V_e}$: this is possible since both the $\Tr_{X_e/X}$ and $\kev_{V_e}$ generate $\Hom(\omega_{X_e/V}, \omega_{X/V})$ (see \autoref{prop:kev_is_generator}). The rest of the diagram commutes by \autoref{prop:kev_is_generator} and basic properties of the trace map.

By \autoref{def.tauRH} and the Noetherianity of $\omega_{X/V}$, we know that there exists $e_0$ such that 
$$\tau_{\RH}(\omega_{X/V})= \kev(\pi_*(\varpi^{1/p^e})\omega^{\RH}_{X_\infty/V_\infty})$$ 
for all $e\geq e_0$. We fix this $e_0$ throughout the rest of the proof. For any $e\geq e_0$, notice that $(\varpi^{1/p^e})\omega^{\RH}_{X_\infty/V_\infty}\subseteq (\varpi^{1/p^e})\Tr((f_{\infty})_* \omega_{Y_\infty/V_\infty})$ by \autoref{thm.InputsFromSectionRH}. By the commutative diagram above we have that 
\begin{equation}
    \label{eq.EasyContainmentTauRHInAlterationMidProof}
    \tau_{\RH}(\omega_{X/V}) \subseteq \Tr_{Y/X}(f_* \varpi^{1/p^e}\omega_{Y/V}).
\end{equation}
for all $e\geq e_0$. This completes the proof of the first part of the theorem.  {\color{black}  Note that this part of the theorem also works for alterations base changed from $V' \to V$, if applicable.  }

Next we will show the existence of $Y_e$ for any fixed $e\geq e_0$. We first prove the following.
\begin{claim}
\label{clm:stabilization}
For any $e\geq e_0$, we have 
\[ 
    \tau_{\RH}(\omega_{X/V})=\kev\Big(\pi_*\big((\varpi^{1/p^e})\Tr((f_\infty)_*\omega_{Y_\infty/X_\infty})\big)\Big)
\]
for some alteration $Y_\infty\to X_\infty$ such that $Y_\infty=Y_1\otimes_{V_{e_1}}V_\infty$ where $f_1$: $Y_1\to X_{e_1}$ is an alteration for some $e_1 \geq e$. {\color{black}In fact, we may even assume that $Y_1 \to X_{e_1}$ is based changed from $Y_1' \to X'_{e_1}$.}
\end{claim}
\begin{proof}[Proof of Claim]
By \autoref{thm.InputsFromSectionRH}, we know that there exists an alteration $Y_\infty\to X_\infty$ (and we may assume  {\color{black}$Y_\infty=Y_1'\otimes_{V'_{e'}}V_\infty$ where $Y_1'\to X_{e_1}'$ is an alteration for some $e_1\geq e$}) such that 
$$(\varpi^{1/p^{e+1}})\Tr((f_\infty)_*\omega_{Y_\infty/V_\infty}) \subseteq \omega^{\RH}_{X_\infty/V_\infty}.$$
It follows that 
\begin{align*}
\tau_{\RH}(\omega_{X/V}) & =\kev(\pi_*(\varpi^{1/p^{e+1}})\omega^{\RH}_{X_\infty/V_\infty}) \\
& \supseteq \kev(\pi_*(\varpi^{1/p^{e}})\Tr((f_\infty)_*\omega_{Y_\infty/X_\infty})) \\
& \supseteq \kev(\pi_*(\varpi^{1/p^{e}})\omega^{\RH}_{X_\infty/V_\infty})\\
& = \tau_{\RH}(\omega_{X/V})
\end{align*}
where the first and last equality follows by our choice of $e_0$, the first containment follows by our choice of $Y_\infty$ and the second containment follows by \autoref{thm.InputsFromSectionRH}. Thus we have equality throughout and the Claim follows. 
\end{proof}

At this point, for any $e>e_0$, we consider the following commutative diagram:
\[
    \xymatrix{
        (f_{1*} \omega_{Y_1/V_{e_1}}) \otimes_{V_{e_1}} V_{\infty} \ar@{=}[r] & (f_{\infty})_* \omega_{Y_\infty/V_\infty} \ar[d] \ar[r]^-{\Tr} & \omega_{X_\infty/V_\infty} \ar[d] \ar@/^4pc/[ddd]^{\kev}\\
        & f_{1*} \omega_{Y_1/V_{e_1}} \ar[r]^-{\Tr} \ar[d]^{\cong}& \omega_{X_{e_1}/V_{e_1}}\ar[d]^{\cong}\ar@/^2pc/[dd]^<<<<{\kev|_{V_{e_1}}} \\
        & f_{1*}\omega_{Y_1/V} \ar[r]^{\Tr} \ar@/_1pc/[rd]_{\Tr_{Y_1/X}} & \omega_{X_{e_1}/V} \ar[d]_{\Tr}\\
        & &\omega_{X/V}.
    }
\]
Here $Y_\infty$ and $Y_1$ (and hence $e_1$) are constructed from \autoref{clm:stabilization}, the two upper unlabeled vertical arrows are induced by $\mathbb{T}_{e_1} : V_{\infty} \to V_{e_1}$ base changed to $Y_1$ and $X$ respectively (in particular, they are split surjective by \autoref{prop:kev_is_generator}), and the isomorphism is chosen to make the composition $$\omega_{X_{e_1}/V_{e_1}}\cong \omega_{X_{e_1}/V}\xrightarrow{\Tr_{X_{e_1}/X}}\omega_{X/V}$$ equal to the map induced by $\kev_{V_{e_1}}$: this is possible since both $\Tr_{X_{e_1}/X}$ and $\kev_{V_{e_1}}$ generate $\Hom(\omega_{X_{e_1}/V}, \omega_{X/V})$ (see \autoref{prop:kev_is_generator}). The rest of the diagram commutes by \autoref{prop:kev_is_generator} and basic properties of the trace map.

It follows from the commutative diagram above and \autoref{clm:stabilization} that 
\[
    \tau_{\RH}(\omega_{X/V})=\kev(\pi_*(\varpi^{1/p^e})\Tr((f_{\infty})_* \omega_{Y_\infty/V_\infty})) = \Tr_{Y_1/X}(f_{1*} \varpi^{1/p^e} \omega_{Y_1/V}).
\]
Set $Y_e := Y_1$.  For any further alteration $f$: $Y \to Y_e$, we immediately see that 
\[
    \tau_{\RH}(\omega_{X/V}) = 
     \Tr_{Y_1/X}(f_* \varpi^{1/p^e} \omega_{Y_1/V}) \supseteq  \Tr_{Y/X}(f_* \varpi^{1/p^e} \omega_{Y/V}).
\]
The reverse containment is \autoref{eq.EasyContainmentTauRHInAlterationMidProof} and we have proven the desired result. The ``particular'' statement follows immediately since the intersection stabilizes as we have shown above. 

{\color{black}  For the final statement, our construction guarantees that $Y_e \to X$ is the base change of some $Y_e' \to X'$ as desired.}
\end{proof}

\begin{remark}
\label{rmk:NormalGemoetricallyIntegral}
Now suppose $X$ is normal, integral, proper and flat over $\Spec(V)$. {We claim that, when $V$ is complete, up to replacing $V$ by another DVR, we may assume that $X_{\infty}$ is integral.}
To this end, observe that $H^0(X, \cO_X)$ is a normal, finite domain extension of $V$ and hence $H^0(X, \cO_X)$ is a complete DVR of mixed characteristic. Therefore, upon replacing $V$ by $H^0(X, \cO_X)$, we may assume $X\to \Spec(V)$ is such that $X_\infty$ is integral: it is enough to show that $X_\infty[1/p]$ is integral, but we know that $X_\infty[1/p]=X[1/p]\otimes_{V[1/p]}V_\infty[1/p]$ is normal (\cite[\href{https://stacks.math.columbia.edu/tag/038N}{Tag 038N}]{stacks-project}) and we have $$H^0(X_\infty[1/p], \cO_{X_\infty[1/p]})\cong V_\infty[1/p] \otimes_{V[1/p]} H^0(X[1/p], \cO_{X[1/p]})= V_\infty[1/p]$$
is a field. Thus $X_\infty[1/p]$ is normal and connected, and thus integral. 
\end{remark}

\begin{definition}
\label{def.tauAlt}
Let $X\to \Spec(V)$ be a separated, finite type, and flat morphism of schemes such that $X$ is integral. We define
$$\tau_{\rm alt}(\omega_{X/V}, \varpi^{1/p^e}) := \bigcap_{Y\to X} \Tr_{Y/X}(f_*\varpi^{1/p^e}\omega_{Y/V})$$
where $f$: $Y\to X$ runs over all alterations taken inside a fixed algebraic closure $\overline{K(X)}$ of the fraction field of $X$ such that $\varpi^{1/p^e}\in \cO_Y$. Here the intersection is taken in the category of sheaves of $\cO_X$-modules.
\end{definition}

Note that, a priori, $\tau_{\rm alt}(\omega_{X/V}, \varpi^{1/p^e})$ is only a sheaf of $\cO_X$-modules. Our next result will show that it is actually a coherent subsheaf of $\omega_{X/V}$ for $e \gg 0$.

\begin{theorem}
\label{thm.tauAlt}
Under \autoref{def.tauAlt} and assume $V$ is complete, we have 
\begin{enumerate}
    \item Up to isomorphisms, $\tau_{\rm alt}(\omega_{X/V}, \varpi^{1/p^e})$ depends only on $X$ and the $\mathbf{Q}$-Cartier divisor $\frac{1}{p^e}\Div(\varpi)$ (i.e., it does not depend on the choice of $V$). 
\item There exists $e_0>0$ such that 
$\tau_{\rm alt}(\omega_{X/V}, \varpi^{1/p^e}) = \tau_{\rm alt}(\omega_{X/V}, \varpi^{1/p^{e'}})$ for all $e, e' \geq e_0$, and for all $e \geq e_0$, there exists an alteration $Y_e \to X$ with $\varpi^{1/p^e} \in \cO_{Y_e}$ such that for every further alteration $Y \to Y_e$ with composition $f : Y \to Y_e \to X$, we have that
\[
\tau_{\rm alt}(\omega_{X/V}, \varpi^{1/p^e}) = \Tr_{Y/X}(f_* \varpi^{1/p^e}\omega_{Y/V}).
\]
In particular, $\tau_{\rm alt}(\omega_{X/V}, \varpi^{1/p^e})\subseteq \omega_{X/V}$ is a coherent sheaf.  
\end{enumerate}
\end{theorem}
\begin{proof}
First, we assume additionally that $X\to \Spec(V)$ is proper. Let $W=H^0(X^{\rm N}, \cO_{X^{\rm N}})$ where $X^{\rm N}$ denotes the normalization of $X$. By \autoref{rmk:NormalGemoetricallyIntegral} and \autoref{thm.SingleAlterationWithPToEpsilon}, we know that there exists $e_0>0$ such that for all $e>e_0$, 
$$\tau_{\rm alt}(\omega_{X^{\rm N}/W}, \varpi^{1/p^e}) = \tau_{\RH}(\omega_{X^{\rm N}/W})=\Tr_{Y/X}(f_*\varpi^{1/p^e}\omega_{Y/W})$$
for every alteration $f$: $Y\to X^{\rm N}$ that factors over some fixed alteration $Y_e\to X^{\rm N}$. If we fix an isomorphism $W\cong \omega_{W/V}$, then the following commutative diagram 
\[
\xymatrix{
\omega_{Y/W} \ar[d]^{\Tr}\ar[r]^\cong & \omega_{Y/V} \ar[d]^{\Tr} \\
\omega_{X^{\rm N}/W} \ar[r]^\cong & \omega_{X^{\rm N}/V}
}\]
shows that 
$$\tau_{\rm alt}(\omega_{X^{\rm N}/V}, \varpi^{1/p^e})\cong \Tr_{Y/X}(f_*\varpi^{1/p^e}\omega_{Y/V})\cong \Tr_{Y/X}(f_*\varpi^{1/p^e}\omega_{Y/W})\cong \tau_{\rm alt}(\omega_{X^{\rm N}/W}, \varpi^{1/p^e})$$
for every alteration $f$: $Y\to X^{\rm N}$ that factors over some fixed alteration $Y_e\to X^{\rm N}$. In particular, this shows that up to isomorphisms $\tau_{\rm alt}(\omega_{X^{\rm N}/V}, \varpi^{1/p^e})$ does not depend on the choice of $V$: if we have $X^{\rm N}\to \Spec(V')$ is proper and flat for another complete DVR $V'$, then $W=H^0(X^{\rm N}, \cO_{X^{\rm N}})$ is a finite DVR extension over $V'$ and so after fixing an isomorphism $W\cong\omega_{W/V'}$, the same argument as above shows that
$$\tau_{\rm alt}(\omega_{X^{\rm N}/V'}, \varpi^{1/p^e})\cong \tau_{\rm alt}(\omega_{X^{\rm N}/W}, \varpi^{1/p^e}).$$
Next, notice that in the intersection defining $\tau_{\rm alt}(\omega_{X/V}, \varpi^{1/p^e})$, we may clearly restrict to alterations $Y\to X$ that factor through $X^{\rm N}$. It follows that 
$$\tau_{\rm alt}(\omega_{X/V}, \varpi^{1/p^e})\cong \Tr_{X^{\rm N}/X}(\tau_{\rm alt}(\omega_{X^{\rm N}/V}, \varpi^{1/p^e})).$$
Therefore, both parts (a) and (b) for $X$ follow from the already established result for $X^{\rm N}$. 

Finally, for a general $X$ as in the theorem, we let $X'\to X$ be a compactification of $X$. Then since every alteration $Y\to X$ can be compactified to an alteration $Y'\to X'$, we know that $\tau_{\rm alt}(\omega_{X/V}, \varpi^{1/p^e}) = \tau_{\rm alt}(\omega_{X'/V}, \varpi^{1/p^e})|_{X}$. Now both parts (a) and (b) follow from the already established result for $X'$ (as $X'\to \Spec(V)$ is proper).
\end{proof}
{\color{black}
We next drop the completeness assumption on $V$ in \autoref{thm.tauAlt}.

\begin{theorem}
\label{thm.tauAltNonComplete}
Under \autoref{def.tauAlt}, there exists $e_0>0$ such that 
$$\tau_{\rm alt}(\omega_{X/V}, \varpi^{1/p^e}) = \tau_{\rm alt}(\omega_{X/V}, \varpi^{1/p^{e'}})$$ for all $e, e' \geq e_0$.
Moreover, for all $e \geq e_0$, there exists an alteration $Y_e \to X$ with $\varpi^{1/p^e} \in \cO_{Y_e}$ such that for every further alteration $Y \to Y_e$ with composition $f : Y \to Y_e \to X$, we have that
\[
\tau_{\rm alt}(\omega_{X/V}, \varpi^{1/p^e}) = \Tr_{Y/X}(f_* \varpi^{1/p^e}\omega_{Y/V}).
\]
In particular, $\tau_{\rm alt}(\omega_{X/V}, \varpi^{1/p^e})\subseteq \omega_{X/V}$ is a coherent sheaf.  

Furthermore, if $V \to \overline{V}$ is the Henselization or the completion of $V$, then for any irreducible component $\overline{Y_e}$ of the base changed map $Y_e \otimes_V \overline{V} \to X \otimes_V \overline{V}=: \overline{X}$ and any further alteration $\overline{Y}\to \overline{Y_e}$, we have 
\[
\tau_{\rm alt}(\omega_{\overline{X}/\overline{V}}, \varpi^{1/p^e}) = \Tr_{\overline{Y}/\overline{X}}(f_* \varpi^{1/p^e}\omega_{\overline{Y}/\overline{V}}).
\]
\end{theorem}

\begin{proof}
The statement that $\tau_{\rm alt}(\omega_{X/V}, \varpi^{1/p^e}) = \tau_{\rm alt}(\omega_{X/V}, \varpi^{1/p^{e'}})$ for all $e, e' \geq e_0$ follows by base change and the later statements (using the coherence of $\tau_{\rm alt}(\omega_{X/V}, \varpi^{1/p^e})\subseteq \omega_{X/V}$). We thus turn our focus to the later statements.

Without loss of generality we may replace $X$ by its normalization to assume $X$ is normal. We consider the extensions 
$$V\to V^h \to \widehat{V}$$  
where $V^h$ is the Henselization of $V$ and form the corresponding base changes $X' := X \otimes_V V^h$, $X'' := X \otimes_V \widehat{V}$, $f ' : X' \to \Spec V^h$ and $f'' : X'' \to \Spec \widehat{V}$.  These base changes are also normal and hence a disjoint union of their irreducible components.

Fix $e_0$ from \autoref{thm.SingleAlterationWithPToEpsilon} that works for the base change $X'' \to \widehat{V}$.  For any $e > e_0$, \autoref{thm.SingleAlterationWithPToEpsilon} guarantees that there exists an alteration $Y_e' \to X'$ so that the base change to $\widehat{V}$, $Y_e'' \to X''$, has the property that 
\[
    \tau_{\alt}(\omega_{X''/\widehat{V}}, \varpi^{1/p^e}) = \Tr_{Y''_e/X''}(f''_* \varpi^{1/p^e}\omega_{Y''_e/\widehat{V}}).
\]
This implies that $\Tr_{Y''_e/X''}(f''_* \varpi^{1/p^e}\omega_{Y''_e/\widehat{V}}) \subseteq \Tr_{Y''/X''}(g''_* \varpi^{1/p^e}\omega_{Y''/\widehat{V}})$ for any alteration $g'' : Y'' \to X''$ so that $\varpi^{1/p^e} \in \Gamma(Y'', \cO_{Y''})$.  But then by flat base change we see that 
\[
    \Tr_{Y'_e/X'}(f'_* \varpi^{1/p^e}\omega_{Y'_e/V^h}) \subseteq \Tr_{Y'/X'}(g'_* \varpi^{1/p^e}\omega_{Y'/V^h})
\]
for any alteration $Y' \to X'$.  In particular, $\tau_{\alt}(\omega_{X'/V^h}, \varpi^{1/p^e})$ is a coherent sheaf and the intersection defining it stabilizes.   Thus we have proven the result when $V = V^h$ is Henselian.  

We now need to descend to $X \to \Spec (V)$.  For each $e > e_0$, we know there exists an alteration $Y_e' \to X'$ which can be used the compute the intersection $\tau_{\alt}(\omega_{X'/V^h}, \varpi^{1/p^e})$.  As $V^h$ is a colimit of pointed \'etale extensions, we know there exists a pointed \'etale extension $V \subseteq W$ so that $X_W := X \otimes_V W$ has an alteration $g : Z_e \to X_W$ that base changes to $Y_e' \to X'$.  If $Z \to Z_e \to X_W$ is a further alteration with $Z$ normal integral, then the base change 
\[
    g' : Z_{V^h} =: Z' \to Y_e' \to X'
\]
is an alteration, although $Z'$ need not be integral.  It is still normal though, and hence it is a disjoint union of integral alterations $Z' = \coprod Z^j$ each of which has a map $g'_j : Z^j \to X'$ factoring through $Y_e' \to X'$.  Hence 
\[
    \Tr_{Z^j/X'}(g'_{j*}\varpi^{1/p^e}\omega_{Z^j/V^h}) = \Tr_{Y'_e/X'}(f'_* \varpi^{1/p^e}\omega_{Y'_e/V^h}).
\]
and thus we have 
\[
    \Tr_{Z'/X'}(g'_{*}\varpi^{1/p^e}\omega_{Z'/V^h}) = \Tr_{Y'_e/X'}(f'_* \varpi^{1/p^e}\omega_{Y'_e/V^h}).
\]
It follows by flat base change that
\[
    \tau_{\alt}(\omega_{X_W/W}, \varpi^{1/p^e}) = \Tr_{Z_e/X_W}(g_* \varpi^{1/p^e}\omega_{Z_e/W}).
\]
That is, the intersection defining the left side is computed by a single alteration.  

We are almost done but the composition $Z_e \to X_W \to X$ is not proper, although it is generically finite.  Nonetheless, by \autoref{damn_lemma}, we can find an alteration $Y_e \to X$, so that $Y_e \times_V W \to X_W$ can be factored through $Z_e \to X_W$. To complete the proof, note that for any further alteration $f : Y \to Y_e \to X$, we have 
\[
\Tr_{Y/X}(f_{*}\varpi^{1/p^e}\omega_{Y/V}) \subseteq \Tr_{Y_e/X}(f_{e*} \varpi^{1/p^e}\omega_{Y_e/V}).
\]
However these are equal after base changing by the faithfully flat map $\Spec (W) \to \Spec (V)$.  Hence this containment is an equality which proves the main part of the theorem. The final statement about compatibility with base change to Henselization or completion follows from the construction.
%
%
%
\end{proof}

}

\begin{definition}
\label{def.tauAltEpsilon}
Let $X\to \Spec(V)$ be a separated, finite type, flat morphism of schemes such that $X$ is integral. We set
$$\tau^a_{\rm alt}(\omega_{X/V}) := \tau_{\rm alt}(\omega_{X/V}, \varpi^{1/p^e})  \text{ for all $e\gg0$.}$$
By \autoref{thm.tauAltNonComplete}, this is well-defined, independent of the choice of $V$ (up to isomorphisms) and $e \gg 0$. Moreover there exists a fixed $e_0>0$ such that for every $e>e_0$, there exists an alteration $Y_e\to X$ such that for all further alterations $f$: $Y\to Y_e\to X$, we have that $\tau^a_{\rm alt}(\omega_{X/V})= \Tr_{Y/X}(f_* \varpi^{1/p^e}\omega_{Y/V})$. Note that we choose the notation $\tau^a_{\rm alt}$ to emphasize that this is a $\varpi$-almost version of the alteration (i.e., there is always a $(\varpi^{1/p^\infty})$-perturbation built into this definition).
\end{definition}

The following properties of $\tau^a_{\rm alt}(\omega_{X/V})$ are immediate consequences.

\begin{corollary}
    \label{cor.TauRHInvertPEqualsJNoetherianNoPair}
    With notation as in \autoref{def.tauAltEpsilon}, we have 
    $$\tau^a_{\rm alt}(\omega_{X/V})[1/p] = \mJ(X[1/p], \omega_{X[1/p]})$$ 
    where the right hand side is the multiplier submodule (i.e., the Grauert-Riemenschneider canonical sheaf) of $X[1/p]$.  More generally, if $X = \Spec R$ is affine, the formation of $\tau^a_{\rm alt}$ commutes with arbitrary localization.  
\end{corollary}
\begin{proof}
    We may choose $e > 0$ and $f : Y \to X$ an alteration such that $$\tau^a_{\rm alt}(\omega_{X/V}) = \Tr_{Y/X}(f_* \varpi^{1/p^e}\omega_{Y/V})$$ by \autoref{thm.tauAltNonComplete}. Enlarging $Y$ if necessary we may assume that $Y[1/p]$ is nonsingular (in fact, it suffices to assume it factors through a nonsingular variety).  Inverting $p$ we notice that the right hand side above becomes the multiplier submodule by a slight modification of \cite[Theorem 8.1]{BlickleSchwedeTuckerTestAlterations}.

    For the statement about commuting with arbitrary localization, note that any alteration which stabilizes the intersection for $\tau^a(\omega_{X/V})$ will also stabilize at each stalk.  The statement follows.
\end{proof}

\begin{corollary}
    \label{cor.TransformationRuleRHUnderFiniteMaps}
    Suppose that $f : Y \to X$ is a finite surjective map between separated integral schemes of finite type and flat over $\Spec(V)$, with induced trace map $\Tr_{Y/X} : f_*\omega_{Y/V} \to \omega_{X/V}$. Then 
    \[
        \Tr_{Y/X}(f_*\tau^a_{\rm alt}(\omega_{Y/V})) = \tau^a_{\rm alt}(\omega_{X/V}).
    \]
\end{corollary}
\begin{proof}
By \autoref{thm.tauAltNonComplete}, we may choose $e \gg 0$ and an alteration $g : Z \to Y$ such that 
    \[
        \Tr_{Z/X}((f\circ g)_* \varpi^{1/p^e}\omega_{Z/V}) = \tau^a_{\rm alt}(\omega_{X/V}) \;\;\;\; \text{ and } \;\;\;\; \Tr_{Z/Y}(g_* \varpi^{1/p^e} \omega_{Z/V}) = \tau^a_{\rm alt}(\omega_{Y/V}).
    \]
    Since $\Tr_{Y/X} \circ \Tr_{Z/Y} = \Tr_{Z/X}$, the result follows.
\end{proof}

\subsection{Behavior under smooth maps} In this subsection we study the behavior of $\tau^a_{\rm alt}(\omega_{X/V})$ under smooth morphisms of schemes.

\begin{theorem}
    \label{thm.SmoothPullbackForRHOnNoetherian}
    Let $f : Y \to X$ be a smooth map of separated integral schemes finite type and flat over $\Spec(V)$. Then 
    \[ 
        f^* \tau^a_{\rm alt}(\omega_{X/V})  \otimes_{\cO_Y} \omega_{Y/X} \cong \tau^a_{\rm alt}(\omega_{Y/V}).
    \]
\end{theorem}
\begin{proof}

We first assume {\color{black}that $V$ is complete and} that $X_\infty$ and $Y_\infty:=Y\otimes_VV_\infty$ are both integral. Let $\widehat{X}_\infty$ be the $p$-adic completion of $X_\infty$. By \autoref{GRchar0Mult} and \autoref{rmk:tauRHInfinityNonProper}, we know that $\omega^{\RH}_{\widehat{X}_\infty/V_\infty}\in D_{\qcoh}(\widehat{X}_\infty)$ and $\mJ(X_\infty[1/p], \omega_{X_\infty[1/p]})\in D_{\qcoh}(X_\infty[1/p])$ have the same image inside $D_{\qcoh}(\widehat{X}_{\infty, \eta})$.\footnote{Here $\widehat{X}_{\infty, \eta}$ denote the rigid generic fiber of the formal scheme $\widehat{X}_\infty$: if $\widehat{X}_\infty={\rm Spf}(A)$, then we have an identification $D_{\qcoh}(\widehat{X}_{\infty, \eta})=D_{\qcoh}(\widehat{A}[1/p])$.} By Beauville–Laszlo patching \cite{BeauvilleLaszlo}, see \cite[Section 0BNI]{stacks-project} and \cite[Section 5]{BhattTannakaDuality}, we know that there exists a unique quasi-coherent sheaf $\omega_{X_\infty/V_\infty}^{\RH}$ on $X_\infty$ such that
\begin{enumerate}
    \item The $p$-adic completion of $\omega_{X_\infty/V_\infty}^{\RH}$ agrees with $\omega^{\RH}_{\widehat{X}_\infty/V_\infty}$.
    \item $\omega_{X_\infty/V_\infty}^{\RH}[1/p]\cong \mJ(X_\infty[1/p], \omega_{X_\infty[1/p]})$.
\end{enumerate}
Furthermore, this $\omega_{X_\infty/V_\infty}^{\RH}$ agrees with the one appearing in \autoref{thm.InputsFromSectionRH} when $X\to\Spec(V)$ is proper by uniqueness of the Beauville--Laszlo patching and the formal GAGA. Similarly, we have a quasi-coherent sheaf $\omega_{Y_\infty/V_\infty}^{\RH}$ that satisfies the analogous (a) and (b) above.

Since $f_{\infty} : Y_{\infty} \to X_{\infty}$ is smooth, by \autoref{SmoothGRdual} we have that
$$\omega_{\widehat{Y}_{\infty}/V_\infty}^{\RH} \cong f^*_{\infty} \omega_{\widehat{X}_{\infty}/V_\infty}^{\RH} \otimes_{\cO_{\widehat{Y}_{\infty}}} \omega_{\widehat{Y}_{\infty}/\widehat{X}_{\infty}}.$$
It is also well-known that $\mJ(Y_\infty[1/p], \omega_{Y_\infty[1/p]}) \cong f^*_{\infty}\mJ(X_\infty[1/p], \omega_{X_\infty[1/p]}) \otimes_{\cO_{Y_\infty[1/p]}}\omega_{Y_\infty/X_\infty}[1/p]$. Therefore we know that 
$$\omega_{{Y}_{\infty}/V_\infty}^{\RH} \cong f^*_{\infty} \omega_{{X}_{\infty}/V_\infty}^{\RH} \otimes_{\cO_{{Y}_{\infty}}} \omega_{{Y}_{\infty}/{X}_{\infty}}$$

At this point, we let $f'$: $Y'\to X'$ be a compactification of $Y\to X$ (note that $Y'\to X'$ is not necessarily smooth, but both $Y'_\infty$ and $X'_\infty$ are integral by our assumption). By \autoref{thm.tauAlt} and \autoref{thm.SingleAlterationWithPToEpsilon}, we know that 
$$\tau^a_{\rm alt}(\omega_{X'/V}) = \tau_{\RH}(\omega_{X'/V})= \kev(\pi_*\omega_{X'_\infty/V_\infty}^{\RH}).$$
Thus after we restrict to the open subsets $X\subseteq X'$ and $X_\infty\subseteq X'_\infty$ respectively, we have 
$$\tau^a_{\rm alt}(\omega_{X/V}) \cong \kev_X(\pi_*\omega_{X_\infty/V_\infty}^{\RH}).$$

Similarly, we also have $\tau^a_{\rm alt}(\omega_{Y/V}) \cong \kev_Y(\pi_*\omega_{Y_\infty/V_\infty}^{\RH}).$ Here we use $\kev_X$, $\kev_Y$ to emphasize that these maps are defined from $\pi_*\omega_{X_\infty/V_\infty}\to \omega_{X/V}$ and $\pi_*\omega_{Y_\infty/V_\infty}\to \omega_{Y/V}$ respectively. But note that both $\kev_X$ and $\kev_Y$ are obtained out of a base change $\kev: V_\infty\to V$. 
Since $(f^* \omega_{X/V}) \otimes_V V_{\infty} = f_{\infty}^* (\omega_{X/V} \otimes_V V_{\infty})$, it is easy to see that 
    \[ 
        (f^* \bT_X) \otimes_{\cO_Y} \omega_{Y/X} \colon (f^* \omega_{X_{\infty}/V_{\infty}})\otimes_{\cO_Y} \omega_{Y/X} \to (f^* \omega_{X/V}) \otimes_{\cO_Y} \omega_{Y/X}
    \] 
    is identified with $\bT_Y : \omega_{Y_{\infty}/V_{\infty}} \to  \omega_{Y/V}$ (by the behavior of relative canonical sheaf under smooth pull back).  
    Under this identification, we clearly have
    \begin{align*}
     \tau^a_{\rm alt}(\omega_{Y/V}) & =   \bT_Y(\pi_*\omega_{Y_{\infty}/V_{\infty}}^{\RH})\\
     & = \bT_Y\big(\pi_*(f^*_{\infty} \omega_{X_{\infty}/V_{\infty}}^{\RH} \otimes_{\cO_{Y_\infty}}
        \omega_{Y_\infty/X_\infty}) \big)  \\
        & = (f^* \bT_X)(f^*\pi_*\omega^{\RH}_{X_{\infty}/V_{\infty}}) \otimes_{\cO_Y} \omega_{Y/X} \\
        & = f^*\kev_X(\pi_*\omega_{X_\infty/V_\infty}^{\RH})  \otimes_{\cO_Y} \omega_{Y/X}\\
        & = \tau^a_{\rm alt}(\omega_{X/V}) \otimes_{\cO_Y} \omega_{Y/X}.
    \end{align*}
This completes the proof when both $X_\infty$ and $Y_\infty$ are integral.

Now we handle the general case{\color{black}, still assuming that $V$ is complete}. Let $X^{\rm N}$ be the normalization of $X$. After replacing $V$ if necessary, we may assume that $X^{\rm N}\otimes_VV_\infty$ is integral: for instance, one can take $X'$ to be a compactification of $X^{\rm N}$ and set $V=H^0(X'^{\rm N}, \cO_{X'^{\rm N}})$, see \autoref{rmk:NormalGemoetricallyIntegral}. Note that $Y\times_XX^{\rm N}$ and $X^{\rm N}\otimes_VV_e$ are integral for all $e$. We next claim the following.

\begin{claim}
\label{clm.IrreducibleComponents}
There exists $e\gg0$ such that each irreducible component of $(Y\times_XX^{\rm N})\otimes_VV_e$ remains integral after base change from $V_e$ to $V_\infty$. 
\end{claim}
\begin{proof}[Proof of Claim]   
First we note that since $V_\infty^{\mathrm{nc}}$ is henselian valuation ring, by \cite[5.4.54]{GabberRameroAlmostringtheory}, there is an equivalent of categories $(V_\infty^{\mathrm{nc}}[1/p])_{\fet}\cong (V_\infty[1/p])_{\fet}$ (this easily implies there is a bijection between finite field extensions of $V_\infty^{\mathrm{nc}}[1/p]$ and $V_\infty[1/p]$). It follows that $V_\infty^{\mathrm{nc}}[1/p]$ is separably closed (equivalently, algebraically closed) inside $V_\infty[1/p]$: otherwise, let $K$ be a nontrivial finite field extension of $V_\infty^{\mathrm{nc}}[1/p]$ inside $V_\infty[1/p]$, the injection 
$K\otimes_{V_\infty^{\mathrm{nc}}[1/p]}K \hookrightarrow K\otimes_{V_\infty^{\mathrm{nc}}[1/p]}V_\infty[1/p]$
then implies that $K\otimes_{V_\infty^{\mathrm{nc}}[1/p]}K$ is a field which is a contradiction. 
By \cite[\href{https://stacks.math.columbia.edu/tag/037P}{Tag 037P}]{stacks-project}, we know that there is a bijection between the irreducible components of $(Y\times_XX^{\rm N})\otimes_VV_\infty^{\mathrm{nc}}$ and those of $(Y\times_XX^{\rm N})\otimes_VV_\infty$ (note that all schemes involved are $p$-torsion free so irreducible components are in bijection with irreducible components after inverting $p$). Since $V_\infty^{\mathrm{nc}}=\varinjlim_{e}V_e$, there exists $e\gg0$ such that there is a bijection between irreducible components of $(Y\times_XX^{\rm N})\otimes_VV_e$ and those of $(Y\times_XX^{\rm N})\otimes_VV_\infty^{\mathrm{nc}}$. 
\end{proof}

Finally, we come back to the proof of the theorem. We consider the following base change diagram (each square is cartesian)
\[
\xymatrix{
\widetilde{Y} \ar[r]\ar[d] & (Y\times_X X^{\rm N})\otimes_VV_e \ar[r] \ar[d] & Y\times_X X^{\rm N} \ar[d] \ar[r] & Y \ar[d] \\
\widetilde{X}:= (X^{\rm N}\otimes_VV_e)^{\rm N} \ar[r] & X^{\rm N}\otimes_VV_e \ar[r] & X^{\rm N} \ar[r] & X
}
\]
where $\widetilde{X}:= (X^{\rm N}\otimes_VV_e)^{\rm N}$ denotes the normalization of $X^{\rm N}\otimes_VV_e$. Note that all vertical maps are smooth by base change. By \autoref{clm.IrreducibleComponents}, each irreducible component of $(Y\times_X X^{\rm N})\otimes_VV_e$ remains integral upon base change to $V_\infty$, thus the same is true for each irreducible component of $\widetilde{Y}$ (since $\widetilde{X}$ and $X^{\rm N}\otimes_VV_e$ are generically the same). Since $\widetilde{X}$ is normal and $\widetilde{Y}\to \widetilde{X}$ is smooth, we know that $\widetilde{Y}$ is normal and hence $\widetilde{Y}=\coprod_{i=1}^n \widetilde{Y}_i$ such that $\widetilde{Y}_i\otimes_{V_e}V_\infty$ is integral for all $i$ (in particular, each $\widetilde{Y}_i$ is smooth over $\widetilde{X}$ and $\widetilde{Y}\otimes_{V_e}V_\infty$ is a disjoint union of integral schemes $\widetilde{Y}_i\otimes_{V_e}V_\infty$). Consider the cartesian diagram
\[
\xymatrix{
\widetilde{Y}=\coprod_{i=1}^n \widetilde{Y}_i \ar[d]_b \ar[r]^-{g} & \widetilde{X} \ar[d]^a \\ 
Y \ar[r]_f & X.
}
\]
It is well-known that the trace map of $b$ can be identified with the pull back of the trace map of $a$, under the following identifications:
\begin{align*}
b_*\omega_{\widetilde{Y} /V} & \cong b_*(g^*\omega_{\widetilde{X}/V}\otimes \omega_{\widetilde{Y} /\widetilde{X}})\\
& \cong b_*(g^*\omega_{\widetilde{X}/V} \otimes b^*\omega_{Y/X}) \\
& \cong b_*g^*\omega_{\widetilde{X}/V}\otimes \omega_{Y/X} \\
& \cong f^*a_*\omega_{\widetilde{X}/V} \otimes \omega_{Y/X} \\
& \xrightarrow{f^*\Tr_a} f^*\omega_{X/V} \otimes \omega_{Y/X} \cong \omega_{Y/V}.
\end{align*}
This induces the following commutative diagram: 
\[
\xymatrix{
b_*\omega_{\widetilde{Y}/ V} \cong b_*\prod_{i=1}^n\omega_{\widetilde{Y}_i/ V} \ar[dd]^{\Tr_b} & b_*\prod_{i=1}^n\tau_{\rm alt}(\omega_{\widetilde{Y}_i/ V}, p^\epsilon) \ar@{_{(}->}[l] \ar[r]^-\cong \ar@{->>}[dd]^{\Tr_b} & b_*(g^*\tau_{\rm alt}(\omega_{\widetilde{X}/V}, p^\epsilon) \otimes b^*\omega_{Y/X})  \ar[d]^\cong  \\
& & f^*a_*\tau_{\rm alt}(\omega_{\widetilde{X}/V}, p^\epsilon) \otimes \omega_{Y/X} \ar@{->>}[d]^{f^*\Tr_a} \\
\omega_{Y/V} & \tau^a_{\rm alt}(\omega_{Y/V}) \ar@{_{(}->}[l] \ar@{.>}[r] & f^*\tau^a_{\rm alt}(\omega_{X/V}) \otimes \omega_{Y/X}
}.
\]
Here the isomorphism on the top row follows from the already established case (since $\widetilde{X}\otimes_{V_e}V_\infty$ and all $\widetilde{Y}_i\otimes_{V_e}V_\infty$ are integral and $\tau_{\rm alt}$ is independent of the choice of $V$, see \autoref{thm.SingleAlterationWithPToEpsilon}) and the two surjections follow from \autoref{cor.TransformationRuleRHUnderFiniteMaps}. After the identification of $\Tr_b$ with $f^*\Tr_a$, one sees that the dotted arrow in the diagram above is an isomorphism, which is the desired result. 

{\color{black}The final reduction to the case where $V$ is complete follows from the ``Furthermore'' part of \autoref{thm.tauAltNonComplete}.}
\end{proof}

\subsection{The completed stalks} 
After completing at a point $x \in X_{p= 0}$, we want to compare $\tau^a_{\rm alt}(\omega_{X/V})$ with those test modules already defined for complete local rings, see \autoref{ss:different-notions-test-ideals} for the domain case. 

\begin{definition}
\label{def.tau+atStalk}
Suppose $R$ is a Noetherian complete equidimensional reduced local ring of mixed characteristic $(0,p>0)$.  Recall that for a fixed rational number $\lambda > 0$ and $f\in R$ we have a $+$-test module
\[
    \tau_+(\omega_R, f^{\lambda}) = \tau_+(\omega_R, \lambda\Div(f)) := \bigcap_S \Tr_{S/R}(f^{\lambda}\omega_S)  =: {\myB}^0(\Spec(R), \lambda\Div(f); \omega_R) \subseteq \omega_R
\]
where the intersection runs over all normal rings $S \supseteq R$ finite over $R$ where some $f^{\lambda} \in S \subseteq K(R^+) = \overline{K(R)}$, see \cite[Proposition 4.17]{BMPSTWW1}.  The choice of $f^{\lambda}$ does not matter.
\end{definition}

\begin{corollary}
    \label{cor.TauAgreesWithCompletionComputation}
    Suppose the notation of \autoref{notation.ForSchemesOverADVR}.
    Let $x \in X_{p=0}$ and let $R = \widehat{\cO_{X,x}}$ be the completion of the local ring $\cO_{X,x}$ at its maximal ideal.  Then 
    \[
        \tau^a_{\rm alt}(\omega_{X/V}) \otimes \widehat{\cO_{X,x}} = \tau_+(\omega_R, p^{1/p^e})
    \]
    for all $e \gg 0$.
\end{corollary}
\begin{proof}
    We set $\omega_R = \omega_{X,x} \otimes R$.  
    Suppose first that $X$ is normal.  In that case, $R$ is a domain. It is enough to show that $\tau^a_{\rm alt}(\omega_{X/V}) \otimes \widehat{\cO_{X,x}} = \tau_+(\omega_R, \varpi^{1/p^e})$ for all $e\gg0$ since $\Div(p)$ is a multiple of $\Div(\varpi)$. There exists $e_0$ such that  $\tau_+(\omega_R, \varpi^{1/p^e})$ is constant for all $e>e_0$ (by Noetherianity of $R$).  We fix an $e>e_0$ such that $\tau^a_{\rm alt}(\omega_{X/V}) = \Tr_{Y/X}(f_*\varpi^{1/p^e}\omega_{Y/V})$ for some alteration $f: Y\to Y_e\to X$ as in \autoref{thm.tauAltNonComplete}. 
    By \cite[Proposition 4.29]{BMPSTWW1} (see \autoref{ss:different-notions-test-ideals}) we have
    \begin{equation} \label{eq:+-test-comparison}
        \tau_+(\omega_R, \varpi^{1/p^e}) = \bigcap_{Z'} \Tr_{Z'/\Spec(R)}(f'_* \varpi^{1/p^e}\omega_{Z'})
    \end{equation}
    where the intersection runs over alterations with fixed geometric generic point $\overline{K(X)} \supseteq K(X)$ and where $f' : Z' \to \Spec(R)$ is base changed from an alteration $Z_x \to \Spec\cO_{X,x}$ and hence comes from an alteration $f : Z \to X$.

    However, the intersection $\bigcap_Z \Tr_{Z/X}(f_* \varpi^{1/p^e}\omega_{Z/V})$ stabilizes (is equal to $\Tr_{Y/X}(f_* \varpi^{1/p^e}\omega_{Y/V})$ for $Y$ as above) and so the intersection also stabilizes after flat base change to $R$ (since all the alterations in \autoref{eq:+-test-comparison} come from alterations over $X$). Thus 
    \[
        \begin{array}{rl}
            \tau^a_{\rm alt}(\omega_{X/V}) \otimes \widehat{\cO_{X,x}} 
            & =  \Tr_{Y/X}(f_* \varpi^{1/p^e}\omega_{Y/V}) \otimes \widehat{\cO_{X,x}}\\
            & =  \Tr_{Y'/\Spec(R)}(f'_* \varpi^{1/p^e}\omega_{Y'})\\
            & =  \bigcap_{Z'}\Tr_{Z'/\Spec(R)}(f'_* \varpi^{1/p^e}\omega_{Z'}) \\
            & =  \tau_+(\omega_R, \varpi^{1/p^e}),
        \end{array}
    \]
    which completes the proof in the normal case.

    In general, let $n : X^{\mathrm{N}} \to X$ denote the normalization.  Base changing to $R$, we obtain $R \to \prod S_i$ where each $S_i$ is a complete normal local domain.  Note that for any normal alteration $f : W \to X$, $\Tr : f_* \omega_{W} \to \omega_X$ factors through $n_* \omega_{X^{\mathrm{N}}} \hookrightarrow \omega_{X}$.  Furthermore, the alteration which computes $\tau^a_{\alt}(\omega_{X^{\mathrm{N}/V}})$ base changes to the one that computes each $\tau_+(\omega_{S_i}, \varpi^{1/p^e})$ as above.  By \autoref{cor.TransformationRuleRHUnderFiniteMaps} and \cite[Lemma 4.18]{BMPSTWW1}, the same alteration computes $\tau_{\alt}^{a}(\omega_{X/V})$ and also $\tau_+(\omega_R, \varpi^{1/p^e})$.
\end{proof}

{\color{black}
\begin{remark}
    There is an alternate pathway to prove the results in this section from Theorem C (\autoref{GRviaAIC}), without the discussion of almost constant systems of alterations (as in \autoref{AlmostStabTraceAlt}).  
    In the first version of this paper on the arXiv, this alternate proof appeared as an appendix.  Specifically, in that appendix, we give a different proof of \autoref{cor.TauAgreesWithCompletionComputation}, namely that 
\[
    \tau_{\RH}(\omega_{X/V}) \otimes \widehat{\cO_{X,x}} = \tau_+(\omega_{\widehat{\cO_{X,x}}}, p^{1/p^e})
\]
for $e \gg 0$, which does not rely on \autoref{ss:GRAlt} or on  \autoref{ss:GRAlt2}.  
This equality would be clear by Theorem C, properties of the trace map, and local duality, if $X_\infty \to X$ was finite, and so our key idea is to study a weaker version of local duality for non-Noetherian $X_{\infty}$ by utilizing the map $\kev$.
The non-Noetherian case of local duality is also discussed in \autoref{remark:MD-non-noetherian}, but the statement therein requires the sheaf on $X_\infty$ to be pulled-back from a finite level, which is not the case in our setting.  Many other results in this section can then be deduced using the quite general \autoref{lem.StabilizingIntersection} below.  We feel that the argument we gave here was ultimately simpler and so we relegated the alternate proof as an appendix.  We encourage the reader interested in local duality and perverse sheaves on $X_{\infty}$ to explore the appendix in the first arXiv version.
\end{remark}
}

\subsection{Comparison with the test module due to Hacon-Lamarche-Schwede} In this subsection, we compare $\tau^a_{\rm alt}(\omega_{X/V})$ with the version of test module introduced and studied in \cite{HaconLamarcheSchwede} when $X\to \Spec(V)$ is projective (or quasi-projective).



\begin{theorem}
\label{thm.ComparisonTauRHvsHLSNoPair}
With notation as in \autoref{def.HaconLamarcheSchwede}, {\color{black} and in particular assuming $V$ is complete}, we have 
$$\tau^a_{\rm alt}(\omega_{X/V}) = \tau_{\myB^0}(\omega_{X/V}, {1/p^e}\Div(p))$$
for all $e\gg0$.
\end{theorem}
\begin{proof}
It is enough to show that $\tau^a_{\rm alt}(\omega_{X/V}) = \tau_{\myB^0}(\omega_{X/V}, \varpi^{1/p^e})$ for all $e\gg0$ since $\Div(p)$ is a multiple of $\Div(\varpi)$.
By \autoref{thm.tauAlt}, we know that for all $e\gg0$, we have
\[ 
    \tau^a_{\rm alt}(\omega_{X/V}) = \bigcap_{Y}\Tr_{Y/X}(f_*\varpi^{1/p^e}\omega_{Y/V})
\]
where $Y$ runs over all alterations such that 
$\varpi^{1/p^e}\in\cO_Y$ (note that the right hand side in fact stabilizes as in \autoref{thm.tauAlt}). 
By \cite[Proposition 4.11]{HaconLamarcheSchwede}, we have 
$$\tau_{\myB^0}(\omega_{X/V},\varpi^{1/p^e})\subseteq \Tr_{Y/X}(f_*\varpi^{1/p^e}\omega_{Y/V})$$
for all $Y$ such that $\varpi^{1/p^e}\in\cO_Y$, thus $\tau_{\myB^0}(\omega_{X/V},\varpi^{1/p^e})\subseteq \tau^a_{\rm alt}(\omega_{X/V})$. 

It remains to show that $\tau^a_{\rm alt}(\omega_{X/V})\subseteq \tau_{\myB^0}(\omega_{X/V},\varpi^{1/p^e})$ for all $e\gg0$. Let $\sL$ be a very ample line bundle on $X$ and let $S=\oplus_{i\geq 0}H^0(X, \sL^i)$ be the section ring of $X$ with respect to $\sL$. By \cite[Proposition 4.7]{HaconLamarcheSchwede}), we know that $\tau_{\myB^0}(\omega_{X/V}, \varpi^{1/p^e})$ is the sheaf associated to the graded $S$-module $\bigoplus_{i>0}\myB^0(X, \frac{1}{p^e}\Div(\varpi); \omega_{X/V}\otimes \sL^i)$. By \cite[Proposition 5.5]{BMPSTWW1} (note that we need a pair version with $\myB^0(X, \frac{1}{p^e}\Div(\varpi); \omega_{X/V}\otimes \sL^i)$ in place of $\myB^0(X;\omega_{X/V}\otimes\sL^i)$, which follows from the same argument there), we have a graded isomorphism 
\[ 
    \bigoplus_{i>0} \myB^0(X, \frac{1}{p^e}\Div(\varpi); \omega_{X/V}\otimes \sL^i) \cong \Big[\bigcap_{T} \Tr_{T/S}(\varpi^{1/p^e}\omega_T)\Big]_{>0}
\]
where $T$ runs over all graded finite domain extensions of $S$ such that $\varpi^{1/p^e}\in T$ and $\Tr_{T/S}$: $\omega_T\to \omega_S$ is the trace map. On the other hand, by \autoref{thm.tauAlt}, we know that for all $e\gg0$, we have that  
\[ 
    \tau^a_{\rm alt}(\omega_{S/V}) = \bigcap_{Z}\Tr_{Z/\Spec(S)}(f_*\varpi^{1/p^e}\omega_{Z/V})
\] 
where $f$: $Z\to \Spec(S)$ runs over all alterations of $\Spec(S)$ such that $\varpi^{1/p^e}\in\cO_{Z}$. In particular, we have that 
\[ 
    \tau^a_{\rm alt}(\omega_{S/V}) \subseteq \bigcap_{T} \Tr_{T/S}(\varpi^{1/p^e}\omega_T).
\]
where $T$ runs over all graded finite domain extensions of $S$ such that $\varpi^{1/p^e}\in T$ as above. Finally, we note that there is an open affine cover $\{U_j\}$ of $X$ such that $Z_j := \Spec(\cO_{U_j}[t, t^{-1}])$ is an open affine cover of $\Spec(S)-V(S_+)$. By \autoref{thm.SmoothPullbackForRHOnNoetherian}, we have 
$$\tau^a_{\rm alt}(\omega_{Z_j/V})\cong g_j^* \tau^a_{\rm alt}(\omega_{U_j/V}) \otimes \omega_{{Z_j}/{U_j}} (\cong \tau^a_{\rm alt}(\omega_{U_j/V})[t, t^{-1}])$$
where $g_j$: $Z_j\to U_j$. It follows that we have 
\small
\[
    \xymatrix{
       \tau^a_{\rm alt}(\omega_{Z_j/V}) \ar@{=}[r] \ar@{=}[d] & \tau^a_{\rm alt}(\omega_{S/V})|_{Z_j} \ar@{=}[d] \ar@{^{(}->}[r] &  \left(\bigoplus_{i>0}\myB^0(X, \frac{1}{p^e}\Div(\varpi); \omega_{X/V}\otimes\sL^i)\right)|_{Z_j}\ar@{=}[d] \\
        g_j^* \tau^a_{\rm alt}(\omega_{U_j/V}) \otimes \omega_{{Z_j}/{U_j}} \ar@{=}[r] & g_j^*\tau^a_{\rm alt}(\omega_{X/V})|_{U_j} \otimes \omega_{{Z_j}/{U_j}} & g_j^*\tau_{\myB^0}(\omega_{X/V}, \varpi^{1/p^e})|_{U_j}\otimes \omega_{{Z_j}/{U_j}}  
    }
\]
\normalsize
where the right vertical identification follows from \cite[Proposition 4.7]{HaconLamarcheSchwede}: $\tau_{\myB^0}(\omega_X, \varpi^{1/p^e})$ is the sheaf associated to the graded $S$-module $\bigoplus_{i>0}\myB^0(X, \frac{1}{p^e}\Div(\varpi); \omega_{X/V}\otimes\sL^i)$. It follows from the diagram above that
$$\tau^a_{\rm alt}(\omega_{X/V})|_{U_j} \subseteq \tau_{\myB^0}(\omega_{X/V}, \varpi^{1/p^e})|_{U_j}.$$ 
Since $\{U_j\}$ covers $X$, $\tau^a_{\rm alt}(\omega_{X/V}) \subseteq \tau_{\myB^0}(\omega_{X/V}, \varpi^{1/p^e})$ as wanted. 
\end{proof}

\section{Test ideals and modules of pairs}
\label{sec.TestIdealsDivisorPairs}

Our goal in this section is to generalize the results from \autoref{sec:SingOverDVR} to pairs $(X, \Gamma)$ where $\Gamma$ is a $\bQ$-Cartier $\bQ$-divisor, and also similarly to define test \emph{ideals} $\utau(\cO_X, \Delta)$ where $\Delta$ is a $\bQ$-divisor so that $K_X + \Delta$ is $\bQ$-Cartier.  In particular, this will assign test ideals $\utau(\cO_X)$ when $X$ is $\bQ$-Gorenstein that agree with previous definitions. All these ideals will incorporate small perturbations that are necessary for our proofs to work.

Unless otherwise stated, we work in the following setting.
\begin{setting} \label{setting:section7}
Let $(V, \varpi, k)$ be a DVR of mixed characteristic $(0,p>0)$ and let $X$ be a finite type, normal, integral scheme over $\Spec(V)$ such that $X \to \Spec(V)$ is flat (equivalently, surjective). We take $\Gamma$ to be a $\bQ$-Cartier $\bQ$-divisor on $X$.
\end{setting}

\begin{remark}
In this section we will often refer to \autoref{sec:SingOverDVR} in which $X$ was assumed to be separated as sometimes it was necessary to consider proper compactifications. However, this does not cause any problems, as all our results are local, so we can always restrict to open affine subsets in which case separatedness is automatic. We shall implicitly do this over and over again in the proofs below, without stating this explicitly.
\end{remark}

\begin{convention}
Throughout this section,  if $Y \to X$ is finite surjective with $Y$ normal and integral, we continue to use $\Tr_{Y/X} : K(Y) \to K(X)$ to denote the trace map. We will often simply write $\Tr$ instead of $\Tr_{Y/X}$ when $Y\to X$ is clear from the context.

Moreover, we will write $\tau_{\rm alt}(\omega_X, \varpi^{1/p^e})$ for $\tau_{\rm alt}(\omega_{X/V}, \varpi^{1/p^e})$ from \autoref{def.tauAlt}, and $\tau_{\rm alt}^a(\omega_X)$ for $\tau_{\rm alt}^a(\omega_{X/V})$ from \autoref{def.tauAltEpsilon} -- which is justified as these objects are independent of the choice of $V$ (see \autoref{thm.tauAltNonComplete} and \autoref{def.tauAltEpsilon}).
\end{convention}

For convenience of the reader, we summarize the results from \autoref{sec:SingOverDVR} that we will use in this section. 
\begin{enumerate}
    \item For all $e\gg0$, we have $\tau_{\rm alt}^a(\omega_X) = \tau_{\rm alt}(\omega_X, \varpi^{1/p^e})$. Moreover, there exists an alteration $Y_e \to X$ with $\varpi^{1/p^e} \in \cO_{Y_e}$ such that for all further alterations $Y \to Y_e$ with composition $f : Y \to Y_e \to X$, we have $\tau_{\rm alt}(\omega_X, \varpi^{1/p^e}) = \Tr_{Y/X}(f_* \varpi^{1/p^e}\omega_{Y})$. (\autoref{thm.tauAltNonComplete})
    \item For $f: Y\to X$ a finite surjective map such that $Y$ is normal and integral, we have $\tau_{\rm alt}^a(\omega_X) = \Tr_{Y/X}(f_*\tau_{\rm alt}^a(\omega_Y))$. (\autoref{cor.TransformationRuleRHUnderFiniteMaps})
    \item $\tau_{\rm alt}^a(\omega_X)[1/p] \cong \mJ(X[1/p], \omega_{X[1/p]})$, and for all $x\in X_{p=0}$ with $R:=\widehat{\cO_{X,x}}$,  we have $\tau^a_{\rm alt}(\omega_X)\otimes \widehat{\cO_{X,x}} = \tau_{+}(\omega_R, p^{1/p^e})$ for all $e\gg0$. (\autoref{cor.TauRHInvertPEqualsJNoetherianNoPair} and \autoref{cor.TauAgreesWithCompletionComputation})
    \item If $X\to \Spec(V)$ is projective and $V$ is complete, then we have $\tau_{\rm alt}^a(\omega_X)=\tau_{\myB^0}(\omega_X, p^{1/p^e})$ for all $e\gg0$, where the latter is defined in \autoref{ss:different-notions-test-ideals}. (\autoref{thm.ComparisonTauRHvsHLSNoPair})
\end{enumerate}
Strictly speaking, the assumption on normality in (b) is unnecessary, but it will be needed when we generalize this result to the case of pairs.




\begin{definition}\label{def:tau_RH_submodule_pairs}
   In the situation of \autoref{setting:section7}, choose $\psi : Y \to X$ a finite surjective map such that $\psi^* \Gamma$ is Cartier.  We define   
    \[
        \tau_{\alt}^a(\omega_X,  \Gamma) := \Tr_{Y/X}\Big(\tau_{\alt}^a(\omega_Y) \otimes_{\cO_Y} \cO_Y(-\psi^* \Gamma)\Big).
    \]
    It follows from \autoref{cor.TransformationRuleRHUnderFiniteMaps} that this is independent of the choice of $\psi : Y \to X$. Moreover, when $\Gamma \geq 0$, we have that $\tau_{\alt}^a(\omega_X,  \Gamma) \subseteq \omega_X$.
If $X = \Spec(R)$ is affine and $\Gamma = t \Div(f)$ for some $f \in \Spec (R)$ and $t \in \mathbf{Q}_{\geq 0}$, then we write 
    \[
        \tau_{\alt}^a(\omega_R, f^t) := \tau_{\alt}^a(\omega_X, t\Div(f))
    \] 
    for the corresponding submodule of $\omega_R$.  
\end{definition}
\begin{remark} \label{rem:a-test-ideal-definition}
Similarly, given a $\bQ$-divisor $\Delta$ such that $K_X+\Delta$ is $\bQ$-Cartier, we define
\[
\tau_{\alt}^a(\cO_X, \Delta) := \tau_{\alt}^a(\omega_X, K_X+\Delta).
\]
\end{remark}
It follows from the projection formula that if $\Gamma = \Gamma' + H$ for some Cartier divisor $H$, then 
\begin{equation}
    \tau_{\alt}^a(\omega_X, \Gamma) = \tau_{\alt}^a(\omega_X, \Gamma') \otimes_{\cO_X} \cO_X(-H).
\end{equation}
In most cases this implies that we may reduce the proofs of the results below to when $\Gamma \geq 0$, in which case $\tau_{\alt}^a(\omega_X, \Gamma)$ is a subsheaf of $\omega_X$. We first observe the following result on the completed stalks of $\tau_{\alt}^a(\omega_X, \Gamma)$ (recall the definition of $\tau_+$ for complete local rings from \autoref{def:tauR+} and \autoref{def:+test}).

\begin{corollary}
    \label{cor.TauRHVsTau+ForDivisorPairs}
    In the situation of \autoref{setting:section7}, let $x \in X_{p=0}\subseteq X$ and $R := \widehat{\cO_{X,x}}$. Then we have
    \[
        \tau_{\alt}^a(\omega_X, \Gamma)\otimes R = \tau_{+}(\omega_{R},  \epsilon \Div(p) + \Gamma|_R)
    \]
    for $1 \gg \epsilon > 0$.  
\end{corollary}
\begin{proof} Without loss of generality we may assume that $X$ is affine and $\Gamma \geq 0$.
    Choose $f : Y \to X$ a finite surjective map for which $f^* \Gamma$ is Cartier. Write $f^*\Gamma = {\rm div}(h)$ for $h \in H^0(Y,\cO_Y)$. Pick $y\in Y$ such that $f(y)=x$, and set $
    S := \widehat{\cO_{Y,y}}$. We get a corresponding extension of rings $R \subseteq S$. 

    By \autoref{cor.TauAgreesWithCompletionComputation} we have that $\tau_{\alt}^a(\omega_Y)_{y} \cdot S = \tau_{+}(\omega_{S}, p^{1/p^e})$ for $e \gg 0$. Thus

    \[    {\def\arraystretch{1.2}
        \begin{array}{rl}
            \tau_{\alt}^a(\omega_X, \Gamma) \otimes R &=  \Tr_{Y/X}\big( h \tau^a_{\alt}(\omega_Y) \big) \otimes R\\
             &=  \Tr_{S/R} \big( h \tau^a_{\alt}(\omega_Y)_{y} \cdot S \big) \\
             &= \Tr_{S/R} \big(\tau_{+}(\omega_{S}, 1/p^e \Div(p) + f^*\Gamma|_{S})) \\
             &=  \tau_{+}(\omega_R, 1/p^e \Div(p) + \Gamma|_{R}),
        \end{array}}
    \]
    where the first equality is  
\autoref{def:tau_RH_submodule_pairs} and  the last equality follows immediately from the definition of $\tau_+$ (see \autoref{def:+test}).

\end{proof}

\subsection{Perturbation}

Unfortunately, we do not know the answer to the following question:
\begin{question}
    \label{quest.tauAltPerturbed}
    For any effective Cartier divisor $D$ on $X$, is it true that 
    \[
        \tau_{\alt}^a(\omega_X, \varepsilon D) = \tau_{\alt}^a(\omega_X)
    \]
    for all $1 \gg \varepsilon > 0$?
\end{question}

We bypass this by building in perturbations.  First we recall a special case of a lemma of Gabber-Liu-Lorenzini.

\begin{lemma}[{\cite[Theorem 8.1]{GabberLiuLorenziniHypersurfacesInProjectiveSchemes}}]
    \label{lem.GlobalDiscriminantDivisor}
    Suppose $X\to\Spec(V)$ is flat, projective and $X$ is integral. Then there exists a finite surjective map $h : X \to \bP_V^n$.  Hence, there exists a Cartier divisor $G \geq 0$ on $\bP_V^n$ such that $h$ is \'etale outside of $h^{-1}G$.  
\end{lemma}

We next recall the following well-known result in our context.


\begin{lemma}[{\cite[Remark 3.11]{BhattScholzepPrismaticCohomology}}]
    \label{lem.ConstructingA_infty}
    Suppose $A$ is a Noetherian complete regular local ring of mixed characteristic $(0,p>0)$.  Then there exists a perfectoid $A$-algebra $A_{\infty}$ that is a $p$-completed filtered colimit of finite regular extensions of $A$. In other words, 
    \[
        A_{\infty} = (\colim_e A_e)^{\wedge_p}.
    \]
    such that each $A_e$ is a Noetherian complete regular local ring finite over $A$. 
    
    In particular $A_{\infty}$ is faithfully flat over $A$ and admits a map $A_{\infty} \to \widehat{A^+}$.
    Moreover, if $B$ is any perfectoid big Cohen-Macaulay $A^+$-algebra, then the induced map $A_{\infty} \to B$ is $p$-complete faithfully flat.
\end{lemma}
\begin{proof}
Let $C(k)$ be a coefficient ring for $A$ (so $C(k)$ is a complete unramified DVR of mixed characteristic $(0,p>0)$), we can form a filtered system of complete unramified DVR $C(k')$ where $k'$ runs over all finite inseparable extensions of $k$, whose completed colimit is $W(k^{1/p^\infty})$ (by a very similar construction as in \autoref{prop:kev_is_generator} without adjoining $p$-power roots of the uniformizer).

If $A$ is unramified, then $A\cong  C(k)\llbracket x_2, \dots, x_d \rrbracket$. In this case we consider the filtered system $C(k')\llbracket p^{1/p^e}, x_2^{1/p^e}, \dots, x_d^{1/p^e}\rrbracket$ and set $A_\infty$ be the $p$-completed colimit. If $A$ is ramified, then we have $A\cong C(k)\llbracket x_1,\dots, x_d\rrbracket /(p-f)$ where $f\in (x_1,\dots,x_d)^2$. In this case we consider the filtered system 
$$C(k')\llbracket x_1^{1/p^e},\dots,x_d^{1/p^e}\rrbracket /(p-f).$$
Note that each item in this filtered system is a complete regular local ring (the maximal ideal is generated by $x_1^{1/p^e},\dots,x_d^{1/p^e}$). The $p$-completed filtered colimit $A_\infty$ is perfectoid (see \cite[Remark 3.11]{BhattScholzepPrismaticCohomology} or \cite[Section 3]{ShimomotoApplicationofAlmostPurity} for more details). Furthermore, 
as all elements of $A^+$ admit compatible systems of arbitrary $p$-power roots, each $A_e$ can be mapped to $A^+$ (each $A_e$ is finite over $A$) inducing $\colim_e A_e \to A^+$ and hence $A_\infty$ maps to $\widehat{A^+}$. Thus $A_\infty$ maps to every perfectoid big Cohen-Macaulay $A^+$-algebra. Finally, since each $A_e$ is a regular local ring and $B$ is big Cohen-Macaulay, $A_e\to B$ is faithfully flat for all $e$ and thus $A_\infty\to B$ is $p$-complete faithfully flat. 
\end{proof}


\begin{proposition}
    \label{prop.PerturbedTestIdealIsReal}
    Suppose $W \to V$ is finite type with $W$ integral regular and suppose we have a finite surjective $h : X \to W$ and an effective divisor $G$ on $W$ such that $h$ is \'etale outside of $h^{-1}G$.  For example, via \autoref{lem.GlobalDiscriminantDivisor}, for $X$ projective we can take $W = \bP^n_V$.
    Fix $1 \gg \epsilon_0 > 0$ so that $\tau_{\rm alt}^a(\omega_X, \epsilon h^*G)$ is constant for all $\epsilon_0 \geq \epsilon > 0$ (which exists since $X$ is Noetherian). Then for all $x \in X_{p = 0}$ and $R := \widehat{\cO_{X,x}}$, we have 
    \[
        \tau_{\rm alt}^a(\omega_X, \epsilon h^*G)\otimes \widehat{\cO_{X,x}} = \tau_{\cB}(\omega_R).
    \]
In particular, the left hand side above is independent of the choice of $h$ and $G$. 
\end{proposition}

\begin{proof}
Begin by setting $a = h(x) \in W$ and $A = \widehat{\cO_{W, a}}$ and suppose that $g \in A$ is a local defining equation for $G$.  Note that $A$ is a complete regular local ring (it is possibly ramified, indeed even $V$ might be a ramified DVR).  Note also that $(h_* \cO_{X})_a \cdot A$ is a product of complete local rings, one of which is $R$.  It follows that $A[g^{-1}] \subseteq R[g^{-1}]$ is \'etale.  

    By \autoref{cor.TauRHVsTau+ForDivisorPairs}, we already know that $\tau_{\rm alt}^a(\omega_X,  \epsilon h^* G) \otimes \widehat{\cO_{X,x}} = \tau_+(\omega_{R}, p^{\epsilon'} g^{\epsilon})$ for $1 \gg \epsilon' > 0$.  By assumption, this remains unchanged after shrinking $\epsilon$ and $\epsilon'$ further, so we moreover have that $\tau_{\rm alt}^a(\omega_X,  \epsilon h^* G) \otimes \widehat{\cO_{X,x}} = \tau_+(\omega_{R}, p^{\epsilon} g^{\epsilon})$ for $1 \gg \epsilon> 0$. Thus it suffices to prove the following claim.
    \begin{claim}
        \label{clm.Tau+PerturbedEqualsTauBigB}
        Suppose $A \subseteq R$ is a finite extension of Noetherian complete local domains of mixed characteristic $(0, p> 0)$ such that $A$ is regular.  Additionally suppose that $0\neq g \in A$ is such that $A[g^{-1}] \subseteq R[g^{-1}]$ is \'etale.   Then
        \[
            \tau_+(\omega_{R}, p^{\epsilon} g^{\epsilon}) = \tau_{+}(\omega_R, g^{\epsilon})= \tau_{\cB}(\omega_R)
        \]     
        for all $1\gg\epsilon >0$. 
    \end{claim}
    \begin{proof}[Proof of Claim]
        Note if $A[g^{-1}] \subseteq R[g^{-1}]$ is \'etale then so is $A[(pg)^{-1}] \subseteq R[(pg)^{-1}]$.  Hence we may remove the $p^{\epsilon}$ term and prove $\tau_{+}(\omega_R, g^{\epsilon})= \tau_{\cB}(\omega_R)$.
        The key computations can be found in \cite[Lemma 5.1.6]{CaiLeeMaSchwedeTuckerPerfectoidSignature1} and \cite[Proposition 5.7]{MaSchwedeSingularitiesMixedCharBCM}.  Following the notation in the former reference, consider  
        \[
            \begin{array}{rcl}
                M_2 & := & \{ \eta \in H^d_{\fram}(R) \;|\;  (g^{1/p^{\infty}}) \eta = 0 \in H^d_{\fram}(R^+)\}\\
                M_3 & := & \{ \eta \in H^d_{\fram}(R) \;|\;  (g^{1/p^{\infty}}) \eta = 0 \in H^d_{\fram}(B) \text{ for some perfectoid BCM $R^+$-algebra $B$ }\} \\
                M_4 & := & \{ \eta \in H^d_{\fram}(R) \;|\;   \eta \mapsto 0 \in H^d_{\fram}(B) \text{ for some perfectoid BCM $R^+$-algebra $B$ }\}.
            \end{array}
        \]        
        By \cite[Proposition 5.7]{MaSchwedeSingularitiesMixedCharBCM}, it suffices to show that $M_2=M_4$. We will establish this by comparing them with $M_3$. First of all, it is clear that $M_2, M_4 \subseteq M_3$. By \cite[Lemma 5.6]{MaSchwedeSingularitiesMixedCharBCM}, we have $M_3\subseteq M_4$ and thus $M_3=M_4$.  Note also that by \cite[Theorem 4.9]{MaSchwedeSingularitiesMixedCharBCM} we may fix a single $B$ that so that $M_3$ and $M_4$ may be computed with respect to that single $B$.  
        
        It remains to show that $M_3 \subseteq M_2$. This argument proceeds essentially identically to the argument in \cite[{Lemma 5.1.6, Proof of $M_3 \subseteq M_1$}]{CaiLeeMaSchwedeTuckerPerfectoidSignature1} and we point out the differences. Let $S$ denote the integral closure of $R$ in the Galois closure of $K(R)/K(A)$. By \autoref{lem.ConstructingA_infty}, we have $A\to A_\infty$ where $A_\infty$ is perfectoid that is a $p$-completed filtered colimit of finite regular extensions of $A$, and that $A_{\infty} \to B$ is $p$-complete faithfully flat.
        We set 
        \[
            R_{\perfd}^{A_{\infty}} = (R \otimes_{A} A_{\infty})_{\perfd} \;\;\; \text{ and } \;\;\; S_{\perfd}^{A_{\infty}} = (S \otimes_{A} A_{\infty})_{\perfd}
        \]
        Following the argument of \cite[{Lemma 5.1.6, Proof of $M_3 \subseteq M_1$}]{CaiLeeMaSchwedeTuckerPerfectoidSignature1} verbatim at this point, 
        we deduce that if $\eta \in M_3$ (that is $(g^{1/p^{\infty}}) \otimes \eta = 0 \in H^d_{\fram}(B)= B\otimes H_\m^d(R)$) then 
        \[
            (g)_{\perfd} \otimes \eta = 0 \in H^d_{\fram}(R_{\perfd}^{A_{\infty}})= R_{\perfd}^{A_{\infty}} \otimes H_\m^d(R).
        \]
        Since $A_\infty$ maps to $\widehat{R^+}$, we know that $R_{\perfd}^{A_{\infty}} = (R \otimes_{A} A_{\infty})_{\perfd}$ maps to $\widehat{R^+}$ by the universal property of the perfectoidization functor. Thus we have $(g^{1/p^{\infty}}) \otimes \eta = 0 \in \widehat{R^+}\otimes H^d_{\fram}(R)$, that is, $(g^{1/p^{\infty}})\eta=0\in H_\m^d(R^+)$ as desired.       
    \end{proof}
    The claim completes the proof.
\end{proof}

{\color{black}

\begin{remark}
\label{rmk.Tau+PerturbedEqualsTauBigB choosing g and h}
There is a slight generalization of \autoref{clm.Tau+PerturbedEqualsTauBigB} that we will need later. Namely, suppose $A \subseteq R \subseteq S$ are finite extensions of Noetherian complete local domains of mixed characteristic $(0, p> 0)$ such that $A$ is regular.  Additionally suppose that $0\neq g \in A$ is such that $A[g^{-1}] \subseteq R[g^{-1}]$ is \'etale and $0\neq h\in R$ such that $R[h^{-1}]\subseteq  S[h^{-1}]$ is \'etale.   Then
        \[
        \tau_{+}(\omega_S, (gh)^{\epsilon})= \tau_{\cB}(\omega_S)
        \]     
        for all $1\gg\epsilon >0$. To see this, we first claim that we may assume $A[g^{-1}] \subseteq R[g^{-1}]$ and $R[h^{-1}]\subseteq  S[h^{-1}]$ are Galois. We can replace $R$ by $R'$, the integral closure of $R$ in the Galois closure of $K(R)/K(A)$. The total quotient ring of $S\otimes_RR'$ is a finite product of fields $\prod K_i$, each of which is a finite extension of $K(R')$. The integral closure of $S\otimes_RR'$ in the Galois closure of $\prod K_i$ over $K(R')$ is thus a finite product of complete normal domains. We replace $S$ by one of them, call it $S'$. We still have that $R'[h^{-1}]\to S'[h^{-1}]$ is finite \'etale. If we show that $\tau_{+}(\omega_{S'}, (gh)^{\epsilon})= \tau_{\cB}(\omega_{S'})$, then by the transformation rule of finite maps, we know that 
        $$\tau_{+}(\omega_S, (gh)^{\epsilon})= \Tr_{S'/S}\tau_{+}(\omega_{S'}, (gh)^{\epsilon})= \Tr_{S'/S}\tau_{\cB}(\omega_S)=\tau_{\cB}(\omega_S)$$
        which is what we want.  From here on out we view $A, R, S \subseteq R^+$ and we also fix maps from $A_{\infty}, R^{A_{\infty}}_{\perfd}, S^{A_{\infty}}_{\perfd}$ to $\widehat{R^+}$ in a compatible way, i.e., we have $A_{\infty}\to R^{A_{\infty}}_{\perfd}\to S^{A_{\infty}}_{\perfd} \to \widehat{R^+}$.

Therefore from now on we assume that $A[g^{-1}] \subseteq R[g^{-1}]$ and $R[h^{-1}]\to S[h^{-1}]$ are Galois. Following the proof of \autoref{clm.Tau+PerturbedEqualsTauBigB}, it is enough to show that 
$$M_3  :=  \{ \eta \in H^d_{\fram}(S) \;|\;  (gh)^{1/p^{\infty}} \eta = 0 \in H^d_{\fram}(B) \text{ for some perfectoid BCM $R^+$-algebra $B$ }\} 
$$
is contained in 
$$M_2  :=  \{ \eta \in H^d_{\fram}(S) \;|\;  (gh)^{1/p^{\infty}} \eta = 0 \in H^d_{\fram}(R^+)\}.$$
For every $\sigma$ in the Galois group of $K(R)$ over $K(A)$, we fix a lift $\sigma'$ of $\sigma$ in the Galois group of $K(R^+)=\overline{K(R)}$ over $K(A)$. Let $B$ be a big Cohen-Macaulay $R^+$-algebra. We will use $\sigma_*B$ to denote the $R^{A_\infty}_{\perfd}$-algebra via the map $R^{A_\infty}_{\perfd}\xrightarrow{\sigma}R^{A_\infty}_{\perfd}\to \widehat{R^+}\to B$. We have a commutative diagram:
\[\xymatrix{
R^{A_\infty}_{\perfd} \ar[r] \ar[d]^{\sigma} & S^{A_\infty}_{\perfd} \ar[r] & \widehat{R^+} \ar[d]^{\sigma'} \\
R^{A_\infty}_{\perfd} \ar[rr] && \widehat{R^+} \ar[r] & B,
}
\]
which endows $\sigma'_*B$ an $S^{A_\infty}_{\perfd}$-algebra structure such that $\sigma'_*B\cong \sigma_*B$ as $R^{A_\infty}_{\perfd}$-algebras. We claim that we have a commutative diagram: 
\[
\xymatrix{
    A_\infty \ar[r] \ar[d] & R^{A_\infty}_{\perfd} \ar[r] \ar[d] & S^{A_\infty}_{\perfd} \ar[d] \\
    B \ar[r] & R^{A_\infty}_{\perfd}\widehat{\otimes}_{A_\infty} B \ar[r] \ar@{=}[d]^a & S^{A_\infty}_{\perfd} \widehat{\otimes}_{A_\infty} B \ar@{=}[d]^a\\
    & \displaystyle \prod_{\sigma} \sigma_*B \cong \prod_{\sigma} \sigma'_*B \ar[r] & \displaystyle \prod_{\rho}\rho_*\big(\prod_{\sigma} \sigma'_*B\big)
}
\]
We explain the two almost isomorphisms, which are $g$-almost and $(gh)$-almost respectively.  Indeed, notice that  $R^{A_\infty}_{\perfd} \widehat{\otimes}_{A_\infty} R^{A_\infty}_{\perfd}$ is $g$-almost isomorphic to $\prod_{\sigma} \sigma_*R^{A_\infty}_{\perfd}$ and likewise $S^{A_\infty}_{\perfd}\widehat{\otimes}_{R^{A_\infty}_{\perfd}} S^{A_\infty}_{\perfd}$ is $h$-almost isomorphic to $\prod_{\rho}\rho_*S^{A_\infty}_{\perfd}$ by \cite[Theorem 10.9]{BhattScholzepPrismaticCohomology}, where $\sigma$ and $\rho$ runs over all elements in the Galois groups of $K(R)/K(A)$ and $K(S)/K(R)$.  The first almost isomorphism follows immediately from 
\[ 
R^{A_\infty}_{\perfd} \widehat{\otimes}_{A_\infty} B \cong \Big(R^{A_\infty}_{\perfd} \widehat{\otimes}_{A_\infty} R^{A_\infty}_{\perfd}\Big) \widehat{\otimes}_{R^{A_\infty}_{\perfd}} B \overset{a}{\cong} \Big(\prod_{\sigma} \sigma_*R^{A_\infty}_{\perfd}\Big) \widehat{\otimes}_{R^{A_\infty}_{\perfd}} B \cong \prod_{\sigma} \sigma_*B .
\]
For the second, we note that
\[
    {\def\arraystretch{2.0}
    \begin{array}{rl}
         S^{A_\infty}_{\perfd} \widehat{\otimes}_{A_\infty} B 
        \cong  S^{A_\infty}_{\perfd} \widehat{\otimes}_{R^{A_\infty}_{\perfd}}  R^{A_{\infty}}_{\perfd} \widehat{\otimes}_{A_\infty} B
        \overset{a}{\cong}   \displaystyle{S^{A_\infty}_{\perfd} \widehat{\otimes}_{R^{A_\infty}_{\perfd}} \Big(\prod_{\sigma} \sigma'_*B \Big)}
        \cong  \displaystyle{\prod_{\sigma} \Big( S^{A_\infty}_{\perfd} \widehat{\otimes}_{R^{A_\infty}_{\perfd}} \sigma'_* B\Big)}
    \end{array}}
\]
With the aforementioned $S^{A_\infty}_{\perfd}$-algebra structure on each $\sigma'_*B$, we have
\[
 {\def\arraystretch{2.0}
    \begin{array}{rl}
& \displaystyle{\prod_{\sigma'} \Big( S^{A_\infty}_{\perfd} \widehat{\otimes}_{R^{A_\infty}_{\perfd}} \sigma'_* B\Big)} \\
    \cong & \displaystyle{\prod_{\sigma'} \big(S^{A_\infty}_{\perfd} \widehat{\otimes}_{R^{A_\infty}_{\perfd}} S^{A_\infty}_{\perfd} \big)\widehat{\otimes}_{S^{A_\infty}_{\perfd}} \sigma'_*B }\\
    \overset{a}{\cong} & \displaystyle{\prod_{\sigma'} \big(\displaystyle{\prod_{\rho}\rho_*S^{A_\infty}_{\perfd}\big) \widehat{\otimes}_{S^{A_\infty}_{\perfd}} \sigma'_*B}}\\
    \cong & \displaystyle{\prod_{\rho, \sigma'}\rho_* \sigma'_*B}
    \end{array}
    }
\]
The last isomorphism is where we use the left $S^{A_{\infty}}_{\perfd}$-action on the second-to-last line.

By \cite[Lemma 5.1.5]{CaiLeeMaSchwedeTuckerPerfectoidSignature1} applied to $A\to R$ and $R\to S$,\footnote{In \cite[Lemma 5.1.5]{CaiLeeMaSchwedeTuckerPerfectoidSignature1}, it is assumed that the base ring $A$ is regular but the proof works for general $A\to R$ finite extensions of complete normal local domains.} we know that there exists a sufficiently large perfectoid $R^+$-algebra $B$ such that  
$$
M_3  =  \Big\{ \eta \in H^d_{\fram}(S) \;\Big|\;  (gh)^{1/p^{\infty}} \eta = 0 \in H^d_{\fram}\Big(\prod_{\rho}\rho_*\big(\prod_{\sigma} \sigma'_*B\big)\Big) \Big\} \\
$$
Since $A_\infty\to B$ is $p$-complete faithfully flat, by the diagram above we know that $$S^{A_\infty}_{\perfd} \to \prod_{\rho}\rho_*\big(\prod_{\sigma} \sigma'_*B\big)$$ is $p$-complete $(gh)$-almost faithfully flat and in particular $p$-complete $(gh)$-almost pure in the sense of \cite[Proposition 2.3.8]{CaiLeeMaSchwedeTuckerPerfectoidSignature1}. In particular, we know that 
\[ 
    H_\m^d(S^{A_\infty}_{\perfd}) \to H_\m^d\Big(\prod_{\rho}\rho_*\big(\prod_{\sigma} \sigma'_*B\big)\Big)
\]
is $(gh)$-almost injective. Since $S^{A_\infty}_{\perfd}$ maps to $\widehat{R^+}$, it follows that $M_3$ is contained in 
$$M_2=\{ \eta \in H^d_{\fram}(S) \;|\;  (gh)^{1/p^{\infty}} \eta = 0 \in H^d_{\fram}(R^+)\}$$
as wanted. 
\end{remark}

}



\begin{definition}[Perturbation-friendly test module]
\label{def:perturbed_tau}
We work in the situation of \autoref{setting:section7}.
Cover $X$ by open quasi-projective (for instance affine) $X_i \subseteq X$.  We let $\overline{X_i}$ be a projective compactification of $X$.  On each $X_i$ define 
\[
    \tau(\omega_{X_i}) := \tau_{\alt}^a(\omega_{\overline{X_i}},  \epsilon' h_i^*G_i)|_{X_i}.
\]
for $1 \gg \epsilon' > 0$, where $h_i$ and $G_i$ are as in \autoref{prop.PerturbedTestIdealIsReal} (i.e., $h_i: \overline{X_i}\to \mathbb{P}^n_{V}$ is finite surjective and \'etale outside $h_i^{-1}G_i$).  By \autoref{prop.PerturbedTestIdealIsReal}, the $\tau(\omega_{X_i})$ are independent of the choices of $h_i$ and $G_i$ and they agree on the overlaps $X_i \cap X_j$ -- they are all coherent subsheaves of $\omega_X$, and hence they glue to make the \emph{mixed characteristic test module}: $\tau(\omega_X)$.

Recall that $\Gamma$ is a $\bQ$-Cartier $\bQ$-divisor on $X$.  We pick $f : Z \to X$ a finite surjective map with $Z$ normal integral such that $f^* \Gamma$ is Cartier.   We define the \emph{mixed characteristic test module} of $(X, \Gamma)$ to be 
    \[
        \utau(\omega_X, \Gamma) = \Tr\Big( f_* \utau(\omega_Z) \otimes_{\cO_Z} \cO_Z(-f^* \Gamma)  \Big).
    \]
Once again, this definition is independent of the choice of $Z$, using the transformation rules for finite maps \autoref{cor.TransformationRuleRHUnderFiniteMaps}.
\end{definition}

We first show that $\utau(\omega_X, \Gamma)$ can be computed on a single sufficiently large alteration when $X$ is quasi-projective, which is an alternate way to conclude independence of the choice of $Z$ above.
\begin{theorem}\label{thm:perturbed_tau_submodule_single_alteration}
In the situation of \autoref{setting:section7} assume that $X \to \Spec V$ is quasi-projective. 
 Then there exists an effective divisor $G^0$ on $X$ so that for every Cartier divisor $G \geq G^0$ on $X$ and every $1 \gg \epsilon > 0$ (depending on $G$),  the following holds.

There exists an alteration $\pi_\epsilon: Y_\epsilon \to X$ with $\pi_\epsilon^*\epsilon G$ and $\pi_{\epsilon}^*\Gamma$ Cartier  such that for every further normal alteration $Y\to Y_{\epsilon}$ with composition $\pi:Y\to Y_{\epsilon}\to X$ we have that 
\[ 
    \tau(\omega_X,\Gamma)= \tau_{\alt}^a(\omega_X,\epsilon G + \Gamma)=\Tr\big(\pi_*\omega_{Y}(-\pi^*\Gamma-\epsilon\pi^* G)\big).
\]

\end{theorem}

{\color{black}Later, in \autoref{thm:onealterationtorulethemall}, we will see we can replace $G^0$ with an ideal $pJ$ where $J$ is the ideal defining the locus where $(X, \Gamma)$ is not smooth over $V$.}

\begin{proof}
We first deal with the case when $\Gamma=0$.  
We may assume $X\to \Spec(V)$ is projective. Let $G' > 0$ and $h : X \to \bP^n$ be as in \autoref{prop.PerturbedTestIdealIsReal} (\cf \autoref{lem.GlobalDiscriminantDivisor}), so that for $\epsilon'> 0$ sufficiently small we have $\tau(\omega_X)=\tau_{\rm alt}^a(\omega_X,  \epsilon'h^*G')$ by \autoref{def:perturbed_tau}.     
Fixing such an $\epsilon'$, by \autoref{def:tau_RH_submodule_pairs}, if $\phi:Z\to X$ is such that $\phi^*\epsilon'h^* G'$ is Cartier, then 
$$\tau_{\rm alt}^a(\omega_X,  \epsilon' h^*G')=\Tr(\phi_*\tau_{\rm alt}^a(\omega_Z)\otimes \sO_Z(-\phi^*\epsilon' h^*G')).$$
   Now by \autoref{thm.tauAltNonComplete}, there is some $e_0$, such that for every $e\geq e_0$ there is an alteration $Y_e\to Z$ with $\varpi^{1/{p^e}}\in \cO_{Y_e}$ 
   such that for every alteration $\psi : Y \to Z$ which factors through $Y_e$ we have $$\tau_{\rm alt}^a(\omega_{Z})=\Tr(\psi_*\varpi^{1/p^e}\omega_{Y}).$$  Now set $G^0=h^*G'+\Div(\varpi)$, and choose $\epsilon_{G_0}$ to satisfy $\epsilon_{G^0}<{1/p^{e_0}}$ and $\epsilon_{G^0}<\epsilon'$.  
   For any $0<\epsilon<\epsilon_{G^0}$, choose $e {\geq e_0}$ such that $1/p^{e+1}<\epsilon<1/p^e$, and choose $Y_\epsilon$ that dominates $Y_e$ and $Y_{e+1}$.  Then for any alteration $\pi:Y\to X$ over $Y_\epsilon$ we have 
   \[ 
        \Tr\left(\pi_*\omega_Y\left(-\pi^*\frac{1}{p^e} G^0\right)\right)\subseteq\Tr\left(\pi_*\omega_Y\left(-\pi^*\epsilon G^0\right)\right)\subseteq \Tr\left(\pi_*\omega_Y\left(-\pi^*\frac{1}{p^{e+1}} G^0\right)\right).
   \]
   Both modules on the outside are equal to $\tau(\omega_X)$ by construction.  

   Now for any $G \geq G^0$, note that we can find $G'' \geq G'$ on $\bP^n_V$ so that $G^1 := h^*G''+\Div(\varpi) \geq G$.
   The arguments above yielding $\epsilon_{G^0}$ and $Y_\epsilon$ are also valid for $G^1$;  choose $\epsilon_0$ to be the minimum of $\epsilon_{G^0}$ and the new $\epsilon_{G^1}$. For every $0<\epsilon <\epsilon_0$, we can further replace $Y_{\epsilon}$ by a larger alteration so that it satisfies the desired conclusion for both $G^0$ and $G^1$ simultaneously.  
   Hence we conclude by noting that:
   \[
   \tau(\omega_X) = \tau_{\alt}^a(\omega_X, \epsilon G^0) \supseteq \tau_{\alt}^a(\omega_X, \epsilon G) \supseteq \tau_{\alt}^a(\omega_X, \epsilon G^1) = \tau(\omega_X)
   \]
   \[
    \tau(\omega_X) = \Tr\big( \pi_* \omega_Y(-\pi^* \epsilon G^0) \big) \supseteq \Tr\big( \pi_* \omega_Y(-\pi^* \epsilon G) \big) \supseteq \Tr\big( \pi_* \omega_Y(-\pi^* \epsilon G^1) \big) = \tau(\omega_X).
   \]
This completes the case of $\Gamma=0$. 

Finally, we handle the general case. Choose a finite map $\phi:Z\to X$ such that $\phi^*\Gamma$ is Cartier.  Then $$\tau(\omega_X,\Gamma)=\Tr(\phi_*\tau(\omega_{Z})\otimes{\sO_Z}(-\phi^*\Gamma))$$
Set $h' = h \circ \phi : Z \to \bP^n_V$. We may enlarge $G'$ so that $h'$ is \'etale outside $(h')^{-1}G'$ as well, and run the above argument for trivial boundary on $Z$, noting that the divisors $G^0$ and $G^1$ produced on $Z$ are pulled back from $X$. 
    Therefore, applying 
    the result for trivial boundary on $Z$
    also
    produces the required $\epsilon_0$ and the required alterations over $X$. 
\end{proof}

{\color{black}
\begin{remark}
\label{rem.BaseChangeToCompletionWithDivisorPairs}
    Suppose $V \to \widehat{V}$ is the completion map and $g : X_{\widehat{V}} \to X$ is the base change.  Note while $X_{\widehat{V}}$ need not be integral, although it is normal and hence it is a disjoint union of normal irreducible components all of which are finite type over $\widehat{V}$, and we can work with each component separately when defining $\tau(\omega_{X_{\widehat{V}}}, g^*\Gamma)$.  
    
    We claim that $\tau(\omega_X, \Gamma) \otimes_V \widehat{V} = \tau(\omega_{X_{\widehat{V}}}, g^*\Gamma)$.  Indeed, after reducing to the projective case, we notice that the $G$ we pick as in \autoref{prop.PerturbedTestIdealIsReal}, pulls back to $g^* G$ which satisfies the same conditions for $X_{\widehat{V}}$.  But now apply \autoref{thm.tauAltNonComplete} to $Z$, a finite cover of $X$ where $\epsilon G + \Gamma$ becomes Cartier.  Indeed, fixing an alteration $f : Y_e \to Z$ as in \autoref{thm.tauAltNonComplete}, we see that each pair of irreducible components of the base change $f \otimes_V \widehat{V}$, say $\overline{Y_e} \to \overline{X}$, satisfy the desired stabilization result.  This forces  
    \[
    \tau(\omega_X, \Gamma) \otimes_V \widehat{V} = \tau(\omega_{X_{\widehat{V}}}, g^*\Gamma)
    \]
    as claimed.
\end{remark}
}

We next list the properties of this test module.

\begin{proposition}
    \label{prop.PropertiesOfUltTauOmegaForPairs}
    We work in the situation of \autoref{setting:section7}. 
    The following hold.
    \begin{enumerate}
        \item If $\Gamma \leq \Gamma'$ then $\tau(\omega_X, \Gamma) \supseteq \tau(\omega_X, \Gamma')$.\label{prop.PropertiesOfUltTauOmegaForPairs.IncreasingCoefficientContainment}
        \item If $D$ is Cartier then $\tau(\omega_X, \Gamma+D) = \tau(\omega_X, \Gamma) \otimes \cO_X(-D)$.\label{prop.PropertiesOfUltTauOmegaForPairs.EasyPrincipalSkoda}
        \item $\utau(\omega_X, \Gamma)[1/p] = \mJ(\omega_{X[1/p]}, \Gamma|_{X[1/p]})$. \label{prop.PropertiesOfUltTauOmegaForPairs.InvertP}
        \item For any finite surjective map $f : X' \to X$, $\Tr(f_*\utau(\omega_{X'}, f^* \Gamma)) = \utau(\omega_X, \Gamma)$.\label{prop.PropertiesOfUltTauOmegaForPairs.FiniteMaps}
        \item  For any projective birational map $f : X' \to X$, $f_* \utau(\omega_{X'}, f^* \Gamma) \supseteq \utau(\omega_X, \Gamma)$. \label{prop.PropertiesOfUltTauOmegaForPairs.BirationalMaps}
        \item Suppose $X$ is quasi-projective and $V$ is complete. Then there exists $G^0 \geq 0$, such that for all Cartier divisors $G \geq G^0 \geq 0$ on $X$, we have that $\utau(\omega_X, \Gamma) = \tau_{\myB^0}(\omega_X, \epsilon G + \Gamma)$ for all $1 \gg \epsilon > 0$ (depending on $G$).\label{prop.PropertiesOfUltTauOmegaForPairs.ComparisonWithB0}
        \item Suppose $X$ is quasi-projective, $x \in X_{p = 0}$, and $R = \widehat{\cO_{X,x}}$.  Then there exists $G^0 \geq 0$ such that for all Cartier divisors $G \geq G^0 \geq 0$ on $X$, we have that
        \[
            \utau(\omega_X, \Gamma) \otimes R = \tau_{+}(\omega_R, \Gamma|_R + \epsilon G|_R)
        \]  
        for all $1 \gg \epsilon > 0$ (depending on $G$).  \label{prop.PropertiesOfUltTauOmegaForPairs.CompletionComparisonWithTau+}
        \item             \label{prop.PropertiesOfUltTauOmegaForPairs.Completion}
        Suppose $x \in X_{p=0}$ and $R = \widehat{\cO_{X,x}}$.  Then 
             \[
                \utau(\omega_X, \Gamma) \otimes R = \tau_{\cB}(\omega_R, \Gamma|_R).%
            \]
        \item For any Cartier divisor $D > 0$ on $X$, we have that 
            \[
                \tau(\omega_X, \Gamma) = \tau(\omega_X, \epsilon D + \Gamma)
            \]
            for all $1 \gg \epsilon > 0$ (depending on $D$).
            \label{prop.PropertiesOfUltTauOmegaForPairs.Perturbation}
        \item For any smooth $f : Y \to X$ we have that 
            \[
                f^* \tau(\omega_X, \Gamma) \otimes \omega_{Y/X} = \tau(\omega_Y, f^* \Gamma).
            \]
            \label{prop.PropertiesOfUltTauOmegaForPairs.SmoothPullBack}
    \end{enumerate}
\end{proposition}
\begin{proof}
    \autoref{prop.PropertiesOfUltTauOmegaForPairs.IncreasingCoefficientContainment} and \autoref{prop.PropertiesOfUltTauOmegaForPairs.EasyPrincipalSkoda} follow easily from the definition (or can be deduced from various other properties below).

    For \autoref{prop.PropertiesOfUltTauOmegaForPairs.InvertP}, simply observe that $\mJ(\omega_{X[1/p]}, \Gamma|_{X[1/p]})$ can be computed from a sufficiently large alteration, \cf \cite[Proof of Theorem 8.1]{BlickleSchwedeTuckerTestAlterations}, and apply \autoref{thm:perturbed_tau_submodule_single_alteration}.  

    We next prove \autoref{prop.PropertiesOfUltTauOmegaForPairs.FiniteMaps} and \autoref{prop.PropertiesOfUltTauOmegaForPairs.BirationalMaps} simultaneously. Let $G_{0, X}$ and $G_{0, X'}$ be the effective Cartier divisors on $X$ and $X'$ that satisfy the conclusion of  \autoref{thm:perturbed_tau_submodule_single_alteration} for $X$ and $X'$ respectively. Fix a Cartier divisor $G\geq G_{0, X}$ such that $f^*G\geq G_{0, X'}$ and fix $0 < \epsilon \ll 1$ (depending on $G$) so that this data satisfies the conclusion of \autoref{thm:perturbed_tau_submodule_single_alteration} for $G$ on $X$ and $f^*G$ on $X'$ respectively.  
    Take $Y_\epsilon$ and $Y'_{\epsilon}$ as in  \autoref{thm:perturbed_tau_submodule_single_alteration} and fix an alteration $Y$ that dominates both $Y_\epsilon$ and $Y'_{\epsilon}$. Consider 
    $$\pi: Y\xrightarrow{\pi'} X' \xrightarrow{f}X.$$
    It follows from \autoref{thm:perturbed_tau_submodule_single_alteration} that 
    $$\tau(\omega_X, \Gamma)=\Tr_{Y/X}(\pi_*\omega_Y(-\epsilon \pi^*G-\pi^*\Gamma)), \text{ and }$$
    $$\tau(\omega_{X'}, f^*\Gamma)=\Tr_{Y/X'}(\pi'_*\omega_Y(-\epsilon \pi^*G-\pi^*\Gamma)).$$
In particular, applying $f_*$ to the surjection
$$\pi'_*\omega_Y(-\epsilon \pi^*G-\pi^*\Gamma) \xrightarrow{\Tr_{Y/X'}} \tau(\omega_{X'}, f^*\Gamma),$$
we obtain the top row in the commutative diagram
\[
    \xymatrix{
        \pi_*\omega_Y(-\epsilon \pi^*G-\pi^*\Gamma) \ar@{->>}[rrd]_{\Tr_{Y/X}} \ar[rr]^-{f_*\Tr_{Y/X'}} & & f_*\tau(\omega_{X'}, f^*\Gamma) \ar@{^{(}->}[r] & K(X') \ar[d]^{\Tr_{X'/X}}  \\
        & & \tau(\omega_X, \Gamma) \ar@{^{(}->}[r] & K(X)
    }.
\]
Now if $f$ is finite, then $f_*\Tr_{Y/X'}$ is surjective and the diagram shows that $\Tr(f_*\utau(\omega_{X'}, f^* \Gamma)) = \utau(\omega_X, \Gamma)$ (viewed as subsheaves of $K(X)$) which proves \autoref{prop.PropertiesOfUltTauOmegaForPairs.FiniteMaps}, while if $f$ is projective birational, then $K(X')=K(X)$ (i.e., $\Tr_{X'/X}$ is an isomorphism) and the diagram shows that $f_* \utau(\omega_{X'}, f^* \Gamma)\supseteq \utau(\omega_X, \Gamma)$ which proves \autoref{prop.PropertiesOfUltTauOmegaForPairs.BirationalMaps}.

  Now we prove \autoref{prop.PropertiesOfUltTauOmegaForPairs.ComparisonWithB0}. We may assume that $X$ is projective, and we first consider the case that $\Gamma=0$ (which is the key case). 
  By \autoref{def:perturbed_tau}, we can fix a Cartier divisor $G'\geq 0$ such that for all Cartier divisors $G''\geq G'$, we have $\utau(\omega_X)=\tau_{\alt}^a(\omega_X, \epsilon'' G'')$ for all $0<\epsilon''\ll 1$. Set $G^0:=G'+\Div(p)$ and we will show that $G^0$ satisfies the conclusion of \autoref{prop.PropertiesOfUltTauOmegaForPairs.ComparisonWithB0} (with $\Gamma=0$). Fix a Cartier divisor $G\geq G^0$, we need to show that $\utau(\omega_X)=\tau_{\myB^0}(\omega_X, \epsilon G)$ for all $0<\epsilon\ll1$. Note that, since $G-\Div(p)\geq G^0-\Div(p)=G'$, by our choice of $G'$, we know that 
    \begin{equation}
    \label{eqn.utauVStauAlt}
    \utau(\omega_X) =\tau_{\alt}^a\big(\omega_X, \epsilon (G-\Div(p))\big)  \text{ for all $0<\epsilon \ll1$. }
    \end{equation}
    Next, we note that by Noetherianity of $X$, there exists $\delta>0$ such that 
    \begin{equation}
    \label{eqn.tauB0constant}
    \tau_{\myB^0}\big(\omega_X, \alpha\Div(p) + \beta(G-\Div(p))\big) \text{ is constant for all $0< \alpha, \beta \leq \delta$. }
    \end{equation}
 At this point, by \autoref{eqn.utauVStauAlt} and \autoref{eqn.tauB0constant}, we fix $0<\epsilon \ll \delta$ so that we have $\utau(\omega_X) =\tau_{\alt}^a(\omega_X, \epsilon (G-\Div(p)))$. Choose a finite cover $f$: $W\to X$ such that $\epsilon f^*(G-\Div(p))$ is Cartier on $W$. We then have that
 \begin{align*}
    \utau(\omega_X) &= \tau^a_{\alt}(\omega_X, \epsilon(G-\Div(p)))\\
    & = \Tr\big(f_* \tau^a_{\alt}(\omega_W) \otimes_{\cO_W} \cO_W(-\epsilon f^*(G-\Div(p)))\big) \\
    & = \Tr\big(f_*\tau_{\myB^0}(\omega_W, \epsilon'\Div(p)) \otimes_{\cO_W} \cO_W(-\epsilon f^*(G-\Div(p)))\big) \\
    & = \tau_{\myB^0}(\omega_X, \epsilon'\Div(p) + \epsilon(G-\Div(p))) \\
    & = \tau_{\myB^0}(\omega_X, \epsilon\Div(p) + \epsilon(G-\Div(p))) \\
    & = \tau_{\myB^0}(\omega_X, \epsilon G)
 \end{align*}
for all $0<\epsilon'\ll \epsilon$ where the first equality follows by our choice of $\epsilon$, the second equality follows from \autoref{def:perturbed_tau}, the third equality follows from \autoref{thm.ComparisonTauRHvsHLSNoPair}, the fourth equality follows from \cite[Theorem 4.21]{HaconLamarcheSchwede}, the fifth equality follows from $\epsilon'\ll \epsilon\ll \delta$ and \autoref{eqn.tauB0constant}. This completes the proof when $\Gamma=0$. 

Finally, we can take a finite cover $\pi$: $X'\to X$ such that $\pi^*\Gamma$ is a Cartier divisor on $X'$. By the already established $\Gamma=0$ case applied to $X'$, we find a Cartier divisor $G'^0\geq 0$ on $X'$. Note that by the argument in the $\Gamma=0$ case, we may as well assume that $G'^0=\pi^*G^0$ where $G^0\geq 0$ is a Cartier divisor on $X$. It follows that for all Cartier $G\geq G_0$ on $X$, $\utau(\omega_{X'})=\tau_{\myB^0}(\omega_{X'}, \epsilon \pi^*G)$ for all $0<\epsilon\ll1$. By \autoref{def:perturbed_tau}, $$\utau(\omega_X, \Gamma)= \Tr(\pi_*\utau(\omega_{X'})\otimes_{\cO_{X'}}\cO_{X'}(-\pi^*\Gamma)),$$ 
and by \cite[Theorem 4.21]{HaconLamarcheSchwede}, 
$$\tau_{\myB^0}(\omega_X, \epsilon G+ \Gamma)= \Tr(\pi_*\tau_{\myB^0}(\omega_{X'}, \epsilon \pi^*G) \otimes_{\cO_{X'}}\cO_{X'}(-\pi^*\Gamma)).$$
Putting these together we see that $\utau(\omega_X, \Gamma)=\tau_{\myB^0}(\omega_X, \epsilon G+ \Gamma)$ for all $0<\epsilon\ll1$ as desired. 

 For \autoref{prop.PropertiesOfUltTauOmegaForPairs.CompletionComparisonWithTau+}, letting $G^0$ be chosen as in \autoref{thm:perturbed_tau_submodule_single_alteration}, for $G \geq G^0+ \Div(p)$ and $0 < \epsilon \ll 1$ we have
 \begin{align*}
 \tau(\omega_X,\Gamma) \otimes R  = \tau_{\alt}^a(\omega_X, \epsilon (G-\Div(p)) + \Gamma) \otimes R  = \tau_+(\omega_R, \epsilon G \vert_R + \Gamma \vert_R)
 \end{align*}
 where the second equality uses \autoref{cor.TauRHVsTau+ForDivisorPairs}.

    For \autoref{prop.PropertiesOfUltTauOmegaForPairs.Completion} we may assume that $X$ is affine and hence we may assume that $\Gamma$ is effective.  Pick $f : Y \to X$ finite with $f^* \Gamma$ is Cartier.  Let $h$ and $G$ be as in \autoref{prop.PerturbedTestIdealIsReal} (i.e., $h: \overline{X}\to \mathbb{P}^n_{V}$ is finite surjective and \'etale outside $h^{-1}G$); also set $h' = h \circ f$ and enlarge $G$ to assume $h'$ is also \'etale outside $(h')^{-1}G$.    Choose $y \in Y$ to be a point lying over $x$.  Let $S = \widehat{\cO_{Y,y}}$ and we see that $R \subseteq S$ is finite.
    For $0< \epsilon \ll 1$, we know that
    \begin{align*}
        \tau(\omega_X, \Gamma) \otimes R &= \tau_{\alt}^a(\omega_X, \epsilon h^*G + \Gamma)\otimes R \\
        &= \tau_+(\omega_R, \epsilon \Div(p) + \epsilon h^*G \vert_R + \Gamma \vert_R ) \\
        & = \Tr_{S/R} ( \tau_+(\omega_S, \epsilon \Div(p) + \epsilon (h')^*G \vert_S + f^*\Gamma \vert_S)) \\
        & = \Tr_{S/R} (\tau_{\cB}(\omega_S) \otimes_S S(-f^* \Gamma|_S)) \\
        & = \tau_{\cB}(\omega_R, \Gamma|_R)
    \end{align*}
    where the second equality uses \autoref{cor.TauRHVsTau+ForDivisorPairs}, the third equality uses \cite[Theorem 6.17]{MaSchwedeSingularitiesMixedCharBCM}, the fourth equality follows from \autoref{clm.Tau+PerturbedEqualsTauBigB}, and the last equality again uses \cite[Theorem 6.17]{MaSchwedeSingularitiesMixedCharBCM}.
    

    For \autoref{prop.PropertiesOfUltTauOmegaForPairs.Perturbation}, it is harmless to assume that $X$ quasi-projective.  The result then follows from \autoref{prop.PropertiesOfUltTauOmegaForPairs.ComparisonWithB0} since we may increase $G$.  Alternately, it can be deduced from \autoref{thm:perturbed_tau_submodule_single_alteration} by increasing $G$, or one may deduce it from \autoref{prop.PropertiesOfUltTauOmegaForPairs.Completion} and the fact that it holds for $\tau_{\cB}$ (\cite[Theorem C]{MaSchwedeSingularitiesMixedCharBCM}) plus the fact that these sheaves are coherent on a Noetherian $X$.

    Finally, for \autoref{prop.PropertiesOfUltTauOmegaForPairs.SmoothPullBack}, we may assume that both $X$ and $Y$ are quasi-projective.  It suffices to prove this for \'etale maps and $X \times \bA^m \to X$ (adjoining variables).

    Thus suppose first that $f : Y \to X$ is \'etale.  Let $\overline{f} : \overline{Y} \to X$ denote the integral closure of $X$ in $K(Y)$.  
    Finally, we pick $G^0$ so that both $G^0,$ and $\overline{f}^* G^0$ 
    satisfy \autoref{thm:perturbed_tau_submodule_single_alteration} on $X$ and ${\overline{Y}}$ respectively.  
    Pick $G \geq G^0$ and $1 \gg \epsilon > 0$.  
    We pick a sufficiently large alteration $X_{\epsilon} \to \overline{Y} \to X$ so that for all further alterations $\pi_X : X' \xrightarrow{\pi_{\overline{Y}}} \overline{Y} \to X$, we have that 
    \begin{eqnarray}        
        \tau(\omega_X, \Gamma) & = & \Tr_{X'/X} {\pi_X}_* \cO_{X'}(K_{X'} - \pi_X^* (\Gamma + \epsilon G)) \label{eq.TestIdealToBaseChange}, \text{ and } \\
        \tau(\omega_{\overline{Y}}, \overline{f}^* \Gamma) & = & \Tr_{X'/\overline{Y}} {\pi_{\overline{Y}}}_* \cO_{X'}(K_{X'} - \pi_X^* (\Gamma + \epsilon G)) \label{eq.TestIdealNotBaseChanged}
    \end{eqnarray}
    Notice that $X' \times_X Y$ is normal and hence a disjoint union of irreducible components $W_1, \dots, W_n$ with closures $\overline{W}_1, \dots, \overline{W}_n$.  Each $\overline{W}_i$ dominates $X'$ and hence $X_{\epsilon}$.  It follows that applying $f^* (-) \otimes \omega_{Y/X}$ yields \autoref{eq.TestIdealNotBaseChanged} restricted to $Y$, which is what we wanted.

    It suffices to prove the result for $f : Y = X \times \bP^m \to X$ and hence we may assume that $X$ is projective over $V$. For convenience of the reader we provide the following diagram, the construction of which is gradually explained below.
            \begin{center}
    \begin{tikzcd}
    \ar[bend right = 30, start anchor={[xshift=-5.2em]center}, end anchor={[xshift=-4.6em]center}]{dd}[swap]{v''} \mathllap{Y'' :=\ } X'' \times \bP^m_V \ar{r}{f''} \ar{d}{\beta} & X'' \ar{d}{\alpha} \ar[bend left = 40, start anchor={[xshift=0.9em]center}, end anchor={[xshift=0.7em]center}]{dd}{u''}  & \\
     \mathllap{Y' :=\ } X' \times \bP^m_V \ar{r}{f'} \ar{d}{v'} & X' \ar{d}{u'} \arrow[ul, phantom, "\lrcorner", very near end] & \\ 
    \mathllap{Y :=\ } X \times \bP^m_V \ar{dr} \ar{r}{f} & X \ar{dr}{h} \arrow[ul, phantom, "\lrcorner", very near end] &\\
    & \bP^n_V \times \bP^m_V \ar{r} & \bP^n_V.
    \end{tikzcd}
    \end{center}
    Choose $u' : X' \to X$ a finite map such that $(u')^* \Gamma$ is Cartier.  Form the base change $f' : Y' = X' \times \bP^m \to X'$.
    On $X'$, using a map $h_{X'} : X' \xrightarrow{u'} X \to \bP^n_V$, we can find a Cartier $G^0$ on $X'$ pulled back from $\bP^n_V$ outside of which $h_{X'}$ is \'etale.  Base changing with $\bP^m_V$ we have  that
    \[
        h_{Y'} : Y' \to Y \to \bP^n_V \times \bP^m_V
    \]
    also a finite surjective map to a nonsingular scheme, which is \'etale outside of $f'^* G^0$.  Choose $1 \gg \epsilon > 0$ that works for  both $G := G^0$ and $G_Y := f'^* G^0$.  Finally, set $u'' : X'' \xrightarrow{\alpha} X' \to X$ a further finite cover such that 
    \[
        (u'')^* \Gamma + \alpha^* \epsilon G
    \]
    is Cartier.  We then have that
    \begin{equation}
        \label{eq.TauOmegaXPreBaseChange}
        \tau(\omega_X, \Gamma) = \Tr_{X''/X} (u'')_* \Big( \tau_{\alt}^a(\omega_{X''}) \otimes_{\cO_{X''}} \cO_{X''}(-(u'')^* \Gamma - \alpha^* \epsilon G)\Big).
    \end{equation}
    Setting $Y'' = X'' \times \bP^m_V$ with associated $v'' : Y'' \xrightarrow{\beta} Y' \to Y$ we see likewise that 
    \begin{equation}
        \label{eq.TauOmegaXPostBaseChange}
        \tau(\omega_Y, f^* \Gamma) = \Tr_{Y''/Y} (v'')_* \Big( \tau_{\alt}^a(\omega_{Y''}) \otimes_{\cO_{Y''}} \cO_{Y''}(-(v'')^* f^* \Gamma - \beta^* \epsilon f'^* G)\Big).
    \end{equation}
    The result then follows from \autoref{thm.SmoothPullbackForRHOnNoetherian} once we observe that the formation of trace commutes with base change and hence applying $f^*(-) \otimes \omega_{Y/X}$ to \autoref{eq.TauOmegaXPreBaseChange} yields \autoref{eq.TauOmegaXPostBaseChange}.
\end{proof}

\begin{remark} \label{remark:choice-of-perturbation-divisor}
We emphasize that any effective Cartier divisor $G^0$ as in the statements of \autoref{def:perturbed_tau},  \autoref{thm:perturbed_tau_submodule_single_alteration}, and \autoref{prop.PropertiesOfUltTauOmegaForPairs} is valid as long as its support contains the fibre over $\varpi$, and if we fix $\phi : Z \to X$ with  $\phi^* \Gamma$ is Cartier, we have that $\Supp \phi^* G^0$ contains the ramification of the induced $Z \xrightarrow{\phi} X \xrightarrow{h} \bP^n_V$.
\end{remark}

We conclude with a corollary pointing out that the multiplier ideal in characteristic zero can be computed from \emph{finite covers} in mixed characteristic and localizing.  Suppose $R$ is a Noetherian domain of essentially finite type over a DVR of mixed characteristic $(0, p> 0)$ such that $p$ is in the Jacobson radical (for instance if $R$ is local of mixed characteristic).  
Define 
\[ 
    \tau_{\mathrm{fin}}(\omega_R) = \bigcap_{S \supseteq R} \Tr(\omega_S)
\]
where $S$ runs over finite extensions of $R$ contained in a fixed choice of $R^+$.  We define $\tau_{\alt}(\omega_R)$ analogously, running over alterations, as in the introduction.

\begin{corollary}
    \label{cor.MultiplierIdealViaFiniteCovers}
    With notation as above, $\tau_{\mathrm{fin}}(\omega_R)[1/p] = \tau_{\alt}(\omega_R)[1/p] = \mJ(\omega_{R[1/p]})$.
\end{corollary}
\begin{proof}
    Note we have 
    \[
        \tau_{\alt}^a(\omega_R)[1/p] \subseteq \tau_{\alt}(\omega_R)[1/p] \subseteq \mJ(\omega_{R[1/p]})
    \]
    where $\tau_{\alt}(\omega_R)$ is defined as in the introduction (cf.\  \autoref{prop.PropertiesOfUltTauOmegaForPairs} \autoref{prop.PropertiesOfUltTauOmegaForPairs.IncreasingCoefficientContainment}).  Hence, all three are equal which proves the second equality in the statement of \autoref{cor.MultiplierIdealViaFiniteCovers}.  Since $\tau_{\fin}(\omega_R)[1/p] \supseteq \tau_{\alt}(\omega_R)[1/p]$ it follows that $\tau_{\fin}(\omega_R)[1/p] \supseteq \mJ(\omega_{R[1/p]})$.

    Next pick a maximal ideal $\fram \in R$.  By assumption $p \in \fram$.  Let $\widehat{R}$ denote the $\fram$-adic completion of $R$.  Since $R \to \widehat{R}$ has geometrically regular fibers as $R$ is excellent, it follows that 
    \[
        \mJ(\omega_{R[1/p]}) \otimes_{R[1/p]} \widehat{R}[1/p] = \mJ(\omega_{\widehat{R}[1/p]})
    \]
    since a resolution base changes to a resolution.  Now, on the other hand we have that 
    \[
        \tau_{\fin}(\omega_R) \otimes_R \widehat{R} \subseteq \bigcap_{S \supseteq R} \Tr(\omega_S \otimes_R \widehat{R})
    \] 
    since the intersection can only become bigger after completion.  The right side is equal to $\tau_{\fin}(\omega_{\widehat{R}}) = \tau_{\alt}(\omega_{\widehat{R}})$ by \cite[Proposition 4.29]{BMPSTWW1}.  Furthermore $\tau_{\alt}(\omega_{\widehat{R}})[1/p] \subseteq \mJ(\omega_{\widehat{R}[1/p]})$, since $\tau_{\alt}(\omega_{\widehat{R}})$ is an intersection of things, each of which after inverting $p$ computes $\mJ(\omega_{\widehat{R}[1/p]})$. Hence 
    \[
        \tau_{\fin}(\omega_R) \otimes_R \widehat{R}[1/p] \subseteq \tau_{\alt}(\omega_{\widehat{R}})[1/p] \subseteq \mJ(\omega_{\widehat{R}[1/p]}) = \mJ(\omega_{R[1/p]}) \otimes_{R[1/p]} \widehat{R}[1/p].
    \]
    Since this holds for all $\widehat{R}$ (running over all maximal ideals) and since the map $R_{\fram} \to \widehat{R}$ is faithfully flat, we see that 
    \[
        \tau_{\fin}(\omega_R)[1/p] \subseteq \mJ(\omega_{R[1/p]}).
    \]
    This completes the proof.
\end{proof}
\subsection{Test ideals}

So far we have only defined test modules.  We can now define test ideals as well.

\begin{setting} \label{setting:section7testideals}
Let $(V, \varpi, k)$ be a DVR of mixed characteristic $(0,p>0)$ and let $X$ be a finite type, normal, integral scheme over $\Spec(V)$ such that $X \to \Spec V$ is flat (equivalently, surjective). We take $\Delta$ to be a $\bQ$-divisor on $X$ such that $K_X+\Delta$ is $\bQ$-Cartier.
\end{setting}

\begin{definition}
    \label{def.UltTestIdealDefinition}
    We work in the setting of \autoref{setting:section7testideals}. 
    We define the \emph{test ideal} of $(X,\Delta)$ to be the fractional ideal:
    \[
        \utau(\cO_X, \Delta) := \utau(\omega_X, K_X + \Delta).
    \]
    When $\Delta \geq 0$, this is an ideal and not simply a fractional ideal.
\end{definition}

It follows directly from the definitions and previous results that this test ideal agrees with those previously defined:

\begin{corollary}\label{cor:one_test_ideal_to_rule_them_all}

Let $(X,\Delta)$ be as in 
\autoref{setting:section7testideals}, and assume additionally that $X\to \Spec(V)$ is quasi-projective.  Furthermore, fix a point $x\in X_{p=0}$ and let $R=\widehat{\sO_{X,x}}$. Then there is a Cartier divisor $G^0\geq 0$ on $X$ such that for any Cartier divisor $G\geq G^0$ and any $0<\epsilon\ll 1$ (depending on $G$), we have:
\begin{enumerate}
    \item $\tau(\sO_X,\Delta)=\tau_{\alt}^a(\sO_X,\Delta+\epsilon G):=\tau_{\alt}^a(\omega_X, K_X+\Delta+\epsilon G)$.
    \item $\tau(\sO_X,\Delta)=\tau_{\myB^0}(\sO_X,\Delta+\epsilon G)$ when $V$ is complete.
    \item $\tau(\sO_X,\Delta)\cdot R=\tau_+(R,(\Delta+\epsilon G)|_R)$
    \item $\tau(\sO_X,\Delta)\cdot R=\tau_{\mathcal{B}}(R,\Delta|_R)$
    \item $\tau(\sO_X,\Delta)[1/p]=\mJ(\sO_{X[1/p]}, \Delta|_{X[1/p]})$
    \item\label{itm:perturbation_test_ideal} For all Cartier $D\geq 0$, $\tau(\sO_X,\Delta)= \tau(\sO_X,\Delta+\epsilon D)$ for all $0<\epsilon\ll1$. 
\end{enumerate}
\end{corollary}
\begin{proof}
These are translations of the previous results to the case of $\Gamma=K_X+\Delta$.  Specifically:
\begin{enumerate}
    \item This follows immediately from the definition of $\tau(\sO_X,\Delta)$.
    \item This follows from \autoref{prop.PropertiesOfUltTauOmegaForPairs}\autoref{prop.PropertiesOfUltTauOmegaForPairs.ComparisonWithB0}.
    \item This follows from \autoref{prop.PropertiesOfUltTauOmegaForPairs}\autoref{prop.PropertiesOfUltTauOmegaForPairs.CompletionComparisonWithTau+}.
    \item This follows from 
\autoref{prop.PropertiesOfUltTauOmegaForPairs}\autoref{prop.PropertiesOfUltTauOmegaForPairs.Completion}.
    \item This follows from \autoref{prop.PropertiesOfUltTauOmegaForPairs}\autoref{prop.PropertiesOfUltTauOmegaForPairs.InvertP}.
    \item This follows from \autoref{prop.PropertiesOfUltTauOmegaForPairs}\autoref{prop.PropertiesOfUltTauOmegaForPairs.Perturbation}.
\end{enumerate}
\end{proof}

\begin{theorem}
    \label{thm.LocalPropertiesOfUltTauOX}
    With notation as in 
    \autoref{setting:section7testideals} the test ideal $\utau(\cO_X, \Delta)$ satisfies the following properties.
    \begin{enumerate}
        \item Restriction:  Suppose $H$ is a normal Cartier divisor on $X$ such that $H$ and $\Delta$ have no common components in support.  Then  \label{thm.LocalPropertiesOfUltTauOX.Restriction}
        \[ 
            \utau(\cO_X, \Delta) \cdot \cO_H \supseteq \utau(\cO_H, \Delta|_H).
        \]        
        \item Finite maps:  If $f : Y \to X$ is a finite surjective map, then 
        \[ 
            \Tr\big( \utau(\cO_Y, f^* \Delta - \Ram) \big)= \utau(\cO_X, \Delta).
        \]
        \item Smooth maps:  If $f : Y \to X$ is a smooth map, then 
        \[
            f^* \utau(\cO_X, \Delta) = \utau(\cO_Y, f^* \Delta)
        \]
        \item Birational maps:  If $f : Y \to X$ is a finite type birational map, then 
            \[
                f_* \utau(\cO_Y, f^* (K_X + \Delta) - K_Y) \supseteq \utau(\cO_X, \Delta).
            \]
            \label{thm.LocalPropertiesOfUltTauOX.Birational}
        \item Subadditivity: \label{thm.LocalPropertiesOfUltTauOX.Subadditivity}
        If $X$ is regular, then for $\Delta_1, \Delta_2 \geq 0$,
        \[
            \tau(\cO_X, \Delta_1 + \Delta_2) \subseteq \tau(\cO_X, \Delta_1) \tau(\cO_X, \Delta_2).
        \]
    \end{enumerate}
\end{theorem}
\begin{proof}
    By using \autoref{prop.PropertiesOfUltTauOmegaForPairs} \autoref{prop.PropertiesOfUltTauOmegaForPairs.Completion}, the first follows from the corresponding local results for $\tau_{\cB}$, see \cite[Proposition 2.10 and Theorem 3.1]{MaSchwedeTuckerWaldronWitaszekAdjoint}, or from \cite[Lemma 4.20 and Theorem 4.23]{HaconLamarcheSchwede}.  The second and third follows from \autoref{prop.PropertiesOfUltTauOmegaForPairs} \autoref{prop.PropertiesOfUltTauOmegaForPairs.FiniteMaps}, \autoref{prop.PropertiesOfUltTauOmegaForPairs.SmoothPullBack}.
    {\color{black} We now prove \autoref{thm.LocalPropertiesOfUltTauOX.Birational}; in the projective case we can simply apply \autoref{prop.PropertiesOfUltTauOmegaForPairs} \autoref{prop.PropertiesOfUltTauOmegaForPairs.BirationalMaps}. More generally, it suffices to check the statement for affine $Y$, and any such $Y$ can be embedded as an affine chart in a projective $\overline{f} : \overline{Y} \to X$.  Then we have that 
    \[
        f_* \utau(\cO_Y, f^* (K_X + \Delta) - K_Y) \supseteq {\overline f}_* \utau(\cO_{\overline{Y}}, f^* (K_X + \Delta) - K_{\overline{Y}}) \supseteq \utau(\cO_X, \Delta).
    \]
    }
    Finally, subadditivity \autoref{thm.LocalPropertiesOfUltTauOX.Subadditivity} follows since it holds for complete local rings from the argument of \cite[Theorem 4.4]{MaSchwedePerfectoidTestideal}, as carried out in \cite[Theorem 6.18]{MurayamaUniformBoundsOnSymbolicPowers}, \cf \cite[Section 8.3]{HaconLamarcheSchwede}, all of which are based on an argument of Takagi in characteristic $p > 0$ \cite[Theorem 2.4]{TakagiFormulasForMultiplierIdeals}.
\end{proof}

Next we show that this test ideal can be computed from a single alteration in the following sense.



\begin{corollary} \label{cor:test-ideal-one-alteration}
In the situation of \autoref{setting:section7testideals} assume that $X$ is quasi-projective.
Then there exists an effective Cartier divisor $G^0$ on $X$ so that for every Cartier divisor $G \geq G^0$ on $X$, the following is satisfied. 
  
  There exists an $\epsilon_0 > 0$ such that for every smaller $\epsilon_0 > \epsilon > 0$ there exists a normal alteration $\pi_{\epsilon} : Y_{\epsilon} \to X$ with $\pi_{\epsilon}^* (\epsilon G + {(K_X+\Delta)})$ Cartier, such that for every further normal alteration $Y \to Y_{\epsilon}$ with composition $\pi : Y \to Y_{\epsilon} \xrightarrow{\pi_{\epsilon}} X$ we have that 
\[
        \tau(\cO_X, \Delta) = \Tr\big( \pi_* \cO_Y(K_Y - \pi^*(K_X + \Delta) - \epsilon \pi^* G)\big)
    \]
\end{corollary}
\begin{proof}
    This follows from \autoref{thm:perturbed_tau_submodule_single_alteration} applied to $\Gamma=K_X+\Delta$.
\end{proof}

Finally we note a global generation result which immediately follows from the corresponding result in \cite{HaconLamarcheSchwede}.

\begin{corollary}
\label{cor.EffectiveGlobalGenerationTauDivisor}
In the situation of \autoref{setting:section7testideals} we assume that $X$ is projective over a complete DVR $V$ and $\Delta \geq 0$.
Let $M$ be a Cartier divisor such that $M-K_X-\Delta$ is big and nef, and let $A$ be a globally generated ample Cartier divisor.  If $n_0=\dim(X_{p=0})$, then there exists Cartier $H \geq 0$ such that  $\tau(\sO_X,\Delta)\otimes \sO_X(nA+M)$ is globally generated by $\myB^0(X,\Delta+\epsilon H; \sO_X(nA+M))$ for all $\epsilon > 0$ and $n\geq n_0$. 

Hence, if $\tau(\cO_X, \Delta) = \cO_X$, then $\myB^0(X,\Delta; \sO_X(nA+M))$ globally generates $\cO_X(nA + M)$.
\end{corollary}

\begin{proof}
Pick $G^0$ as in \autoref{cor:one_test_ideal_to_rule_them_all}.  
As $M - K_X - \Delta$ is big and nef, there exists an effective Cartier divisor $H \geq G^0$ such that $M - K_X - \Delta - \epsilon H$ is ample for $1 \gg \epsilon > 0$. 
The corollary then follows from \cite[Corollary 6.2]{HaconLamarcheSchwede} applied to $\tau_{\myB^0}(\sO_X,\Delta+\epsilon H)$ using the fact that $\tau(\sO_X,\Delta) = \tau_{\myB^0}(\sO_X,\Delta + \epsilon H)$ for $1 \gg \epsilon > 0$ by \autoref{cor:one_test_ideal_to_rule_them_all}(b).  The final statement follows as 
$\myB^0(X,\Delta+\epsilon H; \sO_X(nA+M)) \subseteq \myB^0(X,\Delta; \sO_X(nA+M)) \subseteq H^0(X, \cO_X(nA+M))$.
\end{proof}

\subsection{Openness of $p$-almost splinter locus}\label{ss:open-splinter}

Last we study the question of whether the splinter locus is open.
\begin{definition}
\label{def.PAlmostSplinter}
    Fix a prime number $p>0$. Let $R$ be a Noetherian domain all of whose primes have residue field of characteristic $0$ or $p$.  We say that $R$ is \emph{a $p$-almost splinter} if and only if there exists a rational number $0 < \epsilon \ll 1$ such that for every finite extension of rings $R \subseteq S$ such that $p^{\epsilon} \in S$, we have that the map $R \xrightarrow{p^{\epsilon}} S$ splits as a map of $R$-modules.  The \emph{$p$-almost splinter locus} is the set of primes $Q \in \Spec R$ such that $R_Q$ is a $p$-almost splinter.
\end{definition}

Note that if $R$ is normal, then $R[1/p]$ is normal in characteristic zero and so is a splinter.  Hence the non-$p$-almost splinter locus is always a subset of $V(p) \subseteq \Spec R$.

\begin{setting}
    \label{set.PAlmostSplinter}
    In this subsection, we fix $R$ to be a normal $\bQ$-Gorenstein domain of finite type and flat over a DVR $V$ of mixed characteristic $(0,p>0)$. Take $X = \Spec (R)$. 
\end{setting}

\begin{lemma} \label{lem:a-tau-governs-splinters-noncomplete}  With notation as in \autoref{set.PAlmostSplinter}, if $\tau^a_{\rm alt}(R) = R$, then $R$ is a $p$-almost splinter. 
\end{lemma}
\begin{proof}
Fix $0 < \epsilon \ll 1$ and let $R \subseteq S$ be a finite extension such that $p^{\epsilon} \in S$.  Recall that $\omega_{S/R} = \Hom_R(S, R)$ is reflexive. By assumptions, $\tau_{\rm alt}(R, p^\epsilon) = R$ for $0 < \epsilon \ll 1$, and if $R \hookrightarrow S$ is a finite extension, then $\Spec (S) \to \Spec (R)$ is an alteration. Hence
\[
\Image \Big( \Tr_{S/R} : \omega_{S/R} \xrightarrow{p^{\epsilon}} R \Big) = \Image \Big( \Hom_R(S, R) \xrightarrow{\text{evaluation at}\,\,  p^{\epsilon}} R \Big)
\]
is equal to $R$, and so $\Hom_R(S, R) \xrightarrow{\text{evaluation at}\,\,  p^{\epsilon}} R$ is surjective. This implies that $R \xrightarrow{p^{\epsilon}} S$ splits.
\end{proof}

\begin{lemma} \label{lem:a-tau-governs-splinters-complete} Let $R$ be a complete Noetherian local domain. Then  
$\tau_+(R, p^\epsilon) = R$ for all rational $0 < \epsilon \ll 1$ if and only if $R$ is a $p$-almost splinter.
\end{lemma}
\begin{proof}
Recall that $\omega_{S/R} = \Hom_R(S, R)$, for a finite extension $R \subseteq S$, is reflexive. Then we have that:
\[
\tau_+(R, p^\epsilon) = \bigcap_{\textrm{ finite } R \,\subseteq\, S} \Tr_{S/R}(\omega_{S/R}) =  \bigcap_{\textrm{ finite } R \,\subseteq\, S} \Image \Big( \Hom_R(S, R) \xrightarrow{\text{evaluation at}\,\,  p^{\epsilon}} R \Big), 
\]
where the intersection is taken over all finite ring extensions $R \subseteq S$ such that $p^{\epsilon} \in S$. Thus the statement follows from the fact that $R \xrightarrow{p^{\epsilon}} S$ splits if and only if $\Hom_R(S, R) \xrightarrow{\text{evaluation at}\,\,  p^{\epsilon}} R$ is surjective.
\end{proof}

 \begin{theorem} \label{thm:p-almost-splinter-locus-open} For $R$ as in \autoref{set.PAlmostSplinter}, the $p$-almost splinter locus is open. Specifically if $R_{\mathfrak{q}}$ is a $p$-almost splinter for $\mathfrak{q} \in X_{p=0}$, then $R[1/g]$ is a $p$-almost splinter for some $g \not \in \mathfrak{q}$.
\end{theorem}
In the proof we use the notation $\tau^a_{\alt}(R)$ from Remark \ref{rem:a-test-ideal-definition}. All of its properties follows from that of $\tau^a_{\alt}(\omega_R)$ from \autoref{sec:SingOverDVR}.
\begin{proof}
Observe that the $\mathfrak{q}$-completion $\widehat{R}_{\mathfrak{q}}$ is a $p$-almost splinter as can be seen immediately from \cite[Proposition 4.29]{BMPSTWW1}.

First, we prove that $\tau^a_{\alt}(R) \not \subseteq \mathfrak{q}$. To this end, suppose by contradiction that $\tau^a_{\alt}(R) \subseteq \mathfrak{q}$.  This immediately contradicts the following equalities for $0 < \epsilon \ll 1$: 
\[
\widehat{\tau^a_{\alt}(R)}_{\mathfrak{q}} \overset{(1)}{=}
\tau_+(\widehat{R}_{\mathfrak{q}}, p^\epsilon) \overset{(2)}{=} \widehat{R}_{\mathfrak{q}},
\]

where (1) follows from \autoref{cor.TauAgreesWithCompletionComputation}, and (2) follows from
\autoref{lem:a-tau-governs-splinters-complete}.

Thus we can assume that there exists $g \in \tau^a_{\alt}(R)$ such that $g \not \in \mathfrak{q}$. Since $\tau^a_{\alt}(R)$ is stable under localization, $\tau^a_{\alt}(R[1/g]) = R[1/g]$. Thus $R[1/g]$ is a $p$-almost splinter by \autoref{lem:a-tau-governs-splinters-noncomplete}.
\end{proof}

\begin{remark}
    \label{rem.pgAlmostSplinterIsAlsoOpen}
    For $R$ as in \autoref{set.PAlmostSplinter} and any $0 \neq g \in R$, the results in this section, with the same proof as above, also show that the analogously defined $pg$-almost splinter locus of $R$ is open as well.  
\end{remark}

\section{Test ideals of non-principal ideals}
\label{sec.TestIdealsNonPrincipal}

There are multiple ways to define test ideals associated to ideals $\fra^t$ raised to formal exponents $t \geq 0$.  Our goal in this section is to show that they all (almost) agree and hence that theorems we prove about one can be transferred to another.

We work first in the complete local case.  

\subsection{The complete local case}\label{sec:local_non_principal}

The following subsection is independent from the work done earlier in the paper.

\begin{setting}
\label{set.CompleteLocalTestIdealForNonPrincipal}
    Suppose $(R, \m)$ is a  complete Noetherian local domain with residue field of characteristic $p > 0$.  Suppose $\fra \subseteq R$ is an ideal and $t \geq 0$ is a rational number.  We also fix a set $\{ g_{\lambda} \}_{\lambda \in \Lambda} \subseteq \fra$ (typically a generating set). 
    
    More generally $\fra_1, \dots, \fra_m$ are ideals and $t_1, \dots, t_m \geq 0$ are rational numbers.  In that case we may also consider subsets $\{ g_{1, \lambda}\}_{\lambda \in \Lambda_1} \subseteq \fra_1, \dots, \{ g_{m, \lambda} \}_{\lambda \in \Lambda_m} \subseteq \fra_m$ (typically generating sets).  
    
    We set $B$ to be a perfectoid big Cohen-Macaulay $R^+$-algebra (for instance $\widehat{R^+}$).  
\end{setting}

Our first definition uses all the elements of the ideal $\fra$ (and the powers of $\fra$). 

\begin{definition}[\cite{MaSchwedePerfectoidTestideal,SatoTakagiArithmeticAndGeometricDeformationsOfFPure,RobinsonBCMTestIdealsMixedCharToric,MurayamaUniformBoundsOnSymbolicPowers}, \cf \cite{HaraYoshidaGeneralizationOfTightClosure}]
\label{def.TauForIdealEltWise}
Define $\tau^{\elt}_{B}(\omega_R, \fra^t)$ to be the sum 
\[
    \sum \tau_B(\omega_R, f^s) 
\]
where the sum runs over $f \in \fra^n$ and rational numbers $s \geq t/n$ for various integers $n$. We similarly define 
$$\tau^{\elt}_{\mathcal{B}}(\omega_R, \fra^t) = \sum \tau_{\mathcal{B}}(\omega_R, f^s),$$
this is equal to $\tau^{\elt}_{B}(\omega_R, \fra^t)$ for all $B$ sufficiently large by \cite[Proposition 6.4]{MaSchwedeSingularitiesMixedCharBCM}.

More generally, for mixed test ideals, we define $\tau_B^{\elt}(\omega_R, \fra_1^{t_1} \cdots \fra_m^{t_m})$ to be 
\[
   \sum \tau_B(\omega_R, f_1^{s_1} \cdots f_m^{s_m}) 
\]
where the sum runs over $f_i \in \fra_i^{n_i}$ and rational numbers $s_i \geq t_i/n_i$ for various integers $n_i > 0$. We similarly have $\tau_{\mathcal{B}}^{\elt}(\omega_R, \fra_1^{t_1} \cdots \fra_m^{t_m})$. The superscript ``$\elt$'' is shorthand for \emph{element-wise}.

Finally, if $\Delta \geq 0$ is a $\bQ$-divisor such that $K_R + \Delta = {1 \over l} \Div(h)$ is $\bQ$-Cartier (for some $h \in R$), then we define 
\[
    \tau_B^{\elt}(R, \Delta, \fra^t) := \tau_B(\omega_R, h^{1/l} \fra^t) 
\]
and likewise with $\tau_B(R, \Delta, \fra_1^{t_1} \cdots \fra_m^{t_m})$ and $\tau_{\mathcal{B}}(R, \Delta, \fra_1^{t_1} \cdots \fra_m^{t_m})$.
\end{definition}

This is not exactly the definition as given in the above sources but it is easily seen to coincide up to perturbing the coefficients.  The one given in the above sources is typically phrased in terms of local cohomology.

Our next definition is obtained by choosing a set of generators of $\fra = (g_{\lambda})$.  

\begin{definition}
[\cite{MaSchwedePerfectoidTestideal,SatoTakagiArithmeticAndGeometricDeformationsOfFPure,RobinsonBCMTestIdealsMixedCharToric,MurayamaUniformBoundsOnSymbolicPowers}] 
\label{def.ElementWiseTestElementCompleteLocal}
Define $\tau^{\elt}_{B}(\omega_R, [g_\lambda]^t)$ to be the sum 
\[
    \sum \tau_B\Big(\omega_R, \prod g_{\lambda}^{s_{\lambda}}\big)
\]
where the sum runs over finite formal products $\prod g_{\lambda}^{s_{\lambda}}$ such that the (finitely many nonzero) $s_{\lambda}$ are nonnegative rational numbers with $\sum_{\lambda} s_{\lambda} \geq t$.  We further define 
\[
\tau_{\mathcal{B}}^{\elt}(\omega_R, [g_\lambda]^t) := \sum \tau_{\mathcal{B}}\Big(\omega_R, \prod g_{\lambda}^{s_{\lambda}}\big),
\]
this is equal to $\tau^{\elt}_{B}(\omega_R, [g_\lambda]^t)$ for all $B$ sufficiently large again by \cite[Proposition 6.4]{MaSchwedeSingularitiesMixedCharBCM}. Note that for sets $\{ g_\lambda \}$ with only one element, this clearly coincides with our earlier definitions of $\tau_B(\omega_R, g^{t})$ and $\tau_{\mathcal{B}}(\omega_R, g^{t})$.

More generally, if $\Gamma = s \Div g \geq 0$ is $\bQ$-Cartier, we define $\tau_B^{\elt}(\omega_R,\Gamma, [g_{1, \lambda}]^{t_1} \cdots [g_{m, \lambda}]^{t_m})$ to be 
\[
    \sum \tau_B\Big(\omega_R, g^s (\prod g_{1,\lambda}^{s_{1,\lambda}}) \cdots (\prod g_{m,\lambda}^{s_{m,\lambda}}) \Big)
\]
where again we run over finite formal products for finite sets of nonnegative rational numbers (finitely many nonzero) $s_{i,\lambda} \geq 0$ with $\sum_{\lambda} s_{i,\lambda} \geq t_i$. We similarly define the test modules $\tau_{\mathcal{B}}^{\elt}(\omega_R, [g_{1, \lambda}]^{t_1} \cdots [g_{m, \lambda}]^{t_m})$. If $\Delta$ is such that $K_R + \Delta$ is $\bQ$-Cartier, we can also define $\tau^{\elt}_B(R, \Delta,  [g_{1, \lambda}]^{t_1} \cdots [g_{m, \lambda}]^{t_m})$ and $\tau_{\mathcal{B}}^{\elt}(R, \Delta, [g_{1, \lambda}]^{t_1} \cdots [g_{m, \lambda}]^{t_m})$ in the way analogous to \autoref{def.TauForIdealEltWise}.  Non-effective $\Gamma$ or $\Delta$ can be handled by adding an effective Cartier divisor and then untwisting to obtain a fractional ideal (or fractional submodule of $\omega_R$).
\end{definition}

Finally, there is a third notion which appeared in \cite{HaconLamarcheSchwede}.  We first recall some notation.

Suppose $Y$ is a normal integral scheme proper over a complete Noetherian ring $R$ and $M$ is a Cartier divisor on $Y$.  We set 
\[
    \myB^0(Y, tM; \omega_Y) = \bigcap_{Z \to Y} \Image\Big(H^0(Z, \cO_Z(K_Z - tG)) \xrightarrow{\Tr} H^0(Y, \cO_Y(K_Y))\Big)
\]
where the intersection runs over either alterations $Z \to Y$ or equivalently finite surjective maps.  

\begin{definition}
    \label{def.TauBlup}
    Let $\pi : Y \to X\, {=\Spec(R)}$ be a projective birational map from a normal $Y$ factoring through the blowup of $\fra$ (for instance, we can let $Y$ denote the normalized blowup of $\fra$).  
    Write $\cO_Y(-M) = \fra \cdot \cO_Y$ for some Cartier divisor $M$.  We define $\tau_{\blup}(\omega_R, \fra^t)$ to be 
    \[
        \myB^0(Y, tM; \omega_Y) \subseteq H^0(Y, \omega_Y) \subseteq \omega_R.
    \]
    This is independent of the choice of $Y$ by \cite[Lemma 4.19]{BMPSTWW1}.

    More generally, for mixed test ideals, if $f : Y \to \Spec R$ as above factors through all the blowups of the $\fra_i$ and we write $\fra_i \cO_Y = \cO_Y(-M_i)$, and $\Gamma$ is $\bQ$-Cartier then we define  $\tau_{\blup}(\omega_R, \Gamma, \fra_1^{t_1} \cdots \fra_m^{t_m})$ to be  
    \[
        \myB^0(Y, f^* \Gamma + t_1 M_1 + \dots + t_m M_m; \omega_Y) \subseteq H^0(Y, \omega_Y) \subseteq \omega_R.
    \]

    Finally, if $\Delta$ is such that $K_R + \Delta$ is $\bQ$-Cartier, then we can define
    \[
        \tau_{\blup}(R, \Delta, \fra^t) :=  \myB^0(Y, t M + \pi^*(K_X + \Delta); \omega_Y) \subseteq \omega_R.
    \]
    The ideal $\tau_{\blup}(R, \Delta, \fra_1^{t_1} \cdots \fra_m^{t_m})$ is defined similarly.
\end{definition}    

{\color{black}One advantage of this definition is that it is immediately clear that 
    \begin{equation}
        \label{eq.VariousUnambiguities}
        \tau_{\myB^0}(\omega_R, \underbrace{\fra^{t} \cdots \fra^{t}}_{\text{$n$-times}}\frb^s) = \tau_{\myB^0}(\omega_R, {(\fra^{n}})^t\frb^s) =         \tau_{\myB^0}(\omega_R, {\fra^{nt}}\frb^s).
    \end{equation}
    More generally, we have the following observation.
}

\begin{remark}
    \label{rem.SuperficialMixedIdealStrategy}
    The generalization to $\tau_{\myB^0}(\omega_R, \Gamma, \fra_1^{t_1} \cdots \fra_m^{t_m}) = \tau_{\myB^0}(\omega_R, (g)^s \fra_1^{t_1} \cdots \fra_m^{t_m})$ is superficial.  Indeed, suppose for each $i$ we write $s = a/b, t_i = a_i/b$ for integers $a_i$ and a common denominator $b$, it is straightforward to see that  
    \[
        \tau_{\myB^0}(\omega_R, \fra_1^{t_1} \cdots \fra_m^{t_m}) = \tau_{\myB^0}\big(\omega_R, \fra^{1/b}\big)
    \]
    where $\fra = g^a \fra_1^{t_1 b} \cdots \fra_m^{t_m b}$ (an honest ideal).  
\end{remark}

All of these definitions have various advantages.  For instance $\tau(\omega_S, [g_{\lambda}]^t)$ is easily seen to satisfy summation-type theorems, and subadditivity for this variant can be found in \cite{MurayamaUniformBoundsOnSymbolicPowers} or \cite{MaSchwedePerfectoidTestideal}, \cf \cite{TakagiFormulasForMultiplierIdeals} when $R$ is regular.  On the other hand, if $R$ is regular or even BCM-regular in the sense of \cite{MaSchwedeSingularitiesMixedCharBCM}, it is easy to see that 
\[
    \fra^{\lceil t \rceil} = (g_{\lambda})^{\lceil t \rceil} \subseteq \tau(R, [g_{\lambda}]^t).
\]
We explicitly highlight some key properties of $\tau_{\blup}(\omega_R, \fra^t)$ as several will be useful shortly.

\begin{proposition} \label{prop:easy-properties-for-tau-B0-local}
    With notation as in \autoref{def.TauBlup}:
    \begin{enumerate}
        \item If $t \leq  t'$ then $\tau_{\blup}(\omega_R, \fra^t) \supseteq \tau_{\blup}(\omega_R, \fra^{t'})$.\label{prop:easy-properties-for-tau-B0-local.exponentcomparison}
        \item If $\fra \subseteq \frb$ then $\tau_{\blup}(\omega_R, \fra^t) \subseteq \tau_{\blup}(\omega_R, \frb^t)$.\label{prop:easy-properties-for-tau-B0-local.idealcontainment}
        \item $\tau_{\blup}(\omega_R, \fra^t) = \tau_{\blup}(\omega_R, \overline{\fra}^t)$, where $\overline{\fra}$ is the integral closure (likewise for mixed and ideal variants).  \label{prop:easy-properties-for-tau-B0-local.integralClosure}
        \item \label{prop:easy-properties-for-tau-B0-local.trace-inclusion} If $\pi : Z \to \Spec R$ is any normal alteration with $\fra \cdot \cO_Z= \cO_Z(-G)$ and such that $tG$ is Cartier, then 
        \[
            \tau(\omega_R, \fra^t) \subseteq \Tr\big(H^0(Z, \cO_Z(K_Z -tG))\big).
        \]
    \end{enumerate}
\end{proposition}
\begin{proof}
    Properties \autoref{prop:easy-properties-for-tau-B0-local.exponentcomparison} and \autoref{prop:easy-properties-for-tau-B0-local.idealcontainment} follow from \cite[Lemma 4.7]{BMPSTWW1}.  \autoref{prop:easy-properties-for-tau-B0-local.integralClosure} follows since $\fra \cO_Y = \overline{\fra} \cO_Y = \cO_Y(-M)$.  \autoref{prop:easy-properties-for-tau-B0-local.trace-inclusion} follows since 
    \[ 
        \tau_{\myB^0}(\omega_R, \fra^t) = \myB^0(Y, tM; \omega_Y) = \bigcap_{Z \to Y} \Image\Big(H^0(Z, \cO_Z(K_Z - tG)) \xrightarrow{\Tr} H^0(Y, \cO_Y(K_Y))\Big)
    \]
    where the intersection runs over normal alterations $Z \to Y$.  
\end{proof}


Another particularly important property for $\tau_{\blup}(\omega_R, \fra_1^{t_1} \cdots \fra_m^{t_m})$ is Skoda's theorem.  We give a complete proof here since it is proved in \cite{HaconLamarcheSchwede} only for the case of a single ideal $\fra$.

\begin{theorem}
    \label{thm.SkodaForLocalCase}
    Suppose $(R, \fram)$ is a complete Noetherian reduced local domain with a positive characteristic residue field.  Suppose $\fra_1, \dots, \fra_m \subseteq R$ are ideals and $t_1, \dots, t_m \geq 0$ are rational numbers.  Suppose further that $t_1$ is greater than or equal to the number of generators of $\fra_1$, then 
    \[
        \tau_{\blup}(\omega_R, \fra_1^{t_1} \cdots \fra_m^{t_m}) = \fra_1 \cdot \tau_{\blup}(\omega_R, \fra_1^{t_1-1} \cdot \fra_2^{t_2} \cdots \fra_m^{t_m}).
    \]
    When $R$ is normal and $K_R + \Delta$ is $\bQ$-Cartier, the analogous result holds for the test ideal $\tau_{\blup}(R, \Delta, \fra_1^{t_1} \cdots \fra_m^{t_m})$.
\end{theorem}
\begin{proof}
    For ease of notation, we show that 
    \[
        \tau_{\blup}(\omega_R, \fra^{t} \frb^s) = \fra \cdot \tau_{\blup}(\omega_R, \fra^{t-1} \frb^s)
    \]
    for $t \geq \dim R$ as we can reduce to this case via the observation of \autoref{rem.SuperficialMixedIdealStrategy} (\cf \autoref{rem:candoskodawhenbiggerthandim}).  Set $X := \Spec R$. 
    
    Fix a normal birational map $\pi : Y \to X$ such that $\fra \cO_Y = \cO_Y(-M)$ and $\frb \cO_Y(-N)$ are line bundles. The effective divisors $M$ nad $N$ are fixed throughout the proof.  We now follow the argument of \cite{HaconLamarcheSchwede}, which itself is a modification of arguments from \cite{SchwedeTuckerNonPrincipal,LazarsfeldPositivity2,EinLazarsfeldGlobalGeneration}.  The point is we use Bhatt's vanishing for $\pi_+ : Y^+ \to X$ (where $Y^+$ is the absolute integral closure of $Y$) instead of Kawamata-Viehweg vanishing on a resolution in characteristic zero, or Frobenius plus Serre vanishing in characteristic $p > 0$.

    Pick generators $s_1, \dots, s_n \in \fra$ which then also generate $\cO_Y(-M)$.  Set $V = R^{\oplus {n \choose i}}$ to be the free $R$-module on the $s_i$.  We have an exact Skoda/Koszul complex on $Y$
    \[
        0 \to \sF_{n} \to \sF_{n-1} \to \dots \to \sF_1 \to \sF_0 \to 0
    \]
    where $\sF_i = \wedge^i V \otimes \cO_Y(iM)$.  We tensor this complex with $\cO_{Y^+}$ and note it stays exact since it is locally a resolution of a free module.  Since $-tM$ and $-tN$ pull back to Cartier divisors on some finite cover $\kappa : W \to Y$, we write $\cO_{Y^+}(-tM - sN) := \mu^* \cO_W(-\kappa^* tM - \kappa^* sN)$ to denote the pullback of the corresponding sheaf via $Y^+ \xrightarrow{\mu} W$.  We twist by $\cO_{Y^+}(-tM - sN)$ and apply $\sHom_{\cO_{Y^+}}(-, \cO_{Y^+})$ to obtain the following exact complex:
    \[
        0 \leftarrow \sG_{n} \leftarrow \sG_{n-1} \leftarrow \dots \leftarrow \sG_1 \leftarrow \sG_0 \leftarrow 0
    \]
    where 
    \[
        \sG_i := \wedge^i V \otimes \cO_{Y^+}( (t-i)M + s N)
    \]
    Note $-\kappa^* ((t-i)M + sN)$ is big and semi-ample for $i = 0, \dots, n$.  Hence 
    \[
        0 = \myH^j \myR\Gamma_{\fram} \myR\Gamma(Y^+, \sG_i) = H^j_{\fram}(\myR (\pi_+)_* \sG_i) =: H^j_{\fram}(Y^+, \sG_i)
    \]
    for $j \neq d= \dim R$ by \cite{BhattAbsoluteIntegralClosure} in the form of \cite[Corollary 3.7]{BMPSTWW1} (the right term is a slight abuse of notation). 
    It follows that  
    \[
        0 \leftarrow H^d_{\fram}(Y^+, \sG_{n}) \leftarrow \dots \leftarrow H^d_{\fram}(Y^+,\sG_{1})  \leftarrow H^d_{\fram}(Y^+,\sG_0) \leftarrow 0
    \]
    is also exact and most crucially $H^d_{\fram}(Y^+,\sG_0) \to H^d_{\fram}(Y^+,\sG_1)$ injects.

    Since $M, N$ are effective, we have canonical maps $\cO_Y \to \sG_0$ and $V \otimes \cO_Y \to \sG_1$.   
    Consider the following diagram:
    \begin{equation}
        \label{eq.KeySkodaSquare}
        \xymatrix{
            V \otimes H^d_{\fram}(Y, \cO_Y) \ar@{->>}[d]  &\ar[l] H^d_{\fram}(Y, \cO_Y) \ar@{->>}[d]  \\
            \mathrm{im}_2 \ar@{^{(}->}[d]  & \ar@{_{(}->}[l] \mathrm{im}_1 \ar@{^{(}->}[d]   \\
            H^d_{\fram}(Y^+,\sG_1)  & \ar@{_{(}->}[l] H^d_{\fram}(Y^+,\sG_0)  
        }
    \end{equation}
    where the middle row consists of images of the canonical maps. 
    Let $E$ denote an injective hull of the residue field of $R$, then by local and Grothendieck duality, and \autoref{eq.B0DefinitionViaLocalCohom}, 
    \[ 
        \Hom_R(\mathrm{im}_1, E) = \myB^0(Y, tM+sN; \omega_Y) = \tau_{\myB^0}(\omega_R, \fra^t \frb^s)
    \] 
    and likewise $\Hom_R(\mathrm{im}_2, E) = V \otimes \tau_{\myB^0}(\omega_R, \fra^{t-1} \frb^s)$.  Hence the Matlis dual of the middle horizontal map is the map 
    \[
        V \otimes  \tau_{\myB^0}(\omega_R, \fra^{t-1} \frb^s) \to \tau_{\myB^0}(\omega_R, \fra^{t} \frb^s)
    \]
    which is thus surjective.  But $V$ is the free vector space on the $s_i$ and so we have just shown that  
    \[
        \tau_{\myB^0}(\omega_R, \fra^{t} \frb^s) = \fra \cdot \tau_{\myB^0}(\omega_R, \fra^{t-1} \frb^s)
    \]
    as desired.  

    The ideal case may be reduced to the test module case since if $K_R + \Delta = {1 \over n} \Div(h)$ then
    \[
        \tau_{\myB^0}(R, \Delta, \fra_1^{t_1} \cdots \fra_m^{t_m}) = \tau_{\myB^0}(\omega_R, \fra_1^{t_1} \cdots \fra_m^{t_m} (h)^{1/n}).
    \]
    The proof is complete.
\end{proof}

\begin{remark}\label{rem:candoskodawhenbiggerthandim}
    By passing to a finite \'etale extension of $R$ and replacing $\fra_1$ with a reduction (a smaller ideal with the same integral closure), we may always assume that $\fra_1$ is generated by $\leq \dim R$ elements, see \cite[Proposition 3.10]{HaconLamarcheSchwede}.  Hence, instead of assuming that $t_1 \geq \text{(number of generators of $\fra_1$)}$ we may alternately assume that $t_1 \geq \dim R$ in the statement of \autoref{thm.SkodaForLocalCase}.  
\end{remark}

We include a direct and alternative proof of the subadditivity property for $\tau_{\myB^0}$ which mimics the strategy of \cite[Proof of Theorem 4.3]{IyengarMaSchwedeWalkerMCM}. Note that when $(R,\m)$ is complete and regular (or more generally complete and Gorenstein), $\omega_R=R$ and we will write $\tau_{\myB^0}(R,\fra^t)$ for $\tau_{\myB^0}(\omega_R,\fra^t)$. 

\begin{theorem}
    \label{thm:BlowUpSubadditivity}
    Let $(R,\m)$ be a complete regular local ring of residue characteristic $p>0$. Let $\fra$ and $\frb$ are two ideals. Then for all rational numbers $s, t\geq 0$, we have 
    $$\tau_{\myB^0}(R, \fra^s\frb^t)\subseteq \tau_{\myB^0}(R, \fra^s)\cdot \tau_{\myB^0}(R,\frb^t).$$
    \end{theorem}
    \begin{proof}
    Let $\pi \colon Y \to \Spec(R)$ be a projective birational map with $Y$ normal that factors through the blowup of $\fra\frb$. Let $\fra\cO_Y=\cO_Y(-G)$ and $\frb\cO_Y=\cO_Y(-H)$. Note that these are both big and semi-ample on $Y$. Let $Y^+$ denote the absolute integral closure of $Y$. We will use the following notations ($d=\dim(R)$):
    $$0^{\fra^s}_{H_\m^d(R)} = \ker\big(H_\m^d(R)\to H_\m^d(\myR\pi_*\cO_{Y^+}(sG))\big)$$
    $$0^{\frb^t}_{H_\m^d(R)} = \ker\big(H_\m^d(R)\to H_\m^d(\myR\pi_*\cO_{Y^+}(tH))\big)$$ 
    $$0^{\fra^s\frb^t}_{H_\m^d(R)} = \ker\big(H_\m^d(R)\to H_\m^d(\myR\pi_*\cO_{Y^+}(sG+tH))\big.$$
    We know from \cite{HaconLamarcheSchwede} and \cite{BMPSTWW1} that $\tau_{\myB^0}(R,\fra^t)=\Image\big(H_\m^d(R)\to H_\m^d(\myR\pi_*\cO_Y(sG))\big)^\vee$. Since $H_\m^d(R)=E$, the injective hull of the residue field of $R$, by Matlis duality we know that 
    $$\tau_{\myB^0}(R,\fra^s)= \Ann_R0^{\fra^s}_{H_\m^d(R)} \text{ and } 0^{\fra^s}_{H_\m^d(R)}=0:_E\tau_{\myB^0}(R,\fra^s),$$ 
    and similarly 
    $$\tau_{\myB^0}(R,\frb^t)= \Ann_R0^{\frb^t}_{H_\m^d(R)} \text{ and } 0^{\frb^t}_{H_\m^d(R)}=0:_E\tau_{\myB^0}(R,\frb^t),$$ 
    $$\tau_{\myB^0}(R,\fra^s\frb^t)= \Ann_R0^{\fra^s\frb^t}_{H_\m^d(R)} \text{ and } 0^{\fra^s\frb^t}_{H_\m^d(R)}=0:_E\tau_{\myB^0}(R,\fra^s\frb^t),$$ 
    
    With notation as above, it is enough to show that $0^{\fra^s}_{H_\m^d(R)}: \tau_{\myB^0}(R, \frb^t) \subseteq 0^{\fra^s\frb^t}_{H_\m^d(R)}$. Because if this is true, then since $\tau_{\myB^0}(R,\fra^s\frb^t) = \Ann_R0^{\fra^s\frb^t}_{H_\m^d(R)}$, it follows that 
    $$\tau_{\myB^0}(R,\fra^s\frb^t) \cdot (0^{\fra^s}_{H_\m^d(R)}: \tau_{\myB^0}(R, \frb^t))=0$$
    and thus 
    \begin{align*}
    \tau_{\myB^0}(R,\fra^s\frb^t) &\subseteq \Ann_R(0^{\fra^s}_{H_\m^d(R)}: \tau_{\myB^0}(R, \frb^t)) \\
    & = \Ann_R\big((0:_E\tau_{\myB^0}(R,\fra^s)): \tau_{\myB^0}(R, \frb^t)\big)\\
    & = \Ann_R\big( 0:_E \tau_{\myB^0}(R, \fra^s)\tau_{\myB^0}(R,\frb^t) \big) \\
    & = \tau_{\myB^0}(R, \fra^s)\cdot \tau_{\myB^0}(R,\frb^t)
    \end{align*}
    
    Now we prove $0^{\fra^s}_{H_\m^d(R)}: \tau_{\myB^0}(R, \frb^t) \subseteq 0^{\fra^s\frb^t}_{H_\m^d(R)}$. Suppose $\eta\in H_\m^d(R)$ is such that $\tau_{\myB^0}(R, \frb^t)\eta \in 0^{\fra^s}_{H_\m^d(R)}$. This means the image of $\tau_{\myB^0}(R, \frb^t)\eta$ vanish in $H_\m^d(\myR\pi_*\cO_{Y^+}(sG))$, i.e., we have 
    \begin{align*}
    \text{Image of $\eta$ in $H_\m^d(\myR\pi_*\cO_{Y^+}(sG))$} & \in \Hom\big(R/\tau_{\myB^0}(R, \frb^t), {H_\m^d(\myR\pi_*\cO_{Y^+}(sG))}\big) \\
    & \cong H^0\Big(\myR\Hom\big(R/\tau_{\myB^0}(R, \frb^t), {H_\m^d(\myR\pi_*\cO_{Y^+}(sG))}\big) \Big) \\
    & \cong H^0\Big(\myR\Hom\big(R/\tau_{\myB^0}(R, \frb^t), \myR\pi_*\cO_{Y^+}(sG) \otimes_R^{\mathbf{L}} E\big) \Big) \\
    & \cong H^0\Big(\myR\Hom\big(R/\tau_{\myB^0}(R, \frb^t), E\big) \otimes_R^{\mathbf{L}} \myR\pi_*\cO_{Y^+}(sG)  \Big) \\
    & \cong H^0\Big(\Ann_E\tau_{\myB^0}(R, \frb^t) \otimes_R^{\mathbf{L}} \myR\pi_*\cO_{Y^+}(sG) \Big) 
    \end{align*}
    where the isomorphism on the third line follows from Bhatt's vanishing \cite[Theorem 6.28]{BhattAbsoluteIntegralClosure} (see \cite[Corollary 3.7]{BMPSTWW1}, which implies that $H_\m^d(\myR\pi_*\cO_{Y^+}(sG))$ is the first non-vanishing cohomology of $\myR\pi_*\cO_{Y^+}(sG) \otimes_R^{\mathbf{L}} E$, and it sits in degree $0$) and the isomorphism on the fourth line follows from \cite[Proposition 1.1 (4)]{FoxbyIsomorphismsComplexes} (here we used $R$ is regular so that $\myR\pi_*\cO_{Y^+}(sG)$ is isomorphic to a bounded complex of flat modules). Note that if we tensor the following diagram 
    \[
    \xymatrix{
    0 \ar[r] & 0^{\frb^t}_{H_\m^d(R)} \ar[r] \ar@{=}[d] &  H_\m^d(R)  \ar[r] \ar@{=}[d] & H_\m^d(\myR\pi_*\cO_{Y^+}(tH)) \ar[d] \\
    0 \ar[r] & \Ann_E\tau_{\myB^0}(R, \frb^t)  \ar[r] & E \ar[r] &  E \otimes_R^{\mathbf{L}} \myR\pi_*\cO_{Y^+}(tH)
    }
    \]
    with $\myR\pi_*\cO_{Y^+}(sG)$ (the right vertical map exists again by Bhatt's vanishing explained above), then we obtain that 
    \begin{align*}
    \Ann_E\tau_{\myB^0}(R, \frb^t) \otimes_R^{\mathbf{L}} \myR\pi_*\cO_{Y^+}(sG) & \to E \otimes_R^{\mathbf{L}} \myR\pi_*\cO_{Y^+}(tH) \otimes_R^{\mathbf{L}} \myR\pi_*\cO_{Y^+}(sG) \\
    & \to E \otimes_R^{\mathbf{L}} \myR\pi_*\cO_{Y^+}(sG+tH).
    \end{align*}
    is the zero map (the second map exists by \cite[Tag 0FKU]{stacks-project}). Therefore, we have $\Image(\eta)\in H^0\big(\Ann_E\tau_{\myB^0}(R, \frb^t) \otimes_R^{\mathbf{L}} \myR\pi_*\cO_{Y^+}(sG) \big)$ maps to zero in $H^0(E \otimes_R^{\mathbf{L}} \myR\pi_*\cO_{Y^+}(sG+tH))\cong  H_\m^d(\myR\pi_*\cO_{Y^+}(sG+tH))$, that is, $\eta\in 0^{\fra^s\frb^t}_{H_\m^d(R)}$ as wanted. 
    \end{proof}

We now return to the problem of comparing our various notions of $\tau$.  We first note that we have the following containments:
\[
    \tau_B^{\elt}(\omega_R, [g_\lambda]^t) \subseteq \tau_B^{\elt}(\omega_R, \fra^t) \subseteq \tau_{\widehat{R^+}}^{\elt}(\omega_R, \fra^t) \subseteq \tau_{\blup}(\omega_R, \fra^t),
\]
where the last containment follows from \cite[Remark 6.4]{HaconLamarcheSchwede}. The analogous containments also hold for mixed test modules $\tau(\omega_R, \fra_1^{t_1} \cdots \fra_{m}^{t_m})$ as well as the test ideal variations.

Our goal is to show these containments are equalities after slight adjustments (assuming the $g_{\lambda}$ generate $\fra$).  The key argument follows, which should be viewed as a variant of the summation formula for multiplier ideals \cite{MustataMultiplierIdealOfASum,TakagiFormulasForMultiplierIdeals,JowMillerMultiplierIdealsOfASum}.

\begin{proposition}
    \label{prop.TauBlupVsTauElts}
    Suppose that $(g_{\lambda} \mid \lambda \in \Lambda) = \fra$.  Then for $1 \gg \epsilon > 0$ we have that 
    \[
        \tau_{\blup}(\omega_R, \fra^{t+\epsilon}) = \tau_{\widehat{R^+}}^{\elt}(\omega_R, [g_\lambda]^{t+\epsilon}).
    \]
    The analogous statement also holds for mixed test modules and mixed test ideals.
\end{proposition}
\begin{proof}
    Without loss of generality we may assume that $\{g_{\lambda} \mid \lambda \in \Lambda \} = \{g_1, \dots, g_v\}$ is finite and $t$ is rational.  Fix $\epsilon' > 0$ such that $\tau_{\blup}(\omega_R, \fra^{t+\epsilon})$ is constant for all $0 < \epsilon \leq \epsilon'$ and such that $\tau_{\widehat{R^+}}^{\elt}(\omega_R, [g_\lambda]^{t+\epsilon})$ is also constant for such $\epsilon$.  Such an $\epsilon'$ exists by Noetherianity.  

    Let $S \supseteq R$ be a finite normal extension such that $\{g_{\lambda}^{1/N}\} \subseteq S$ for $N$ so that $N\epsilon' \geq 2v$ which implies $N(t + \epsilon') - v \geq N(t + \epsilon'/2)$.  We may also assume that $Nt, N\epsilon'/2 \in \mathbf{Z}$.  
    Consider the ideal $\fra_N = (g_1^{1/N}, \dots, g_v^{1/N}) \subseteq S$.  It follows easily that $\overline{\fra_N^N} = \overline{\fra S}$.  Hence 
    \[
        \begin{array}{cll}
             &     \tau_{\blup}(\omega_R, \fra^{t+\epsilon'})\\
             = &  \Tr(\tau_{\blup}(\omega_S, (\fra S)^{t+\epsilon'})) & \text{($\myB^0$ under finite maps, \cite[Lemma 4.18]{BMPSTWW1}) }\\
            = &  \Tr(\tau_{\blup}(\omega_S, (\fra_N)^{N(t+\epsilon')})) & \text{(integral closure agnostic $\tau_{\blup}$)} \\
             = &  \Tr((\fra_N)^{N(t+\epsilon') - v}\tau_{\blup}(\omega_S, (\fra_N)^{v} )) & \text{(Skoda theorem, \cite[Theorem 6.6]{HaconLamarcheSchwede})}\\
             \subseteq & \Tr((\fra_N)^{N(t+\epsilon'/2)} \tau_{\widehat{R^+}}(\omega_S)) & \text{($\tau_{\blup}(\omega_S, (\fra_N)^{v} ) \subseteq \tau_{\blup}(\omega_S) = \tau_{\widehat{R^+}}(\omega_S)$)}\\
             = & \Tr\Big(\sum g_1^{a_1/N} \cdots g_v^{a_v/N}\tau_{\widehat{R^+}}(\omega_S)\Big) & \text{($\sum a_i = N(t+\epsilon'/2)$, $a_i \in \mathbf{Z}_{\geq 0}$)}\\
             = & \sum \tau_{\widehat{R^+}}(\omega_R, g_1^{a_1/N} \cdots g_v^{a_v/N}) & \text{($\myB^0$ under finite maps, \cite[Lemma 4.18]{BMPSTWW1})}\\
             \subseteq & \tau_{\widehat{R^+}}(\omega_R, [g_\lambda]^{t+\epsilon'/2}).
        \end{array}
    \]
    The result follows.  The statement for mixed test modules and mixed test ideals follows similarly.
\end{proof}

{\color{black}
\begin{theorem}
    \label{thm.CommonTestIdealNonPrincipal}
    Suppose $(R, \fram)$ is a Noetherian complete local domain of residue characteristic $p > 0$ and $\fra \subseteq R$ is an ideal. Pick a finite generating set $(g_{\lambda}) = \fra$.
    Suppose that $A \subseteq R$ is a Noether-Cohen normalization (i.e., $A\to R$ is a finite extension such that $A$ is complete and regular).  Choose $h \in A$ such that $A[1/h] \subseteq R[1/h]$ is \'etale, let $g = \prod_{\lambda} g_{\lambda}$, and set $d_0 = h g$.  Then for every $d \in R$ divisible by $d_0$ and 
    for $1 \gg \epsilon > 0$ (depending on $d$) the following ideals are equal.
    \begin{enumerate}
        \item $\tau_{\cB}^{\elt}(\omega_R, [g_\lambda]^{t})$\label{thm.CommonTestIdealNonPrincipal.BCMGensNoEpsilon}
        \item $\tau_{\cB}^{\elt}(\omega_R, \fra^{t})$\label{thm.CommonTestIdealNonPrincipal.BCMIdealNoEpsilon}
        \item $\tau_{\cB}^{\elt}(\omega_R, d^{\epsilon} [g_\lambda]^{t + \epsilon})$\label{thm.CommonTestIdealNonPrincipal.BCMGensWithEpsilon}
        \item $\tau_{\cB}^{\elt}(\omega_R, d^{\epsilon} \fra^{t + \epsilon})$\label{thm.CommonTestIdealNonPrincipal.BCMIdealWithEpsilon}
        \item $\tau_{\widehat{R^+}}^{\elt}(\omega_R, d^{\epsilon} [g_\lambda]^{t + \epsilon})$\label{thm.CommonTestIdealNonPrincipal.PerturbedGens}
        \item $\tau_{\widehat{R^+}}^{\elt}(\omega_R, d^{\epsilon} \fra^{t+\epsilon})$\label{thm.CommonTestIdealNonPrincipal.PerturbedIdeal}
        \item $\tau_{\blup}(\omega_R, d^{\epsilon} \fra^{t + \epsilon})$\label{thm.CommonTestIdealNonPrincipal.Blup}
    \end{enumerate} 
    The analogous result also holds for mixed test modules.  
\end{theorem}
\begin{proof}
First of all, we notice that \autoref{thm.CommonTestIdealNonPrincipal.BCMGensNoEpsilon}=\autoref{thm.CommonTestIdealNonPrincipal.BCMGensWithEpsilon} and \autoref{thm.CommonTestIdealNonPrincipal.BCMIdealNoEpsilon}=\autoref{thm.CommonTestIdealNonPrincipal.BCMIdealWithEpsilon} for all $d$ and $0<\epsilon\ll1$, because for all $0\neq g\in R$, $\tau_{\mathcal{B}}(\omega_R, f^s)= \tau_{\mathcal{B}}(\omega_R, f^sg^\epsilon)$ for all $0<\epsilon\ll 1$ by \cite[Proposition 6.4]{MaSchwedeSingularitiesMixedCharBCM} (by Noetherianity these test ideals are equal to a finite sum of test ideals for principal divisors, and so we only need to apply op.cit.\ a finite number of times). Also, for any $d$ and $\epsilon$, we clearly have \autoref{thm.CommonTestIdealNonPrincipal.PerturbedGens}$\subseteq$\autoref{thm.CommonTestIdealNonPrincipal.PerturbedIdeal}$\subseteq$\autoref{thm.CommonTestIdealNonPrincipal.Blup} and \autoref{prop.TauBlupVsTauElts} implies that they are equal for $1 \gg \epsilon > 0$ when applied to the mixed pair $(\omega_R, d^0 \fra^t)$. Also, note that we have  \autoref{thm.CommonTestIdealNonPrincipal.BCMGensWithEpsilon}$\subseteq$
\autoref{thm.CommonTestIdealNonPrincipal.BCMIdealWithEpsilon}$\subseteq$\autoref{thm.CommonTestIdealNonPrincipal.PerturbedIdeal}=\autoref{thm.CommonTestIdealNonPrincipal.PerturbedGens} for all $d$ and $0<\epsilon\ll1$, where the last equality is what we have already shown. Therefore, it only remains to show that \autoref{thm.CommonTestIdealNonPrincipal.PerturbedGens}$\subseteq$\autoref{thm.CommonTestIdealNonPrincipal.BCMGensWithEpsilon}.


We fix a $d = d'g h$, a multiple of $d_0$ and consider 
\[\tau_{\cB}^{\elt}(\omega_R, d^{\alpha} [g_\lambda]^{t + \beta}) \text{ and } \tau_{\widehat{R^+}}^{\elt}(\omega_R, d^{\alpha} [g_\lambda]^{t + \beta})\]
for various $\alpha, \beta>0$. Note that there exists $\epsilon_0$ such that both these families are constant for all $0<\alpha, \beta \leq \epsilon_0$ (by Noetherianity). Fix this $\epsilon_0$ and fix $\epsilon_1<\epsilon_0$. We can write 
\[ 
    \tau_{\widehat{R^+}}^{\elt}(\omega_R, d^{\epsilon_0} [g_\lambda]^{t + \epsilon_0}) = \sum_j \tau_{\widehat{R^+}}\big(\omega_R, d^{\epsilon_0}\prod_{\lambda} g_{\lambda}^{s_{j,\lambda}}\big)
\]
where the right hand side is a finite sum for various finite collections of $\{s_{j,\lambda}\}_{\lambda}$ such that $\sum_{\lambda} s_{j,\lambda}\geq t+ \epsilon_0$. At this point, we examine each single $\tau_{\widehat{R^+}}\big(\omega_R, d^{\epsilon_0}\prod_{\lambda} g_{\lambda}^{s_{j,\lambda}}\big)$. Let $S$ be the normalization of $R[d^{\epsilon_1}\prod_\lambda g_{\lambda}^{s_{j,\lambda}}]$ inside $R^+$. 
Now we have that
\begin{align*}
\tau_{\widehat{R^+}}\big(\omega_R, d^{\epsilon_0}\prod_{\lambda} g_{\lambda}^{s_{\lambda_j}}\big) & = \Tr_{S/R}\Big((d^{\epsilon_1}\prod_{\lambda} g_{\lambda}^{s_{j,\lambda}})\cdot \tau_{\widehat{R^+}}(\omega_S, d^{\epsilon_0-\epsilon_1})\Big) \\
& \subseteq \Tr_{S/R}\Big((d^{\epsilon_1}\prod_\lambda g_{\lambda}^{s_{j,\lambda}})\cdot \tau_{\widehat{R^+}}(\omega_S, (gh)^{\epsilon})\Big) \\
& = \Tr_{S/R}\Big((d^{\epsilon_1}\prod_\lambda g_{\lambda}^{s_{j,\lambda}})\cdot \tau_{\mathcal{B}}(\omega_S)\Big) \\
& =\tau_{\mathcal{B}}\big(\omega_R, d^{\epsilon_1}\prod_\lambda g_{\lambda}^{s_{j,\lambda}}\big)
\end{align*}
for all $\epsilon\ll1$, where the first and last equality follows from \cite[Theorem 6.17]{MaSchwedeSingularitiesMixedCharBCM}, the second to last equality follows from \autoref{clm.Tau+PerturbedEqualsTauBigB} and \autoref{rmk.Tau+PerturbedEqualsTauBigB choosing g and h} since $A[1/h]\to R[1/h]$ is finite \'etale and $R[1/g]\to S[1/g]$ is finite \'etale. 
Finally, taking a sum, we see that 
$$\tau_{\widehat{R^+}}^{\elt}(\omega_R, d^{\epsilon_0} [g_\lambda]^{t + \epsilon_0}) \subseteq \tau_{\cB}^{\elt}(\omega_R, d^{\epsilon_1} [g_\lambda]^{t + \epsilon_0}) = \tau_{\cB}^{\elt}(\omega_R, d^{\epsilon_0} [g_\lambda]^{t + \epsilon_0})$$
where the last equality follows from our choice of $\epsilon_0$. Thus, by our choice of $\epsilon_0$ again, we have \autoref{thm.CommonTestIdealNonPrincipal.PerturbedGens}$\subseteq$\autoref{thm.CommonTestIdealNonPrincipal.BCMGensWithEpsilon} for all $0<\epsilon\leq \epsilon_0$ as we desired. 

 The statement for mixed test modules follows similarly.
\end{proof}
}

\begin{remark}
    \label{rem.ExplicitD0}
    It is worth remarking that the $d_0$ constructed above is quite explicit.  Indeed, following the proof, if $A[h^{-1}] \subseteq R[h^{-1}]$ is \'etale and $h_{\lambda} \in A$ is a multiple of $g_{\lambda}$, where $\{g_{\lambda}\}_\lambda$ is a generating set of $\mathfrak{a}$, then we can take
    \[
        d_0 := h \prod_{\lambda} h_{\lambda}.
    \]
\end{remark}

In view of \autoref{thm.CommonTestIdealNonPrincipal}, we make the following definition.

\begin{definition} \label{def:common-tau-nonprincipal}
    With notation as above if $(g_\lambda \mid \lambda \in \Lambda) =: \fra \subseteq R$, we define 
\[
    \utau(\omega_R, \fra^t)
\]
to be the common submodule of $\omega_R$ from \autoref{thm.CommonTestIdealNonPrincipal}.  The mixed test modules, denoted $\utau(\omega_R, \Gamma, \fra_1^{t_1} \cdots \fra_m^{t_m})$, are defined similarly (or by combining the $\Gamma$ and $\fra_i^{t_i}$ into a single $\fra^t$).  Likewise if $R$ is normal, then
\[ 
    \utau(R, \Delta, \fra^t) := \utau(\omega_R, K_R + \Delta, \fra^t) = \utau(\omega_R, h^{1/n} \fra^t)
\] 
where $\Div(h) = n(K_R + \Delta)$.   The mixed test ideals $\utau(R, \Delta, \fra_1^{t_1} \cdots \fra_m^{t_m})$ are defined analogously.
\end{definition}


Our results imply Takagi's summation formula in mixed characteristic \cite{TakagiFormulasForMultiplierIdeals}, \cf \cite{MustataMultiplierIdealOfASum,JowMillerMultiplierIdealsOfASum}.
\begin{corollary}[\cf {\cite{TakagiFormulasForMultiplierIdeals}}]
    \label{cor.SummationInLocalCase}
    Suppose $(R, \m)$ is a Noetherian complete local domain, $\Gamma$ is $\bQ$-Cartier, 
    and $\fra, \frb \subseteq R$ are ideals.  Then 
    \[
        \utau(\omega_R, \Gamma, (\fra + \frb)^t) = \sum_{t_1 + t_2 = t} \utau(\omega_R, \Gamma, \fra^{t_1} \frb^{t_2})
    \]
    where the $t_i \geq 0$.
    The analogous result also holds for larger sums $\utau(\omega_R, \Gamma, (\fra_1 + \dots + \fra_m)^t)$.
\end{corollary}
\begin{proof}
    By taking a finite extension, we may assume that $\Gamma = 0$.
    Writing $\fra = (g_1, \dots, g_v)$ and $\frb = (g_{v+1}, \dots, g_w)$ we see 
    \[
        \utau(\omega_R, (\fra + \frb)^t) = \tau^{\elt}_{\cB}(\omega_R, [g_1, \dots, g_w]^t).
    \]
    Likewise 
    \[
         \sum_{t_1 + t_2 = t} \utau(\omega_R, \Gamma, \fra^{t_1} \frb^{t_2}) = \sum_{t_1 + t_2 = t} \tau^{\elt}_{\cB}(\omega_R, \Gamma, [g_1, \dots, g_v]^{t_1} [g_{v+1}, \dots, g_w]^{t_2})
    \]
    which also equals $\tau^{\elt}_{\cB}(\omega_R, [g_1, \dots, g_w]^t)$ by definition.
   The formula for larger sum is completely analogous.  This completes the proof.
\end{proof}

\begin{remark}
    The essential part of this argument is the proof of \autoref{prop.TauBlupVsTauElts}.  It is not difficult to see that this argument, or slight variants thereof, proves that the summation formula is a formal consequence in any multiplier/test ideal theory that allows small perturbations in the coefficients, has Skoda's theorem, and satisfies transformation rules under finite maps.
\end{remark}

\begin{remark}
    \autoref{thm.CommonTestIdealNonPrincipal} answers a variant of \cite[Question 9.1]{MaSchwedePerfectoidTestideal}, and also a question contained  in \cite[Remark 6.4]{HaconLamarcheSchwede} (in the local case). More specifically, in \cite[Definition 3.1 and 3.5]{MaSchwedePerfectoidTestideal}, there are four versions of mixed characteristic test ideals defined using perfectoid algebras that are constructed via Andr\'{e}'s flatness lemma (see \cite[Section 2.5]{AndreDirectsummandconjecture}, \cite[Section 2]{BhattDirectsummandandDerivedvariant}, \cite[Theorem 7.14]{BhattScholzepPrismaticCohomology}), and \autoref{thm.CommonTestIdealNonPrincipal} shows that, if one uses a sufficient large perfectoid big Cohen-Macaulay $R^+$-algebra $B$, then all four notions agree. If one uses the original construction in \cite{MaSchwedePerfectoidTestideal}, then at least the perturbed element-wise defined $\tau([f_1,\dots, f_n]^t)$ and the $\tau(\mathfrak{a}^t)$ defined via perturbed blowups are the still the same (when $f_1,\dots,f_n$ is a generating set of $\mathfrak{a}$). This can be proved by a similar argument as in \autoref{thm.CommonTestIdealNonPrincipal} by observing that adjoining compatible system of $p$-power roots to $A_\infty$ via Andr\'{e}'s construction maps to $\widehat{R^+}$\footnote{Actually the perfectoid algebra constructed in \cite{MaSchwedePerfectoidTestideal} only admits a $p$-almost map to $\widehat{R^+}$, but this is not an issue since in the test ideal definition in \cite[Section 3]{MaSchwedePerfectoidTestideal}, we built in the small $p$-perturbation.} and noting that one can take $d=1$ since $R=A$ is regular in \cite{MaSchwedePerfectoidTestideal}. We do not carry out the details here because we think that the test ideal defined via $\widehat{R^+}$ or $B$ is better behaved (and seems the ``correct" definition).
\end{remark}

{\color{black}
    \subsubsection{More precise choices of test elements.}
    \begin{lemma}
        \label{lem.TestElementsFormAnIdeal}
        Suppose $(R, \fram)$ is a complete local Noetherian domain of residue characteristic $p > 0$, $\Gamma$ is a $\bQ$-Cartier $\bQ$-divisor on $\Spec (R)$ (if $\Gamma \neq 0$ we assume $R$ is normal), and $\fra \subseteq R$ is an ideal.
        Consider the following set:
        \[
             D := \{ 0 \neq d \in R \;|\; \utau(\omega_R, \Gamma, \fra^t) = \tau_{\myB^0}(\omega_R,\Gamma, d^{\epsilon} \fra^t) \;\;\forall 1 \gg \epsilon > 0 \}.
        \]
        Then $D \cup \{ 0\}$ is a nonzero radical ideal of $R$.  In particular, it is closed under sum. 
    \end{lemma}
    \begin{proof}
        If we can show $D$ is an ideal it is certainly radical.  To show $D$ is nonempty, pick $h$ with $H = \Div (h)$ such that $\Gamma + H \geq 0$ and let $\frb$ define $\Supp(\Gamma + H)$.  Choose a finite generating set $\{f_\lambda\}_\lambda$ of $\frb$ and $\{g_\lambda\}_\lambda$ of $\fra$ and let $d_0\in R$ so that $A[d_0^{-1}] \subseteq R[d_0^{-1}]$ is \'etale for some Cohen-Noether normalization $A \subseteq R$. Now pick $d \in R$ that is a multiple of $d_0fg$ where $f=\prod f_\lambda$ and $g=\prod g_\lambda$. Then for any $1 \gg \epsilon > 0$ we have that
        \begin{align*}
            \tau_{\myB^0}(\omega_R, \Gamma + H, d^{\epsilon} \fra^t) & \subseteq \tau_{\myB^0}(\omega_R, \Gamma+H, d^{\epsilon/2} \fra^{\epsilon/2} \fra^{t}) \\
            & = \tau_{\myB^0}(\omega_R, \Gamma+H, d^{\epsilon/2} \fra^{t+\epsilon/2}) \\
            & = \utau(\omega_R, \Gamma+H, \fra^t)\\
            & \subseteq  \tau_{\myB^0}(\omega_R, \Gamma + H, d^{\epsilon} \fra^t) 
        \end{align*} 
        where the second equality and inclusion follow from \autoref{thm.CommonTestIdealNonPrincipal}. Thus we have equalities throughout. As $H$ is Cartier, it follows that $\tau_{\myB^0}(\omega_R, \Gamma, d^{\epsilon} \fra^t) =\tau_{\myB^0}(\omega_R, \Gamma, d^{\epsilon/2} \fra^{t+\epsilon/2}) = \utau(\omega_R, \Gamma, \fra^t)$  and so  $D$ is non-empty.  
    
        It suffices to show that $D$ is closed under addition as it is easily seen to be closed under multiplication from $R$.  Suppose $d, d' \in D$ with $d + d' \neq 0$.  Then we also know that $\utau(\omega_R, \Gamma, \fra^t) = \tau_{{\myB^0}}(\omega_R, \Gamma, d^{\epsilon} d'^{\epsilon'} \fra^t)$ for all $1 \gg \epsilon, \epsilon' > 0$.  It immediately follows from the summation formula \autoref{prop.TauBlupVsTauElts} in the mixed test ideal case that 
        \[
            \utau(\omega_R, \Gamma, \fra^t) = \tau_{\myB^0}(\omega_R, \Gamma, (d,d')^{\epsilon} \fra^t)
        \]
        for $1 \gg \epsilon > 0$.
        However, $d + d' \in (d, d')$ and hence for all $\epsilon > 0$ we  have that 
        \[
            \utau(\omega_R, \Gamma, \fra^t) \supseteq \tau_{\myB^0}(\omega_R, \Gamma, (d+d')^{\epsilon} \fra^t).
        \]
        But the reverse containment is automatic for $1 \gg \epsilon > 0$ and so the result follows.
    \end{proof}

    \begin{corollary}
        \label{cor.CompleteLocalOptimalTestElementChoice}
        Suppose $(R, \fram)$ is a complete local Noetherian domain of residue characteristic $p > 0$ and $\fra \subseteq R$ is an ideal and $0 \neq c \in R$.  Suppose $0 \neq d \in \sqrt{\fra}$ and that for each $Q \in \Spec (R)$ with $d \notin Q$, we have that $R_Q$ is regular.  Then 
        \[
            \utau(\omega_R, c \fra^t) = \tau_{\myB^0}(\omega_R, (pd)^{\epsilon} c \fra^t)
        \]
        for all $1 \gg \epsilon > 0$.  As a consequence, if  we set $\frg = \sqrt{p \fra J}$ where $J$ is the ideal defining the singular locus of $R$, then $\utau(\omega_R, \fra^t) = \tau_{\myB^0}(\omega_R, \frg^{\epsilon} \fra^t)$.
        
        Suppose additionally, at the minimal primes $Q_1, \dots, Q_t$ of $pR$, we have that $R_{Q_i}$ is a DVR such that $p$ is a uniformizer (i.e., $\Div(p)$ is reduced) and that $R[d^{-1}]/(p)$ is regular.  
        Then we have that 
        \[
            \utau(\omega_R, \fra^t) = \tau_{\myB^0}(\omega_R, d^{\epsilon} \fra^t)
        \]
        for all $1 \gg \epsilon > 0$.  In particular, if we set $\frg = \sqrt{\fra J}$ where both $R$ and $R/p$ are regular outside of $V(J)$, then $\utau(\omega_R, \fra^t) = \tau_{\myB^0}(\omega_R, \frg^{\epsilon} \fra^t)$.
    \end{corollary}
    \begin{proof}
        As we may pull $c$ out of either side, we may assume $c = 1$.
        For the first statement, 
        let $Z \subseteq \Spec (R)$ denote the singular locus and suppose that 
        \[
            Q \in (\Spec R) \setminus (Z \cup V(\fra) \cup V(p)).
        \]
        Thanks to by \cite{Heitmann.EtaleLocusCompleteLocal}, there is an unramified complete regular local ring $A$ with a finite extension $A \subseteq R$ and an $h \in A$, $h \notin Q$ such that $A[h^{-1}] \subseteq R[h^{-1}]$ is \'etale.  Pick a generating set $\{g_\lambda\}_\lambda$ of $\fra$ so that $g_\lambda \notin Q$ for each $g_\lambda$. Thus $g:=\prod g_\lambda \notin Q$.  It follows from \autoref{thm.CommonTestIdealNonPrincipal} that for $1 \gg \epsilon > 0$
        \[
            \tau(\omega_R, \fra^t) = \tau_{\myB^0}(\omega_R, (hg)^{\epsilon/2} \fra^{t+\epsilon/2}) \supseteq \tau_{\myB^0}(\omega_R, h^{\epsilon} g^{\epsilon} \fra^{t}) \supseteq \tau(\omega_R, \fra^t)
        \]
        and hence $hg \in D$ while $hg \notin Q$.  In particular, $Q \notin V(D)$ and we thus have that 
        \[
            V(D) \subseteq Z \cup V(\fra) \cup V(p).
        \]
        But by hypothesis $Z \cup V(\fra) \cup V(p) \subseteq V(pd)$ and so it follows that $pd \in D$.  The consequence about $\frg$ follows from the summation formula.
        
        For the second statement, the proof is essentially the same.  Now we consider
        \[
            Q \in (\Spec R) \setminus (Z \cup V(\fra)).
        \]
        Arguing in the same way, this time using a more precise application of \cite{Heitmann.EtaleLocusCompleteLocal}, we obtain that $V(D) \subseteq V(d)$ and hence that $d \in D$.  As above, the consequence about $\frg$ follows from the summation formula.
    \end{proof}


    One can be somewhat more precise as well, weakening the condition that $d \in \fra$ (respectively $\frg \subseteq \sqrt{\fra J}$), for example:

    \begin{corollary}
        Suppose $(R, \fram)$ is a complete local Noetherian domain of residue characteristic $p>0$.  Suppose additionally that $0 \neq f \in R$.  Suppose there exists $0 \neq d \in \fra$ such that for each $Q \in \Spec (R)$ with $d \notin Q$, we have that $R_Q$ and   $(R/f)_Q$ are nonsingular.  Then 
        \[ 
            \tau(\omega_R, f^s) = \tau_{\myB^0}(\omega_R, (pd)^{\epsilon} f^s )
        \]
        for all $1 \gg \epsilon > 0$.  
    \end{corollary}
    \begin{proof}
        Write $s = {a \over b}$ for integers $a,b > 0$ and consider $S := R[f^{1 \over b}] \subseteq R^+$.  For all points $Q \in \Spec (S)$ with $d \notin Q$ we have that $S_Q$ is nonsingular by hypothesis.  Hence, we can apply \autoref{cor.CompleteLocalOptimalTestElementChoice} to $S$ and obtain that
        \[
            \tau(\omega_S) = \tau_{\myB^0}(\omega_S, (pd)^{\epsilon} ).
        \]
        Hence we also have 
        \[
            \tau(\omega_S, f^{s} ) = (f^{1 \over b})^a \tau(\omega_S) = (f^{1 \over b})^a\tau_{\myB^0}(\omega_S, (pd)^{\epsilon}) = \tau_{\myB^0}(\omega_S, (pd)^{\epsilon} f^{s}).
        \]
        Applying the trace map to both sides we obtain 
        \[
            \tau(\omega_R, f^{s} ) = \tau_{\myB^0}(\omega_R, (pd)^{\epsilon} f^{s})
        \]
        thanks to the transformation rules for finite maps.  Indeed, for the left side see \cite[Theorem 6.17]{MaSchwedeSingularitiesMixedCharBCM}.  See \cite[Lemma 4.18]{BMPSTWW1} for the right side.
    \end{proof}
}

We summarize what we know about $\tau(R, \Delta, \fra^t \frb^s)$ in the complete local case.

\begin{corollary}[Summary of results in the complete local case]
    \label{cor.SummaryOfResultsCompleteLocal}
    Suppose $R$ is a Noetherian complete local normal domain of residue characteristic $p > 0$, $\fra, \frb$ are ideals and $\Gamma = K_R + \Delta$ is a $\bQ$-Cartier divisor.  Then $\tau(R, \Delta, \fra^t \frb^s)=\tau(\omega_R, \Gamma, \fra^s \frb^t)$ satisfies the following.
    \begin{enumerate}
        \item The summation formula in the form of \autoref{cor.SummationInLocalCase}.
        \item Skoda's theorem in the form of \autoref{thm.SkodaForLocalCase}.
        \item The restriction theorem: if $R$ is $\bQ$-Gorenstein and $R/(f)$ is normal, then $\tau(R, \Delta, \fra^t) \cdot R/(f) \supseteq \tau(R/(f), \Delta|_R, (\fra R/(f))^t)$.
        \item Perturbation: $\tau(R, \Delta, \fra^t) = \tau(R, \Delta, \fra^{t+\epsilon})$ for $1 \gg \epsilon > 0$.
        \item Singularity measurement: $\tau(R, \Delta, \fra^t)$ agrees with $R$ after localizing at all points on the complement of 
        $(\Spec R)_{\mathrm{sing}} \cup \Supp \Delta \cup V(\fra)$.
        \item Subadditivity when $R$ is regular and $\Gamma = 0$, in the form of \autoref{thm:BlowUpSubadditivity}.
        {\color{black}\item Suppose $(R, \Delta)$ is BCM-regular (that is, $\tau(R, \Delta) = R$), for instance suppose $R$ is regular and $\Delta = 0$.  Then $\fra \subseteq \tau(R, \Delta, \fra^1)$.}
    \end{enumerate}
\end{corollary}
\begin{proof}
    These are straightforward consequences of what we have already done.
    \begin{enumerate}
        \item This follows from \autoref{cor.SummationInLocalCase}.
        \item This follows from \autoref{thm.SkodaForLocalCase}.
        \item This follows from the main result of \cite{MaSchwedeTuckerWaldronWitaszekAdjoint} after reduction to the principal case (say via summation).
        \item Perturbation is definitional.
        \item Singularity measurement follows from \cite[Theorem 4.1]{MaSchwedeTuckerWaldronWitaszekAdjoint} after writing the ideal as a sum of test ideals of principal ideals $\tau(R, \Delta, f^s)$ (for instance, using summation).
        \item For subadditivity, 
            if the relevant $G = \Div(g)$, then for $1 \gg \epsilon = 2/N$, 
            \[
                \begin{array}{rcl}
                    \tau(R, \fra^t \frb^s) = \tau_{\myB^0}(R, \epsilon G, \fra^t \frb^s) & = & \tau_{\myB^0}(R, (g \fra^N)^{t/N} (g \frb^N)^{s/N})\\
                    & \subseteq & \tau_{\myB^0}(R, (g \fra^N)^{t/N} ) \cdot \tau_{\myB^0}(R, (g \frb^N)^{s/N}) = \tau(R, \fra^t) \cdot \tau(R, \frb^s).
                \end{array}
            \]
            via \autoref{thm:BlowUpSubadditivity}.  The argument of \cite{MaSchwedePerfectoidTestideal,MurayamaUniformBoundsOnSymbolicPowers} \cf \cite{TakagiFormulasForMultiplierIdeals}, can also be adapted to work if one uses \autoref{thm.CommonTestIdealNonPrincipal}.
        {\color{black}\item For the containment $\fra \subseteq \tau(R, \Delta, \fra^1)$, take $f \in \fra$.  Then by Skoda's theorem (b), we have that $(f) = f \tau(R, \Delta) = \tau(R, \Delta, (f)^1) \subseteq \tau(R, \Delta, \fra^1)$.  Hence $\fra \subseteq \tau(R, \Delta, \fra^1)$ as desired.}
    \end{enumerate}    
\end{proof}

{\color{black}
Before moving to the global case, we present an example.  Suppose that $(R, \fram)$ is a complete regular local ring and $x_1, \dots, x_n$ is a regular system of parameters (a minimal generating set for $\fram$).  Suppose that $\fra$ is an ideal generated by monomials in the $x_i$ (for instance, $\fra = (x_1^2, x_1x_2^2, x_2^3)$).  We will show that 
\[ 
    \tau(R, \fra^t)
\]
is generated by monomials in the $x_i$.  Indeed, it is made up of the same ``monomials'' that appear in Howald's theorem for multiplier ideals \cite{HowaldMultiplierIdeals} or the corresponding result for test ideals in positive characteristic \cite[Theorem 4.8]{HaraYoshidaGeneralizationOfTightClosure}, \cf \cite{BlickleMultiplierIdealsAndModulesOnToric}.  A variant of this result in mixed characteristic appeared previously in \cite{RobinsonBCMTestIdealsMixedCharToric}.

To explain this precisely, we fix some notation.   Consider an ideal $\fra := (m_1, \dots, m_s)$ generated by the monomials $\displaystyle{m_j = \underline{x}^{\underline{a_j}}:= \prod_{i=1}^n x_i^{a_{ji}}}$ in the $x_i$.  One can consider the set of all exponents $\underline{\lambda}$ such that $\underline{x}^{\underline{\lambda}} \in \fra$.  These form a subset of the lattice $L = \bZ^n$ and we set $P$ to be the convex hull of this set, which is called the Newton polygon of $\fra$.

\begin{proposition}[{\cf \cite{HowaldMultiplierIdeals,HaraYoshidaGeneralizationOfTightClosure,RobinsonBCMTestIdealsMixedCharToric}}]
\label{prop.HowaldsTheoremForTau}
    With notation as above, we have that 
    \[
        \tau(R, \fra^t) = \big( {\underline x}^{\underline{\lambda}} \mid \underline{\lambda}+{\bf 1} \in \mathrm{Int}(tP) \cap L\big) = ( x^{\lfloor \lambda \rfloor} \mid \lambda \in tP ).
    \]
\end{proposition}
\begin{proof}
    The second equality in the statement is simply \cite[Remark 1]{HowaldMultiplierIdeals}, hence we will show that $\tau(R, \fra^t)$ agrees with the right side.

    We will see that this is a consequence of the summation formula.  Indeed, by \autoref{cor.SummationInLocalCase} we already know that 
    \[
        \tau(R, \fra^t) = \sum_{\substack{t_1+ \dots + t_s \\\geq t}} \tau(R, \prod_j m_j^{t_j}) 
    \]
    where the $t_j \geq 0$.  Now, $\prod_{j=1}^{s} m_j^{t_j} = \underline{x}^{\underline{b}}$ where 
    \[ 
        \underline{b} = (a_{11}t_1 + \dots + a_{s1}t_s, \dots, {a_{1n}t_1 + \dots + a_{sn}t_s}) = t_1 \underline{a_1} + \dots + t_s\underline{a_s}
    \]   
    which is by definition those elements $\underline{b}$ in $tP$.
    As $R$ is regular and $\Div(\underline{x})$ has SNC support, we also know from \autoref{thm.SkodaForLocalCase} (applied in the trivial principal case) and from \cite[Theorem 4.1]{MaSchwedeTuckerWaldronWitaszekAdjoint} that
    \[
        \tau(R, \underline{x}^{\underline b}) = ( \underline{x}^{\lfloor \, \underline{b}\, \rfloor}).
    \]
    The result follows.
\end{proof}
}

\subsection{Global versions of test modules of non-principal ideals}


    

\begin{setting} \label{setting:global-version-test-ideals-projective}
    Suppose $(V, \varpi)$ is a Noetherian DVR of mixed characteristic $(0, p > 0)$ and $X$ is a {\color{black}normal,} integral, quasi-projective scheme over $V$.  We fix $\frc$ an ideal sheaf on $X$ and take a finite set $\Lambda$ and Cartier divisors $C_{\lambda}$ for $\lambda \in \Lambda$, such that $\sum_{\lambda \in \Lambda} \cO_X(-C_{\lambda}) = \frc$. We also fix $t \in \bQ_{\geq 0}$ and a $\bQ$-Cartier divisor $\Gamma$ on $X$.
\end{setting}
In most of the work below we can easily replace $X$ by a projective compactification.
\begin{definition} \label{def.GlobalyTauOmegaXNonPrincipal}
With notation as in \autoref{setting:global-version-test-ideals-projective}, we define
\[
\utau(\omega_X, \frc^t) := \sum_{\sum s_{\lambda} \geq t} \utau(\omega_X, \sum_{\lambda} s_\lambda C_\lambda),
\]
where the sum is taken over tuples of non-negative rational numbers $(s_{\lambda})_{\lambda \in \Lambda}$ such that $\sum_\lambda s_{\lambda} \geq t$. As we will see in \autoref{remark:def-global-test-ideal-non-principal-independent-basis}, this definition is independent of the choice of $C_{\lambda}$.
\end{definition}

\begin{remark}
\label{rem:def.GlobalyTauOmegaXNonPrincipal.Triple}
Similarly we can define 
\[ 
    \utau(\omega_X, \Gamma, \frc^t) =  \sum_{\sum s_{\lambda} \geq t} \utau(\omega_X, \Gamma + \sum_{\lambda} s_\lambda C_\lambda)
\] 
for a $\bQ$-Cartier $\bQ$-divisor $\Gamma$. More generally, let $\frc_1, \dots, \frc_m$ be ideal sheaves on $X$. For $1 \leq j \leq m$, pick divisors $C_{\lambda,j}$ on $X$ indexed by $\lambda \in \Lambda_j$ such that $\sum_{\lambda \in \Lambda_j} \cO_X(-C_{\lambda,j}) = \frc_j$. Let $t_1,\ldots, t_m \in \bQ$. Then we define
\[
\utau(\omega_X, \Gamma, \frc_1^{t_1}\ldots \frc_r^{t_r}) :=  \sum_{(s_{\lambda,1}), \dots, (s_{\lambda,r})} \utau(\omega_X, \Gamma + \sum_j \sum_{\lambda \in \Lambda_j} s_{\lambda,j} C_{\lambda,j})
\]
where the outer sum is taken over tuples of non-negative rational numbers $(s_{\lambda,j})_{\lambda \in \Lambda_j}$ such that $\sum_{\lambda \in \Lambda_j} s_{\lambda,j} \geq t_j$. 
\end{remark}

{\color{black}One advantage of this definition is that various properties can be directly reduced to the divisor case.  For instance, we see that if $V \to \widehat{V}$ is the completion map and $X_{\widehat{V}}$ is the base change, then $\tau(\omega_{X}, \frc^t) \otimes_V \widehat{V} = \tau(\omega_{X_{\widehat{V}}}, (\frc \cO_{\widehat{V}})^t)$, see \autoref{rem.BaseChangeToCompletionWithDivisorPairs}.}

Recall from \autoref{prop.PropertiesOfUltTauOmegaForPairs} \autoref{prop.PropertiesOfUltTauOmegaForPairs.Perturbation} that for every $(s_{\lambda})_{\lambda \in \Lambda}$ such that $\sum_\lambda s_{\lambda} \geq t$ we can find $1 \gg \epsilon > 0$ such that  
\[
\utau(\omega_X, \sum_{\lambda} s_\lambda C_\lambda) = \utau(\omega_X, \sum_{\lambda} (s_\lambda+\epsilon) C_\lambda).
\]
In particular, we obtain that:
\[
\utau(\omega_X, \frc^t) = \sum_{\sum s_{\lambda} > t} \tau(\omega_X, \sum_{\lambda} s_\lambda C_\lambda).
\]
By Noetherianity, $\utau(\omega_X, \frc^t)$ is a sum of $\utau(\omega_X, \sum_{\lambda} s_\lambda C_\lambda)$ for some finite number of tuples $(s_{\lambda})_{\lambda \in \Lambda}$ such that $\sum s_{\lambda} > t$. Moreover, since localization commutes with finite sums, we get that $\utau(\omega_X, \frc^t)$ is a coherent sheaf on $X$ by \autoref{thm:perturbed_tau_submodule_single_alteration}. 

\begin{lemma}  \label{lem:restriction-of-tau-elt-to-R}
    For all $x\in X_{p=0}$ and $R=\widehat{\cO_{X,x}}$, we have that 
\begin{equation} 
\utau(\omega_X, \Gamma, \frc^t) \otimes R = \utau(\omega_R, \Gamma|_R, (\frc R)^t),
\end{equation}
    where the right hand side is the (local) test module from \autoref{def:common-tau-nonprincipal} (noting that if $\Gamma|_R = s \Div(g)$ then $\utau(\omega_R, \Gamma|_R, (\frc R)^t) := \utau(\omega_R, (gR)^s (\frc R)^t)$).  The analogous statement for mixed test modules also holds.
\end{lemma}
\begin{proof}
    This follows by definition of $\utau(\omega_X, \Gamma, \frc^t)$, by \autoref{prop.PropertiesOfUltTauOmegaForPairs} \autoref{prop.PropertiesOfUltTauOmegaForPairs.Completion}, and by \autoref{thm.CommonTestIdealNonPrincipal} (equality between \autoref{thm.CommonTestIdealNonPrincipal.BCMGensNoEpsilon} and \autoref{thm.CommonTestIdealNonPrincipal.BCMIdealNoEpsilon}).  The mixed test module case follows similarly.
\end{proof}

\begin{remark} \label{remark:def-global-test-ideal-non-principal-independent-basis}
We claim that 
 \[ 
        \utau(\omega_X, \frc^t) = \sum_{n \in \bN} \Bigg( \sum_{\substack{ C \mbox{ \scriptsize Cartier} \\ \O_X(-C) \subseteq \frc^n}} \utau( \omega_X, (t/n)C) \Bigg),
    \]
and so, $\utau(\omega_X, \frc^t)$ is independent of the choice of $C_{\lambda}$.
Indeed, after reducing to the projective case, it is enough to show this equality after restricting to $R = \widehat{\cO_{X,x}}$ for every $x \in X_{p=0} \subseteq X$. Write $C_{\lambda}|_R = {\rm div}(g_\lambda)$ for $g_\lambda \in R$. Since the right hand side of the above displayed equality is a finite sum by Noetherianity again, we reduce to checking that:
\[
\sum_{\sum s_{\lambda} > t} \utau(\omega_X,  \prod_{\lambda} g_{\lambda}^{s_{\lambda}}) = \sum_{n\in \bN} \sum_{f \in \frc^n } \utau(\omega_R, f^{\frac{t}{n}}),
\]    
which follows from \autoref{def:common-tau-nonprincipal} and \autoref{thm.CommonTestIdealNonPrincipal}.

{\color{black}This also lets us reduce the mixed test module case to the single ideal module case.  Consider $\tau(\omega_X, \frc_1^{t_1}\cdots \frc_r^{t_r})$.  As the $t_i$ are rational, we can write $t_i = a_i/b$ with a common denominator. We notice that 
\[
        \tau(\omega_X, \frc_1^{t_1}\cdots \frc_r^{t_r})
        = \tau(\omega_X, (\frc_1^{a_1}\cdots \frc_r^{a_r})^{1/b})
\]
as this can be checked locally at each stalk even after completion (and it is clear for $\tau_{\myB^0}(\omega_R, \dots)$, see the discussion after \autoref{def.TauBlup}).  Hence, the results in this subsection also apply to $\utau(\omega_X, \Gamma, \frc_1^{t_1}\ldots \frc_r^{t_r})$ with minimal work, and in other cases we will do this reduction with minimal comment.}
\end{remark}

\begin{remark} \label{remark:ultimate-tau-omega-global-explicit}
Let $\sL$ be an  ample line bundle on $X$, and let $\pi : Y \to X$ is a proper birational map from a normal $Y$ factoring through the blowup of $\frc$ with $M$ the Cartier divisor so that $\O_Y(-M) = \fra \cdot \O_Y$, and $\mathcal{N}_i$ is the subsheaf of $\omega_{X} \otimes \sL^i$ generated by $\myB^0(Y,tM;\omega_{Y} \otimes \pi^*\sL^i) \subseteq H^0(X,\omega_{X} \otimes \sL^i)$. Following \cite[Section 6.1]{HaconLamarcheSchwede} define
$$\tau_{\myB^0}(\omega_X,\frc^t) = \mathcal{N}_i \otimes \sL^{-i} \quad \mbox{ for } i \gg 0.$$
\end{remark}



We will need the following lemma.
\begin{lemma}
\label{lem.KevinsNonPrincipalLemma}
With assumptions as in  \autoref{setting:global-version-test-ideals-projective}, let $x \in X_{p=0}$ and  $R = \widehat{\cO_{X,x}}$. Then
  \[ 
        \tau_{\myB^0}(\omega_X,\Gamma,\frc^{t}) \otimes R \subseteq \tau_{\myB^0}(\omega_R, \Gamma|_R, (\frc R)^{t}).
    \] 
\end{lemma}
\begin{proof}
    After reducing to the case that $\Gamma \geq 0$ if necessary, we may incorporate $\Gamma$ into $\frc$ and so assume that $\Gamma = 0$.
We choose a proper birational map $\pi : Y \to X$ and $M$ as in \autoref{remark:ultimate-tau-omega-global-explicit}.  Now, $\myB^0(Y, tM; \omega_Y \otimes \pi^* \sL^i)$ is the intersection of the images of 
    \[
        H^0(Z, \omega_Z \otimes \cO_Z(-t\nu^* M) \otimes \nu^* \pi^* \sL^i ) \to H^0(Y, \omega_Y \otimes \pi^* \sL^i)
    \] 
    for alterations $\nu : Z \to Y$ with $Z$ normal. Those images admit maps to the images of 
    \[
        H^0(Z_R, \omega_{Z_R} \otimes \cO_{Z_R}(-t\nu^* M) \otimes \nu^* \pi^* \sL^i|_{Z_R} ) \to H^0(Y_R, \omega_{Y_R} \otimes \pi^* \sL^i|_{Y_R}).
    \]
    where $Z_R$ and $Y_R$ are the base changes to the complete local ring $R$.  Of course, on $R$, $\sL$ is trivial, and the intersection of the latter images becomes $\tau_{\myB^0}(\omega_R, \fra^{t})$ by \cite[Proposition 4.29]{BMPSTWW1}.  The desired containment follows.  
\end{proof}

In what follows, suppose $\pi : Y \to X$ a birational map from a normal $Y$ factoring through the blow up of $\frc$ and write $\frc \cO_Y = \cO_Y(-M)$.  We define: 
\begin{equation}
    \myB^0(X, \Gamma+ \epsilon G, \frc^t; \omega_X \otimes \sL^n) := \myB^0(Y, \pi^* \Gamma + \epsilon G + t M, \omega_Y \otimes \pi^* \sL^n)
\end{equation}
viewed as a submodule of $H^0(Y, \omega_Y \otimes \pi^* \sL^n) \subseteq H^0(X, \omega_X \otimes \sL^n)$.

We can now obtain a generalization, and a more precise version, of \autoref{prop.PropertiesOfUltTauOmegaForPairs} \autoref{prop.PropertiesOfUltTauOmegaForPairs.ComparisonWithB0}.  {\color{black}

We will need the following lemma.
\begin{lemma}
    \label{lem.QDivisorSupportEffectiveGame}
    Suppose $\Gamma$ is a $\bQ$-divisor on $X$, an integral normal scheme which is quasi-projective over a Noetherian ring.  Then 
    \[
        \Supp(\Gamma) \supseteq \bigcap_{L} \Supp(L + \Gamma) = \{x \in X \;|\; \Gamma \text{ is not Cartier at $x$}\}
    \]
    where the intersection runs over Cartier $L$ such that $L + \Gamma \geq 0$.
\end{lemma}
\begin{proof}    
    We first show $\supseteq$.
    Pick $x \in X \setminus \Supp(\Gamma)$.  It suffices to find $L$ such that $L + \Gamma \geq 0$ and $x \notin \Supp L$.  We fix $\sA$ a very ample line bundle.  Write $\Gamma = \Gamma_1 - \Gamma_2$ where $\Gamma_i$ are effective and have no common components.  Fix $I = \cO_X(-\lceil \Gamma_2\rceil)$.  Then for $m \gg 0$ we have that $I \otimes \sA^m$ is globally generated.  Pick a global section $s \in \Gamma(X, I \otimes \sA^m) \subseteq \Gamma(X, \sA^m)$ which does not vanish at $x$ (viewed as a section of $\sA^m$) and let $L$ be the corresponding effective Cartier divisor with $\cO_X(L) \cong \sA^m$.  We see that $x \notin \Supp(L)$ by construction and observe that $L \geq \lceil \Gamma_2\rceil$ so that $L + \Gamma \geq 0$.

    For the second equality $\supseteq$ is clear.  For the converse, choose $s \in K(X)$ such that $\Gamma$ and the principal divisor $\Div(s)$ agree at $x$.  By construction, $x \notin \Supp(\Gamma - \Div(s))$.  Thus by the first part, there exists $L'$ such that $\Gamma - \Div(s) + L'$ is effective and $x$ is not in its support.  Fix $L = -\Div(s) + L'$.
\end{proof}

Now we fix some notation:

\begin{notation}
    \label{notation.GlobalTestElementChoice}
    Assume the conventions of  \autoref{setting:global-version-test-ideals-projective}. 
    Additionally suppose $\frg \subseteq \cO_X$ is an ideal sheaf with the following properties.
    First, 
    \[
        V(\frg) \supseteq (\text{non-Cartier locus of $\Gamma$}) \cup V(\frc)
    \] 
    which holds for instance if $V(\frg) \supseteq \Supp(\Gamma) \cup V(\frc)$.
    Second, so that either:
    \begin{enumerate}
        \item $\frg \subseteq \sqrt{p\cO_X}$ and $X \setminus V(\frg)$ is smooth over $V$, or 
        \item $X$ is projective over $V$, $\Div_X(p) = X_{p=0}$ is reduced, and $X \setminus V(\frg)$ is smooth over $V$.
    \end{enumerate}
\end{notation}


In the second case, the condition that $X$ is projective can frequently be removed (note, there is already an implicit quasi-projective hypothesis).  However, removing the projectivity assumption adds substantial technicality to various proofs as the fact that $X_{p=0}$ is reduced in condition (b) may not be preserved when passing to a projective compactification.  We choose not to add these technicalities.

\begin{theorem}
    \label{thm.MainResultOnNonprincipalV2}
    Fix notation as in \autoref{notation.GlobalTestElementChoice} and additionally assume that $V$ is complete. 
    Then for every $1 \gg \epsilon > 0$ we have that 
    \begin{equation} 
    \utau(\omega_X,\Gamma, \frc^t) = \tau_{\myB^0}(\omega_X, \Gamma, \frg^{\epsilon}\frc^{t}). \label{thm.MainResultOnNonprincipalV2.comparisonWithB0} 
    \end{equation}
    As a consequence, for $X$ projective over $V$, for $\Gamma \geq 0$ $\bQ$-Cartier and $\sL$ is ample, then we have that 
    \[
        H^0(X, \utau(\omega_X,\Gamma, \frc^t) \otimes \sL^n) = \myB^0(X, \Gamma, \frg^{\epsilon} \frc^t; \omega_X \otimes \sL^n)
    \]
    for $1 \gg \epsilon > 0$ and $n \gg 0$.
\end{theorem}
\begin{proof}    
    It is harmless to assume that $X$ is projective and so we do.
    Recall that by definition, for a fixed finite set of Cartier divisors $C_{\lambda}$ with $\sum \cO_Y(-C_{\lambda}) = \frc$, we have that
    \[
        \utau(\omega_X, \Gamma,\frc^t) = \sum_{1 \leq i \leq r} \utau\Big(\omega_X, \Gamma+\sum_{\lambda \in \Lambda} s^{(i)}_\lambda C_\lambda\Big),
    \]
    where $(s^{(1)}_\lambda)_{\lambda \in \Lambda}, \ldots, (s^{(r)}_\lambda)_{\lambda \in \Lambda}$ are some tuples of non-negative rational numbers such that the finite sums $\sum_\lambda s^{(i)}_{\lambda} > t$ for each $1 \leq i \leq r$.  We may also assume these $C_\lambda$ and $s_{\lambda}^{(i)}$ can be used to compute the multiplier module via the classical summation formula on $X[1/p]$.

    \begin{claim}  \label{claim:MainResultOnNonprincipalClaim} There exists an effective Cartier divisor $G^0$ such that for every $G \geq G^0$, for every $1 \gg \epsilon >0$ (depending on $G$), and for every $1 \leq i \leq r$:
        \[
        \utau\Big(\omega_X, \Gamma+\sum_{\lambda} s^{(i)}_\lambda C_\lambda\Big) = \tau_{\myB^0}\Big(\omega_X, \Gamma+\epsilon G + \sum_{\lambda} s^{(i)}_\lambda C_\lambda\Big).
        \]
        \end{claim}
        \begin{proof}[Proof of Claim]
        For every $1 \leq i \leq r$ we take an effective Cartier divisor $G_i^{0}$, containing all the $C_{\lambda}$ in its support, such that the following holds. For every $G \geq G_i^{0}$ and every $1 \gg \epsilon>0$ (depending on $G$)  the statement of \autoref{prop.PropertiesOfUltTauOmegaForPairs} \autoref{prop.PropertiesOfUltTauOmegaForPairs.ComparisonWithB0} is satisfied for $\Gamma +  \sum_{\lambda} s^{(i)}_\lambda C_\lambda$.
        
        Then we define $G^0$ to be a Cartier divisor such that $G^0 \geq G_i^0$ for every $1 \leq  i \leq r$ (for example $G^0 = \sum_{i=1}^r  G_i^0$). Then given $G \geq G^0$ we have that $G \geq G^0_i$ for every $1 \leq i \leq r$, and so the statement of the claim is satisfied for every $1 \gg \epsilon > 0$ (depending on $G$).
        \end{proof}

        As we can make $G$ larger, we can choose such a $G$ so that $\cO_X(-G) \subseteq \frg$ and hence 
        \[
            \utau(\omega_X, \Gamma,\frc^t) = \sum_i \tau_{\myB^0}\Big(\omega_X, \epsilon G + \Gamma + \sum_{\lambda} s^{(i)}_\lambda C_\lambda\Big) \subseteq \tau_{\myB^0}(\omega_X, \Gamma,\frg^{\epsilon} \frc^t).
        \]
        Fix $x \in X$ a closed point, we need to show equality at the stalks at $x$.  As $x \in X_{p=0}$, 
        let $R = \widehat{\cO_{X,x}}$, a complete local ring of mixed characteristic.   
        Expanding to $R$, by \autoref{lem:restriction-of-tau-elt-to-R}, we see that 
        \[
            \utau(\omega_X, \Gamma, \frc^t) \otimes R = \utau(\omega_R, \Gamma, (\frc R)^t).
        \]
        On the other hand, by \autoref{lem.KevinsNonPrincipalLemma}, we have that 
        \[
            \utau_{\myB^0}(\omega_X, \Gamma,\frg^{\epsilon} \frc^t) \otimes R \subseteq \utau_{\myB^0}(\omega_R, \Gamma,(\frg R)^{\epsilon} (\frc R)^t).
        \]
        As $\cO_{X,x}$ is excellent, we see that $\Spec (R)$ is regular outside of $V(\frg R)$.
        We now wish to apply \autoref{cor.CompleteLocalOptimalTestElementChoice} but $\Gamma$ need not be effective and so we do not see how absorb it into $\fra$ without changing $\frg$.  However, fix a Cartier $L$ on $X$ such that $L + \Gamma$ is effective, we see that if $g \in R$ and $V(g) \supseteq V(\frg R) \cup \Supp((L + \Gamma)|_R)$, then $g$ satisfies the conditions of \autoref{cor.CompleteLocalOptimalTestElementChoice} by our hypotheses and so $\utau(\omega_R, (\Gamma + L)|_R, (\frc R)^t) = \utau_{\myB^0}(\omega_R, (\Gamma+L)|_R, g^{\epsilon} (\frc R)^t)$ and hence, as $L$ is Cartier, that 
        \[
            \utau(\omega_R, \Gamma|_R, (\frc R)^t) = \utau_{\myB^0}(\omega_R, \Gamma|_R, g^{\epsilon} (\frc R)^t)
        \]  
        for $1 \gg \epsilon > 0$.
        The set of $g$ satisfying {this equation} generates an ideal $J \subseteq R$ by \autoref{lem.TestElementsFormAnIdeal}.  
        However, by \autoref{lem.QDivisorSupportEffectiveGame}, we see that $V(J) \subseteq V(\frg)$.  Hence $\frg R \subseteq \sqrt{J}$ and thus, just as in the proof of \autoref{cor.CompleteLocalOptimalTestElementChoice}, we see that 
        \begin{equation} 
            \label{eq.usefulEqn.thm.MainResultOnNonprincipalV2}
            \utau(\omega_R, \Gamma|_R, (\frc R)^t) = \utau_{\myB^0}(\omega_R, \Gamma|_R, \frg^{\epsilon} (\frc R)^t)
        \end{equation}
        for $1 \gg \epsilon > 0$.


        The final consequence follows since $\tau_{\myB^0}(\omega_X, \Gamma+ \epsilon G, \frc^t) \otimes \sL^n$ is the subsheaf of $\omega_X \otimes \sL^n$ generated by $\myB^0(X, \Gamma+ \epsilon G, \frc^t, \omega_X \otimes \sL^n)$ for $n \gg 0$ by definition.
\end{proof}
}

{\color{black}
\begin{proposition} \label{prop:tau-nonprincipal-is-global-intersection}
    Fix notation as in \autoref{notation.GlobalTestElementChoice}.  Then for every $1 \gg \epsilon > 0$ (depending on $\frg$), we have that 
    \[
        \utau(\omega_X, \Gamma, \frc^t) = \bigcap_{Y \to X} \Tr \big( \pi_* \cO_Y(K_Y - \pi^* \Gamma - \epsilon N - tM) \big)
    \]
    where the intersection is over all normal alterations $\pi : Y \to X$ such that $\frc \cO_Y = \cO_Y(-M)$ and $\frg \cO_Y = \cO_Y(-N)$ are line bundles, and where $tM$, $\epsilon N$, and $\pi^*\Gamma$ are Cartier.

    Moreover, for every $x \in X_{p=0}$ and $R = \widehat{\cO_{X,x}}$, we have that
    \[
    \utau(\omega_X, \Gamma, \frc^t) \otimes R = \bigcap_{Y \to X} \Tr \big( \pi_* \cO_Y(K_Y - \pi^* (\epsilon G+\Gamma)- tM) \otimes R \big). 
    \]
\end{proposition}

We emphasize that the above statement is quite subtle as we are taking the intersection in the category of sheaves (and not the category of quasi-coherent sheaves). In particular, the fact that the right hand side of the first displayed equation is quasi-coherent is not clear at this point and is only a consequence of the above proposition. In the proof we shall use the following fact: if $\mathcal{F} \subseteq \mathcal{G}$ is an inclusion of $\cO_X$-sheaves and the stalks $\mathcal{F}_x=\mathcal{G}_x$ agree for every $x \in X$, then $\mathcal{F} = \mathcal{G}$. Since completion is faithfully flat, it is enough to check the equality of stalks after tensoring by $R = \widehat{\cO_{X,x}}$.

\begin{proof}
Without loss of generality we may assume that $X$ is projective over $V$.
The inclusion $\subseteq$ follows from the observation that 
\[
\utau(\omega_X, \Gamma, \frc^t) 
\subseteq \Tr \big( \pi_* \cO_Y(K_Y - \pi^*\Gamma - \epsilon N - tM) \big) 
\]
for every $Y \to X$ as above.  Indeed, this containment can be checked after completion at every closed point where the statement is clear in view of \autoref{lem:restriction-of-tau-elt-to-R}.  


We next prove the inclusion $\supseteq$. On $X[1/p]$, the left side $\utau(\omega_X, \Gamma, \frc^t)[1/p]$ is $\mJ(\omega_{X[1/p]},\Gamma, (\frc[1/p]^t))$ by the summation formula \cite{TakagiFormulasForMultiplierIdeals} and \autoref{prop.PropertiesOfUltTauOmegaForPairs} \autoref{prop.PropertiesOfUltTauOmegaForPairs.InvertP}, whereas the right side stabilizes in characteristic zero to be the multiplier ideal for $Y$ large enough.

Now fix $x \in X_{p=0}$ and $R = \widehat{\cO_{X,x}}$.  It suffices to check our stalks at $x$ agree after passing to the completion $R$.  To this end, we argue as follows:
\begin{align*}
& \left(\bigcap_{Y \to X} \Tr \big( \pi_* \cO_Y(K_Y - \pi^* \Gamma -\epsilon N - tM) \big)\right) \otimes R\\
\overset{(\dagger)
}{\subseteq} & \bigcap_{Y \to X}\left( \Tr \big( \pi_* \cO_Y(K_Y - \pi^* \Gamma - \epsilon N - tM) \big) \otimes R \right) \\
\overset{(1)}{=} & \tau_{\myB^0}(\omega_R, \Gamma|_R, (\frg R)^{\epsilon}(\frc R)^{t}) \\
\overset{(2)}{=} & \utau(\omega_R, \Gamma|_R, (\frc R)^t) \\
\overset{(3)}{=} & \utau(\omega_X, \Gamma, \frc^t) \otimes R,
\end{align*}
where $(\dagger)$ is immediate, (1) follows from  the definition of $\tau_{\myB^0}(\omega_R,\Gamma|_R, (\frg R)^{\epsilon} (\frc R)^t)$ (see \autoref{def.TauBlup}), (2) follows from the equality 
\[
    \tau_{\myB^0}(\omega_R,\Gamma|_R, (\frg R)^{\epsilon} \frc^{t}) = \utau(\omega_R,\Gamma|_R, (\frc R)^t)
\]
which we observed as \autoref{eq.usefulEqn.thm.MainResultOnNonprincipalV2} within the proof of \autoref{thm.MainResultOnNonprincipalV2} (we did not use that $V$ is complete for that statement), 
 and (3) is  \autoref{lem:restriction-of-tau-elt-to-R}. 

Finally, the second part of the proposition is equivalent to $(\dagger)$ being an equality of $R$-submodules of $\omega_R$, which is a consequence of our chain of containments above.
\end{proof}
}

\begin{lemma}
    \label{lem.StabilizingIntersection}
    Let $X$ be a Noetherian scheme, and $\sF$ a coherent sheaf on $X$. Suppose that $\{\sG_\lambda\}_{\lambda \in \Lambda}$ is a collection of coherent subsheaves of $\sF$ so that, for any two $\lambda_1, \lambda_2 \in \Lambda$, there exists $\lambda \in \Lambda$ with $\sG_\lambda \subseteq \sG_{\lambda_1} \cap \sG_{\lambda_2}$. Let $\sG = \bigcap_{\lambda \in \Lambda} \sG_\lambda$, where the intersection is taken as subsheaves of $\sF$. If $\sG$ happens to be coherent, and moreover for all $x \in X$ we have that
    \begin{equation}
    \label{eq:completedstalksintersection}
    \widehat{\sG_x} = \bigcap_{\lambda \in \Lambda} \widehat{(\sG_\lambda)_x}
    \end{equation}
    for the ($\fram_x$-adically) completed stalks, then there exists $\lambda_0 \in \Lambda$ such that $\sG = \sG_\lambda$ for all $\lambda \in \Lambda$ with $\sG_\lambda \subseteq \sG_{\lambda_0}$. In other words, the intersection defining $\sG$ stabilizes.
\end{lemma}

\begin{proof}
    As $\sG$ and $\sG_\lambda$ are coherent for all $\lambda \in \Lambda$,  the subsets $Z_\lambda = \Supp (\sG_\lambda / \sG)$ are closed. By Noetherianity, we may choose $\lambda_0 \in \Lambda$ so that $Z_{\lambda_0}$ is minimal among $Z_\lambda$ for $\lambda \in \Lambda$. By way of contradiction, assume $Z_{\lambda_0}$ is non-empty and pick $x \in Z_{\lambda_0}$ which is the generic point of an irreducible component. Thus, $(\sG_{\lambda_0})_x/\sG_x$ has finite length and in particular is $\fram_x$-power torsion, hence unaffected by $\fram_x$-adic completion. Using the descending chain condition on submodules of 
    \[ 0 \neq (\sG_{\lambda_0})_x/\sG_x = \widehat{(\sG_{\lambda_0})_x/\sG_x} = \widehat{(\sG_{\lambda_0})_x} / \widehat{\sG_x}, \]
    it follows that $\widehat{(\sG_\lambda)_x}/\widehat{\sG_x}$ are constant for all $\sG_\lambda \subseteq \sG_{\lambda_0}$ sufficiently small. In other words, there is some $\lambda_1 \in \Lambda$ with $\sG_{\lambda_1} \subseteq \sG_{\lambda_0}$ so that $\widehat{(\sG_\lambda)_x} = \widehat{(\sG_{\lambda_1})_x}$ for all $\lambda \subseteq \Lambda$ with $\sG_\lambda \subseteq \sG_{\lambda_1}$. Using \eqref{eq:completedstalksintersection}, we then have
    \[
    \widehat{\sG_x} = \bigcap_{\lambda \in \Lambda} \widehat{(\sG_\lambda)_x} = \widehat{(\sG_{\lambda_1})_x} \qquad \mbox{ or equivalently } \qquad 0 = (\sG_{\lambda_1})_x/\sG_x  = \widehat{(\sG_{\lambda_1})_x} / \widehat{\sG_x}
    \]
    so that $x \not\in Z_{\lambda_1}$ and $Z_{\lambda_1} \subsetneq Z_{\lambda_0}$ is a proper subset, contradicting the minimality of $Z_{\lambda_0}$. Thus, we must have that $Z_{\lambda_0} = \emptyset$ or $\sG_{\lambda_0} = \sG$. If $\lambda \in \Lambda$ with $\sG_{\lambda} \subseteq \sG_{\lambda_0}$, this gives
    \[
    \sG \subseteq \sG_{\lambda} \subseteq \sG_{\lambda_0} = \sG,
    \]
    and so equality holds throughout as well.
\end{proof}

{\color{black}

\begin{theorem}
    \label{thm:onealterationtorulethemall}
        Fix notation as in \autoref{notation.GlobalTestElementChoice} and choose $1 \gg \epsilon > 0$ depending on $\frg$.  
        Then there exists $\pi_{\epsilon} : Y_{\epsilon} \to X$ factoring through the blowups of $\frc, \frg$, with $\frc \cO_{Y_{\epsilon}} = \cO_{Y_{\epsilon}}(-M_{\epsilon})$, $\frg \cO_{Y_{\epsilon}} = \cO_{Y_{\epsilon}}(-N_{\epsilon})$ so that $t M_{\epsilon}$ and $\epsilon N_{\epsilon}$ are Cartier, 
        such that the following holds. 
        For every further alteration $\pi : Y \to Y_{\epsilon} \xrightarrow{\pi_{\epsilon}} X$ where we write $\frc \cO_Y = \cO_Y(-M)$ and $\frg \cO_Y = \cO_Y(-N)$, 
        we have that 
        \begin{equation}
            \label{thm.eq.NonPrincipalSingleAltWhiskey}
            \Tr\big(\pi_* \cO_Y(K_Y - \pi^* \Gamma - \epsilon N - tM)\big) = \tau(\omega_X,\Gamma, \frc^t).
        \end{equation}
        Furthermore, if $X = \Spec (S)$, where $S$ is the completion of a mixed characteristic normal local domain of essentially finite type over a mixed characteristic DVR $V$, then $\tau(\omega_S, \Gamma, \frc^t)$ as defined in \autoref{def:common-tau-nonprincipal} is also computed by a single alteration in an analogous way.
    \end{theorem}
    \begin{proof}
        Fix $1 \gg \epsilon > 0$ from \autoref{prop:tau-nonprincipal-is-global-intersection}.
        Let $\Lambda$ be the set indexing  alterations $Y \to X$ with $\frc \cO_Y = \cO_Y(-M)$, $\frg \cO_Y = \cO_Y(-N)$ and $t M$ and $\epsilon N$ Cartier.  For $\lambda \in \Lambda$ and the corresponding alteration $\pi \colon Y \to X$ with $\frc \cO_Y = \cO_Y(-M), \frg \cO_Y = \cO_Y(-N)$ define 
    \begin{align*}
    \mathcal{G}_{\lambda} &:= \Tr\big(\pi_* \cO_Y(K_Y - \pi^* \Gamma - \epsilon N - tM)\big),  \text{ and }\\
    \mathcal{G} &:= \bigcap_{\lambda \in \Lambda} \mathcal{G}_{\lambda}.
    \end{align*}
    Here, $\mathcal{G} = \utau(\omega_X, \Gamma, \frc^t)$ by \autoref{prop:tau-nonprincipal-is-global-intersection}. 
    
    The theorem follows from \autoref{lem.StabilizingIntersection} provided that we can verify that its assumptions are satisfied.     
    Specifically, we need to show that 
    \[
        \utau(\omega_X, \Gamma, \frc^t) \otimes R = \bigcap_{Y \to X} \Tr \big( \pi_* \cO_Y(K_Y - \pi^* \Gamma - \epsilon N - tM) \otimes R \big)
    \]
    for every $x \in X$ and $R = \widehat{\cO_{X,x}}$, 
    where the left hand side is equal to $\tau(\omega_R, \Gamma, \frc^t)$ by \autoref{lem:restriction-of-tau-elt-to-R}. When $x$ has residue characteristic zero, this is classical (see \cite{BlickleSchwedeTuckerTestAlterations}). When $x \in X_{p=0}$, this follows immediately from the second part of \autoref{prop:tau-nonprincipal-is-global-intersection}.  This proves the first part of the theorem.
    
    The final statement is simply a combination of the main part of the theorem with \autoref{lem:restriction-of-tau-elt-to-R}, \cf \autoref{prop:tau-nonprincipal-is-global-intersection}.
    \end{proof}
}

\begin{theorem}
\label{thm.PropertiesOfTauOmegaNonPrincipal}
With notation as above, let $\frb, \frc$ be ideal sheaves, and let  $t, s \in\bQ$.     We have the following properties:
\begin{enumerate}
\item For any Cartier divisor $D$ and $0<\epsilon_1,\epsilon_2,\epsilon_3\ll 1$ (depending on the data of  $X$, $\frb$, $\frc$, and $D$), we have 
\[\tau(\omega_X, \Gamma, \frb^t)=\tau(\omega_X,\epsilon_1D + \Gamma,\frb^{t+\epsilon_2}\frc^{\epsilon_3}).\] 
\item  If $\frb\subseteq\frc$ and $s \geq t$, then  $\tau(\omega_X,\Gamma, \frb^s)\subseteq\tau(\omega_X,\Gamma, \frc^t)$.

\item $\tau(\omega_X,\Gamma, \fra^{nt})=\tau(\omega_X,\Gamma, (\fra^n)^t) {\color{black}= \tau(\omega_X, \Gamma, \underbrace{\fra^t \cdots \fra^t}_{\text{$n$-times}})}$ for every $n \in \bN$.




\item $\tau(\omega_X,\Gamma,\frc^t)[1/p]=\mathcal{J}(\omega_{X[1/p]},\Gamma|_{X[1/p]}, (\frc[1/p])^t)$.

\item 
If $t \geq \dim(X)$ or $\frb$ can locally be generated by $t$ elements, then $\tau(\omega_X,\Gamma,\frb^t \frc^s) = \frb \cdot \tau(\omega_X,\Gamma,\frb^{t-1}\frc^s)$.
\item 
$\tau(\omega_X, \Gamma, (\frb + \frc)^t) = \sum \tau(\omega_X, \Gamma, \frb^{t_1} \frc^{t_2})$ where the sum runs over $t_1, t_2 \geq 0$ with $t_1 + t_2 = t$.
\item Suppose that $f : Y \to X$ is a finite surjective map, then $\Tr\big( f_* \utau(\omega_Y, f^* \Gamma,(\frc \cO_Y)^t)\big) = \utau(\omega_X, \Gamma, \frc^t).$
{\color{black}\item Suppose that $f : Y \to X$ is a projective birational map, then $f_* \utau(\omega_Y, f^* \Gamma,(\frc \cO_Y)^t) \supseteq \utau(\omega_X, \Gamma, \frc^t).$}
\item Suppose that $f : Y \to X$ is a smooth map, then $\utau(\omega_Y, f^* \Gamma, (\frc \cO_Y)^t) = f^* \utau(\omega_X,\Gamma, \frc^t) \otimes \omega_{Y/X}$. \label{thm.PropertiesOfTauOmegaNonPrincipal.SmoothPullback}
\end{enumerate}

\end{theorem}

\begin{proof}\hphantom{a}

\begin{enumerate}
\item The terms $\epsilon_1 D, \frb^{\epsilon_2}, \frc^{\epsilon_3}$, can all be absorbed inside $\frg$ as making that ideal smaller changes nothing.  Part (a) then from \autoref{thm:onealterationtorulethemall}.  Alternately, one may reduce to the projective case, verify it at a closed point, and observe the statement then holds in a neighborhood.  Taking a minimum of the $\epsilon_i$ that work on a finite affine cover also proves (a). 

\item[(b),(c)] These follow in various ways.  For instance, simply choose an alteration and $\frg$ that works for both pairs in \autoref{thm:onealterationtorulethemall}. Alternately, one can check it after passing to  the projective case, and then checking it at the completed stalks of closed points, where \autoref{eq.VariousUnambiguities} applies.

\addtocounter{enumi}{2}

\item 
By choosing an alteration dominating a log resolution over the locus of points with residue characteristic zero, this follows immediately from \autoref{thm:onealterationtorulethemall}.

\item Using \autoref{lem:restriction-of-tau-elt-to-R}, this follows from
\autoref{thm.SkodaForLocalCase} thanks to  \autoref{rem:candoskodawhenbiggerthandim}.

\item After passing to a finite cover to trivialize $\Gamma$, this follows from \autoref{cor.SummationInLocalCase} as we can check it at the completed stalks.

\item This follows immediately from, for example, \autoref{thm:onealterationtorulethemall}.

{\color{black}\item This also follows from \autoref{thm:onealterationtorulethemall}.}

\item This follows from the definition of $\utau$ via summation (\autoref{def.GlobalyTauOmegaXNonPrincipal}) and the principal case (\autoref{prop.PropertiesOfUltTauOmegaForPairs} \autoref{prop.PropertiesOfUltTauOmegaForPairs.SmoothPullBack}).

\end{enumerate}
\end{proof}

\subsection{Global versions of test ideals of non-principal ideals}

We now move to define test ideals instead of test modules.  
\begin{definition}
    \label{def.UltimateTestIDEALDefinition}
With notation as in \autoref{setting:global-version-test-ideals-projective}, assume in addition that $X$ is normal, and suppose that $\Delta$ is a $\bQ$-divisor such that $K_X+\Delta$ is $\bQ$-Cartier.  Then we define the test ideal $\tau(\sO_X,\Delta,\frc^t)=\tau(\omega_X,K_X+\Delta,\frc^t)$.  For a sequence of ideal sheaves $\frc_1,\dots \frc_r$ and rational numbers $t_i \geq 0$ we also define mixed test ideals
\[\tau(\sO_X,\Delta,\frc_1^{t_1}\dots\frc_r^{t_r}):=\tau(\omega_X,K_X+\Delta,\frc_1^{t_1}\dots\frc_r^{t_r}).\]
As before, the $\frc_1^{t_1}\dots\frc_r^{t_r}$ may be reinterpreted as a single $\frc^t$ as the $t_i$ are rational.  
When $\Delta\geq 0$ we note that this is a genuine ideal, not just a fractional ideal.

{\color{black}
Choose an ideal sheaf $\frg$ so that 
$V(\frg) \supseteq (\text{non-Cartier locus of $\Delta$}) \cup V(\frc)$ and so that either:
    \begin{enumerate}
        \item $\frg \subseteq \sqrt{p\cO_X}$ and $X \setminus V(\frg)$ is smooth over $V$, or 
        \item $X$ is projective over $V$, $X_{p=0}$ is reduced\footnote{If $X$ is normal, this simply means that $\Div_X(p)$ is reduced.}, and $X \setminus V(\frg)$ is smooth over $V$.
    \end{enumerate}}
\end{definition}
{\color{black}

\begin{lemma}
    \label{lem.frgAndfrdComparison}
    An ideal sheaf $\frg$ satisfies the conditions of  \autoref{def.UltimateTestIDEALDefinition} if and only if it satisfies the conditions of \autoref{notation.GlobalTestElementChoice} for $\Gamma = K_X + \Delta$.
    
\end{lemma}
\begin{proof}
The complement of $V(\frg)$ is non-singular in either case.  Hence, 
    we only need to compare the non-Cartier locus of $\Delta$ with the non-Cartier locus of $K_X + \Delta$ on the non-singular locus of $X$. But $K_X$ is Cartier on the non-singular locus so the result follows.
\end{proof}
}

We now collect together various properties satisfied by our test ideals.  First we point out it can be computed from a single alteration.

{\color{black}
\begin{corollary}
    \label{thm.FinalAlterationStabilizationForTestIdeal}
    Fix notation as in \autoref{def.UltimateTestIDEALDefinition}. 

    There exists $\pi_{\epsilon} : Y_{\epsilon} \to X$ satisfying the conditions of \autoref{thm:onealterationtorulethemall} for $\frg, \frc$
    such that the following holds.
      For every further alteration $\pi : Y \to Y_{\epsilon} \xrightarrow{\pi_{\epsilon}} X$ where we write $\frc \cO_Y = \cO_Y(-M), \frg \cO_Y = \cO_Y(-N)$ 
    we have that for $1 \gg \epsilon > 0$ with $\epsilon $
    \begin{equation}
        \label{thm.eq.NonPrincipalSingleAltWhiskeyIdeal}
        \Tr\big(\pi_* \cO_Y(K_Y - \pi^* (K_X + \Delta) - \epsilon N - tM)\big) = \tau(\cO_X, \Delta, \frc^t).
    \end{equation}
    Furthermore, if instead $X = \Spec (R)$, where $R$ is the completion of a mixed characteristic normal local domain of essentially finite type over a mixed characteristic DVR $V$, then $\tau(R, \Delta, \frc^t)$ as defined in \autoref{def:common-tau-nonprincipal} is also computed by a single alteration in the same way.
\end{corollary}
\begin{proof}
    In view of \autoref{thm:onealterationtorulethemall} this follows from the fact that we can take $\frg$ for either $\Gamma$ or $\Delta$, which we proved in  \autoref{lem.frgAndfrdComparison}.  
\end{proof}
}

\begin{corollary}
    \label{cor.FinalCompletionGlobalizationLocalizationForTestIdeal}
    With notation as in \autoref{def.UltimateTestIDEALDefinition}, for all $x\in X_{p=0}$ and $R=\widehat{\cO_{X,x}}$, we have that 
    \begin{equation} 
    \utau(\cO_X, \Delta, \frc^t) \cdot R = \utau(R, \Delta|_R, (\frc R)^t),
    \end{equation}
        where the right hand side is the test ideal from \autoref{def:common-tau-nonprincipal}.  

    Furthermore, {\color{black}if $V$ is complete,} for every $1 \gg \epsilon>0$:
        \begin{equation} 
        \utau(\cO_X,\Delta, \frc^t) = \tau_{\myB^0}(\cO_X, \Delta, \frg^{\epsilon} \frc^{t}). 
        \end{equation}
\end{corollary}
\begin{proof}
    Apply \autoref{lem:restriction-of-tau-elt-to-R}, \autoref{thm.MainResultOnNonprincipalV2} as well as  \autoref{lem.frgAndfrdComparison} after setting $\Gamma = K_X + \Delta$.
\end{proof}

We list some other useful properties below.

\begin{theorem}
    \label{thm.FinalPropertiesOfTestIdealsTriples}
With notation as in  \autoref{def.UltimateTestIDEALDefinition}, let $\frb, \frc$ be ideal sheaves, and let  $t, s \in\bQ$.     We have the following properties:
\begin{enumerate}
\item For any Cartier divisor $D$ and $0<\epsilon_1,\epsilon_2,\epsilon_3\ll 1$ (depending on the data of  $X$, $\frb$, $\frc$, and $D$), we have 
\[\tau(\cO_X, \Delta,  \frb^t)=\tau(\cO_X, \Delta +  \epsilon_1D,\frb^{t+\epsilon_2}\frc^{\epsilon_3}).\] \label{thm.FinalPropertiesOfTestIdealsTriples.Perturbation}
\item  If $\frb\subseteq\frc$ and $s \geq t$, then  $\tau(\cO_X, \Delta, \frb^s)\subseteq\tau(\cO_X, \Delta, \frc^t)$.\label{thm.FinalPropertiesOfTestIdealsTriples.BiggerIdeals-Exponents}

\item $\tau(\cO_X, \Delta,\fra^{nt})=\tau(\cO_X, \Delta, (\fra^n)^t) {\color{black}= \tau(\cO_X, \Delta, \underbrace{\fra^t \cdots \fra^t}_{\text{$n$-times}})}$ for every $n \in \bN$.\label{thm.FinalPropertiesOfTestIdealsTriples.Unambiguity}




\item $\tau(\cO_X, \Delta, \frc^t)[1/p]=\mathcal{J}(\cO_{X[1/p]}, \Delta|_{X[1/p]}, (\frc[1/p])^t)$.  \label{thm.FinalPropertiesOfTestIdealsTriples.InvertPMultiplierIdeals}

\item 
If $t \geq \dim(X)$ or $\frb$ can locally be generated by $t$ elements, then $\tau(\cO_X, \Delta, \frb^t \frc^s) = \frb \cdot \tau(\cO_X, \Delta, \frb^{t-1}\frc^s)$.
\label{thm.FinalPropertiesOfTestIdealsTriples.Skoda}
{\color{black}\item Suppose $(X, \Delta)$ is BCM-regular (that is, $\tau(\cO_X, \Delta) = \cO_X)$), then $\frb \subseteq \tau(\cO_X, \Delta, \frb^1)$.  \label{thm.FinalPropertiesOfTestIdealsTriples.EasyContainmentInBCMRegular}}
\item 
$\tau(\cO_X, \Delta, (\frb + \frc)^t) = \sum \tau(\cO_X, \Delta, \frb^{t_1} \frc^{t_2})$ where the sum runs over $t_1, t_2 \geq 0$ with $t_1 + t_2 = t$.
\item If $X$ is regular, then $\tau(\cO_X, \frb^s \frc^t) \subseteq \tau(\cO_X, \frb^s) \cdot \tau(\cO_X, \frc^t)$.\label{thm.FinalPropertiesOfTestIdealsTriples.Subadditivity}
{\item Suppose that $f : Y \to X$ is a finite surjective map, then $\Tr\big( f_* \utau(\cO_Y, \Delta_Y, (\frc \cO_Y)^t)\big) = \utau(\cO_X, \Delta, \frc^t)$, where $K_Y+\Delta_Y = f^*(K_X+\Delta)$. \label{thm.FinalPropertiesOfTestIdealsTriples.FiniteMaps}}
{\color{black}\item Suppose that $f : Y \to X$ is a finite type birational map, then $ f_* \utau(\cO_Y, \Delta_Y, (\frc \cO_Y)^t) \supseteq \utau(\cO_X, \Delta, \frc^t)$, where $K_Y+\Delta_Y = f^*(K_X+\Delta)$. \label{thm.FinalPropertiesOfTestIdealsTriples.Birational}}
\item Suppose that $f : Y \to X$ is a smooth map, then $\utau(\cO_Y, f^*\Delta, (\frc \cO_Y)^t) = f^* \utau(\cO_X, \Delta; \frc^t)$.\label{thm.FinalPropertiesOfTestIdealsTriples.SmoothPullback}
\item Suppose that $H \subseteq X$ is a normal Cartier divisor such that $H$ and $\Delta$ have no common components in support, then \[
    \tau(\cO_X, \Delta, \frb^t) \cdot \cO_H \supseteq \tau(\cO_H, \Delta|_H, (\frb \cO_H)^t).
\]%
\label{thm.FinalPropertiesOfTestIdealsTriples.Restriction}%
\end{enumerate}
\end{theorem}
\begin{proof}
Subadditivity \autoref{thm.FinalPropertiesOfTestIdealsTriples.Subadditivity}, and \autoref{thm.FinalPropertiesOfTestIdealsTriples.EasyContainmentInBCMRegular} follow from \autoref{cor.FinalCompletionGlobalizationLocalizationForTestIdeal} and \autoref{cor.SummaryOfResultsCompleteLocal}.  The rest are immediate from \autoref{thm.PropertiesOfTauOmegaNonPrincipal} and \autoref{thm.LocalPropertiesOfUltTauOX}.  {\color{black}Note, for \autoref{thm.FinalPropertiesOfTestIdealsTriples.Birational}, one can reduce to the projective case by arguing as in the proof of  \autoref{thm.LocalPropertiesOfUltTauOX}.}
\end{proof}

\begin{remark}[Test ideals of linear series]
    \label{rem.TestIdealsOfLinearSeries}
    One other sort of test ideal appeared in \cite{HaconLamarcheSchwede}.  Suppose $X$ is normal and projective over $V$.  Suppose $D$ is a Cartier divisor on $X$ and $t \geq 0$ is a rational number.  Then they defined 
    \[
        \tau_{\myB^0}(\omega_X, \Gamma, t || D ||) = \sum_{F \in | mD | } \tau_{\myB^0}(\omega_X, \Gamma + {t \over m} F)
    \] 
    for $m$ sufficiently divisible.  Let $\frb_{m}$ denote the base ideal of $|mD|$.  We claim that there exists a $G^0 \geq 0$ such that for each $G \geq G^0$ and all $1 \gg \epsilon > 0$ (depending on $G$)
    \begin{equation}
        \tau_{\myB^0}(\omega_X, \Gamma +\epsilon G, t || D ||) = \tau(\omega_X, \Gamma, \frb_m^{t/m})
    \end{equation}
    for $m$ sufficiently divisible.  

    We quickly show this (assuming $\Gamma = 0$ for simplicity of exposition).
    Indeed, the containment
    \[
        \tau_{\myB^0}(\omega_X, \Gamma + \epsilon G, t || D ||) \subseteq \tau(\omega_X, \Gamma, \frb_m^{t/m})
    \]
    is clear since if $F \in |mD|$ then $\cO_X(-F) \subseteq \frb_m$.  On the other hand, $\frb_m^n \subseteq \frb_{mn}$, and so $\tau(\omega_X, \frb_m^{t/m})$ increases as $m$ becomes more divisible and hence it must stabilize (fix that $m$).  Now, 
    \[ 
        \sum_{F \in |mD| }\cO_X(-F) = \frb_m
    \] 
    and so for some sufficiently large $n$
    \[
        \tau(\omega_X, \frb_m^{t/m}) = \sum_{\substack{F_1, \dots, F_w \\ \in |mD|} } \; \sum_{\substack{s_1 + \dots + s_w \\= {t/m}}} \tau\Big(\omega_X, s_1 F_1 + \dots + s_w F_w \Big).
    \]
    The sum is finite and so we may assume all relevant $s_j = t {a_j \over mn}$ are rational numbers where we can pick $a_j, n \in \bZ$.  Note then ${t / m} = \sum s_j = {t \over mn} \sum {a_j}$.  Hence each 
    \[ 
        s_1 F_1 + \dots + s_w F_w = {t \over nm}(a_1 F_1 + \dots + a_w F_w) = {t \over nm} F
    \]  
    for $F = a_1 F_1 + \dots + a_w F_w \in |nmD|$.   But then we have that 
    \[
        \tau(\omega_X, \frb_m^{t/m}) = \sum_{F \in |nmD| } \tau\big(\omega_X, {t \over nm} F\big)
    \]
    which is also a finite sum and so we may find such a $G^0$, so that for $G$ and $\epsilon$ as above, that 
    \[
        \tau(\omega_X, \frb_m^{t/m}) = \sum_{\substack{F_1, \dots, F_l \\ \in |nmD|} } \tau_{\myB^0}(\omega_X, \epsilon G + {t \over nm} F_i) \subseteq \tau_{\myB^0}(\omega_X, \epsilon G, t || D ||)
    \]
    which completes the proof.

    Suppose now that $A$ is a globally generated ample divisor, and $\Gamma = K_X + \Delta$.  If $L - (K_X + \Delta) - t D$ is big and nef, then we immediately see that 
    \begin{equation}
        \tau(\cO_X, \Delta, \frb_m^{t/m}) \otimes \cO_X(nA + L) = \tau_{\myB^0}(\cO_X, \Delta + \epsilon G, t || D ||)  \otimes \cO_X(nA + L)
    \end{equation}
    is globally generated for $n \geq n_0 := \dim X_{p = 0}$ (still assuming $m$ sufficiently divisible).  Indeed, this ideal is a sum of ideals of the form $\tau(\cO_X, \Delta + {t \over nm} D_{nm})$ for effective $D_{nm} \sim nmD$.  Hence, since each $L - (K_X + \Delta + {t \over nm} D_{nm})$ is big and nef, we can simply apply \autoref{cor.EffectiveGlobalGenerationTauDivisor} and note that a sum of globally generated sheaves is globally generated.
\end{remark}


\section{Characteristic zero analog of our result} \label{ss:intro-char0}

In what follows we explain how to recover the multiplier ideal sheaves from the intersection complex. Let $X$ be an algebraic variety of dimension $d$ defined over $\mathbf{C}$. We fix an embedding of $X$ into a smooth ambient variety $Y$, and define a $\mathcal D_X$-module to be a $\mathcal D_Y$-module supported on $X$.

First, we recall the characteristic zero Riemann-Hilbert correspondence generalizing the correspondence between local systems and vector bundles with integrable connections:
\[
\RH \colon  D_{\cons}^b(X, \mathbf{C}) \xrightarrow{\simeq} D_{\mathrm{rh}}^b(X, \mathcal D_X).
\]
Here the left hand side denotes the derived category of $\mathbf{C}$-constructible sheaves while the right hand side denotes the derived category of regular holonomic $\mathcal D_X$-modules. Moreover, $D_{\cons}^b(X, \mathbf{C})$ admits a subcategory of perverse sheaves $\mathrm{Perv}_{\cons}(X, \mathbf{C})$ (see  \autoref{sec:preliminaries_perverse_sheaves}) for which the functor $\RH$ restricts to:
\[
\RH \colon \mathrm{Perv}_{\cons}(X, \mathbf{C}) \xrightarrow{\simeq} \mathcal {\rm Mod}_{\rm rh} (\mathcal D_X),
\]
where ${\rm Mod}_{\rm rh}(\mathcal D_X)$ denotes the category of regular holonomic $\mathcal D_X$-modules. Also, note that this  functor admits a shift by $d$ with respect to the usual Riemann-Hilbert between local systems and vector bundles with integrable connections, for example, $\RH(\mathbf{C}) = \cO_X[-d]$.

In the $p$-adic situation, the analog of $\RH$ incorporates additional Galois action structure. However, in characteristic zero, additional structure, namely Hodge filtration, must be introduced by hand. This is achieved by ways of Hodge modules which are triples  $\left(V, M, F_{{\bullet}} M\right)$ consisting of a $\mathcal{D}_X$-module $M$ with filtration $F_{{\bullet}} M$ and a perverse $\mathbf{Q}$-sheaf $V$ such that $\mathrm{RH}(V \otimes \mathbf{C})=M$, subject to various conditions. In what follows, by abuse of notation, we denote the \emph{trivial Hodge module} $(\mathbf{Q}_X[d], \cO_X, F_{\bullet} \cO_X)$ with a trivial filtration by $\mathbf{Q}_X[d]$. 


In practice, one often works with the split de Rham complex. Specifically, we have a functor (the reader should be warned that this notation is non-standard):
\[
\mathrm{RH}^{\text {Higgs }}: D_{\mathrm{HM}}^b(X) \rightarrow D^b(X, \cO_X)
\]
between the derived category of Hodge modules and the usual derived category of coherent sheaves. In the case when $X$ is smooth this functor is defined as 
\begin{align*}
\mathrm{RH}^{\mathrm{Higgs}}\left(V, M, F_{{\bullet}} M\right) &:=\operatorname{GR}_{\bullet} \mathrm{DR}(M)=\bigoplus F_k \mathrm{DR}(M) / F_{k-1} \mathrm{DR}(M), \text{ where } \\
\mathrm{DR}(M) &:= (M \to M \otimes \Omega^1_X \to \cdots \to M \otimes \Omega^{d}_X) \in D^{[-d,0]}(X, \mathbf{C}) \text{ and } \\F_k \mathrm{DR}(M) &:= (F_{k}M \to F_{k+1}M \otimes \Omega^1_X \to \cdots \to F_{k+d}M \otimes \Omega^{d}_X) \in D^{[-d,0]}(X, \mathbf{C}).
\end{align*}
For example, 
\begin{equation} \label{eq:rrh-char0-trivial}
\mathrm{RH}^{\mathrm{Higgs}}(\mathbf{Q}_X[d])= \cO_X[d] \oplus \Omega_X^1[d-1] \oplus \cdots \oplus \Omega_X^d[0].
\end{equation}
By performing this operation on the ambient space $Y$, this construction generalizes to the case when $X$ is not necessarily smooth. We emphasize that the $p$-adic Riemann-Hilbert correspondence discussed in this paper is a direct analog in mixed characteristic of $\mathrm{RH}^{\text {Higgs}}$ and not of the classical $\mathrm{RH}$ functor!

One can define a duality functor $\mathbf{D} \colon D_{\mathrm{HM}}^b(X) \to D_{\mathrm{HM}}^b(X)$ on Hodge modules which restricts to Verdier duality on constructible sheaves. Moreover, there is the Grothendieck duality functor $\mathbf{D} \colon D^b(X,\cO_X) \to D^b(X,\cO_X)$ such that
\[
\mathbf{D}(-) = \myR\sHom(-, \omega^{\mydot}_X).
\]
Functor $\mathrm{RH}^{\mathrm{Higgs}}$ satisfies the following properties:
\begin{itemize}
\item it interchanges Grothendieck duality and Hodge module duality
      \[ \mathbf{D}{\rm RH}^{\rm Higgs}(M) \simeq {\rm RH}^{\rm Higgs}(\mathbf{D}M),\]
\item it is compatible with proper pushforwards, specifically for $\pi \colon Y \to X$ proper,
  $${\rm RH}^{\rm Higgs}(\myR\pi_*M) \simeq \myR\pi_*{\rm RH}^{\rm Higgs}(M).$$
\end{itemize}
Moreover, as its $p$-adic analog, one can check that this functor is left $t$-exact for perverse $t$-structures on both sides, but we will not need that in this subsection.

For simplicity, we state the characteristic zero analog of the main result of this paper, only for Grauert-Riemenschneider sheaves, but one can extend the definition to all multiplier ideals sheaves. From now on we assume that $X$ is normal.
\begin{definition} \label{def:GR-intro}
We define the \emph{Grauert-Riemenschneider sheaf}
\[
\mathcal J(X, \omega_X) = \pi_*\omega_Y
\]
for a resolution of singularities $\pi \colon Y \to X$.
\end{definition}
\noindent One can check that this definition is independent of the choice of the resolution. When $X$ is Gorenstein, this agrees with the multiplier ideal sheaf. 

One of the most important examples of perverse sheaves is the intersection complex ${\rm IC}_X \in \mathrm{Perv}_{\cons}(X, \mathbf{Q})$. This constructible complex underlies a Hodge module which, by abuse of notation, we denote by the same symbol ${\rm IC}_X$. This Hodge module is self-dual, that is $\mathbf{D}({\rm IC}_X) = {\rm IC}_X$. 

\begin{remark}
To state the following theorem, we need to introduce a couple of maps.
Firstly, the map $\mathbf{Q}_X[d] \to {\rm IC}_X$ of Hodge modules induces by applying $H^0\mathbf{D}{\rm RH}^{\rm Higgs}$:
\[
\psi' \colon H^0\mathbf{D} \RH^{\rm Higgs} (\mathbf{Q}_X[d]) \leftarrow H^0 \mathbf{D} {\rm RH}^{\rm Higgs}({\rm IC}_X) = H^0{\rm RH}^{\rm Higgs}({\rm IC}_X).
\]
Secondly, assuming for now that $X$ is smooth and using \autoref{eq:rrh-char0-trivial}, we get the map $\cO_X[d] \to {\rm RH}^{\rm Higgs}(\mathbf{Q}_X[d])$. Then by applying Grothendieck duality $H^0\mathbf{D}(-) = H^0\myR\sHom(-, \omega^{\mydot}_X)$, we obtain:
\[
\psi'' \colon \omega_X \leftarrow H^0\mathbf{D}{\rm RH}^{\rm Higgs} (\mathbf{Q}_X[d]).
\]
In general, since $\omega_X$ is $S_2$, we can define this map on the smooth locus and extend to $X$ by reflexivity. Finally, we define $\psi \colon H^0{\rm RH}^{\rm Higgs} ({\rm IC}_X) \to \omega_X$ as the composition $\psi'' \circ \psi'$.
\end{remark}

\begin{theorem} \label{thm:mainthm-in-char-zero}
For a normal algebraic variety $X$ over $\mathbf{C}$, the following holds:
\[
\mathcal J(X, \omega_X) = {\rm Image}(\psi \colon H^0{\rm RH}^{\rm Higgs} ({\rm IC}_X) \to \omega_X).
\]
\end{theorem}
\begin{proof}
Let $\pi \colon Y \to X$ be a resolution and $Z := \pi(\mathrm{Exc}(\pi))$. First, we observe that by the decomposition theorem for Hodge modules: 
\[
\myR\pi_*{\rm IC}_Y = {\rm IC}_X \oplus M,
\]
where $M \in D_{\mathrm{HM}}^b(X)$ is supported on $Z$.

Second, by \autoref{eq:rrh-char0-trivial}, the fact that ${\rm IC}_Y = \mathbf{Q}_Y[d]$ since $Y$ is smooth, and the compatibility of $\RH$ with proper pushforward, we have that:
\begin{align*} H^0{\rm RH}^{\rm Higgs}(\myR\pi_*{\rm IC}_Y) &= H^0\myR\pi_*{\rm RH}^{\rm Higgs}({\rm IC}_Y) \\ 
&= H^0\myR\pi_*(\cO_Y[d] \oplus \cdots \oplus \omega_Y[0]) = \pi_*\omega_Y \oplus \textrm{a sheaf supported on } Z.
\end{align*}
Combining the above and using torsion-freeness of $\omega_X$ we finally get: 
\[
\Image(\psi) = \Image( H^0{\rm RH}^{\rm Higgs}(\myR\pi_*{\rm IC}_Y) \to \omega_X ) =  \Image(\pi_*\omega_Y \to \omega_X). \qedhere
\]
\end{proof}
We do not know how to prove \autoref{thm:mainthm-in-char-zero} using purely complex geometric methods but not the decomposition theorem.  On the other hand, in the $p$-adic setting, the decomposition theorem for perverse sheaves is false. Instead, our proof uses the fact that the $p$-adic Riemann-Hilbert correspondence is left $t$-exact for the perverse $t$-structures (see \autoref{sec:ReviewRH}). 
\begin{remark}
For the sake of this paper, it is easier to work with the dual variant of \autoref{thm:mainthm-in-char-zero} which may be stated as:
\[
\cO^{\rm GR}_{X} = \Image(\phi \colon {}^p H^0(\cO_X[d]) \to {}^p H^0({{\rm RH}^{\rm Higgs}}({\rm IC}_X)))
\]
where $\cO^{\rm GR}_{X} := \mathbf{D}(\mathcal J(X, \omega_X))$, and $\phi$ defined as the composition 
\[
\cO_X[d] \to {\rm RH}^{\rm Higgs}(\mathbf{Q}_X[d]) \to {\rm RH}^{\rm Higgs}({\rm IC}_X).
\]
\end{remark}

\section{Further questions}

We do not know if the $p$-perturbation and the $\epsilon G$ perturbation in this paper is necessary (\cf \autoref{quest.tauAltPerturbed}).  Explicitly, we do not know, or have an approach, to the following question.

\begin{conjecture}[\cf {\cite{SmithTightClosureParameter}}]
    Suppose that $X$ is normal integral and finite type over a complete mixed characteristic DVR $V$, and $\Gamma$ is a $\bQ$-Cartier $\bQ$-divisor.  Then  
    \[
        \tau(\omega_X, \Gamma, \frc^t) = \Tr_{Y/X} \big(\pi_* \cO_Y(K_Y - \pi^* \Gamma - tM)  \big)
    \]
    for some sufficiently large alteration $\pi : Y \to X$ with $\frc \cO_Y = \cO_Y(-M)$ and $tM$ and $\pi^* \Gamma$ Cartier?
\end{conjecture}

Even if we cannot answer the above question, based on the theory of test elements \cite{HochsterHunekeTC1}, we hope that the following holds.
\begin{question}
    Suppose $X$ is $\bQ$-Gorenstein and quasi-projective over $V$.  Suppose $G^0 > 0$ is a Cartier divisor such that $\cO_X(-G^0) \subseteq \tau(\cO_X)$, then does that $G^0$ satisfy \autoref{cor:test-ideal-one-alteration}?  In other words, for any $G \geq G^0$ and $1 \gg \epsilon > 0$ (depending on $G$), do we have that 
    \[
        \tau(\cO_X) = \Tr_{Y/X}\big(\pi_* \cO_Y(K_Y - \pi^*(K_X + \epsilon G))  \big)
    \]
    for some sufficiently large alteration $\pi : Y \to X$ such that $\pi^*(K_X + \epsilon G)$ is Cartier.
    One can ask the same question for pairs and triples $(X, \Delta, \frc^t)$.
\end{question}

Another natural way to try to generalize this paper is to replace the complete DVR $V$ with a complete Noetherian local domain $(R, \fram)$. 

\begin{conjecture}  Suppose $X$ is quasi-projective over a Noetherian complete\footnote{or even excellent with a dualizing complex} local ring $R$.  Then for some $G^0 \geq 0$, all Cartier $G \geq G^0$ and all $1 \gg \epsilon > 0$ (depending on $G$), the intersection over sufficiently large alterations $\pi : Y \to X$
\[
    \bigcap_{\pi : Y \to X} \Tr_{Y/X}\big(\pi_* \cO_Y(K_Y - \pi^* \epsilon G)  \big)
\]
stabilizes.  We expect the same stabilization with pairs $(\omega_X, \Gamma)$ and triples $(\omega_X, \Gamma, \frc^t)$.
\end{conjecture}

One should also expect that the adjoint-variant of the test ideal \cite{MaSchwedeTuckerWaldronWitaszekAdjoint}, \cf \cite{TakagiPLTAdjoint,TakagiAdjointIdealsAndACorrespondence,HaconLamarcheSchwede} can also be defined by a single alteration.  
Suppose $X$ is normal, integral and quasi-projective over a mixed characteristic complete DVR $V$ and $D$ a prime divisor on $X$.  For each alteration $Y \to X$ with fixed $K(Y) \subseteq K(X^+) = \overline{K(X)}$, compatibly pick $D_Y$ a prime divisor dominating $D$ (this is the data of a valuation on $K(X^+)$ extending the DVR corresponding to $D$).  

\begin{conjecture}
      With notation as above, suppose $\Delta$ is a $\bQ$-divisor without common components with $D$, and $K_X + \Delta + D$ is $\bQ$-Cartier.     Then for every sufficiently large effective Cartier divisor $G$ without common components with $D$, and all $1 \gg \epsilon > 0$ (depending on $G$), and all normal alterations $\pi : Y \to X$ (depending on $\epsilon G$), the intersection
    \[
        \adj^D(\cO_X, D+\Delta) = \bigcap_{\pi : Y \to X} \Tr_{Y/X} \big( \pi_* \cO_Y(K_Y - \pi^* (K_X + \Delta + D + \epsilon G) + D_Y) \big)
    \]
    stabilizes and hence, after inverting $p$, agrees with the adjoint ideal in characteristic zero.

    Furthermore, for each closed point $x \in X$ with $R = \widehat{\cO_{X,x}}$, we have that $\adj^D(\cO_X, D+\Delta) \cdot R$ agrees with the adjoint ideal defined in \cite{MaSchwedeTuckerWaldronWitaszekAdjoint} for sufficiently large perfectoid big Cohen-Macaulay algebras.  Finally, up to small perturbation, it also agrees with the adjoint ideal defined in \cite{HaconLamarcheSchwede}.
\end{conjecture}
\appendix

\newpage
\section[Uniform approximation of Abhyankar valuation ideals in mixed characteristic, an appendix by Rankeya Datta]{Uniform approximation of Abhyankar valuation ideals\\ in mixed characteristic\\ An appendix by Rankeya Datta}
\label{Datta-valuation-ideals}

Discovering uniform behavior in Noetherian rings has been a central theme in commutative algebra since the 1980s. Some examples are Koll{\'a}r's effective Nullstellensatz using uniform bounds for annihilators of local cohomology modules \cite{KollarEffectiveNullstellensatz} and the uniform Artin-Rees theorems of O'Carroll \cite{OCarrollUniformAR}, Duncan-O'Carroll \cite{OCarrollDuncanAR} and Huneke \cite{HunekeUniform}. In a different direction, Ein-Lazarsfeld-Smith used multiplier ideals to prove that for radical ideals of the coordinate ring of a smooth characteristic $0$ affine variety, there is a uniform bound for inclusions of symbolic powers of these ideals in their ordinary powers \cite{EinLazSmithSymbolic}. Their result was generalized by Hochster-Huneke to all ideals of all equicharacteristic regular rings using the celebrated theory of tight \cite{HochsterHunekeComparisonOfSymbolic}, by Ma-Schwede to radical ideals of mixed characteristic regular rings with geometrically reduced formal fibers \cite{MaSchwedePerfectoidTestideal} and to all ideals of all regular rings by Murayama \cite{MurayamaUniformBoundsOnSymbolicPowers}. The key tool for \cite{MaSchwedePerfectoidTestideal,MurayamaUniformBoundsOnSymbolicPowers} was a mixed characteristic analog of multiplier ideals \cite{MaSchwedePerfectoidTestideal,MaSchwedeSingularitiesMixedCharBCM,RobinsonBCMTestIdealsMixedCharToric,HaconLamarcheSchwede,BMPSTWW1} whose creation was inspired by the techniques introduced in Andr{\'e}'s proof of the direct summand conjecture \cite{AndreDirectsummandconjecture,AndrePerfectoidAbhyankarLemma, AndreWeaklyFunctorialBigCM} and Scholze's theory of perfectoid spaces \cite{ScholzePerfectoidspaces}. 

After their work on symbolic powers, Ein-Lazarsfeld-Smith utilized multiplier ideals and resolution of singularities in characteristic $0$ to prove another surprising uniformity result for the graded family of valuation ideals of an Abhyankar valuation that is centered on a smooth point of a characteristic $0$ variety \cite{EinLazSmithValuations}. As a consequence, they also deduced Izumi's comparison result for a pair of Abhyankar valuations that share a common smooth center. The author later proved the prime characteristic versions of the \cite{EinLazSmithValuations} results using the theory of test ideals \cite{DattaUniformApproximation}. The purpose of this appendix is to use the theory of mixed characteristic test ideals developed in this paper to prove a mixed characteristic version of \cite{EinLazSmithValuations,DattaUniformApproximation}. The main result is:

\begin{theorem}[see \autoref{thm:uniform-approximation-Abhyankar-valuation-ideals}]
\label{thm:uniform-approximation-Abhyankar-valuation-ideals-intro}
Let $(V,\fram_V,\kappa_V)$ be a DVR of mixed characteristic $(0,p)$ such that $\kappa_V$ is perfect. Let $(B,\fram_B,\kappa_B)$ be a regular local ring that is an essentially of finite type faithfully flat $V$-algebra. Let $F$ be the fraction field of $B$ and $\nu$ be an $\bR$-valued valuation of $F$ such that $(B,\fram_B,\kappa_B)$ is an Abhyankar center of $\nu$. Let $\Gamma_\nu$ be the value group of $\nu$ and suppose $\Gamma_\nu/\nu(\Frac(V)^\times)$ is torsion-free. Then there exists $e \in \mathbb{R}_{\geq 0}$ such that for all $m \in \mathbb{R}$ and $\ell \in \mathbb{Z}_{> 0}$, 
\[\ba_{\nu,m}(B)^\ell \subseteq \ba_{\nu,\ell m}(B) \subseteq \ba_{\nu,m-e}(B)^\ell.\]
\end{theorem}

We refer the reader to \autoref{subsec:Abhyankar valuations and Abhyankar centers} for the definition of Abhyankar centers. Here $\ba_{\nu,m}(B)$ is the \emph{valuation ideal} of $B$ consisting of those elements $x \in B$ such that $\nu(x) \geq m$. The inclusion $\ba_{\nu,m}(B)^\ell \subseteq \ba_{\nu,\ell m}(B)$ holds always. The interesting assertion is that there is an $\ell$-th power valuation ideal that contains $\ba_{\nu,\ell m}(B)$ in a uniform manner.  

The first key ingredient for the proof of \autoref{thm:uniform-approximation-Abhyankar-valuation-ideals-intro} is a mixed characteristic local monomialization result of Knaf-Kuhlmann \cite{KnafKuhlmannAbhyankar} for Abhyankar valuations. This result is recalled in the form we need in \autoref{thm:local-uniformization}. We need to assume torsion-freeness of $\Gamma_\nu/\nu(\Frac(V)^\times)$ in \autoref{thm:uniform-approximation-Abhyankar-valuation-ideals-intro} because of its importance in Knaf-Kuhlmann's monomialization result. 

The second key ingredient is the robust theory of mixed characteristic test ideals developed in this paper. We use this theory to define an asymptotic version of the test ideal of an ideal pair (\autoref{def:asymptotic-test-ideals}), show a sub-additivity property for asymptotic test ideals (\autoref{lem:asymptotic-test-ideals-properties}\autoref{lem:asymptotic-test-ideals-properties:subadditivity}) that follows from the sub-additivity property established in \autoref{thm.FinalPropertiesOfTestIdealsTriples}\autoref{thm.FinalPropertiesOfTestIdealsTriples.Subadditivity}, and then reduce \autoref{thm:uniform-approximation-Abhyankar-valuation-ideals-intro} to showing that a certain infinite intersection of colon ideals is non-zero (\autoref{lem:asymptotic-test-ideals-properties}\autoref{lem:asymptotic-test-ideals-properties:reduction-main-thm}). We tackle the non-triviality of the intersection of colon ideals in \autoref{prop:colon-intersection-not-zero}, whose proof proceeds by showing that one can always pass to birational extensions using the birational transformation law for test ideals of ideal pairs (\autoref{thm.FinalPropertiesOfTestIdealsTriples}\autoref{thm.FinalPropertiesOfTestIdealsTriples.BiggerIdeals-Exponents}). This allows us to further reduce to the case where the valuation ideals are monomial in a regular system of parameters via Knaf-Kuhlmann's monomialization result. One finishes the proof by using the `expected' characterization of the mixed characteristic test ideal of a monomial ideal (\autoref{prop.HowaldsTheoremForTau}) in terms of the convex hull of the tuples of exponents of the monomials in the ideal. The characteristic $0$ version of \autoref{prop.HowaldsTheoremForTau} was proved by Howald for multiplier ideals of monomial ideals \cite{HowaldMultiplierIdeals} and the positive characteristic version by Hara-Yoshida for test ideals of monomial ideals \cite{HaraYoshidaGeneralizationOfTightClosure}. For a variant of \autoref{prop.HowaldsTheoremForTau} in mixed characteristic, see also the work of Robinson \cite{RobinsonBCMTestIdealsMixedCharToric}.


We highlight an important corollary of \autoref{thm:uniform-approximation-Abhyankar-valuation-ideals-intro}, namely an Izumi theorem for valuations sharing a common center. We note that this Izumi result has been established in greater generality in \cite{SpivakovskyRondIzumi}. 

\begin{theorem}[see \autoref{thm:Izumi}]
Let $(V,\fram_V,\kappa_V)$ be a DVR of mixed characteristic $(0,p)$ with perfect residue field $\kappa_V$. Let $(B,\fram_B,\kappa_B)$ be a a regular local ring that is an essentially of finite type faithfully flat $V$-algebra. Let $\nu, w$ be two $\bR$-valued valuations of $\Frac(B)$ centered on $B$. Assume the following:
\begin{enumerate}[(i)]
\item $B$ is regular and an Abhyankar center of $\nu$.
\item $\Gamma_\nu/\nu(\Frac(V)^\times)$ is torsion-free.
\end{enumerate}
Then there exists a real number $C > 0$ such that for all $x \in B - \{0\}$, $\nu(x) \leq Cw(x)$.
\end{theorem}

\subsection{Abhyankar valuations and Abhyankar centers}
\label{subsec:Abhyankar valuations and Abhyankar centers}
Throughout this appendix, $F/K$ will denote a finitely generated field extension. Let $w$ be a valuation of $F$ with valuation ring $(\sO_w,\fram_w,\kappa_w)$ and let $\nu$ be the restriction of $w$ to $K$ with valuation ring $(\sO_\nu,\fram_\nu,\kappa_\nu)$. We will \emph{not} assume that $\nu$ is the trivial valuation of $K$. We let $\Gamma_w$ (resp. $\Gamma_\nu$) be the value group of $w$ (resp. $\nu$). Then we have the following fundamental inequality \cite[VI.10.3, Corollary\ 1]{Bourbaki1998}:
\begin{equation}
\label{eq:Abhyankar-inequality}
\dim_{\bQ}(\Gamma_w/\Gamma_\nu \otimes_{\bZ} \bQ) + \td{\kappa_w/\kappa_\nu} \leq \td{F/K}.
\end{equation}
Moreover, if equality holds in the above inequality that $\Gamma_w/\Gamma_\nu$ is a finitely generated abelian group and $\kappa_w/\kappa_\nu$ is a finitely generated field extension.

\begin{definition}
\label{def:Abhyankar-valuation}
With $F/K$ and $w$ as above, we say $w$ is an \emph{Abhyankar valuation} of $F/K$ if equality holds in \autoref{eq:Abhyankar-inequality}. 
\end{definition}

\begin{remark}
\label{rem:finitely-generated-value-group}
If $w$ is an Abhyankar valuation of $F/K$ and $\Gamma_\nu$ is a finitely generated (torsion free) abelian group, then $\Gamma_w$ will be a finitely generated (torsion free) abelian group by the discussion above. 
\end{remark}

We say that $w$ is \emph{centered} on a local domain $(A,\fram_A,\kappa_A)$ or that $(A,\fram_A,\kappa_A)$ is a \emph{center} of $w$ if $(\sO_w,\fram_w,\kappa_w)$ dominates $(A,\fram_A,\kappa_A)$. Said differently, $w$ is centered on $(A,\fram_A,\kappa_A)$ if for all $x \in A - \{0\}$, $w(x) \geq 0$ and for all $x \in \fram_A - \{0\}$, $w(x) > 0$. More generally, if $A$ is a domain that is not necessarily local, we say that \emph{$w$ is centered on $A$ at $\frp \in \Spec(A)$} if $A \subseteq \sO_w$ and $w$ dominates the local domain $A_\frp$ in the above sense. The last definition generalizes to integral schemes in the obvious way.

Closely related to \autoref{eq:Abhyankar-inequality} is the following fundamental result due to Abhyankar, which allows one to generalize the notion of an Abhyankar valuation.

\begin{theorem}\cite{AbhyankarValuationsCenteredLocal}
\label{thm:Abhyankar-inequality-center}
Let $w$ be a valuation of a field $F$ with valuation ring $(\sO_w,\fram_w,\kappa_w)$ and value group $\Gamma_w$. Suppose $w$ is centered on a Noetherian local domain $(A,\fram_A,\kappa_A)$ such that $\Frac(A) = F$. Then 
\begin{equation}
\label{eq:Abhyankar-inequality-center}
\dim_{\bQ}(\Gamma_w \otimes_{\bZ} \bQ) + \td{\kappa_w/\kappa_A} \leq \dim(A).
\end{equation}
If equality holds in the above inequality then $\Gamma_w$ is a finitely generated abelian group and $\kappa_w/\kappa_A$ is a finitely generated field extension.
\end{theorem}

\begin{definition}
\label{def:Abhyankar-center}
In the situation of \autoref{thm:Abhyankar-inequality-center}, we will say that $(A,\fram_A,\kappa_A)$ is an \emph{Abhyankar center of $w$} if equality holds in \autoref{eq:Abhyankar-inequality-center}. More generally, if $X$ is a locally noetherian integral scheme with function field $F$ and $w$ is centered on $x \in X$, then we say $x$ is an \emph{Abhyankar center of $w$} if $\sO_{X,x}$ is an Abhyankar center of $X$.
\end{definition}

\begin{example}
\label{eg:discrete-valuations-Abhyankar-center}
Suppose $w$ is a rank $1$ discrete valuation of $F$, that is, $\Gamma_w \cong \bZ$. Then $\sO_w$ is an Abhyankar center of itself.
\end{example}

\autoref{def:Abhyankar-valuation} and \autoref{def:Abhyankar-center} are close related in the following setting.

\begin{proposition}
\label{prop:Abhyankar-valuation-versus-center}
Let $F/K$ be a finitely generated field extension, $w$ be a valuation of $F$ and $\nu$ denote the restriction of $w$ to $K$. Suppose $\nu$ is centered on a Noetherian local domain $(R,\fram_R,\kappa_R)$ such that $\Frac(R) = K$, $R$ is universally catenary and $R$ is an Abhyankar center of $\nu$. Then the following are equivalent:
\begin{enumerate}
\item $w$ is an Abhyankar valuation of $F/K$.\smallskip
\item For every center $(A,\fram_A,\kappa_A)$ of $w$ with fraction field $F$ such that $A$ is an essentially of finite type local extension of $R$, $A$ is an Abhyankar center of $w$.\smallskip
\item There exists an Abhyankar center $(A,\fram_A,\kappa_A)$ of $w$ with fraction field $F$ such that $A$ is an essentially of finite type local extension of $R$.
\end{enumerate}
\end{proposition}

\begin{proof}
Since $F/K$ is finitely generated, it is always possible to construct a center $(A,\fram_A,\kappa_A)$ of $w$ with fraction field $F$ such that $A$ is an essentially of finite type local extension of $R$. We omit the details.

Since $R$ is an Abhyankar center of $\nu$, we have 
\begin{equation}
\label{eq:Abhyankar-center}
\dim_{\bQ}(\Gamma_\nu \otimes_{\bZ} \bQ) + \td{\kappa_\nu/\kappa_R} = \dim(R).
\end{equation}
Furthermore, since $R$ is universally catenary, if $(R,\fram_R,\kappa_R) \hookrightarrow (A,\fram_A,\kappa_A)$ is a local essentially of finite type extension of domains with $\Frac(A) = F$, then by the dimension formula \cite[\href{https://stacks.math.columbia.edu/tag/02IJ}{Tag 02IJ}]{stacks-project} we have
\begin{equation}
\label{eq:catenary}
\dim(A) = \dim(R) + \td{F/K} - \td{\kappa_A/\kappa_R}.
\end{equation}
Using the diagram of field extensions
\[\begin{tikzcd}[cramped]
	{\kappa_R} && {\kappa_A} \\
	\\
	{\kappa_\nu} && {\kappa_w}
	\arrow[hook, from=1-1, to=1-3]
	\arrow[hook, from=1-1, to=3-1]
	\arrow[hook, from=1-3, to=3-3]
	\arrow[hook, from=3-1, to=3-3]
\end{tikzcd}\]
we also have 
\begin{align*}
\td{\kappa_w/\kappa_A} &\stackrel{}{=} \td{\kappa_w/\kappa_\nu} + \td{\kappa_\nu/\kappa_R} - \td{\kappa_A/\kappa_R}\\
&\stackrel{\autoref{eq:catenary}}{=} \td{\kappa_w/\kappa_\nu} + \td{\kappa_\nu/\kappa_R} + \dim(A) - \dim(R) - \td{F/K}\\
&\stackrel{\autoref{eq:Abhyankar-center}}{=} \td{\kappa_w/\kappa_\nu} + \dim(A) -\dim_{\bQ}(\Gamma_\nu \otimes_{\bZ}\bQ) - \td{F/K}.
\end{align*}
Since $\dim_{\bQ}(\Gamma_w/\Gamma_\nu \otimes_{\bZ} \bQ) = \dim_{\bQ}(\Gamma_w \otimes_{\bZ} \bQ) - \dim_{\bQ}(\Gamma_\nu \otimes_{\bZ} \bQ)$, we then get
\[ \dim_{\bQ}(\Gamma_w \otimes_{\bZ} \bQ) + \td{\kappa_w/\kappa_A} = \dim(A) + \dim_{\bQ}(\Gamma_w/\Gamma_\nu \otimes_{\bZ} \bQ) + \td{\kappa_w/\kappa_\nu} - \td{F/K}.\]
From this last equality it is clear that the equality
\[\dim_{\bQ}(\Gamma_w \otimes_{\bZ} \bQ) + \td{\kappa_w/\kappa_A} = \dim(A)\] 
holds if and only if the equality
\[\dim_{\bQ}(\Gamma_w/\Gamma_\nu \otimes_{\bZ} \bQ) + \td{\kappa_w/\kappa_\nu} = \td{F/K}\] 
holds, proving the equivalence of $(a), (b), (c)$. 
\end{proof}

Since regular local rings are universally catenary, we immediately obtain:

\begin{corollary}
Let $F/K$ be a finitely generated field extension, $w$ be a valuation of $F$ and $\nu$ denote the restriction of $w$ to $K$. Suppose $\sO_\nu$ is a DVR. Then the following are equivalent:
\begin{enumerate}
\item $w$ is an Abhyankar valuation of $F/K$.\smallskip
\item For every center $(A,\fram_A,\kappa_A)$ of $w$ with fraction field $F$ such that $A$ is an essentially of finite type local extension of $\sO_\nu$, $A$ is an Abhyankar center of $w$.\smallskip
\item There exists an Abhyankar center $(A,\fram_A,\kappa_A)$ of $w$ with fraction field $F$ such that $A$ is an essentially of finite type local extension of $\sO_\nu$.
\end{enumerate}
\end{corollary}

The key fact about Abhyankar valuations is that they satisfy local uniformization in arbitrary characteristic \cite{KnafKuhlmannAbhyankar}. We state the result in the generality we need.

\begin{theorem}\cite[Theorem\ 1.2]{KnafKuhlmannAbhyankar}
\label{thm:local-uniformization}
Let $F/K$ be a finitely generated field extension. Let $w$ be a valuation of $F$ with value group $\Gamma_w$ and valuation ring $(\sO_w,\fram_w,\kappa_w)$. Let $\nu$ be the restriction of $w$ to $K$ with valuation ring $(\sO_\nu,\fram_\nu,\kappa_\nu)$. Assume that $\sO_\nu$ is an excellent DVR, $w$ is an Abhyankar valuation of $F/K$, $\Gamma_w/\Gamma_\nu$ is a torsion-free abelian group and $\kappa_w/\kappa_\nu$ is a separable extension (for example, if $\kappa_\nu$ is perfect). Then for any finite set $Z \subset \sO_w$, $w$ is centered on a local domain $(A,\fram_A,\kappa_A)$ satisfying all of the following properties:
\begin{enumerate}
\item $A$ is an essentially of finite type local extension of $\sO_\nu$ and $\Frac(A) = F$,
\item $A$ is regular local and $\dim(A) = \dim_{\bQ}(\Gamma_w \otimes_{\bZ} \bQ)$,
\item $Z \subset A$ and there exists a regular system of parameters $x_1,\dots,x_d$ of $A$ such that for all $z \in Z$,
\[z = ux_1^{a_1}\cdots x_d^{a_d}\]
for some $u \in A^\times$ and $a_1,\dots,a_d \in \bZ_{\geq 0}$.
\end{enumerate}
\end{theorem}

\begin{remark}
In \cite[Theorem\ 1.2]{KnafKuhlmannAbhyankar} the assumption is that the valued field $(K,\nu)$ satisfies the technical condition of being \emph{defectless}. However, as explained on \cite[Page 835]{KnafKuhlmannAbhyankar}, if $\sO_\nu$ is an excellent DVR then $(K,\nu)$ is automatically defectless. Thus, we work with the alternate condition of excellence, which is automatic for mixed characteristic DVRs.
\end{remark}

\autoref{thm:local-uniformization} implies the following:

\begin{proposition}
\label{prop:birational-domination-regular-local}
Let $F/K$ be a finitely generated field extension. Let $w$ be a valuation of $F$ with value group $\Gamma_w$ and valuation ring $(\sO_w,\fram_w,\kappa_w)$. Let $\nu$ be the restriction of $w$ to $K$ with valuation ring $(\sO_\nu,\fram_\nu,\kappa_\nu)$. Assume that $\sO_\nu$ is an excellent DVR, $\Gamma_w/\Gamma_\nu$ is torsion-free and $\kappa_w/\kappa_\nu$ is separable. Suppose $w$ is centered on a local domain $(R,\fram_R,\kappa_R)$ such that $R$ is an essentially of finite type local extension of $\sO_\nu$, $\Frac(R) = F$ and $R$ is an Abhyankar center of $w$. Then there exists a local birational extension of domains
\[(R,\fram_R,\kappa_R) \hookrightarrow (A,\fram_A,\kappa_A)\]
such that 
\begin{enumerate}[(i)]
\item $w$ is centered on $A$,\label{prop:birational-domination-regular-local.i}
\item $A$ is an essentially of finite type local extension of $\sO_\nu$,\label{prop:birational-domination-regular-local.ii}
\item $A$ is regular and $\dim(A) = \dim_{\bQ}(\Gamma_w \otimes_{\bZ} \bQ)$,\label{prop:birational-domination-regular-local.iii}
\item $\kappa_A = \kappa_w$ and $A$ has a regular system of parameters $\{x_1,\dots,x_d\}$ with the property that $\{w(x_1),\dots,w(x_d)\}$ is a $\bZ$-basis of $\Gamma_w$.\label{prop:birational-domination-regular-local.iv}
\end{enumerate}
\end{proposition}

\begin{proof}
$w$ is an Abhyankar valuation of $F/K$ by \autoref{prop:Abhyankar-valuation-versus-center}. Choose a finitely generated $\sO_\nu$-subalgebra, $\sO_\nu[r_1,\dots,r_n]$, of $F$ such that $R = \sO_\nu[r_1,\dots,r_n]_\frp$ for $\frp \in \Spec(\sO_\nu[r_1,\dots,r_n])$. 
 Since $w$ is an Abhyankar valuation of $F/K$, $\Gamma_w$ is a finitely generated torsion-free abelian group by \autoref{rem:finitely-generated-value-group}. Thus, one can choose a $\bZ$-basis $\{\alpha_1,\dots,\alpha_d\}$ of $\Gamma_w$ such that $\alpha_i > 0$. Choose $y_i \in \sO_w$ such that $w(y_i) = \alpha_i$ for all $i$. Furthermore, let $\{z_1,\dots,z_m\}$ be generators of the field extension $\kappa_w/\kappa_\nu$, which is finitely generated because $w$ is an Abhyankar valuation of $F/K$. Let $\overline{z_i} \in \sO_w^\times$ be a lift of $z_i$. Let
\[Z \coloneqq \{r_1,\dots,r_n,y_1,\dots,y_d,\overline{z_1},\dots,\overline{z_m}\}.\]
Applying \autoref{thm:local-uniformization}, we can find a regular local ring $(A,\fram_A,\kappa_A)$ with fraction field $F$ that satisfies \autoref{prop:birational-domination-regular-local.i}, \autoref{prop:birational-domination-regular-local.ii}, \autoref{prop:birational-domination-regular-local.iii} as well as the condition that $Z \subset A$ and every element of $Z$ is a monomial upto a unit in $A$ with respect to a regular system of parameters $\{x_1,\dots,x_d\}$ of $A$. Since $\{w(y_1),\dots,w(y_d)\}$ is a $\bZ$-basis of $\Gamma_w$, it follows that $\{w(x_1),\dots,w(x_d)\}$ generates $\Gamma_w$ as well, and so, must be a $\bZ$-basis of $\Gamma_w$. By construction, $\sO_\nu[r_1,\dots,r_n] \subset A$. Thus, since $w$ is centered on $A$, it follows that $\fram_A \cap \sO_\nu[r_1,\dots,r_n] = \fram_w \cap \sO_\nu[r_1,\dots,r_n] = \frp$. Hence, we get an induced local birational extension $R = \sO_\nu[r_1,\dots,r_n]_\frp \hookrightarrow A$. Finally, we certainly have $\kappa_A \subset \kappa_w$. But $\overline{z_1},\dots,\overline{z_n}$ must be units in $A$ since $(\sO_w,\fram_w,\kappa_w)$ dominates $(A,\fram_A,\kappa_A)$ and $\overline{z_1},\dots,\overline{z_n}$ are units in $\sO_w$. Thus, $\kappa_A = \kappa_w$ because the classes of the $\overline{z_i}$ in $\kappa_w$ generate $\kappa_w$ over $\kappa_\nu$, and so, \autoref{prop:birational-domination-regular-local.iv} is satisfied as well.
\end{proof}

\subsection{Real-valued valuations admitting an Abhyankar center} 

\begin{definition}
\label{def:valuation-ideals}
Let $\nu$ be an $\bR$-valued valuation of a field $K$. Suppose $\nu$ is centered on a domain $A$. Then for $m \in \bR$ the \emph{valuation ideal of $A$ of order $m$ along $\nu$}, denoted $\ba_{\nu,m}(A)$, is defined as follows:
\[\ba_{\nu,m}(A) \coloneqq \{a \in A \colon \nu(a) \geq m\} \cup \{0\}.\]
If $\nu, A$ are clear from context then we just use $\ba_m$ instead of $\ba_{\nu,m}(A)$.
\end{definition}

The collection of valuation ideals $\{\ba_m \colon m \in \bR_{\geq 0}\}$ satisfies the following properties:
\begin{enumerate}
\item[(1)] For $m,n \in \bR$, $\ba_m\cdot\ba_n \subseteq \ba_{m+n}$.
\item[(2)] For $m,n \in \bR$, if $n > m$ then $\ba_n \subseteq \ba_m$.
\item[(3)] If $\nu$ is centered at $\frp \in \Spec(A)$ and $\frp$ is finitely generated by $x_1,\dots,x_d$, then for $m \coloneqq \min\{\nu(x_1),\dots,\nu(x_d)\}$, we have $\ba_{\nu,m}(A) = \frp$.
\end{enumerate}
$(1)$ and $(2)$ imply that $\{\ba_m \colon m \in \bR_{\geq 0}\}$ is a \emph{graded family} of ideals of $A$. In addition, as a consequence of the Archimedean property of $\bR$, the following result is straightforward.

\begin{lemma}
\label{lem:valuation-ideals-primary}
Let $A$ be a domain with fraction field $K$ and $\frp \in \Spec(A)$. Suppose $\nu$ is a $\bR$-valued valuation of $K$ that is centered on $A$ at $\frp$. Then for all $m \in \bR_{> 0}$, $\ba_{\nu,m}(A)$ is a $\frp$-primary ideal of $A$. Furthermore, $\ba_{\nu,m}(A)A_\frp = \ba_{\nu,m}(A_\frp)$.
\end{lemma}

The other result we will need is a precise characterization of the valuation ideals $\ba_m$ when the center $A$ has the properties of \autoref{prop:birational-domination-regular-local}.

\begin{lemma}
\label{lem:valuation-ideals-monomial}
Let $\nu$ be a $\bR$-valued valuation of a field $K$. Suppose that $\nu$ is centered on a Noetherian local domain $(A,\fram_A,\kappa_A)$ such that $\{x_1,\dots,x_d\}$ generates $\fram_A$ and has the property that $\{\nu(x_1),\dots,\nu(x_d)\}$ is $\bZ$-linearly independent. Then for all $m \in \bR_{\geq 0}$,
\[\ba_m(A) = \left(x_1^{a_1}\cdots x_d^{a_d} \colon \sum_{i=1}^d a_i\nu(x_i) \geq m, a_i \in \mathbb{Z}_{\geq 0}\right)\]
\end{lemma}

\begin{proof}
We may assume $m > 0$ as otherwise the result is true because $x_1^0\cdots x_d^0 = 1 \in \ba_0(A)$. Then we know by \autoref{lem:valuation-ideals-primary} and the fact that $A$ is Noetherian that there exists an integer $n > 0$ such that $\fram_A^n \subset \ba_m(A)$. We know that $\fram_A^n$ is generated by monomials in $x_1,\dots,x_d$ of total degree $n$. Moreover, every element of $A/\fram_A^n$ is the class of a polynomial $f$ in $x_1,\dots,x_d$ of total degree $\leq n-1$ whose non-zero terms have coefficients in $A^\times$. Suppose $ux^{a_1}_1\cdots x^{a_d}_d$ and $vx^{b_1}_1\cdots x^{b_d}_d$ are two non-zero terms of $f$ with $u, v \in A^\times$ such that $(a_1,\dots,a_d) \neq (b_1,\dots,b_d)$.
    Then 
\[\nu(ux^{a_1}_1\cdots x^{a_d}_d) = \sum_{i=1}^d a_i\nu(x_i) \neq \sum_{i=1}^d b_i\nu(x_i) = \nu(vx^{b_1}_1\cdots x^{b_d}_d),\]
because $\{\nu(x_1),\dots,\nu(x_d)\}$ is $\bZ$-linearly independent. Thus, $\nu(f)$ is the infimum of the valuations of its non-zero terms, and so, in order for $\nu(f) \geq m$ to hold, each non-zero term of $f$ must have order at least $m$ along $\nu$, that is, each term of $f$ must be in $\ba_m(A)$. This proves that $\ba_m(A)$ is generated by monomials in $x_1,\dots,x_d$.
\end{proof}

\subsection{Uniform approximation of Abhyankar valuation ideals} We now prove the main result, which is:

\begin{theorem}
\label{thm:uniform-approximation-Abhyankar-valuation-ideals}
Let $(V,\fram_V,\kappa_V)$ be a DVR of mixed characteristic $(0,p)$ such that $\kappa_V$ is perfect. Let $(B,\fram_B,\kappa_B)$ be a regular local ring that is an essentially of finite type extension domain of $V$ that dominates $V$. Let $F$ be the fraction field of $B$ and $\nu$ be a $\bR$-valued valuation of $F$ such that $(B,\fram_B,\kappa_B)$ is an Abhyankar center of $\nu$. Suppose $\Gamma_\nu/\nu(\Frac(V)^\times)$ is torsion-free. Then there exists $e \in \mathbb{R}_{\geq 0}$ such that for all $m \in \mathbb{R}$ and $\ell \in \mathbb{Z}_{> 0}$, 
\[\ba_{\nu,m}(B)^\ell \subseteq \ba_{\nu,\ell m}(B) \subseteq \ba_{\nu,m-e}(B)^\ell.\]
\end{theorem}

\begin{remarks}
\label{rem:reductions-main-theorem}
We make a few comments and reductions. 
\begin{enumerate}
    \item Since $B$ dominates $V$ and $\nu$ is centered on $B$, it follows that $\nu$ is also centered on $V$. Thus, if $\nu'$ is the restriction of $\nu$ to $K \coloneqq \Frac(V)$, then the valuation ring of $\nu'$ is $V$, and so, $V$ is an Abhyankar center of $\nu'$. Since a DVR is universally catenary, it follows by \autoref{prop:Abhyankar-valuation-versus-center} that $\nu$ is an Abhyankar valuation of the finitely generated field extension $F/K$. Thus, $\kappa_\nu/\kappa_V$ is a finitely generated, and hence, a separable extension since $\kappa_V$ is perfect.\label{rem:reductions-main-theorem.a}
    
    
    \item One can weaken the assumption that $\kappa_V$ is perfect if one assumes instead that $\kappa_\nu/\kappa_V$ is separable since this is the assumption one needs in the local monomialization theorem of Knaf-Kuhlmann \autoref{thm:local-uniformization}.\label{rem:reductions-main-theorem.b} 

    \item The inclusion $\ba_{\nu,m}(B)^\ell \subseteq \ba_{\nu,\ell m}(B)$ clearly holds always. \label{rem:reductions-main-theorem.c}
    
    \item We can write $B = R_\frp$, where $R \subseteq B$ is a finitely generated $V$-algebra and $\frp \in \Spec(R)$. Since $R$ is excellent (and hence has open regular locus) and $
   B = R_\frp$ is regular, we may assume $R$ is also regular. Moreover, by \autoref{lem:valuation-ideals-primary}, we have $\ba_{\nu,m}(R)B = \ba_{\nu,m}(B)$. Consequently, if we can show that there exists $e \in \mathbb{R}_{\geq 0}$ such that for all $m \in \mathbb{R}$ and $\ell \in \mathbb{Z}_{> 0}$, 
\[\ba_{\nu,\ell m}(R) \subseteq \ba_{\nu,m-e}(R)^\ell\]
then the conclusion for $B$ would follow by ideal expansion. Note that $\Spec(R)$ is a quasi-projective scheme over the DVR $V$. \label{rem:reductions-main-theorem.d}
\end{enumerate}
\end{remarks}

With $R$ as in \autoref{rem:reductions-main-theorem.d} above, for an ideal $\bb$ of $R$ and a rational number $t \geq 0$, we now consider the \emph{test ideal} 
\[\tau(R,\bb^t),\] 
which we define to be the global sections of the test ideal sheaf $\tau(\sO_{\Spec(R)},\tilde{\bb}^t)$ introduced in \autoref{def.UltimateTestIDEALDefinition} with $\Delta = 0$. Note $K_{\Spec(R)}$ is Cartier because $R$ is regular. Here $\tilde{\bb}$ denotes the ideal sheaf associated with the ideal $\bb$. 

We will next define an \emph{asymptotic} version of the above test ideal for the graded family of valuation ideals $\{\ba_{\nu,m}(R) \colon m \in \bR\}$. Fix $m \in \bR$. For any integer $\ell > 0$, we have $\ba_{\nu,m}(R)^\ell \subset \ba_{\nu,\ell m}(R)$. Thus, using \autoref{thm.FinalPropertiesOfTestIdealsTriples} parts \autoref{thm.FinalPropertiesOfTestIdealsTriples.Unambiguity} and \autoref{thm.FinalPropertiesOfTestIdealsTriples.BiggerIdeals-Exponents}, we get
\[\tau(R,\ba_m^t) = \tau(R,(\ba_m^\ell)^{t/\ell}) \subseteq \tau(R,\ba_{\ell m}^{t/\ell}).\]
Thus, for a fixed $m$ and a fixed $t \in \bQ_{\geq 0}$, the collection $\{\tau(R,\ba_{\ell m}^{t/\ell}) \colon \ell \in \mathbb{Z}_{> 0}\}$ is filtered under inclusion. For example, $\tau(R,\ba_{\ell_1 m}^{t/\ell_1}), \tau(R,\ba_{\ell_2 m}^{t/\ell_2})$ are both contained in $\tau(R,\ba_{\ell_1\ell_2 m}^{t/\ell_1\ell_2})$. Since $R$ is Noetherian, $\{\tau(R,\ba_{\ell m}^{t/\ell}) \colon \ell \in \mathbb{Z}_{> 0}\}$ has a unique largest element under $\subseteq$, which must also coincide with $\sum_{\ell \in \mathbb{Z}_{> 0}} \tau(R,\ba_{\ell m}^{t/\ell})$. This leads to the following definition, upon taking $t = 1$.

\begin{definition}
\label{def:asymptotic-test-ideals}
For the graded family of valuation ideals $\ba_\bullet = \{\ba_{\nu,m}(R) \colon m \in \bR\}$ and $m \in \bR$, we define the \emph{$m$-th asymptotic test ideal of $\ba_\bullet$}, denoted $\tau(R,m\cdot\ba_\bullet)$, to be 
\[\tau(R,m\cdot\ba_\bullet) \coloneqq \sum_{\ell \in \mathbb{Z}_{> 0}} \tau(R,\ba_{\ell m}^{1/\ell})\]
\end{definition}

By the above discussion, $\tau(R,m\cdot\ba_\bullet) = \tau(R,\ba_{\ell m}^{1/\ell})$ for a highly divisible integer $\ell \gg 0$.

\begin{remark}
\label{rem:graded-family-more-general}
Asymptotic test ideals can be defined for any graded family of ideals. However, we will focus only on the case of graded families of valuation ideals.
\end{remark}

\begin{lemma}
\label{lem:asymptotic-test-ideals-properties}
Let $(V,\fram_V,\kappa_V)$ be a DVR of mixed characteristic $(0,p)$ and $R$ be a finite type $V$-algebra extension domain. Suppose $R$ is regular and $\nu$ is an $\bR$-valued valuation of $\Frac(R)$ that is centered at $\frp \in \Spec(R)$ such that $\frp \cap V = \fram_V$. Let $\ba_\bullet = \{\ba_m \colon m \in \bR\}$ be the graded family of valuation ideals of $R$ associated with $\nu$. Then we have the following:
\begin{enumerate}
 \item For all $m \in \bR$, $\ba_m \subseteq \tau(R,\ba_m) \subseteq \tau(R,m\cdot \ba_\bullet)$.\label{lem:asymptotic-test-ideals-properties:containment}\smallskip
 \item For all $m \in \bR$ and $\ell \in \bZ_{> 0}$, $\ba_{\ell m} \subseteq \tau(R,(\ell m)\cdot\ba_\bullet) \subseteq \tau(R,m\cdot\ba_\bullet)^\ell$.\label{lem:asymptotic-test-ideals-properties:subadditivity}\smallskip
 \item If $\bigcap_{m \in \bR} (\ba_m \colon \tau(R,m\cdot\ba_\bullet)) \neq 0$, then there exists $e \in \mathbb{R}_{\geq 0}$ such that for all $m \in \mathbb{R}$ and $\ell \in \mathbb{Z}_{> 0}$, $\ba_{m}^\ell \subseteq \ba_{\ell m} \subseteq \ba_{m-e}^\ell$.\label{lem:asymptotic-test-ideals-properties:reduction-main-thm}
\end{enumerate}
\end{lemma}

\begin{proof}
\autoref{lem:asymptotic-test-ideals-properties:containment} Since $R$ is regular, $\ba_m \subseteq \tau(R,\ba_m)$ follows by \autoref{thm.FinalPropertiesOfTestIdealsTriples}\autoref{thm.FinalPropertiesOfTestIdealsTriples.EasyContainmentInBCMRegular}, while $\tau(R,\ba_m) \subseteq \tau(R,m\cdot\ba_\bullet)$ follows by \autoref{def:asymptotic-test-ideals}.

\autoref{lem:asymptotic-test-ideals-properties:subadditivity} $\ba_{\ell m} \subseteq \tau(R,(\ell m)\cdot\ba_\bullet)$ follows by \autoref{lem:asymptotic-test-ideals-properties:containment}. Choose a highly divisible integer $n > 0$ such that $\tau(R,(\ell m)\cdot \ba_\bullet) = \tau(R,\ba_{n\ell m}^{1/n})$ \text{ and } $\tau(R,m\cdot\ba_\bullet) = \tau(R,\ba_{n\ell m}^{1/n\ell})$. Then 
\begin{align*}
\tau(R,(\ell m)\cdot \ba_\bullet) = \tau(R,\ba_{n\ell m}^{1/n}) &= \tau(R,\ba_{n\ell m}^{\ell/n\ell})\\
&\stackrel{\autoref{thm.FinalPropertiesOfTestIdealsTriples}\autoref{thm.FinalPropertiesOfTestIdealsTriples.Unambiguity}}{=} \tau(R,\underbrace{\ba_{n\ell m}^{1/n\ell}\cdots\ba_{n\ell m}^{1/n\ell}}_{\ell-\text{times}}))\\
&\stackrel{\autoref{thm.FinalPropertiesOfTestIdealsTriples}\autoref{thm.FinalPropertiesOfTestIdealsTriples.Subadditivity}}{\subseteq} \tau(R,\ba_{n\ell m}^{1/n\ell})^\ell
= \tau(R,m\cdot\ba_\bullet)^\ell, 
\end{align*}
as desired.

\autoref{lem:asymptotic-test-ideals-properties:reduction-main-thm} Let $x \in \bigcap_{m \in \bR} (\ba_m \colon \tau(R,m\cdot\ba_\bullet))$ be a non-zero element. Let $e \coloneqq \nu(x)$. Since $x\cdot \tau(R,m\cdot \ba_{\bullet}) \subseteq \ba_{m}$, we must have $\tau(R,m\cdot\ba_\bullet) \subseteq \ba_{m-e}$, and so,
\[
\ba_{\ell m} \stackrel{\autoref{lem:asymptotic-test-ideals-properties:containment}}{\subseteq} \tau(R,(\ell m)\cdot \ba_\bullet) \stackrel{\autoref{lem:asymptotic-test-ideals-properties:subadditivity}}{\subseteq} \tau(R,m\cdot\ba_\bullet)^\ell \subseteq \ba_{m-e}^\ell,
\]
proving \autoref{lem:asymptotic-test-ideals-properties:reduction-main-thm}.
\end{proof}

Using \autoref{lem:asymptotic-test-ideals-properties}\autoref{lem:asymptotic-test-ideals-properties:reduction-main-thm} the proof of \autoref{thm:uniform-approximation-Abhyankar-valuation-ideals} reduces to:

\begin{proposition}
\label{prop:colon-intersection-not-zero}
Let $(V,\fram_V,\kappa_V)$ be a DVR of mixed characteristic $(0,p)$ such that $\kappa_V$ is perfect. Let $R$ be a finite type $V$-algebra extension domain. Suppose $R$ is regular and $\nu$ is a $\bR$-valued valuation of $\Frac(R)$ that is centered at $\frp \in \Spec(R)$ such that $\frp \cap V = \fram_V$. Suppose $\frp$ is an Abhyankar center of $\nu$ and $\Gamma_\nu/\nu(\Frac(V)^\times)$ is torsion-free. Let $\ba_\bullet = \{\ba_m \colon m \in \bR\}$ be the graded family of valuation ideals of $R$ associated with $\nu$. Then 
\[\bigcap_{m \in \bR} (\ba_m \colon \tau(R,m\cdot\ba_\bullet)) \neq 0.\]
\end{proposition}

\begin{proof}
Using \autoref{prop:birational-domination-regular-local} and a simple spreading out argument, there exists a finite type birational morphism of affine schemes
\[\pi \colon \Spec(S) \to \Spec(R)\]
such that $S$ is a regular domain and $\nu$ is centered at an Abhyankar point $\frq \in \Spec(S)$ that lies over $\frp \in \Spec(R)$ such that $S_\frq$ satisfies properties \autoref{prop:birational-domination-regular-local.i}-\autoref{prop:birational-domination-regular-local.iv} of \autoref{prop:birational-domination-regular-local}. We let $\bb_\bullet$ denote the corresponding graded family of valuation ideals of $S$ associated with $\nu$. Note that for all $m \in \bR$, 
\begin{equation}
\label{eq:valuation-ideals-contract}
\bb_m \cap R = \ba_m.
\end{equation}
Thus, for all $t \in \bQ_{\geq 0}$
\begin{equation}
\label{eq:trivial-containment}
 \tau(S,\bb_m^t) \stackrel{\autoref{thm.FinalPropertiesOfTestIdealsTriples}\autoref{thm.FinalPropertiesOfTestIdealsTriples.BiggerIdeals-Exponents}}{\supseteq} \tau(S,(\ba_mS)^t).
\end{equation}

\begin{claim}
\label{claim:birational}
If $\bigcap_{m \in \bR} (\bb_m \colon \tau(S,m\cdot\bb_\bullet)) \neq 0$, then $\bigcap_{m \in \bR} (\ba_m \colon \tau(R,m\cdot\ba_\bullet)) \neq 0$.
\end{claim}

\textbf{Proof of \autoref{claim:birational}.}
Choose canonical divisors $K_R$ of $\Spec(R)$ and $K_S$ of $\Spec(S)$ that agree on the locus where $\pi$ is an isomorphism. Let $K_{S/R} = K_S - \pi^*(K_R)$ denote the relative canonical divisor, which is Cartier in our setting. By the birational transformation property (\autoref{thm.FinalPropertiesOfTestIdealsTriples}\autoref{thm.FinalPropertiesOfTestIdealsTriples.Birational}), for all $m \in \bR$, $t \in \bQ_{\geq 0}$ we have 
\begin{equation}
\label{eq:birational-containment}
\pi_*\utau(\Spec(S),-K_{S/R},\widetilde{\ba_mS}^t) \supseteq \utau(\Spec(R),\widetilde{\ba_m}^t).
\end{equation}
The definition of test ideals for triples (\autoref{rem:def.GlobalyTauOmegaXNonPrincipal.Triple}) and   \autoref{prop.PropertiesOfUltTauOmegaForPairs}\autoref{prop.PropertiesOfUltTauOmegaForPairs.EasyPrincipalSkoda} imply
\[\sO_{\Spec(S)}(K_{S/R}) \otimes_{\sO_{\Spec(S)}} \tau\big{(}\Spec(S),\widetilde{\ba_mS}^t\big{)} =  \tau(\Spec(S),-K_{S/R},\widetilde{\ba_mS}^t).
\] 
Thus, using \autoref{eq:trivial-containment}, and taking global sections in \autoref{eq:birational-containment}, we then get 
\begin{equation}
\label{eq:birational-containment-2}
\left(\omega_{S/R}\cdot\tau(S,\bb_m^t)\right) \cap R \supseteq \left(\omega_{S/R}\cdot\tau(S,(\ba_mS)^t)\right) \cap R \supseteq \tau(R,\ba_m^t),
\end{equation}
where $\omega_{S/R} \coloneqq \Gamma(\Spec(S),\sO_{\Spec(S)}(-K_{S/R}))$. Here we are identifying $\omega_{S/R} \otimes_S \tau(S,(\ba_mS)^t)$ with the fractional ideal $\omega_{S/R}\cdot\tau(S,(\ba_mS)^t)$ using the flatness of $\omega_{S/R}$. For a given $m \in \bR$, choose a highly divisible integer $\ell \gg 0$ such that 
\[\tau(R,m\cdot\ba_\bullet) = \tau(R,\ba^{1/\ell}_{m\ell}) \text{ and } \tau(S,m\cdot\bb_\bullet) = \tau(S,\bb^{1/\ell}_{m\ell}).\]
Then by \autoref{eq:birational-containment-2}, for all $m \in \bR$, we get
\begin{equation}
\label{eq:asymptotic-birational-containment}
\left(\omega_{S/R}\cdot\tau(S,m\cdot\bb_\bullet)\right) \cap R \supseteq \tau(R,m\cdot\ba_\bullet).
\end{equation}
Let $\mathfrak{J}$ be the non-zero ideal $\bigcap_{m \in \bR} (\bb_m \colon \tau(S,m\cdot\bb_\bullet))$ of $S$. Then $\omega_{S/R}^{-1}\cdot\mathfrak{J}$ is a non-zero fractional ideal of $S$, and hence, $(\omega_{S/R}^{-1}\cdot\mathfrak{J})\cap R$ is a non-zero ideal of $R$. Hence for all $m \in \bR$, we get 
\begin{align*}
((\omega_{S/R}^{-1}\cdot\mathfrak{J})\cap R) \cdot \tau(R,m\cdot\ba_\bullet) &\stackrel{\autoref{eq:asymptotic-birational-containment}}{\subseteq} \left((\omega_{S/R}^{-1}\cdot\mathfrak{J}) \cdot (\omega_{S/R}\cdot\tau(S,m\cdot\bb_\bullet))\right) \cap R\\
&\stackrel{\autoref{eq:valuation-ideals-contract}}{\subseteq} \bb_m \cap R = \ba_m.
\end{align*}
Thus, $(\omega_{S/R}^{-1}\cdot\mathfrak{J})\cap R \subseteq \bigcap_{m \in \bR} (\ba_m \colon \tau(R,m\cdot\ba_\bullet))$, proving \autoref{claim:birational}.

Let $x_1,\dots,x_d \in S$ such that their images in $S_\frq$ form a regular system of parameters of $S_\frq$ with the property that $\{\nu(x_1),\dots,\nu(x_d)\}$ is a $\bZ$-basis of $\Gamma_\nu$. We finish the proof by showing that 
\[x_1\cdots x_d \in \bigcap_{m \in \bR} (\bb_m \colon \tau(S,m\cdot\bb_\bullet)).\]
This is clear if $m \leq 0$ so suppose $m > 0$. Let $A \coloneqq \widehat{S_\frq}$ be the $\frq S_\frq$-adic completion of $S_\frq$. Since $\bb_m$ is $\frq$-primary, we have $\bb_m = \bb_mA \cap S$. Thus, it suffices to show that 
\[x_1\cdots x_d \in (\bb_mA\colon \tau(S,m\cdot\bb_\bullet)A).\] 
By \autoref{lem:valuation-ideals-primary}, $\bb_mS_\frq$ is the valuation ideal of $S_\frq$ of order $m$ along $\nu$, and by our choice of $S$ and $x_1,\dots,x_d$, we have 
\[\bb_mS_\frq = \left(x_1^{a_1}\cdots x_d^{a_d} \colon \sum_{i=1}^d a_i\nu(x_i) \geq m, a_i \in \mathbb{Z}_{\geq 0}\right)\]
by \autoref{lem:valuation-ideals-monomial}. Thus, $\bb_mA$ is also generated by the monomials $x_1^{a_1}\cdots x_d^{a_d}$ such that $\sum_{i=1}^d a_i\nu(x_i) \geq m$. Now, choose an integer $\ell \gg 0$ such that $\tau(S,m\cdot\bb_\bullet) = \tau(S,\bb^{1/\ell}_{m\ell})$. Since $\frq \in \Spec(S)_{p=0}$, by \autoref{cor.FinalCompletionGlobalizationLocalizationForTestIdeal} we get
\[\tau(S,m\cdot\bb_\bullet)A = \tau(A,(\bb_{m\ell} A)^{1/\ell}),\]
where the right hand side is the complete local universal test ideal from \autoref{def:common-tau-nonprincipal}. Then by \autoref{prop.HowaldsTheoremForTau} we have
\[\tau(A,(\bb_{m\ell} A)^{1/\ell}) = \big(x_1^{b_1}\cdots x_d^{b_d} \colon (b_1+1,\dots,b_d+1) \in \mathrm{Int}((1/\ell)P) \cap \bZ^d\big),\]
where $\mathrm{Int}((1/\ell)P)$ is the relative interior of $(1/\ell)P$, where $P$ is the convex hull of the subset of $\bZ^d$ consisting of those $(a_1,\dots,a_d) \in \bZ^d$ such that $x_1^{a_1}\cdots x_d^{a_d} \in \bb_{m\ell}A$. But $x_1^{a_1}\cdots x_d^{a_d} \in \bb_{m\ell}A$ if and only if $x_1^{a_1}\cdots x_d^{a_d} \in \bb_{m\ell}S_\frq$, which is the valuation ideal of $S_\frq$ of order $m\ell$ along $\nu$. Thus, $(1/\ell)P$ is the convex hull of 
\[\bigg\{\left(\frac{a_1}{\ell},\dots,\frac{a_d}{\ell}\right) \colon a_i \in \bZ_{\geq 0}, \sum_{i=1}^a \frac{a_i}{\ell}\nu(x_i) \geq m\bigg\}.\]
Then clearly $\sum_{i=1}^d (b_i+1)\nu(x_i) \geq m$, or equivalently, $(x_1\cdots x_d)x_1^{b_1}\cdots x_d^{b_d} \in \bb_mA$. Thus, $x_1\cdots x_d \in (\bb_m A \colon \tau(A,(\bb_{m\ell} A)^{1/\ell})) = (\bb_m A \colon \tau(S,m\cdot\bb_\bullet)A)$, as desired.
\end{proof}

\subsection{An Izumi theorem} As a consequence of \autoref{thm:uniform-approximation-Abhyankar-valuation-ideals}, one obtains an Izumi theorem for two $\bR$-valued valuations sharing a common center when the center is Abhyankar for one of the two valuations. A more general version of the result we are about to state holds by \cite{SpivakovskyRondIzumi}.

\begin{theorem}
\label{thm:Izumi}
Let $(V,\fram_V,\kappa_V)$ be a DVR of mixed characteristic $(0,p)$ with perfect residue field $\kappa_V$. Let $(R,\fram_R,\kappa_R)$ be an essentially of finite type extension domain of $V$ such that $\fram_R \cap V = \fram_V$. Let $\nu, w$ be two $\bR$-valued valuations of $\Frac(R)$ centered on $R$. Assume that
\begin{enumerate}[(i)]
\item $R$ is an Abhyankar center of $\nu$,
\item $\Gamma_\nu/\nu(\Frac(V)^\times)$ is torsion-free, and
\item there exists a local essentially of finite type birational extension
\[(R,\fram_R,\kappa_R) \hookrightarrow (A,\fram_A,\kappa_A)\]
such that $A$ is regular and $\nu, w$ are both also centered on $A$.
\end{enumerate}
Then there exists a real number $C > 0$ such that for all $x \in R - \{0\}$, $\nu(x) \leq Cw(x)$.
\end{theorem}

\begin{proof}
Note that $V$ is an Abhyankar center of the restriction of both $\nu$ to $\Frac(V)$. Thus, because $V$ is universally catenary, $A$ is also an Abhyankar center of $\nu$ by \autoref{prop:Abhyankar-valuation-versus-center}. If we can show the existence of a constant $C$ for the center $A$, then the same $C$ will also work for $R$. Thus, replacing $R$ by $A$ we may assume $R$ is regular. Let 
\[\delta \coloneqq \inf\{w(x) \colon x \in \fram_R -\{0\}\}.\]
Note that $\delta > 0$ because if $\fram_R$ is generated by $x_1,\dots,x_d$, then $\delta = \min\{w(x_1),\dots,w(x_d)\}$.
It is also clear that for all $\ell \in \bZ_{\geq 0}$, $\fram_R^\ell \subseteq \ba_{w,\delta\ell}(R)$.

By \autoref{thm:uniform-approximation-Abhyankar-valuation-ideals}, choose a real number $e > 0$ such that for all $m \in \bR$ and $\ell \in \bZ_{\geq 0}$, 
\[\ba_{\nu,m\ell}(R) \subseteq \ba_{\nu,m-e}(R)^\ell.\]
Fix any $m_0 > e$. Then for all $\ell \in \bZ_{\geq 0}$, we have
\begin{equation}
\label{eq:Izumi-containments}
\ba_{\nu,m_0\ell} \subseteq \ba_{m_0-e}(R)^\ell \subseteq \fram_R^\ell \subseteq \ba_{w,\delta\ell}(R).
\end{equation}
We claim that 
\[C \coloneqq 2m_0/\delta\]
will work. Indeed, suppose for contradiction that there exists $x_0 \in \fram_R -\{0\}$ such that $\nu(x_0) > Cw(x_0)$. Choose $\ell \in \bZ_{\geq 0}$ such that $\delta(\ell-1)\leq w(x_0) < \delta\ell$. Note that $\ell \geq 2$ since $\delta\leq w(x_0)$. Thus, multiplying this chain of inequalities by $C$, we get 
\[m_0\ell \leq 2m_0(\ell-1) \leq Cw(x_0) < \nu(x_0).\]
Then $x_0 \notin \ba_{w,\delta\ell}(R)$ but $x_0 \in \ba_{\nu,m_0\ell}(R)$, which contradicts \autoref{eq:Izumi-containments}.
\end{proof}



\newpage

\bibliographystyle{skalpha}
\bibliography{MainBib}















\end{document}